\documentclass{NP-IV}
\usepackage[latin1]{inputenc}
\usepackage[T1]{fontenc}
\usepackage{pgf}
\usepackage{amsmath}
\usepackage{amscd}
\usepackage{mathrsfs}
\usepackage[all]{xy}
\usepackage{graphicx}
\usepackage{pstricks} 
\usepackage{amsfonts}   
\usepackage{pict2e}   
\usepackage{makeidx}
\usepackage[colorlinks=true,linkcolor=red,citecolor=red]{hyperref}

\usepackage{amssymb}
\usepackage{latexsym}
\usepackage{bold-extra}

\renewcommand{\b}{\mathcal{B}}
\newcommand{\Bs}{\ensuremath{\mathscr{B}}}
\newcommand{\bs}[1]{\boldsymbol{#1}}

\newcommand{\chidr}{\chi^{\phantom{1}}_{\mathrm{dR}}}

\newcommand{\GS}{\mathfrak{G}}

\newcommand{\pres}{\Psi}
\newcommand{\RR}{\bs{\mathcal{R}}}

\renewcommand{\H}{\mathscr{H}}

\newcommand{\Hdr}{\mathrm{H}_{\mathrm{dR}}}
\newcommand{\hdr}{\mathrm{h}_{\mathrm{dR}}}

\newcommand{\M}{\mathrm{M}}

\newcommand{\lr}[1]{\langle#1\rangle}

\newcommand{\N}{\mathrm{N}}

\newcommand{\Q}{\mathrm{Q}}
\renewcommand{\O}{\mathscr{O}}

\renewcommand{\P}{\mathrm{P}}
\newcommand{\Prod}{\mathcal{P}}
\newcommand{\R}{\mathcal{R}}

\newcommand{\SF}{{S_{\Fs}}}

\newcommand{\sm}[1]{\begin{smallmatrix}#1\end{smallmatrix}}
\newcommand{\is}[1]{i_{#1}^{\mathrm{sol}}}

\newcommand{\simto}{\stackrel{\sim}{\to}}

\newcommand{\U}{\bs{\mathrm{U}}}

\newcommand{\wt}[1]{#1}

\DeclareMathOperator\rk{rank}

\newcommand\alg{\mathrm{alg}}

\newcommand\wti{\widetilde}

\newcommand{\eps}{\varepsilon}

\newcommand{\Cs}{\mathscr{C}}

\newcommand\Fs{\mathscr{F}}
\newcommand\Gs{\mathscr{G}}

\newcommand\Ts{\mathscr{T}}

\newcommand\Rc{\mathcal{R}}

\newcommand\NN{\mathbb{N}}
\newcommand\ERRE{\mathbb{R}}

\newcommand{\E}[2]{\ensuremath{\mathbb{A}^{#1,\mathrm{an}}_{#2}}}

\newcommand\wc{{\mkern 2mu\cdot\mkern 2mu}}
\newcommand\va{|\wc|}

\newcommand\wKa{\widehat{K^\alg}}

\newcommand{\comment}[1]{
\noindent{\red \framebox{\framebox{\framebox{
\begin{minipage}{450pt}#1
\end{minipage}}}}}
}

\newcommand{\smallcomment}[1]{
{\red \framebox{\framebox{\framebox{
#1}}}}
}

\newcommand{\comm}[1]{
\noindent{\magenta \framebox{\framebox{\framebox{
\begin{minipage}{450pt}#1
\end{minipage}}}}}
}

\usepackage{amsthm}
\swapnumbers

\makeatletter
\def\swappedhead#1#2#3{%
  \thmname{#1}\;%
  \thmnumber{\@upn{\the\thm@headfont#2\@ifnotempty{#1}}}%
  \thmnote{\,{\the\thm@notefont(#3)}}{.~}}
\makeatother

\newtheoremstyle{dotless-thm}
  {10pt}
  {10pt}
  {\itshape}
  {}
  {\bfseries}
  {}
  {.0em}
  {}
\theoremstyle{dotless-thm}

\newtheorem{theorem}{\textbf{Theorem}}[subsection]
\newtheorem{thm-intro}{\textbf{\textsc{Theorem}}}
\newtheorem{rk-intro}[thm-intro]{\textbf{\textsc{Remark}}}
\newtheorem{cor-intro}[thm-intro]{\textbf{\textsc{Corollary}}}
\newtheorem{proposition}[theorem]{\textbf{Proposition}}
\newtheorem{lemma}[theorem]{\textbf{Lemma}}
\newtheorem{corollary}[theorem]{\textbf{Corollary}}
\newtheorem{definition}[theorem]{\textbf{Definition}}
\newtheorem{remark}[theorem]{\textbf{Remark}}

\newtheorem{hypothesis}[theorem]{\textbf{Hypothesis}}
\newtheorem{setting}[theorem]{\textbf{Setting}}

\newtheorem{notation}[theorem]{\textbf{Notation}}
\numberwithin{equation}{section}

\title[Convergence Newton 
polygons III : global decomposition]{
The convergence Newton polygon of a $p$-adic 
differential equation III : global decompositions}

\author{J\'er\^ome Poineau}
\email{jerome.poineau@unicaen.fr}
\address{Laboratoire de mathématiques 
Nicolas Oresme
Université de Caen, BP 5186,
F-14032 Caen Cedex}

\author{Andrea Pulita}
\email{andrea.pulita@univ-grenoble-alpes.fr}
\address{Univ. Grenoble Alpes, CNRS, IF, 38000 Grenoble, France.}

\date{\today}



\begin{abstract}
In this paper we deal with locally free 
$\O_X$-modules with connection over a 
Berkovich curve $X$ and the radii of convergence  
of their solutions as functions on $X$.
As a main result, we prove local and global 
decomposition theorems of such objects according to 
the radii of 
convergence of their solutions. 
\end{abstract}

\begin{document}
\maketitle

\begin{center}
Version of \today
\end{center}

\makeatletter
\renewcommand\tableofcontents{%
    \subsection*{\contentsname}%
    \@starttoc{toc}%
    }
\makeatother

\begin{small}
\setcounter{tocdepth}{3} \tableofcontents
\end{small}

\setcounter{section}{0}

\section*{\textsc{Introduction}}
\addcontentsline{toc}{section}{\textsc{Introduction}}

Following an original idea of Bernard Dwork, if a differential 
equation has two 
solutions with different radii of convergence, 
then it should correspond to a decomposition of the 
equation. If the equation is given in a cyclic basis by 
a differential operator, this decomposition corresponds to 
a factorization of the operator 
into two operators of lower order. 
Decomposition theorems provide an insight about the 
nature of indecomposable objects and  
constitute a central tool in the 
classification of 
$p$-adic differential equations.  
As an example, they are a crucial step to 
obtain the $p$-adic local monodromy theorem \cite{An}, 
\cite{Ked}, \cite{Me}. 
On the other hand, they will also be one of the key 
ingredients needed to establish local and global index 
results \cite{NP-V, NP-VI}, generalizing \cite{Ch-Me-IV}. 

The main contributions to the decomposition results 
are due to Robba \cite{Ro-I}, \cite{Robba-Hensel-original},
\cite{Robba-Hensel},
Dwork-Robba \cite{Dw-Robba}, 
Christol-Mebkhout \cite{Ch-Me-III}, \cite{Ch-Me-IV}, and Kedlaya 
\cite{Kedlaya-book-2}, \cite{Kedlaya-draft}. 
We also recall \cite[p. 97-107]{Correspondance-Malgrange-Ramis}, 
\cite{Levelt}, and \cite{Ramis-Devissage-Gevrey} 
for the decomposition by the 
\emph{formal} slopes of a differential equation 
over the field of power series $K((T))$, which we prove 
to be indeed a decomposition by the radii of convergence 
over a trivially valued base field and, in this sense, they can be seen as a consequence of probably prior results by \cite{Ro-I} (cf. Section \ref{Formal differential equations}). Finally we 
quote \cite{Andre-Slope-Filtration} offering a Tannakian 
approach to this problem for categories admitting 
a notion of \emph{slopes filtrations} satisfying some 
formal properties. 

If our differential equation is defined on a curve, the radii 
of convergence of the solutions can change from a point 
to another of the curve, therefore when we speak about 
the radii we mean the \emph{function defined on the 
whole curve} associating to a point of the curve the 
values of the radii of the solutions at that point.  
In all the quoted contexts, the decomposition results are 
of local nature and the decompositions are obtained over 
small domains of the curve, where the radii are all affine 
functions, so that their value at a point 
or their slopes are the only essential datum of the function.
There indeed are very few examples of global nature 
of such decomposition theorems by the radii 
(mainly on annuli or disks), and all with restrictive 
assumptions.\footnote{Probably, the most important 
example of global nature is given in 
\cite[Ch.12]{Kedlaya-book-2}, 
where a global decomposition is obtained 
over open disks for differential equations 
whose radii turn out to be globally constant (cf. also 
\cite[Theorem 5.4.2]{Kedlaya-draft} for a finer result).} 
A large part of the literature is indeed 
devoted to the following two cases: 
differential equations defined over a germ of punctured 
disk, or over the Robba ring. In the language of Berkovich 
curves this corresponds respectively to a germ of 
segment out of a rational point, or a germ of segment 
out of a point of type $2$ or $3$. 

We present here a general 
decomposition theorem of a global nature in the framework of 
Berkovich quasi-smooth curves, that works without any technical assumption 
(solvability, exponents, Frobenius, harmonicity, de Rham cohomology, ...).

A theoretical point of a crucial importance is the definition itself 
of the radii of convergence. 
The former definition of radii relates them to 
the spectral norm of the connection. 
This (partially) fails in Berkovich geometry,  
because there are solutions converging more than the 
natural bound prescribed by the spectral norm. In 
particular, the spectral definition does not exists at 
points of type $1$. 
Here, we deal with a more geometric definition of the 
radii due to F. Baldassarri in \cite{Balda-Inventiones}, following the ideas of \cite{Dv-Balda}. 
Baldassarri-Di Vizio improve the former definition by introducing 
\emph{over-solvable radii}, 
i.e. radii that are larger than the spectral bound (cf. 
\eqref{eq : spectral - solvable - oversolvable}).  
These radii are not intelligible in terms of spectral norm even if the 
point is of type $2$, $3$, or $4$. 
Moreover they ``\emph{normalize}'' the usual spectral radii of 
convergence, with respect to a semi-stable formal model of the curve. 
They are also able to prove the continuity of the 
smallest radius. 
In \cite{NP-II} we introduce in that 
picture the notion of \emph{weak triangulation} as a 
substitute of Baldassarri's semi-stable model (cf. \cite{Balda-Inventiones}). 
In fact 
such a semi-stable 
model produces a triangulation 
(see \cite{Duc} for instance).\footnote{The notion of weak triangulation will evolve 
in our next papers 
into the notion of \emph{pseudo-triangulation}, 
which permits to better describe the idea of 
finiteness of a curve (cf. 
\cite{NP-IV}). 
Namely, finite curves are those 
admitting a finite pseudo-triangulation and a finite 
number of connected components, and we will show that 
they are characterized by the fact that their de Rham 
cohomology is finite dimensional (cf. \cite{NP-VI}).}
This notion permits to generalize the definition of the 
radii to a larger class of curves, and 
it has the advantage for us of being completely within 
the framework of Berkovich curves. 
Exploiting this point of view, we have proved in 
\cite{NP-I} and \cite{NP-II} that there exists a 
locally finite graph outside which all 
the radii are  locally constant, 
as firstly conjectured by Baldassarri (cf. \cite[Introduction]{Balda-Inventiones}).  
This proves that there are 
relatively  few numerical invariants of the equation encoded in the 
radii of convergence. The finiteness theorem have been 
proved in \cite{NP-II} by recreating the notion of \emph{generic disk} 
in the framework of Berkovich curves, as in the very original point of 
view of Bernard Dwork and Philippe Robba. 
The global decomposition theorem presented here
enlarges the picture, it makes evident that Baldassarri-Di Vizio's idea for the 
definition of the radii is the good one, and gives to it a more operative 
meaning.\\
\if{
Here we add another result to the picture by providing a global 
decomposition theorem that makes evident that Baldassarri's 
setting is the good one, and gives to it a more operative meaning.
}\fi

\if{
With this in mind in this paper we obtain our global 
decomposition by \emph{augmenting} and then \emph{gluing} a local 
decomposition over the local ring $\O_{X,x}$ by the radii. 
Such a decomposition generalizes that obtained by 
Dwork and Robba in a Berkovich neighborhood of a point of the affine 
line \cite{Dw-Robba}. \\
}\fi

We now come more specifically into the content of the paper. Let 
$(K,|.|)$ be a complete ultrametric valued field of characteristic $0$. 
Let $X$ be a quasi-smooth $K$-analytic curve, in the sense of 
Berkovich theory\footnote{Quasi-smooth means that $\mathrm{dim}_{\H(x)} \Omega^1_X\otimes\H(x)
= \mathrm{dim}_x X$ (cf. \cite[3.1.11]{Duc}). 
This notion
corresponds to the notion called ``rig-smooth'' in the 
rigid analytic setting.}. We assume, 
without loss of generality, that $X$ is connected. 
Let $\Fs$ be a locally free $\O_X$-module of finite rank $r$ 
endowed with an integrable connection $\nabla$. 

In \cite{NP-II} we  
explained how to associate to each point $x\in X$ the so called  
\emph{convergence Newton polygon of $\Fs$ at $x$}. Its 
slopes are the logarithms of the radii of 
convergence $\R_{S,1}(x,\Fs)\leq\cdots\leq\R_{S,r}(x,\Fs)$ of a 
conveniently chosen basis of solutions of $\Fs$ at $x$ (cf. Section \ref{multiradius}). 
Here $S$ is a weak-triangulation.

Following \cite{NP-I}, we then define for all 
$i\in\{1,\ldots,r=\mathrm{rank}(\Fs)\}$ a 
\emph{locally finite} graph $\Gamma_{S,i}(\Fs)$, as the locus of 
points that do not admit a virtual open disk in $X-S$ 
as a neighborhood on which 
$\R_{S,i}(-,\Fs)$ is constant (cf. Section 
\ref{section controlling graphs}). 

We say that the index $i$ separates 
the radii (globally over $X$) if, for all $x\in X$, one has 
\begin{equation}
\R_{S,i-1}(x,\Fs)\;<\;\R_{S,i}(x,\Fs)\;.
\end{equation}
More suggestively, this means that the 
convergence Newton polygon of $\Fs$ at $x$ 
has a break over $i$. The following theorem asserts that, 
if this break maintains over all the points of $X$, then the 
differential equation itself decomposes.
\begin{thm-intro}[cf. Theorem \ref{MAIN Theorem} and Proposition 
\ref{Prop. : independence on S}]
\label{Intro : Thm : Main}
If the index $i$ separates the radii of $\Fs$, 
then there exists a sub-differential 
equation
$(\Fs_{\geq i},\nabla_{\geq i})\subseteq(\Fs,\nabla)$ of rank 
$r-i+1$ together with an exact sequence
\begin{equation}\label{eq : sequence intro}
0\to\Fs_{\geq i}\to\Fs\to\Fs_{<i}\to 0
\end{equation}
such that for all $x\in X$ one has 
\begin{equation}
\R_{S,j}(x,\Fs)\;=\;\left\{\begin{array}{lcl}
\R_{S,j}(x,\Fs_{<i})&\textrm{ if }&j=1,\ldots,i-1\bigskip\\
\R_{S,j-i+1}(x,\Fs_{\geq i})&\textrm{ if }&j=i,\ldots,r\;.\\
\end{array} \right.
\end{equation}
Moreover $\Fs_{\geq i}$ is independent on $S$, in the sense that
if $i$ separates the radii of $\Fs$ with respect to another weak-triangulation $S'$, then the resulting sub-object $\Fs_{\geq i}$ is the 
same.
\end{thm-intro}
The decomposition results  
in the quoted literature 
are all \emph{direct sum} decompositions of local 
nature. Up to our knowledge, 
there are no examples of non-direct 
decompositions.
However, in the global setting the above decomposition 
can be non-direct. 
In Section \ref{An explicit counterexample.} 
we provide an explicit example where 
\eqref{eq : sequence intro} does not split, and in Section \ref{Conditions to have a direct sum decomposition} 
we provide criteria to guarantee that $\Fs_{\geq i}$ 
is a direct summand of $\Fs$. More precisely we have the following
\begin{thm-intro}[cf. Theorems 
\ref{Thm : 5.7 deco good direct siummand} and 
\ref{Thm : criterion self direct sum}]
\label{Intro : Thm : direct sum}
In each one of the following two situations $\Fs_{\geq i}$ is a direct 
summand of $\Fs$:
\begin{enumerate}
\item The index $i$ separates the radii of $\Fs^*$ and of $\Fs$, and  
$(X,S)$ is either different from a virtual open disk with empty 
triangulation, or, if $X$ is a virtual open disk $D$ with empty 
triangulation, there exists at least one point $x\in D$ such that 
one of $\R_{\emptyset,i-1}(x,\Fs)$ or $\R_{\emptyset,i-1}(x,\Fs^*)$ 
is spectral non-solvable at $x$.

\item One has $\Bigl(\Gamma_{S,1}(\Fs)\cup\cdots\cup
\Gamma_{S,i-1}(\Fs)\Bigr)\subseteq\Gamma_{S,i}(\Fs)$.
\end{enumerate}
In both cases $(\Fs^*)_{\geq i}$ is isomorphic to 
$(\Fs_{\geq i})^*$, and it is 
also a direct summand of $\Fs^*$.
\end{thm-intro}
In our next paper 
\cite[Section 2]{NP-IV} 
we will provide an operative description of 
the controlling graphs, which can be used in 
particular to check the assumptions of 
Theorem \ref{Intro : Thm : direct sum}. 

It is worth pointing out the fact that the above theorems 
establish an link between \emph{algebraic} 
properties of the connection and \emph{topological} 
properties of the radii of the 
solutions, such as the relative position of their controlling 
graphs.
\if{$\Fs_{\geq i}$ is a direct 
summand of $\Fs$ is a completely algebraic 
notion that is controlled by the condition ii) of Theorem 
\ref{Intro : Thm : direct sum} which is completely 
topological.}\fi 
Moreover, 
condition ii) of Thm. \ref{Intro : Thm : direct sum}  
implies i), and it has the advantage that it involves 
only the radii of $\Fs$. Nevertheless i) is more natural, 
and general, by the following reason. If one is allowed 
to choose arbitrarily the pseudo-triangulation, 
then in order to guarantee the existence of 
$\Fs_{\geq i}$ one has to choose it as small as 
possible, and the same is true for condition i). 
On the other hand, to fulfill ii) one is induced to choose 
it quite large. 
For more precise statements see Remarks 
\ref{Remark : changing tr F_S,I} and 
\ref{Remark : bad condition}.
%
%

An interesting Corollary is the following:

\begin{thm-intro}[cf. \protect{Corollary \ref{Prop : Crossing points shsh}}]\label{cor-intro : filtration intro}
Let $\Fs$ be a differential equation over $X$. 
There exists a locally finite subset $\mathfrak{F}$ of $X$ such that, if 
$Y$ is a connected component of $X-\mathfrak{F}$, then:
\begin{enumerate}
\item For all $i< j$ one has either 
$\R_{S,i}(y,\Fs)=\R_{S,j}(y,\Fs)$ for all $y\in Y$, or 
$\R_{S,i}(y,\Fs)<\R_{S,j}(y,\Fs)$ for all $y\in Y$. 
\item Let $1=i_1<i_2<\ldots<i_h$ be the indexes separating the 
global radii $\{\R_{S,i}(-,\Fs)\}_i$ over $Y$. Then one has a filtration
\begin{equation}\label{eq: intro: - filtration}
0\;\neq\;(\Fs_{|Y})_{\geq i_h}\;\subset\;
(\Fs_{|Y})_{\geq i_{h-1}}\;\subset\;
\cdots\;\subset\;(\Fs_{|Y})_{\geq i_1}\;=\;\Fs_{|Y}\;,
\end{equation}
such that the rank of $(\Fs_{|Y})_{\geq i_k}$ is $r-i_k+1$ and its 
solutions at each point of $Y$ are the solutions of $\Fs$ with radius 
larger than $\R_{S,i_k}(-,\Fs)$.
\end{enumerate}
\end{thm-intro}
Note that we do not endow $Y$ with a 
weak-triangulation and that the radii of the Theorem 
\ref{cor-intro : filtration intro} are 
those of $\Fs$ viewed as an equation over $X$ endowed with $S$. 
In section \ref{Crossing points and global decomposition.} 
we also provide conditions making \eqref{eq: intro: - filtration} 
a graduation.

A direct corollary of the above results is the 
Christol-Mebkhout 
decomposition over the Robba ring \cite{Ch-Me-III} 
(cf. Section \ref{Some interesting particular cases (annuli).}). 
Moreover, our global decomposition result 
will be applied to differential 
equations over elliptic curves to prove that they always 
globally decompose,  because their radii are constant functions \cite[Corollary 3.3.12]{NP-IV}.

\if{
is the following 
classification result:

\begin{thm-intro}[cf. Cor. \ref{Cor Ell curves}]
Assume that $X$ is either 
a Tate curve, 
or that $p\neq 2$ and $X$ is an elliptic curve with good reduction. 
Consider a pseudo-triangulation $S$ of $X$ formed by an individual point. 
Let $\Fs$ be a differential equation over $X$ of rank $r$. Then
\begin{enumerate}
\item For all $i=1,\ldots,r$ the radius $\R_{S,i}(-,\Fs)$ 
is a constant function on $X$, and $\Gamma_{S,i}(\Fs)=\Gamma_S$;
\item One has a direct sum decomposition as 
\begin{equation}
\Fs\;=\;\bigoplus_{0<\rho\leq 1}\Fs^{\rho}\;,
\end{equation}
where $\R_{S,j}(-,\Fs^{\rho})=\rho$ for all $j=1,\ldots,
\mathrm{rank}\,\Fs^{\rho}$.
\end{enumerate}
\end{thm-intro}

}\fi

The proof of the existence of $\Fs_{\geq i}$ 
(cf. Theorem \ref{MAIN Theorem}) is obtained as follows. 
Firstly, in section \ref{Robba-deco}, we prove a local decomposition
theorem by the spectral radii for differential modules over 
the field $\H(x)$, of a point $x$ of type $2$, $3$, or $4$ of a 
Berkovich curve. This generalizes to curves the classical  
Robba's decomposition theorem, originally proved 
for points of type $2$ of the affine line (cf. \cite{Ro-I}).
Then, in section \ref{Dwork-Robba (local) decomposition}, 
we descends that decomposition to $\O_{X,x}\subseteq \H(x)$, i.e. to a neighborhood of $x$ in $X$. 
This is a generalization to curves of  Dwork-Robba's 
decomposition result, originally proved for points of 
type $2$ of the affine line (cf. \cite{Dw-Robba}). 
The language of generic disks introduced in \cite{NP-II} 
permits to extend smoothly these proofs to curves, 
up to minor implementations.
Such local decompositions are also present in 
\cite{Kedlaya-draft}, 
where one makes everywhere a systematic use of the spectral norm 
of the connection, following \cite{Kedlaya-book-2}. 
Methods involving spectral norms work thanks to the Hensel 
factorization theorem of \cite{Robba-Hensel}, 
and \cite[Lemme 1.4]{Ch-Dw}. 
We presents here the same local decomposition results 
as a consequence of the original, and slight more 
geometric, techniques of 
\cite{Ro-I} and \cite{Dw-Robba}.

Now, Robba's and Dwork-Robba's local decompositions 
(over $\H(x)$ and $\O_{X,x}$ respectively) take into 
account only spectral radii (at the points of type $2$, 
$3$, and $4$). One of the reasons is that 
over-solvable radii are truncated by the localization 
process (cf. Section \ref{Local nature of spectral radii.}). 
A consequence of this truncation is that Dwork-Robba's 
local decompositions do not glue (and no global decomposition is given 
in literature, as we already mentioned).
In order to deal with this issue, we \emph{augment} the Dwork-Robba decomposition 
of the stalk $\Fs_x$, for all $x\in X$, by taking in account the decomposition 
of the trivial submodule of $\Fs_x$ 
coming from the existence of solutions converging in a disk 
containing $x$, i.e. taking in account over-solvable radii. 
This augmented decomposition of $\Fs_x$ 
glues without any obstructions, 
and it provides the existence of $\Fs_{\geq i}$. 
In this gluing process, 
Theorem \ref{Intro : Thm : Main} actually 
uses in a crucial way the continuity of the radii 
$\R_{S,i}(-,\Fs)$ (cf. proof of Proposition 
\ref{Prop : extended by continuity}). 
Implicitly, we also use the 
local finiteness of $\Gamma_{S,i}(\Fs)$, 
since for $i\geq 2$ the continuity is an indirect 
consequence of the finiteness (cf. \cite{NP-I} and \cite{NP-II}).

\textbf{Structure of the paper.}
In Section \ref{Secn1} we give some basic notations about differential modules.
In Section \ref{Radii and filtrations by the radii} 
we define the radii and their graphs together with 
their elementary properties.
In Section \ref{Robba-deco} we give Robba's local decomposition 
by the spectral radii over $\H(x)$, for a point of type $2$, $3$, or $4$. 
In Section \ref{Dwork-Robba (local) decomposition} 
we descend that decomposition to the local ring 
$\O_{X,x}\subseteq\H(x)$ following Dwork-Robba's original 
techniques.
In Section \ref{Proof of Main Th} we obtain the global decomposition 
Theorem \ref{MAIN Theorem}, and the criteria 
to have a direct sum decomposition (cf. Theorems 
\ref{Thm : 5.7 deco good direct siummand} and 
\ref{Thm : criterion self direct sum}).
In Section \ref{An explicit counterexample.} we provide explicit 
counterexamples of the basic pathologies of over-solvable radii 
(incompatibility with duality, and exact sequence, a link between 
super-harmonicity property and presence of Liouville numbers, by 
means of the Grothendieck-Ogg-Shafarevich formula). 
In Appendix \ref{Comparison with the general definition of radii} we
discuss the definition of the radius.

\textbf{Acknowledgments.}
We wish to thank Yves Andr\'e, Francesco Baldassarri, Gilles Christol, 
Kiran S. Kedlaya, Adriano Marmora, Zoghman Mebkhout, and 
Nobuo Tsuzuki for helpful comments.
\vspace{1cm}
\section{Definitions and notations}\label{Secn1}
Let $(K,\va)$ be a ultrametric complete valued field of characteristic $0$. Denote by $\wti K :=\{x\in K,|x|\leq 1\}/\{x\in K,|x|< 1\}$ its residue field and by~$p$ be the characteristic exponent of the latter (either~1 or a prime number). Let~$K^\alg$ be an algebraic closure of~$K$. The absolute value~$\va$ on~$K$ extends uniquely to it. We denote by $(\wKa,\va)$ its completion.
By a complete valued field extension of $K$ we mean an isometric embedding $K\to L$ of complete valued fields. 

\subsection{Affine line, disks and annuli.}
\label{Section : disks and annulii}
Let~$\E{1}{K}$ be the Berkovich affine line with coordinate~$T$. 
Let~$L$ be a complete valued extension of~$K$, \emph{i.e.} a ring morphism from~$K$ to a complete valued field~$L$ that is an isometry. 
Let $c\in L$. We set 
\begin{eqnarray}
D_{L}^+(c,R) &\;=\;& 
\big\{x\in \E{1}{L}\, \big|\, |(T-c)(x)|\le R\big\}\;,\quad\qquad\qquad 
R\geq 0\\
D_{L}^-(c,R) &\;=\;& \big\{x\in \E{1}{L}\, \big|\, |(T-c)(x)|<R\big\}
\;,\qquad\qquad \quad R>0\label{eq : open disk}\\
C_{L}^+(c;R_{1},R_{2}) &\;=\;& 
\big\{x\in \E{1}{L}\, \big|\, R_{1}\le |(T-c)(x)|\le R_{2}\big\}\;, 
\qquad R_2\geq R_1\geq 0\\
C_{L}^-(c;R_{1},R_{2}) &\;=\;& \big\{x\in \E{1}{L}\, \big|\, 
R_{1} < |(T-c)(x)| < R_{2}\big\}\;.\qquad R_2>R_1> 0
\label{eq : punctured disk is an annulus}
\end{eqnarray} 
\begin{definition}\label{def:modulusannulus}
Let $C$ be an open or closed annulus. 
If $C= C^\pm(c;R_{1},R_{2})$ in some coordinate, we define the \emph{modulus} of~$C$ to be
$\mathrm{Mod}(C) := \frac{R_{2}}{R_{1}} \in [1,+\infty].$
It is independent of the chosen coordinate (cf. \cite[3.6.23.2]{Duc}). A similar definition can be given if the annulus is semi-open.
\end{definition}

If $D\subseteq \E{1}{L}$ is an open disk, 
we denote by $\O(D)$ (resp. $\b(D)$) the ring of 
\emph{analytic} (resp. \emph{bounded analytic}) functions on $D$. If $D =D_L^{-}(c,R)$ is open, then 
\begin{eqnarray}
\O(D)&\;:=\;&\Bigl\{\sum_{n\geq 0}a_n(T-c)^n,
\; a_n\in L, \; \lim_{n}|
a_n|\rho^n=0,\;\forall\;\rho<R\Bigr\}\;,\label{eq :O(D)}\\
\b(D)&\;:=\;&\Bigl\{\sum_{n\geq 0}a_n(T-c)^n,
\;a_n\in L,\;
 \sup_{n}|a_n|R^n<+\infty\Bigr\}\;.\label{eq :B(D)}
\end{eqnarray}
We write $\O_L(D)$ and $\b_L(D)$ instead of 
$\O(D)$ and $\b(D)$ respectively when we want to specify the base 
field. 

The ring $\b_L(D)$ is a Banach algebra with respect to the sup-norm 
$\|.\|_D$ on $D$. If $D=D_L^-(c,R)$, this 
norm will  be denoted by 
$|.|_{c,R}$. For $f\in\b(D_L^-(c,R))$, one has 
\begin{equation}\label{eq : sup-norm}
|f|_{c,R}=
\sup_{n\geq 0} \Bigl(\Bigl|\frac{1}{n!}\Bigl(\Bigl(\frac{d}{dT}\Bigr)^n(f)\Bigr)(c)\Bigr|_L\cdot R^n\Bigr).
\end{equation}
This norm turns out to be multiplicative. Hence, it defines 
a point of the Berkovich line 
$\mathbb{A}^{1,\mathrm{an}}_K$. We will denote it 
by $x_{c,R}$. 

More generally, for every complete valued field extension 
$L/K$ and every pair~$(t,R)$, with $t\in L$ and $R\geq 0$, the map that sends a polynomial $f\in K[T]$ to 
\begin{equation}\label{norm crho}
|f|_{t,R}\;:=\;
\sup_{n\geq 0} \Bigl(\Bigl|\frac{1}{n!}\Bigl(\Bigl(\frac{d}{dT}\Bigr)^n(f)\Bigr)(t)\Bigr|_L\cdot R^n\Bigr).
\end{equation} 
defines a point of 
$\mathbb{A}^{1,\mathrm{an}}_K$. We will denote it by $x_{t,R}$.

Let $c\in K$. The ring $\O(D^-_K(c,R))$ is endowed with 
the family of norms $\{x_{c,\rho}\}_{\rho<R}$. It is then a Fr\'echet space.

\begin{definition}\label{Radius of a point.} 
The \emph{radius} of a point $x\in\mathbb{A}^{1,\mathrm{an}}_{K}$ is the real number
\begin{equation}\label{eq : radius of a point jtq}
r(x)\;:=\;
\max\Bigl\{r\geq 0\textrm{ such that there exist $L/K$ and $t\in L$ satisfying }x=x_{t,r}\Bigr\}
\end{equation} 
(see also~\cite[4.2]{Ber}).
\end{definition}
Remark that the radius~$r(x)$ depends on the coordinate~$T$. It 
is invariant by algebraic extensions of~$K$, but not by 
arbitrary extensions. We refer to  \cite[Section 1]{NP-I} for additional properties of $r(x)$. This radius coincides with the infimum of the radii 
of the virtual open disks containing~$x$, however the 
above definition will be more meaningful over general 
curves.

\subsubsection{Bounded derivations.}
\label{Section: Bounded Omega 1}
We maintain the above notations. 
Let $L/K$ be a complete valued extension of $K$. Let $T$ be a coordinate on the affine line $\mathbb{A}^{1,\mathrm{an}}_L$.
Let $t\in L$ and $\rho>0$ and 
consider the $L$-rational open disk $D:=D^-(t,\rho)$.
The structure of the 
Berkovich's spectrum of $(\b_L(D),\|.\|_{D})$ is 
actually richer than expected (cf. \cite{Esct}). 
It seems natural to expect that 
the $\b_L(D)$-module 
of continuous differentials 
$\widehat{\Omega}^1_{\b_L(D)/L}$ is not one-dimensional, though no proofs exist to 
our knowledge. This is related to the fact that the 
polynomials are not dense in $(\b_L(D),\|.\|_D)$ 
(cf. proof of Proposition \ref{Prop: cont der on b} below).

\begin{hypothesis}
In the sequel of this paper, we will 
always consider derivations $d$ on $\b_L(D)$ of 
the form $f(T)\frac{d}{dT}$, where $T$ is an $L$-rational coordinate 
of $D$ and $f$ is invertible in $\b_L(D)$.
\end{hypothesis}

Endow $\O_L(D)$ with the normoid structure defined by the 
family of norms $\{|.|_{t,\rho'}\}_{\rho'<\rho}$.
In the framework of normoid spaces it makes sense to 
speak of bounded endomorphisms of $\O(D)$ 
(cf. \cite{Banachoid}). Namely, 
$\varphi:\O_L(D)\to \O_L(D)$ is bounded if, 
for all $\rho'<\rho$, there exists $\rho''<\rho$ 
and a real constant $C>0$ such that $|\varphi(f)|_{t,\rho'}\leq C|f|_{t,\rho''}$, for all $f\in\O_L(D)$.

\begin{proposition}\label{Prop: cont der on b}
The following hold:
\begin{enumerate}
\item\label{i Prop: cont der on b} for all $\rho'\leq\rho$ and all $f\in\b_L(D)$ we 
have $|\frac{d}{dT}f|_{t,\rho'}\leq(\rho')^{-1}
|f|_{t,\rho'}$.
\item\label{ii Prop: cont der on b} Any bounded 
derivation $d$ of $\O_L(D)$ is of the 
form $f\frac{d}{dT}$, with $f\in\O_L(D)$. Moreover, if 
$d$ generates the space of bounded derivations of $\O_L(D)$, then $f$ 
is invertible and $f,f^{-1}\in\b_L(D)$. In particular, $\b_L(D)$ is then globally stabilized by $d$ and $d$ is a bounded endomorphism of $\b_L(D)$.
\item\label{iii Prop: cont der on b} Let $(A,\|.\|)$ be a normed ring and $A\to \b_L(D)$ 
be an injective ring homomorphism such that the two 
norms $\|.\|$ and $\|.\|_D$ are equivalent on $A$. Let 
$d:\b_L(D)\to\b_L(D)$ be a derivation of the form 
$f\frac{d}{dT}$, with $f,f^{-1}\in\b_L(D)$. If 
$d(A)\subseteq A$, then $d$ is bounded on $A$.
\end{enumerate}
\end{proposition}
\begin{proof}
All the statements are classical, 
see for instance \cite{Ch-Ro, Ch}. Let us prove just item 
\eqref{ii Prop: cont der on b}. 
By continuity, $d$ is determined by its values on 
the ring of polynomials $L[T-t]$, which is dense in 
$\O_L(D)$. By $L$-linearity and Leibniz 
rule, we have $d(\sum_{n\geq 0}a_n(T-t)^n)=
\sum_{n\geq 0}a_n\cdot n\cdot (T-t)^{n-1}d(T-t)$ and 
we see that $f=d(T-t)$. If $d$ generates the space of 
bounded derivations, then $f$ has to be invertible 
because $d/dT$ is also a generator of this space. 
The claim then follows from the fact that invertible 
functions are bounded.
\end{proof}

\subsection{Curves and weak triangulations}

In the rest of the text, we will work over a quasi-smooth 
$K$-analytic curve $X$
in the sense of Berkovich. For every $x\in X$ we denote 
by $\O_{X,x}$ (resp. $\mathscr{M}_{X,x}\subset\O_{X,x}$) the ring of 
analytic functions on an unspecified neighborhood of $x$ 
(resp. the ideal of $\O_{X,x}$ of functions that are 
zero at $x$), and by $\H(x)$ the completion of the 
valued field $\O_{X,x}/\mathscr{M}_{X,x}$ with respect to the absolute value induced by $x$.

Recall that 
its structure may be described using the semi-stable 
reduction theorem (see \cite{Semi-Stable-red-thm}) or, 
equivalently, Ducros's 
notion of triangulation (see \cite{Duc}).

\begin{definition}
A virtual open disk is a connected analytic space that becomes isomorphic to a disjoint union of open disks after base change to $\widehat{K^{\mathrm{alg}}}$.

A virtual open annulus is a connected analytic space that becomes isomorphic to a disjoint union of open annuli whose orientations are invariant under $\mathrm{Gal}(\wKa/K)$ after base change to $\widehat{K^{\mathrm{alg}}}$.

\end{definition}

As in \cite{NP-II}, we will use a slight generalization of triangulation that we recall here.

\begin{definition}
A subset $S$ of $X$ is said to be a \emph{weak triangulation} of $X$ if 
\begin{enumerate}
\item $S$ is locally finite and only contains points of type 2 or 3;
\item any connected component of $X - S$ is a virtual open disc or annulus.
\end{enumerate}
\end{definition}
Weak triangulations always exist \cite{NP-II}.
\begin{definition}\label{Def. skeleton}
The \emph{skeleton}~$\Gamma_{C}$ of a virtual open annulus~$C$ is the set of points that have no neighborhoods isomorphic to a virtual open disk. It is homeomorphic to an open segment.

The \emph{skeleton}~$\Gamma_{S}$ of a weak triangulation~$S$ of~$X$ is the union of $S$ and the skeletons of the connected components of $X - S$ that are virtual annuli. It is a locally finite graph.
\end{definition}

Let us finally introduce special neighborhoods of points. 

\begin{definition}\label{def:starshapedopen}
A connected open neighborhood $U$ of a point $x\in X$ will be 
called \emph{star-shaped} if
\begin{enumerate}
\item\label{def:starshapedopen-i} $U$ is a virtual open disk containing $x$, 
if $x$ is of type $1$ or $4$;
\item\label{def:starshapedopen-ii} $U$ is an open subset of $X$ such 
that $\{x\}$ is a weak triangulation of 
$U$, if $x$ is of type $2$ or $3$.
\end{enumerate}

In case \eqref{def:starshapedopen-i}, we define the \emph{canonical weak triangulation}~$S_{U}$ of $U$ to be the empty set, and 
its \emph{pointed skeleton}~$\Gamma_{S_{U}}$ to be 
the open segment connecting $x$ to the boundary of the 
disk $U$.

In case \eqref{def:starshapedopen-ii}, we define the 
\emph{canonical weak triangulation} of~$U$ to be $S_{U} : =\{x\}$ 
and its \emph{pointed skeleton} to be $\Gamma_{S_U}$.
\end{definition}

\begin{definition}\label{Def : star-shaped}
Let $x\in X$. Let~$Y$ be an 
affinoid neighborhood of~$x$ in~$X$. Denote by~$\partial Y$ its boundary in the sense of Berkovich (cf. \cite[Definition~2.5.7]{Ber}). 

We say that~$Y$ is a \emph{star-shaped} affinoid neighborhood of~$x$ in~$X$ if the connected component $U$ of $Y-\partial Y$ containing $x$ is a 
star-shaped open neighborhood of $x$ in $X$, 
and the other connected 
components of $Y-\partial Y$ are all virtual open disks. 

We define the \emph{canonical weak triangulation} of $Y$ to be $S_{Y} := S_{U} \cup \partial Y$, and its \emph{pointed skeleton} to be $\Gamma_{S_{Y}} := \Gamma_{S_{U}} \cup \partial Y$.
\end{definition}

%
%
%
\begin{notation} \label{Notation branch - germ - direction}
Let $x \in X$. Since the curve~$X$ has a tree structure, it makes sense to speak of segments in~$X$, hence of germs of segments out of~$x$, defined as equivalent classes. We define the degree $\mathrm{deg}(b)$ of a germ of segment $b$ of $X$ to be the number of germs of segments in $X_{\wKa}$ over~$b$.


If $x$ is of type~2 or~3, then any germ of segment~$b$ out of~$x$ may be represented by the skeleton of a virtual annulus, hence admits a canonical metric. As a result, for any function $F \colon X \to \ERRE$, the notion of slope~$\partial_bF$ of~$F$ along~$b$ makes sense. For $F$ sufficiently regular, we may also define a notion of laplacian 
\begin{equation}
dd^c F:=
\sum_{b}\mathrm{deg}(b)\partial_b F,
\end{equation}
where $b$ runs through the set of germs of segments out of~$x$.

We refer to \cite[Section 1.5]{NP-IV} for the precise definitions.

\end{notation}

\subsection{Differential  modules and trivial submodules} 
\label{section : Differential  modules and trivial submodules}

Let $A$ be a differential ring, \textit{i.e.} a commutative ring endowed with a non-zero derivation $d:A\to A$.
\begin{remark}\label{Remark : bidual}
We recall that if $\M$ is a finitely generated $A$-module, 
then $\M$ is a projective $A$-module 
if, and only if, $\M$ is locally free over $A$ and its rank 
function is locally constant over $Spec(A)$ with respect 
to the Zariski topology (cf. \cite[II, \S 5.2, Theorem 1]{BouAlgCom}).
In this case,  $\M$ is canonically isomorphic to its bi-dual $\M^{**}$ (cf. \cite[Chapter II, \S 2, n.8, Corollary 4]{Bou-Alg-1-3}).
\end{remark}

\begin{definition}\label{Def.diff.mod}
A \emph{differential module} $(\M,\nabla)$ over $(A,d)$ 
is a finitely generated projective $A$-module
together with a connection $\nabla:\M\to\M$, i.e. a 
$\mathbb{Z}$-linear map satisfying the Leibniz rule 
$\nabla(am)=d(a)m+a\nabla(m)$ for all $a\in A$, 
$m\in\M$. Morphisms of differential modules are 
$A$-linear maps commuting with the connections.
\end{definition}
Denote by $A\langle d\rangle$ the Weyl algebra of 
differential polynomials. As an additive group one has 
$A\langle d\rangle=\oplus_{n\geq 0}A\circ d^n$, and 
the multiplication $\circ$ is the unique one satisfying 
$a\circ d=d\circ a+ d(a)$, for all $a\in A$. 
A differential module over $A$ is naturally an 
$A\langle d\rangle$-module, where the multiplication by $d$ in $\M$ is the map $\nabla$. Reciprocally, 
an $A\langle d\rangle$-module which is 
finitely generated and projective as $A$-module is a differential module. 
Morphisms between differential modules coincide with $A\langle d\rangle$-linear 
morphisms. The $A$-linear algebra constructions on $\M$ such as $-\otimes_A-$ or $\mathrm{Hom}_A(-,-)$ 
are naturally 
$A\langle d\rangle$-modules 
(see for instance \cite[Section 5.3]{Kedlaya-book-2}). As an example the dual module 
$\M^*:=\mathrm{Hom}_A(\M,A)$ is endowed with the 
connection  
$(\nabla^*(\alpha))(m)=\alpha\circ\nabla-d\circ\alpha$, where $\alpha:\M\to A$; 
and $\M\otimes_A\N$ is endowed with the connection 
\begin{equation}\label{eq : tensor product connection}
\nabla'\;:=\;\nabla\otimes\mathrm{Id}_{\N}+\mathrm{Id}_{\M}\otimes\nabla\;.
\end{equation}
For every $A\langle d\rangle$-modules $\M,\N$ we set
\begin{eqnarray}
\Hdr^0(\M,\N)&\;:=\;&\mathrm{Ker}(\nabla':
\M\otimes_A\N\to\M\otimes_A\N)\;,\\
\Hdr^1(\M,\N)&\;:=\;&\mathrm{Coker}(\nabla':
\M\otimes_A\N\to\M\otimes_A\N)\;.
\end{eqnarray}

Notice that $\Hdr^i(\M,\N)\cong \Hdr^i(\N,\M)$.
We often write $\Hdr^i(\M):=\Hdr^i(\M,A)$, if no 
confusion is possible. 
The following lemmas are some of the reasons why 
we mainly deal with 
$A\langle d\rangle$-modules that are  
projective as $A$-modules.
\begin{lemma}\label{Lemma : H^0 exact}
Let $E:0\to\N\to\M\to\Q\to 0$ be an  exact sequence of 
$A\langle d\rangle$-modules and $\mathrm{P}$ an $A\langle d\rangle$-module. If $\mathrm{Tor}_1^A(\mathrm{P},\Q)=0$, we have a long exact sequence 
\begin{equation}\label{eq: long exact H^i}
0\to\Hdr^0(\N,\mathrm{P})\to\Hdr^0(\M,\mathrm{P})\to\Hdr^0(\Q,\mathrm{P})\to
\Hdr^1(\N,\mathrm{P})\to\Hdr^1(\M,\mathrm{P})\to\Hdr^1(\Q,\mathrm{P})\to0.
\end{equation}
In particular, if $\mathrm{P}$ is a flat $A$-module, the functor 
$\N\mapsto\Hdr^0(\mathrm{P},\N)\cong\Hdr^0(\N,\mathrm{P})$ is left exact, while $\N\mapsto\Hdr^1(\mathrm{P},\N)\cong\Hdr^1(\N,\mathrm{P})$ is right exact.
\end{lemma}
\begin{proof}
Since $\mathrm{Tor}^1_A(\mathrm{P},\Q)=0$, the sequence 
$E\otimes_AP$ remains exact. Then, the snake lemma 
applied to the diagram $\nabla:E\otimes_A\mathrm{P}
\to 
E\otimes_A\mathrm{P}$ gives the desired sequence. 
\end{proof}
\begin{lemma}\label{Lemma: H^0(M,-) left exact}
Let  $\M$ be a differential module over $A$ (cf. Definition \ref{Def.diff.mod}), 
then for every $A\langle d\rangle$-module $\N$
we have a functorial isomorphism
\begin{equation}\label{eq : omega = hom}
\Hdr^0(\M,\N)\;\cong\;\mathrm{Hom}_{A\langle d\rangle}(\M^*,\N)\;.
\end{equation}
\end{lemma}
\begin{proof}
We have $\M\cong\M^{**}$. Moreover, $\M^*$ is projective of finite type too and we have 
$\mathrm{Hom}_A(\M^*,\N)\cong\M^{**}\otimes_A\N
\cong\M\otimes_A\N$ 
(cf. \cite[Remark 5.3.4]{Kedlaya-book-2}). Hence 
$\mathrm{Hom}_{A\langle d\rangle}(\M^*,\N)=\Hdr^0(\mathrm{Hom}_A(\M^*,\N))=
\Hdr^0(\M\otimes_A\N)=\Hdr^0(\M,\N)$.
\end{proof}

Let $(A',d')$ be another ring with derivation, and let 
$A\to A'$ be a ring morphism commuting with the 
derivations. In this situation, we say that $A'$ is a 
differential ring over $A$. 
Sometimes, with an abuse, we will denote 
$d'$ by the same symbol $d$.
Let $\M$ be an $A\lr{d}$-module. The scalar extension of 
$\M$ is the $A'\langle d'\rangle$-module 
$A'\langle d'\rangle\otimes_{A\langle d\rangle}\M$ and 
it coincides with $\M\otimes_AA'$ endowed 
with the connection 
$\nabla':=\nabla\otimes\mathrm{Id}_{A'}+
\mathrm{Id}_{\M}\otimes d'$. 
A \emph{solution} of $\M$ with values 
in $A'$ is an element in the kernel of 
$\nabla'$:
\begin{equation}\label{eq : convention on solutions-}
\Hdr^0(\M,A')
\;:=\;\mathrm{Ker}(\nabla':
\M\otimes_AA'\to\M\otimes_AA')\;.
\end{equation}
The kernel of the derivation 
$\Hdr^0(A',A')=A'^{d'=0}:=\mathrm{Ker}(d':A'\to A')$ is 
 a sub-ring of $A'$. The group $\Hdr^0(\M,A')$ 
is naturally an sub-$A'^{d'=0}$-module of $\M$, 
and the rule $\M\mapsto \Hdr^0(\M,A')$ is a covariant 
functor. 
If $A'^{d'=0}$ is a field and if the dimensions are finite, 
we set 
$\hdr^i(\M,A'):=\mathrm{dim}_{A'^{d'=0}}\Hdr^i(\M,A')$ 
and 
\begin{equation}
\chidr(\M,A')\;:=\;\hdr^0(\M,A')-\hdr^1(\M,A')\;.
\end{equation}

Let  $\M$ be a $A\lr{d}$-module. 
The $A^{d=0}$-module $\Hdr^0(\M,A)$ can be seen as an 
$A^{d=0}\langle d\rangle$-module, where the 
multiplication by $d$ is given by the zero map. In this case, we have $A\langle d\rangle\otimes_{A^{d=0}\langle d\rangle}\Hdr^0(\M,A)\cong A\otimes_{A^{d=0}}\Hdr^0(\M,A)$, where $A\otimes_{A^{d=0}}\Hdr^0(\M,A)$ is endowed with the connection $\nabla(a\otimes v):=d(a)\otimes v$. 
The canonical arrow
\begin{equation}\label{eq : j_M}
j_{\M}\;:\;A\otimes_{A^{d=0}}\Hdr^0(\M,A)\longrightarrow\M \;,
\end{equation}
defined by $j_{\M}(a\otimes v):=a\cdot v\in\M$, is a 
morphism of $A\langle d\rangle$-modules which is 
functorial in $\M$. 
We denote by 
\begin{equation}\label{def : M_A}
\M_A
\end{equation} 
the image of $j_{\M}$.
If $f:\M\to\N$ is a morphism of $A\langle d\rangle$-modules, then $f(\M_A)\subseteq\N_A$ 
and the rule $\M\mapsto\M_A$ is a functor.
If $\N\to\M\to\P$ is exact, then $\N_A\to\M_A\to\P_A$ 
is a complex (i.e. the image of $\N_A$ in $\M_A$ lies in 
the kernel of $\M_A\to\P_A$).

\begin{proposition}\label{Proposition : M_A minimal}
Let $\M$ be an $A\lr{d}$-module. Then $\M_A$ is the smallest sub-$A\lr{d}$-module of $\M$ satisfying 
\begin{equation}\label{eq : Hdr0(M_A,A)=Hdr0(M,A)}
\Hdr^0(\M_A,A)\; =\; \Hdr^0(\M,A)\;.
\end{equation}  
In particular $(\M_A)_A=\M_A$.
\end{proposition}
\begin{proof}
Let $\N\subset \M$ be a submodule such that the natural injection (cf. Lemma \ref{Lemma : H^0 exact}) $\Hdr^0(\N,A)\simto \Hdr^0(\M,A)$ is an equality. By functoriality we have an isomorphism 
$\Hdr^0(\N,A)\otimes_{A^{d=0}} A\simto \Hdr^0(\M,A)\otimes_{A^{d=0}} A$ and a morphism $i:\N_A\to\M_A$ commuting with the canonical morphisms $j_{\N}$ and $j_{\M}$. Since $j_{\M}$ is surjective onto $\M_A$, 
then $i$ is surjective too. On the other hand, the composite map $\N_A\to\N\to\M$ is injective, and hence $i$ is injective too. It follows that 
$\M_A=\N_A\subset\N$, which proves that $\M_A$ is minimum as required.
\end{proof}
\begin{definition}
We say that the $A\lr{d}$-module $\M$ is \emph{trivial} 
if $j_{\M}$ is an isomorphism. 
We say that $\M$ is \emph{trivialized by $A'$} if $\M\otimes_AA'$ is 
trivial as a differential module over $A'$. 
\end{definition}
\begin{remark}\label{Remark : Rank H^0 < rank M}
If $\Hdr^0(\M,A)$ is a free $A^{d=0}$-module $\Hdr(\M,A)\cong (A^{d=0})^I$, 
then $\M$ is trivial if, and only if, it is isomorphic, as 
$A\lr{d}$-module, 
to $A^I$ with the connection 
$\nabla((a_i)_{i\in I})=(d(a_i))_{i\in I}$. 
If $\M$ is moreover a differential module, the set $I$ is finite, because in this case $\M$ is of finite type by definition. 
\end{remark}
\begin{remark}\label{rk : trivial implies surjective nabla}
Assume that the derivation $d:A\to A$ is a surjective map. Then for 
every trivial $A\lr{d}$-module $\M$ the connection $\nabla:\M\to\M$ 
is surjective too. Indeed, so is the connection 
$d\otimes \mathrm{Id}$ of $A\otimes_{A^{d=0}}\Hdr^0(\M,A)$.
\end{remark}
\begin{lemma}[\protect{\cite[Lemmas 5.3.3 and 5.3.4]{Kedlaya-book-2}}]
\label{Lemma : Ext^1=H^1}
Let $\M$ and $\N$ be two $A\lr{d}$-modules.
Denote by $\mathrm{Ext}^1(\M,\N)$ the Yoneda extension group of 
exact sequences $0\to\N\to\P\to\M\to 0$. If $\M$ is 
a differential module (cf. Definition \ref{Def.diff.mod}), then 
$\mathrm{Ext}^1(\M,\N)\simto \Hdr^1(\M^*\otimes\N,A)$.\hfill$\Box$
\end{lemma}

\begin{lemma}\label{Lemma :devisage}
Let $E:0\to\N\to\M\to\Q\to 0$ be an exact sequence of 
$A\lr{d}$-modules. Let $A'$ be a differential ring over $A$ and assume that $\Q$ as an $A$-module satisfies $\mathrm{Tor}_1^A(\Q,A')=0$. Then, the following hold.
\begin{enumerate}
\item\label{i Lemma :devisage} The sequence $\Hdr^0(E,A')$ is left exact, 
moreover 
$\Hdr^0(\N,A')=\Hdr^0(\M,A')\cap(\N\otimes_AA')$. 

\item\label{ii Lemma :devisage} If the connection $\nabla':\N\otimes_AA'\to\N\otimes_AA'$ is surjective (i.e. $\Hdr^1(\N,A')=0$), 
then the sequence $\Hdr^0(E,A')$ is also right exact. In particular, this 
holds if the derivation $d:A'\to A'$ is surjective 
and $\N$ is trivialized by $A'$ (cf. Remark \ref{rk : trivial implies surjective nabla}). 

\item\label{iii Lemma :devisage} 
Assume that the derivation 
$d:A'\to A'$ is surjective, and that 
$\N$ and $\Q$ are both trivialized by $A'$. 
Then $\M$ is trivialized by $A'$ too. 
\end{enumerate}
\end{lemma}
\begin{proof}
Let us first assume $A=A'$. Item \eqref{i Lemma :devisage} follows easily from \eqref{eq: long exact H^i}.

\eqref{ii Lemma :devisage} If $\nabla:\N\to\N$ is 
surjective, we have $\Hdr^1(\N)=0$ and the claim follows again from \eqref{eq: long exact H^i}. 
If $d:A\to A$ is surjective and $\N$ is trivial, 
the connection $\nabla:\N\to\N$ is 
surjective by Remark \ref{rk : trivial implies surjective nabla}. 

\eqref{iii Lemma :devisage} By item ii) and the right exactness of tensor product, the sequence $A\otimes_{A^{d=0}}\Hdr^0(E,A)$ is right exact. Hence, the five lemma applied to the diagram
$j_{E}:A\otimes_{A^{d'=0}}\Hdr^0(E,A)\to E$ 
implies that the middle map $j_{\M}$ is an isomorphism, and $\M$ is trivial. 
%

Now, if $A\neq A'$, then the condition  
$\mathrm{Tor}_1^A(\Q,A')=0$ guarantee that 
$E\otimes_AA'$ is exact. Therefore, 
all items can be proved after base change to $A'$ and 
without loss of generality we may assume $A=A'$.
\end{proof}

\if{
\begin{remark}
Let $A'$ be a differential ring over $A$. 
With the notation of Lemma \ref{Lemma :devisage}, if 
$\mathrm{Tor}_1^A(\Q,A')=0$, then $E\otimes_AA'$ is 
exact, and hence $\Hdr^0(E,A')$ remains left exact by \eqref{eq: long exact H^i}. 
Under this condition the statements i), ii) and iii) of 
Lemma \ref{Lemma :devisage}
hold replacing the word ``$\nabla:\N\to\N$ is epi'' by 
``$\nabla':\N\otimes A'\to\N\otimes A'$ is epi''; ``\emph{trivial}'' by ``\emph{trivialized by $A'$}''; 
and ``\emph{$d$ is surjective}'' by 
``\emph{$d'$ is surjective}''.
\end{remark}
}\fi

Recall that a derivation $d$ on an integral domain $A$ extends 
uniquely to a derivation on its fraction field $F(A)$. We will abuse 
notation and also denote it by the same symbol $d$. Recall that $F(A)^{d=0}$ is a field.
\begin{lemma}\label{Lemma : Trivial iff solutions}
Assume that $A$ is an integral domain such that 
$A^{d=0}=F(A)^{d=0}$. Then 
\begin{enumerate}
\item\label{i Lemma : Trivial iff solutions} 
If $\M$ is a $A\langle d\rangle$-module  which is flat 
over $A$, then $j_{\M}$ is injective.
\item\label{ii Lemma : Trivial iff solutions}  
Let $0\to\N\to\M\to\Q\to 0$ be an exact sequence of 
$A\lr{d}$-modules such that $\M$ is trivial and 
$\Q$ is flat as $A$-module. 
Then  $\N$ and $\Q$ are both trivial differential modules. 
In particular, $\M$, $\N$ and $\Q$ are all free 
$A$-modules, the sequence splits and 
$\M\cong\N\oplus\Q$ as $A\langle d\rangle$-modules.
\item\label{iii Lemma : Trivial iff solutions} Assume that the derivation $d:A\to A$ 
is a surjective map.
Let $0\to\N\to\M\to\Q\to 0$ be an exact sequence 
of $A\lr{d}$-modules such that $\N$ and $\Q$ are trivial. 
Then, $\M$ is trivial. 
\end{enumerate}
\end{lemma}
\begin{proof}
If $A=F(A)$, then item \eqref{i Lemma : Trivial iff solutions} is classical (cf. 
\cite[Lemma 5.1.5]{Kedlaya-book-2}). 
If $A\neq F(A)$, we may deduce the injectivity of $j_{\M}$ from the 
injectivity of $j_{\M\otimes_A F(A)}$ as follows.
By functoriality of $j_{\M}$, we have the commutative diagram
\begin{equation}
\xymatrix{
\M\ar[rr]&&\M\otimes_AF(A)\\
\Hdr^0(\M,A)\otimes_{A^{d=0}}A\ar[r]^-{a}\ar[u]^{j_{\M}}&
\Hdr^0(\M,A)\otimes_{A^{d=0}}F(A)\ar[r]^-{b}&
\Hdr^0(\M\otimes_{A}F(A))\otimes_{F(A)^{d=0}}F(A)\ar[u]_{j_{\M\otimes_AF(A)}}\;.}
\end{equation}
It is then enough to prove that 
the composite horizontal map in the bottom 
is injective. Since $\M$ is flat as $A$-module,  
$\Hdr^0(\M,-)$ is left exact by Lemma 
\ref{Lemma : H^0 exact}. 
It follows that we have an inclusion
$\Hdr^0(\M,A)\subseteq\Hdr^0(\M,F(A))
=\Hdr^0(\M\otimes_AF(A))$. Since 
$A^{d=0}=F(A)^{d=0}$ is a field, all the modules are 
flat over $A^{d=0}$ and both the maps 
$a$ and $b$ are injective.
\if{
We deduce that the maps induced by scalar extension
$\Hdr^0(\M,A)\otimes_{A^{d=0}} A\to 
\Hdr^0(\M,A)\otimes_{A^{d=0}} F(A)$ and
$\Hdr^0(\M,A)\otimes_{A^{d=0}} F(A)\to 
\Hdr^0(\M\otimes_AF(A),F(A))\otimes_{A^{d=0}} F(A)$ are injective. 
So the composite map 
$\Hdr^0(\M,A)\otimes_{A^{d=0}} A\to 
\Hdr^0(\M\otimes_AF(A),F(A))\otimes_{A^{d=0}} F(A)$ is injective. 
}\fi

\eqref{ii Lemma : Trivial iff solutions} Denote by 
$E:0\to\N\stackrel{i}{\to}\M\stackrel{p}{\to}\Q\to 0$ 
the exact sequence, 
we consider the morphism of sequences 
$j_{E}:\Hdr^0(E,A)\otimes_{A^{d=0}}A\to E$. 
We know that $\Hdr^0(E,A)$ is left exact (cf. Lemma 
\ref{Lemma : H^0 exact}) and that
$\Hdr^0(E,A)\otimes_{A^{d=0}}A$ remains left exact 
because $A^{d=0}$ is a field and $A$ is flat over 
$A^{d=0}$.
From the surjectivity of $p\circ j_{\M}$ we 
deduce the surjectivity of $j_{\Q}$. 
By \eqref{i Lemma : Trivial iff solutions} one has its injectivity, hence $\Q$ is trivial.
Moreover, this also shows that the map $\Hdr^0(p,A)\otimes\mathrm{Id}_A:\Hdr^0(\M,A)\otimes_{A^{d=0}}A\to\Hdr^0(\Q,A)\otimes_{A^{d=0}}A$ is surjective, because $p$ is, and both $j_{\M}$ and $j_{\Q}$ are isomorphisms.
Then, the snake lemma applied to the diagram $j_{E}:\Hdr^0(E,A)\otimes_{A^{d=0}}A\to E$ shows that 
$\N$ is trivial. Since $A^{d=0}$ is a field, the 
$\Hdr^0(E,A)$ is a sequence of free 
$A^{d=0}$-modules which splits. One sees that 
$\N,\M,\Q$ are free $A$-modules by Remark 
\ref{Remark : Rank H^0 < rank M} and that, for all of 
them, we can consider bases in their $\Hdr^0(-)$. 
The claim follows.

\eqref{iii Lemma : Trivial iff solutions}  
Since $\N$ is trivial, and $d:A\to A$ is surjective, we 
have $\Hdr^1(\N,A)=0$ 
(cf. Remark \ref{rk : trivial implies surjective nabla}). Since 
$-\otimes_{A^{d=0}}A$ is an exact functor, the 
sequence $\Hdr^0(E,A)\otimes_{A^{d=0}}A$ is exact.
Now, the snake lemma applied to the diagram $j_{E}:
\Hdr^0(E,A)\otimes_{A^{d=0}}A\to E$ implies that 
$\M$ is trivial.
\end{proof}

\begin{corollary}\label{Corollary : M_A unique}
Assume that $A$ is an integral domain such that $A^{d=0}=F(A)^{d=0}$. If $\M$ is a differential module over $A$, then
\begin{enumerate}
\item\label{Corollary : M_A unique item i} As an $A$-module, 
$\M_A$  is free of finite type and its rank coincide with the dimension of 
$\Hdr^0(\M,A)$ over $A^{d=0}$:
\begin{equation}\label{eq : dimHdr0(M,A)=rank(M_A)}
\mathrm{dim}_{A^{d=0}}\Hdr^0(\M,A)\;=\;\mathrm{rank}_A(\M_A)\;.
\end{equation}
\item\label{Corollary : M_A unique item ii} If $\M/\M_A$ has no torsion, then  
$\M_A$ is the unique sub-differential module of $\M$ satisfying 
\eqref{eq : Hdr0(M_A,A)=Hdr0(M,A)} and \eqref{eq : dimHdr0(M,A)=rank(M_A)}. In particular this holds if $\M/\M_A$ is projective. 
\hfill$\Box$
\end{enumerate}
\end{corollary}
\begin{proof}
By Lemma \ref{Lemma : Trivial iff solutions} we have $\M_A\cong\Hdr^0(\M,A)\otimes_{A^{d=0}}A$. 
Item \eqref{Corollary : M_A unique item i} follows then immediately. 

Item \eqref{Corollary : M_A unique item ii} follows from Proposition \ref{Proposition : M_A minimal}. 
Indeed, let $\N\subset \M$ be a sub-differential module of $\M$ satisfying 
\eqref{eq : Hdr0(M_A,A)=Hdr0(M,A)} and \eqref{eq : dimHdr0(M,A)=rank(M_A)}. By minimality of $\M_A$ it follows that $\M_A\subset\N$ and by the fact that 
$j_{\N}$ is injective we deduce easily that $\M_A$ and $\N$ have the same rank.
The quotient $\N/\M_A$ is then a torsion submodule of $\M/\M_A$. 
Therefore $\N/\M_A=0$ and the claim follows.
\end{proof}
\if{
\begin{lemma}
Let $A\to A'$ be a morphism commuting with the derivations. Assume 
that all $A\lr{d}$-modules that are finitely presented as $A$-modules 
are also locally free as $A$-modules (i.e. differential modules).
\begin{enumerate}

\item There exists a largest differential sub-module 
$\M_{A'}\subseteq\M$, with the property that a 
differential sub-module of $\M$ is trivialized by $A'$ 
if, and only if, it is contained in $\M_{A'}$. 
In particular $\M_{A'}$ is trivialized by $A'$.

\item The rule $\M\mapsto \M_{A'}$ is a covariant 
left exact functor of the category of differential modules over $A$ into 
its full subcategory of differential modules trivialized by $A'$.
\if{
\item Assume moreover that $A$ has no proper non-zero ideal stable 
by $d$. Then a module $\M$ is trivial if, and only if, its rank equals 
$\mathrm{dim}_{A^{d=0}}\,\Hdr^0(\M,A)$.
}\fi
\end{enumerate}
\end{lemma}
\begin{proof}
......

......

iii) Firstly we notice that ii), together with point iii) of Lemma 
\ref{Lemma :devisage}, imply that the 
sum in $\M$ of two differential sub-modules that are 
trivialized by $A'$ is trivialized by $A'$, so $\M_{A'}$ can be defined as 
the union of all the sub-modules of $\M$ 
trivialized by $A'$. 
Reciprocally if $\N\subseteq \M_A$, then it is trivial by ii).

----

iv) The functor $\M\mapsto\Hdr^0(\M,A')$ is left exact by point i) of 
Lemma \ref{Lemma :devisage}. Let $E:0\to\N\to\M\to \P\to 0$ be an 
exact sequence of differential modules over $A$. Since 
$(A')^{d'=0}$ is a field, $\Hdr^0(E,A')$ is formed by finite 
dimensional vector spaces. Hence 
$\Hdr^0(E,A')\otimes_{(A')^{d'=0}}A'$ remains left exact.

PROBLEMA : Questo non mi dice che $\M\mapsto \M_{A'}$ è esatto, 
ma solo che $\M\mapsto \M_{A}$ è esatto ...
Se fosse $(A')^{d'=0}=A^{d=0}$ allora si potrei concludere ...

----

-----

-----

\end{proof}

\begin{remark}
The inclusion $\Hdr^0(\M_{A'},A')\subseteq 
\Hdr^0(\M,A')$ is rarely an equality. In other words the module 
$\M_{A'}$ does not  control the solutions of $\M$ with 
value in $A'$. Nevertheless if $A=A'$ it is an equality. 
Lemma \ref{Lemma : Trivial iff solutions} will be used in section  
\ref{Proof of Main Th} 
(cf. Prop. \ref{Prop : existence of (F_S,i)_x})
to obtain differential sub-modules of the trivial sub-module $\M_A$. 
In the following way :
if $\M$ comes by scalar extension from a differential 
module $\N$ over a sub-ring $B\subset A$, then 
$\N_B\otimes_{B}A\subseteq \M_{A}$, and $\Hdr^0(N_B,B)=\Hdr^0(N,B)$.
\end{remark}
}\fi

\begin{remark}\label{Remark : example of non direct summand}
The trivial sub-module $\M_A$ is possibly not 
a direct factor of $\M$. We 
provide explicit examples of this in section 
\ref{An explicit counterexample.}.
\end{remark}

\begin{remark}[Duality and trivial sub-modules]\label{Remark : iso trivial dual}
The duality endo-functor $\M\mapsto \M^*$ is an additive and exact
equivalence of the category of differential modules. 
Unfortunately, it is not compatible with the functor $\M\mapsto\M_A$. 
Namely there exists a canonical composite morphism
\begin{equation}
(\M^*)_A\;\to\;\M^*\;\to\;(\M_A)^{*}
\end{equation}
which is often not an isomorphism. 
The dimensions of $(\M_A)^*$ and $(\M^*)_A$ can be different 
(see the counterexamples of section \ref{An explicit counterexample.}).
However, even when the two dimensions are equal, there is no reason to 
have an isomorphism. This means 
that $\Hdr^0(\M^*,A)$ can not be identified to the dual of 
$\Hdr^0(\M,A)$. This is closely related to the assumption of Theorem 
\ref{Thm : 5.7 deco good direct siummand}. 
\end{remark}

\begin{remark}
The dual convention, which is often used in 
literature\footnote{This is indeed
the convention adopted in all the papers by Christol, Dwork, 
Mebkhout, Robba,\ldots and in the theory of $D$-modules.}, consists in defining the  
solutions of an $A\langle d\rangle$-module $\M$ 
with values in $A'$ as the morphisms 
$\mathrm{Hom}_{A\langle d\rangle}(\M,A')=
\mathrm{Hom}_{A'\langle d\rangle}(\M\otimes_AA',A')=\Hdr^0(\M^*,A')$. 
This convention has the privilege of enjoying 
automatically left exactness. 
However, as soon as we deal with differential modules (cf. Definition \ref{Def.diff.mod}),
duality is an equivalence and every statement admits a 
dual statement. 
\end{remark}

\subsubsection{Change of derivation.} 
\label{Sect : Change of derivation.}
Let $d_1,d_2:A\to A$ be two derivations satisfying $d_2=a d_1$ with 
$a$ invertible in $A$.
The two categories of differential modules with respect to $d_1$ and  
to $d_2$ are isomorphic as follows: 
to a $d_1$-module $(\M,\nabla_1)$ one associates the $d_2$-module 
$(\M,\nabla_2)$ with $\nabla_2=a\cdot \nabla_1$. Then an $A$-linear 
morphism commutes with $\nabla_1$ if, and only if, it commutes with 
$\nabla_2$. In other words the functor is the identity on the 
morphisms. 
A differential module is $d_1$-trivial if, and only if, it is $d_2$-trivial.

\subsubsection{Filtrations of cyclic modules, and factorization of 
operators.}\label{Filtrations of cyclic modules, and factorization of 
operators.}
We say that a differential module 
$\M$ is \emph{cyclic} if it is of the form 
$A\lr{d}/A\lr{d}P$, for some unitary 
differential polynomial $P\in A\langle d\rangle$. 
The latter module will be denoted by $\M_P$.
If $A$ is a field, any differential module is cyclic 
(cf. \cite[Ch.II, Lemme 1.3]{Deligne-Reg-Sing}, \cite{Katz-cyclic-vect}). 

If $P$ factorizes in $A\lr{d}$ as $P=P_1\cdot P_2$, then 
the right multiplication  $L\mapsto L\cdot P_2$ by $P_2$  identifies 
$A\lr{d}P_1$ to $A\lr{d}P$ and one has an exact sequence
\begin{equation}\label{eq : notation M_P}
0\to\M_{P_1}\to\M_P\to\M_{P_2}\to 0\;.
\end{equation}
If $A$ is a field, then the converse is also true: given an exact 
sequence of differential modules $0\to\N \to\M_P\to\Q\to 0$, there 
exists a factorization $P=P_1\cdot P_2$ such that 
$\N\cong\M_{P_1}$ and $\Q\cong\M_{P_2}$ (cf. \cite[3.5.6]{Ch} for 
more details).

To end this section, we quote the following lemma by B.Malgrange that gives the equality of the de Rham 
cohomology spaces of $\M_P$ with the kernel and the 
cokernel of $P$ as an endomorphism of $A$.
\begin{lemma}[cf. 
\protect{\cite[Section 3, p.153]{Malgrange-Irreg}}]\label{Lemma : Coh M_P = Coh P}
Let $P\in A\langle d\rangle$. Then
\begin{eqnarray}
\Hdr^0(\M_P,A)&\;\cong\;&\mathrm{Ker}(P:A\to A)\;\\
\qquad\qquad\qquad
\Hdr^1(\M_P,A)&\;\cong\;&\mathrm{Coker}(P:A \to A)\;.\qquad\qquad\qquad\Box
\end{eqnarray}
\end{lemma}

\section{Radii and filtrations by the radii}
\label{Radii and filtrations by the radii}
\begin{setting}\label{Hypothesis : X connected}
In order to simplify some 
statements, from now on until the end of the paper, we will assume 
without loss of generality that the curve $X$ is connected.
\end{setting}

In this section we introduce the radii of convergence of a differential 
equation over the quasi-smooth $K$-analytic curve $X$, and we recall some results of \cite{NP-I} and 
\cite{NP-II}. Without explicit mention of the contrary, we assume everywhere 
that the curve $X$ is endowed with a weak triangulation $S$.

By a \emph{differential equation} or \emph{differential module} over $X$ we mean
a coherent $\O_{X}$-module $\Fs$ endowed with an (integrable) 
connection~
$\nabla:\Fs\to\Fs\otimes\Omega^1_X$. We 
recall that $\nabla$ is by definition a $K$-linear map of 
sheaves such that for all open $U\subset X$, all 
$f\in\O_X(U)$ and all $m\in\Fs(U)$, one has 
$\nabla(f\cdot m)=m\otimes d(f)+f\cdot\nabla(m)$. 
Proposition \ref{Lemma : coherent implies loc free} below shows that
$\Fs$ is automatically a locally free $\O_X$-module of finite rank. 
This is locally a differential module as defined in section 
\ref{section : Differential  modules and trivial submodules}. 
In the sequel of the article, we shall switch freely between the different 
terminologies. 
In particular, we denote by $\Hdr^0(X,\Fs)$ the kernel of the 
connection acting on the global sections $\Fs(X)$.


The following Lemma is closely inspired by 
\cite[II, \S 5.2, Theorem 1]{BouAlgCom}. Recall that $X$ is connected (cf. Setting \ref{Hypothesis : X connected})
\begin{lemma}\label{Lemma : loc free iff stalk free}
Let $\Fs$ be a coherent $\O_X$-module. The following assertions are equivalent. 
\begin{enumerate}
\item\label{Lemma : loc free iff stalk free-i} 
$\Fs$ is locally free over $\O_X$;
\item\label{Lemma : loc free iff stalk free-ii} 
For every $x\in X$, the stalk 
$\Fs_x$ is a free $\O_{X,x}$-module;
\item\label{Lemma : loc free iff stalk free-iii} 
For every rigid point $x\in X$, the stalk 
$\Fs_x$ is a free $\O_{X,x}$-module.
\end{enumerate}
\end{lemma}
\begin{proof}
Clearly $(\ref{Lemma : loc free iff stalk free-i})\Rightarrow 
(\ref{Lemma : loc free iff stalk free-ii})\Rightarrow (\ref{Lemma : loc free iff stalk free-iii})$. Moreover, $(\ref{Lemma : loc free iff stalk free-iii})\Rightarrow (\ref{Lemma : loc free iff stalk free-ii})$ 
because if $x$ is not a rigid point, then $\O_{X,x}$ is a field and 
$\Fs_x$ is automatically free. 

Let us prove that $(\ref{Lemma : loc free iff stalk free-ii})\Rightarrow (\ref{Lemma : loc free iff stalk free-i})$. Let $x\in X$. We need to prove that it admits a neighborhood in~$X$ on which $\Fs$ is free. Let $(y_1,\ldots,y_q)$ be a basis of the stalk $\Fs_x$. Up to shrinking $X$, we 
may assume that $y_1,\ldots,y_q$ are global sections of $\Fs$. 
Let $\phi:\O_X^q\to\Fs$ be the morphism of 
$\O_X$-modules sending the canonical basis of $\O_X^q$ onto the 
sections $(y_1,\ldots,y_q)$. We have an exact sequence
\begin{equation}
0\to\mathcal{R}\to\O_X^q\xrightarrow{\phi}\Fs\to \mathcal{Q}\to 0\;,
\end{equation}
where $\mathcal{R}=\mathrm{ker}(\phi)$ and 
$\mathcal{Q}=\mathrm{coker}(\phi)$. 

Since $\O_{X}$ is coherent, so are $\mathcal{R}$ and $\mathcal{Q}$. By assumption, we have $\mathcal{R}_{x}=0$ and $\mathcal{Q}_{x}=0$. Since the support of a coherent sheaf is closed, the sheaves  $\mathcal{R}$ and $\mathcal{Q}$ are zero is the neighborhood of~$x$, which concludes the proof.
\end{proof}

\begin{proposition}\label{Lemma : coherent implies loc free}
Let $\Fs$ be a coherent $\O_X$-module with a connection $\nabla$. 
Then $\Fs$ is a locally free $\O_X$-module of finite rank.
\end{proposition} 
\begin{proof}
By Lemma \ref{Lemma : loc free iff stalk free}, it is enough to prove that, for 
every rigid point $x\in X$, the stalk $\Fs_x$ is a free $\O_{X,x}$-module. Since $\O_{X,x}$ is a PID, it is enough to prove that $\Fs_{x}$ has no torsion.

Assume by contradiction that $\Fs_{x}$ has torsion. Then, there exists an affinoid neighborhood~$U$ of~$X$ and an element $a \in \O_{X}(U)$ such that the multiplication by $a$ map $\gamma_{a} \colon \Fs(U) \to \Fs(U)$ is not injective. We may assume that $U$ is a virtual closed disk.

Let $L$ be a finite extension of~$K$ such that $U_{L}$ is a disjoint union $\bigcup_{i=1}^n D_{i}$ of closed disks over~$L$. For each $i$, the ring $\O(D_{i})$ is a principal ideal domain, hence, by \cite[Proposition 9.1.2]{Kedlaya-book-2}, $\Fs_{L}(D_{i})$ has no torsion. It follows that $\Fs_{L}(U_{L}) = \prod_{i=1}^n \Fs_{L}(D_{i})$ has no torsion either.

Finally, note that, since $L$ is a finite extension of~$K$, we have $\Fs_{L}(U_{L}) \simeq \Fs(U)\otimes_{K} L$. By faithfull flatness of $L/K$, the map $\gamma_{a} \otimes_{K}L \colon \Fs_{L}(U_{L}) \to \Fs_{L}(U_{L})$ is not injective, which yields a contradiction.
\end{proof}

\if{
\framebox{JéràƒÂƒà'Žme : Another proof, tu préfère laquelle ?}

\begin{proof}
If $\O_{X,x}$ is a field, then $\Fs_x$ is free and finite dimensional, 
and there is nothing to prove. In the other cases $x$ is $L$-rational 
for a finite Galois extension $L/K$. 
A basis of neighborhoods of $x$ is given by the virtual disks $D$. 
Now $D=(D\otimes_KL)/\mathrm{G}$, 
where $\mathrm{G}:=\mathrm{Gal}(L/K)$, and $D\otimes_KL$ is a 
disjoint union of disks. Then $\Fs_{|D}$ is locally free if, and only if, 
$\Fs_{|D\otimes_KL}$ is locally free. So we can assume $K=L$, and 
$D=D_K^-(0,R)$, with $x=0$. Then $\O_{X,x}$ is a discrete 
valuation ring, with ideals generated by $T^n$, 
for some $n\geq 1$. 
None of them is stable by $d/dT$, so $\Fs_x$ is free by 
\cite[9.1.2]{Kedlaya-book-2}.
\end{proof}
\framebox{JéràƒÂƒà'Žme : Another proof, tu préfère laquelle ?}
\begin{proof}
If $\O_{X,x}$ is a field, then $\Fs_x$ is free and finite dimensional, 
and there is nothing to prove. In the other cases $x$ is $L$-rational 
for a finite Galois extension $L/K$. A basis of neighborhoods of $x$ is 
then given by virtual disks $D$. If $T$ is a coordinate of 
$D_{\widehat{K^{\mathrm{alg}}}}$, then $\O_{X,x}$ is a discrete 
valuation ring with maximal ideal generated by the minimal polynomial 
$P(T)$. The derivation $d/dT$ commutes with $\mathrm{Gal}(L/K)$ 
and it acts on $\O_{X,x}$.
\if{

\framebox{NON ESISTONO COORDINATE SU UN DISCO VIRTUALE}

 then all derivation 
$d:\O_{X,x}\to\O_{X,x}$ is of the form $d=f(T)d/dT$, $f\in\O_{X,x}$. 
}\fi
Now $\O_{X,x}$ has no non trivial ideal stable by $d/dT$. 
Indeed %
an ideal $J$ is generated by 
$P^n$, so $(d/dT)^k(P^n)$ lies in $K-\{0\}$ for a convenient 
$k$. So if $d/dT(J)\subseteq J$, then $J=\O_{X,x}$. 
A classical property of such kind of rings is that any 
$\O_{X,x}$-module of finite type with a connection has no torsion (cf. 
\cite[9.1.2]{Kedlaya-book-2}). 
Since $\O_{X,x}$ is a principal ideal domain then $\Fs_{x}$ is free.
\end{proof}
}\fi

\begin{corollary}
The category of locally free of finite rank $\O_X$-modules with 
connection is an abelian category.\hfill$\qed$
\end{corollary}
For all $x\in X$, we set
\begin{equation}
\Fs(x)\;:=\;\Fs\otimes_{\O_X}\H(x)\;.
\end{equation}
\begin{proposition}\label{Prop. : iso on a point implies iso global} 
Let $\sigma:\Fs\to\Fs'$ be a morphism 
between two differential equations over $X$. 
If there exists $x\in X$ such that $\sigma(x):\Fs(x)\to\Fs'(x)$ is an 
isomorphism (resp. monomorphism, epimorphism) over $\H(x)$, then 
$\sigma$ is globally an isomorphism (resp. monomorphism, 
epimorphism).
\end{proposition}
\begin{proof}
The ranks of the kernel and the cokernel of $\sigma$ are locally 
constant functions.
%
%
\end{proof}

\subsection{Generic disks}
\label{Generic disks}
In this section we recreate the framework of Dwork generic points and Dwork generic disks, in the context of Berkovich analytic curves. 

\subsubsection{Existence of a big field.}
In this paper, we need a large field $\Omega$ in the following three contexts
\begin{enumerate}
\item in Proposition \ref{Lemma pts of type 4} to describe the structure of $\pi_\Omega^{-1}(x)$; 

\item to prove the independence of the radii 
with respect to the center $t_x$ (cf. Definition \ref{def: Multiradius} and \cite[Section 2.3]{NP-II});
\item when one needs a spherically complete field to make differential 
modules free over the generic disks 
$D(x)$ by Lazard's theorem \cite{Lazard} 
(cf. Section \ref{Section : Dwork generic points and Dwork generic disks}).

\end{enumerate}
\begin{setting}\label{Setting : Omega}
By convenience of notations, 
from now on we fix a large complete value field extension $\Omega/K$ 
such that
\begin{enumerate}
\item\label{Setting : Omega -1} $\Omega$ is algebraically closed and spherically complete;
\item\label{Setting : Omega -2} $|\Omega|=\mathbb{R}_{\geq 0}$;
\item\label{Setting : Omega -3} For every $x\in X$ there exists an isometric $K$-linear embedding 
\begin{equation}
\H(x)\;\subset\;\Omega.
\end{equation}
\end{enumerate}
\end{setting}
Item \eqref{Setting : Omega -3} is possible thanks to 
Proposition \ref{Prop: Bigfield} below. 

The field $\Omega$ does not play any essential role, and this choice is 
made merely for convenience of notations.


\begin{proposition}\label{Prop: Bigfield}
Let~$L$ be a valued field. Let~$I$ be a set and $(L_{i})_{i\in I}$ be a 
family of valued extensions of~$L$. Then there exists a valued 
extension~$M$ of~$L$ such that, for every~$i\in I$, there exists an 
$L$-linear isometric embedding 
$j_{i} \colon L_{i} \hookrightarrow M$.
\end{proposition}
\begin{proof}
For every~$i\in I$, consider an isometric embedding $h_{i} \colon L 
\hookrightarrow L_{i}$. 
Fix a set~$\mathcal{W}$ of cardinality at least 
$\sum_{i\in I} \mathrm{Card}(L_{i}) + \aleph_{0}$.

Consider the collection~$\Ts$ of tuples $(M,\sigma_{M},h_{M},J, (j_{i})_{i\in J})$ such that
\begin{enumerate}
\item $M$ is a subset of~$\mathcal{W}$;
\item $\sigma_{M}$ is a valued field structure on~$M$ (which will be implicit in what follows);
\item $h_{M}$ is an isometric embedding of~$L$ into~$M$;
\item $J$ is a subset of~$I$;
\item for every~$i\in J$, $j_{i}$ is an $L$-linear isometric embedding of~$L_{i}$ into~$M$;
\item the subfield generated by the union of the images of the~$j_{i}$'s is dense in~$M$. 
\end{enumerate}
By the first condition, $\Ts$ is a set.

Remark that, if~$M$ is a valued field that satisfies all the conditions but 
the first one, then there exists a valued field~$M'$, that is isometrically 
isomorphic to~$M$, and satisfy all the condition. Indeed, the last 
condition implies that the cardinality of~$M$ is at most 
$\sum_{i\in I} \mathrm{Card}(L_{i}) + \aleph_{0}$, hence there 
exists a bijection between~$M$ and a subset~$M'$ 
of~$\mathcal{W}$. 
Transporting the structures, we are done. For this reason, from now 
on, we will forget this first condition.



We endow~$\Ts$ with the following order relation: $(M,\sigma_{M},h_{M},J, (j_{i})_{i\in J}) \le (M',\sigma_{M'},h_{M'},J', (j'_{i})_{i\in J'})$ if
\begin{enumerate}
\item there exists an $L$-linear isometric embedding $j_{M',M}$ of~$M$ into~$M'$;
\item $J$ is a subset of~$J'$~;
\item for every~$i\in J$, we have $j'_{i} = j_{M',M} \circ j_{i}$. 
\end{enumerate}


By an inductive limit construction, it is easy to check that every totally ordered subset of~$\Ts$ has an upper-bound. Hence, by Zorn's lemma, there exists a maximal element $(M,J, (j_{i})_{i\in J})$ in~$\Ts$.

We claim that~$J=I$, which proves the lemma. By contradiction, 
assume that $J\subset I$ and choose $k\in I- J$. It is known 
that there exists a valued field~$M'$ that is both an extension of~$M$ 
and~$L_{k}$: we have isometric embeddings 
$j_{M',M}\colon M \to M'$ and $j'_{k}\colon L_{k}\to M'$. Moreover, 
we may assume that $j_{M',M} \circ h_{M} = j'_{k} \circ h_{k}$, and 
that the subfield generated by $j_{M',M}(M) \cup j_{k}(L_{k})$ is 
dense in~$M'$. Set $h_{M'} = j_{M,M'} \circ h_{M}$ and, for every 
$i\in J$, $j'_{i} = j_{M',M} \circ j_{i}$. Then the tuple 
$(M',\sigma_{M'},h_{M'},J\cup\{k\}, (j'_{i})_{i\in J\cup\{k\}})$ 
is an element of~$\Ts$ that is greater than 
$(M,\sigma_{M},h_{M},J, (j_{i})_{i\in J})$.
\end{proof}

\if{\begin{proposition}
Let~$L$ be a valued field. Let~$I$ be a set and $(L_{i})_{i\in I}$ be a 
family of valued extensions of~$L$. Then, there exists a valued 
field~$\Omega$ such that, for every~$i\in I$, there exists an isometric 
$K$-linear embedding $j_{i} \colon L_{i} \hookrightarrow \Omega$.
\end{proposition}
\begin{proof}
Consider the class $\Ts$ of tuples $(M,J, (j_{i})_{i\in J})$ such 
that
\begin{enumerate}
\item $M$ is a valued extension of~$K$~;
\item $J$ is a subset of~$I$~;
\item for every~$i\in J$, $j_{i}$ is a $K$-linear isometric embedding of~$L_{i}$ into~$M$~;
\item the algebra generated by the union of the images of the~$j_{i}$'s is dense in~$M$. 
\end{enumerate}

By the last condition, we may bound the cardinality of~$M$ and deduce that~$\Ts$ is a set.

We endow~$\Ts$ with the following order relation~: $(M,J, (j_{i})_{i\in J}) \le (M',J', (j'_{i})_{i\in J'})$ if
\begin{enumerate}
\item there exists an isometric $K$-linear 
embedding $j_{M',M}$ of~$M$ into~$M'$~;
\item $J$ is a subset of~$J'$~;
\item for every~$i\in J$, we have $j'_{i} = j_{M',M} \circ j_{i}$. 
\end{enumerate}

By an inductive limit construction, it is easy to check that every totally ordered subset of~$\Ts$ has an upper-bound. Hence, by Zorn's lemma, there exists a maximal element $(M,J, (j_{i})_{i\in J})$ in~$\Ts$.

We claim that~$J=I$, which proves the lemma. By contradiction, assume that $J\subset I$ and choose $k\in I- J$. It is known that there exists a valued field~$M'$ that is both an extension of~$M$ and~$L_{k}$~: we have isometric embeddings $j_{M',M}\colon M \to M'$ and $j'_{k}\colon L_{k}\to M'$. Moreover, we may assume that $j_{M',M}(M) \cup j_{k}(L_{k})$ is dense in~$M'$. For every $i\in J$, set $j'_{i} = j_{M',M} \circ j_{i}$. Then the tuple $(M',J', (j'_{i})_{i\in J'})$ is an element of~$\Ts$ that is greater than $(M,J, (j_{i})_{i\in J})$.
\end{proof}
}\fi

\subsubsection{Dwork generic points and Dwork generic disks}
\label{Section : Dwork generic points and Dwork generic disks}
For every complete valued field extension $L_1\subset L_2$ denote by 
\begin{equation}\label{eq : def of pi_L/K}
\pi_{L_2/L_1}:X_{L_2}\to X_{L_1}
\end{equation}
the canonical projection, where $X_{L_i}:=X\widehat{\otimes}_KL_i$.
\begin{definition}[Canonical Dwork generic point]\label{Def : Dw gen pt}
Let $x\in X$. The map $\mathscr{M}(\H(x))\to X$ lifts canonically to 
a map $\mathscr{M}(\H(x))\to X_{\H(x)}=X\widehat{\otimes}_K\H(x)$ by the universal property of 
the Cartesian diagram $X_{\H(x)}/\H(x) \to X/K$.
We denote by 
\begin{equation}
t_x\in X_{\H(x)}
\end{equation}
the $\H(x)$-rational point so obtained. 
We refer to $t_x$ as the \emph{canonical Dwork generic point} associated with $x$.
More generally, again by the universal property of the Cartesian diagram, 
for every complete valued extension $L/\H(x)$ there is a unique 
$L$-rational point lifting $t_x$. We will call it again by $t_x$.
\begin{equation}
\xymatrix{X\ar[d]&X_{\H(x)}\ar[l]_-{\pi_{\H(x)/K}}\ar[d]&X_L\ar[l]_-{\pi_{L/\H(x)}}\ar[d]\\
\mathscr{M}(K)&\ar[l]\mathscr{M}(\H(x))\ar[lu]^-{x}\ar@/_{1pc}/[u]_-{t_x}&
\ar[l]\mathscr{M}(L).\ar@/_{1pc}/[u]_-{t_x}}
\end{equation} 
\end{definition}
\begin{definition}[Canonical Dwork generic disk]
Let $x\in X$. The canonical Dwork genric disk for $x$ is the greatest open 
disk contained in $\pi_{\H(x)/K}^{-1}(x)$ centered at the canonical Dwork 
generic point $t_x$. We denote it by $D(x)$.
\end{definition}
The existence of Dwork generic disks was proven by Dwork when $X$ is 
the affine line, and in \cite{NP-II} in general. 
We will clarify in the next section the structure of the fiber 
$\pi_{\H(x)/K}^{-1}(x)$ and the existence of Dwork generic disks.

In the applications it is convenient to allow more freeness for the choice of 
$t_x$ and of the generic disk $D(x)$. Namely, we want to have the 
same base field over which every Dwork generic disk exists.
\begin{definition}[Dwork generic points, cf. \cite{NP-I, NP-II}]\label{Notation: Dwork generic disk and Dwork generic point}
Let $L/K$ be a complete valued field extension containing isometrically 
$\H(x)$. Let $x\in X$. We call \emph{Dwork generic point} for $x$ any 
$L$-rational point in $\pi_{L/K}^{-1}(x)$.
\end{definition}
\begin{definition}[Dwork generic disks, cf. \cite{NP-I, NP-II}]
\label{Def : Dwork gen disk}
Let $x\in X$ be a point of type $2$, $3$, or $4$. 
Let $L/\H(x)$ be a complete valued field extension. 
We call Dwork generic disk for $x$ any open disk $D$ such that
\begin{enumerate}
\item $D$ contains a Dwork generic point $t_x$ for $x$;
\item $D$ is the largest open disk in $\pi_{L/K}^{-1}(x)$ containing $t_x$.
\end{enumerate}
We denote any unspecified Dwork generic disk by 
\begin{equation}\label{eq : D(x)}
D(x)\subset X_L\;.
\end{equation}
If a Dwork generic point $t_x$ for $x$ is chosen, we will always assume that 
\begin{equation}\label{eq : t_x in D(x) dsdfd}
t_x\in D(x)\;.
\end{equation}

We extend this definition to points $x\in X$ of type $1$ by defining 
$D(x):=\{t_x\}$, so that \eqref{eq : t_x in D(x) dsdfd} holds for every $x\in X$.
\end{definition}
The following proposition proves that, up to enlarging $L$, Dwork generic 
disks are all isomorphic by an isomorphism preserving the choice of the 
Dwork generic point in them.
\begin{proposition}[\protect{\cite[Corollary 2.20]{NP-II}}]\label{Prop: Omega acts transitively}
If $L/K$ is algebraically closed and spherically complete, the group 
$\mathrm{Gal}^{\mathrm{cont}}(L/K)$ acts 
transitively on the set of Dwork generic disks in $\pi_{L/K}^{-1}(x)$ 
as well as on the set of $L$-rational points of 
$\pi^{-1}_{L/K}(x)$.\hfill$\qed$
\end{proposition}
Thanks to Proposition \ref{Prop: Omega acts transitively}, all Dwork generic 
disks become isomorphic over a sufficiently large field extension of the 
base field $K$.\footnote{Indeed, by Proposition \ref{Prop: Omega acts transitively}, for every pair of Dwork generic disks $D_1$ and $D_2$ 
there is a continuous Galois automorphism 
$\sigma\in\mathrm{Gal}^{\mathrm{cont}}(\Omega/K)$ 
realizing a $K$-linear isomorphism of Banach rings
$\sigma:\O_\Omega(D_2)\simto\O_\Omega(D_1)$. In order to obtain an 
isomorphism of analytic spaces over $\Omega$ (i.e. an $\Omega$-linear 
isomorphism) it is enough to apply a torsion by Galois on the coefficients 
of the series. More specifically, if $t_1$ is a Dwork generic points for 
$D_1$ and $T_1$ is a coordinate on $D_1$, 
then $t_2:=\sigma(t_1)$ is a Dwork generic point for $D_2$ 
and $T_2:=\sigma\circ T_1$ is a coordinate function for $D_2$.
The Galois automorphism $\sigma$ induces an isomorphism sending 
a series 
$\sum_{n\geq 0} a_n (T_2-t_2)^n\in\O_\Omega(D_2)$ into 
$\sum_{n\geq 0} \sigma(a_n) (T_1-t_1)^n$. 
We may compose this automorphism with the one associating to 
$\sum_{n\geq 0} a_n (T_1-t_1)^n$ the series 
$\sum_{n\geq 0} \sigma^{-1}(a_n) (T_1-t_1)^n$ and we obtain an 
$\Omega$-linear isomorphism.} 
 The choice of the Dwork generic disk and Dwork generic 
point is then irrelevant and we often say 
``\emph{the}'' generic disk to indicate one of them.
Merely by convenience of notation, we are now going to fix once 
for all a Dwork generic point $t_x$ and the corresponding Dwork generic 
disk. Recall Setting \ref{Setting : Omega}.

\begin{notation}[Choice of a Dwork generic point]
From now on, for all $x\in X$, we choose an $\Omega$-rational 
point in $\pi_{\Omega/K}^{-1}(x)$ which will be denoted by 
\begin{equation}
t_x\;\in\;\pi_{\Omega/K}^{-1}(x)\;.
\end{equation}
We also maintain the notations of Definition \ref{Def : Dwork gen disk} : the Dwork 
generic disk containing $t_x$ will be denoted by $D(x)$ and if $x$ is a 
point of type $1$, then we set $D(x)=\{t_x\}$.
\end{notation}

\subsubsection{The structure of the fiber.}
Dwork generic points are interesting in our framework because 
of the important features of the Taylor solutions around them. 
In particular the link with the spectral norm of the connection 
(cf. Sections \ref{Robba-deco} and \ref{Dwork-Robba (local) decomposition} for instance). 
In this section we recall some structural results of \cite{NP-I,NP-II} about 
the nature of the fiber $\pi_{L/K}^{-1}(x)$. In particular, the fact that Dwork 
generic disks do exist. We begin with an algebraically closed base field $K$  
and will remove this assumption in a moment.

\begin{proposition}[\protect{\cite[Corollaire~3.14]{Angie} and \cite[Proposition 2.26]{NP-II}}]\label{Lemma pts of type 4}
Let $x\in X$. Assume that $K$ is algebraically closed and let 
$L$ be a complete valued field extension of $K$. 
The inverse image 
\begin{equation}
\pi_{L/K}^{-1}(x)\;=\;\mathscr{M}(\H(x)\widehat{\otimes}_KL)
\end{equation}
is a connected analytic space over $L$. 
The tensor product norm defining the structure of Banach algebra of 
$\H(x)\widehat{\otimes}_KL$ is multiplicative, 
hence it is an element of the spectrum $\mathscr{M}(\H(x)\widehat{\otimes}_KL)$. Let us denote it by
\begin{equation}
\sigma_{L/K}(x)\;\in\;\pi_{L/K}^{-1}(x)\;.
\end{equation} 
It enjoys the following properties:
\begin{enumerate}
\item $\sigma_{L/K}(x)$ is the unique point 
of $\pi_{L/K}^{-1}(x)$ such that 
$\pi_{L/K}^{-1}(x)-\{\sigma_{L/K}(x)\}$ is a 
(possibly empty) 
disjoint union of virtual open disks $D\subseteq X_L$;
\item $\sigma_{L/K}(x)$ is  the maximum element with respect to 
the natural partial order of $\mathscr{M}(\H(x)\widehat{\otimes}_KL)$.
\footnote{As the 
points of the spectrum are semi-norms, they are real valued functions and 
the partial order is induced by the order of $\mathbb{R}$.} \hfill$\qed$
\end{enumerate}
\end{proposition}
If $x$ is of type 1, then $\pi^{-1}(x)=\{\sigma_{\Omega/K}(x)\}$ and these disks are empty for all $\Omega/K$.\footnote{Indeed, $x$ has a neighborhood isomorphic to a virtual open 
disk, in which case we may apply the results of \cite[Section 1.0.2]{NP-I}.}  

Otherwise, if $x$ is of type $2$, $3$, or $4$, and 
if $\Omega$ is an algebraically closed complete 
valued extension of $K$, then they are not empty 
if, and only if, $\H(x)$ embeds into $\Omega$.\footnote{Indeed, 
if $\H(x)\subseteq\Omega$, the point $t_x$ 
of Definition \ref{Def : Dw gen pt}
is an $\Omega$-rational point, hence different from 
$\sigma_{\Omega/K}(x)$.  On the other hand, if 
$\pi_{\Omega/K}^{-1}(x)-\{\sigma_{\Omega/K}(x)\}$ has a non empty 
connected component $D$, then $D$ is an open disk with an 
$\Omega$-rational point (because $\Omega$ is algebraically closed) 
and hence we have an isometric embedding $\H(x)\subseteq\Omega$.}

If no confusion is possible, we will write $\pi_\Omega=\pi_{\Omega/K}$ and $\sigma_\Omega:=\sigma_{\Omega/K}$.
\begin{definition}\label{Def. sigma_L}
Assume that $K$ is algebraically closed. 
For all field extensions $L/K$ one has a canonical section
\begin{equation}
\sigma_{L/K}\;:\;X\to X_L
\end{equation}
associating to $x$ the the point 
$\sigma_{L/K}(x)\in X_L$ of 
Proposition \ref{Lemma pts of type 4}. 
\end{definition}
\begin{remark}
The map $\sigma_L$ is defined in \cite{Angie}, see also \cite[p.98]{Ber}.
\end{remark}
%
%
%

\label{Generic disks.-1}
If now $K$ is general, then 
$X\cong X_{\widehat{K^{\mathrm{alg}}}}/
\mathrm{Gal}(K^{\mathrm{alg}}/K)$. 
If $L'$ is a complete valued extension of $\wKa$, 
this allows to describe  
$\pi_{L'/K}^{-1}(x)$ as follows. 
Consider the composite arrow 
$\pi_{L'/K}:X_{L'}
\xrightarrow{\quad} 
X_{\wKa}\to X$, where the first projection 
$\pi_{L'/\wKa}$ is that of Proposition 
\ref{Lemma pts of type 4}.

One has a disjoint union 
\begin{equation}\label{eq : piOmega/K}
\pi_{L'/K}^{-1}(x)\;=\;
\bigsqcup_{y \in \pi^{-1}_{\widehat{K^{\mathrm{alg}} } / K }(x)}\;
\pi_{L'/\widehat{K^{\mathrm{alg}}}}^{-1}(y)\;.
\end{equation}
By Proposition \ref{Lemma pts of type 4}, 
\begin{equation}\label{eq: disks disjoint union dwork gen}
\pi_{L'/K}^{-1}(x) - \bigsqcup_{y \in \pi^{-1}_{\wKa / K }(x)} \{\sigma_{L'/\wKa}(y)\}
\;\;=\;\;
\pi_{L'/K}^{-1}(x) - \sigma_{L'/\wKa}(\pi^{-1}_{\wKa/K}(x))
\end{equation}
is a (possibly empty) disjoint union of virtual open $L'$-disks. 

%

Finally, let us consider the general situation where 
$L/K$ is any complete valued field extension, we may 
consider its algebraic closure $\widehat{L^{\mathrm{alg}}}$ and 
combine the above results. Every connected component of 
\begin{equation}\label{eq : dw gen dsk disj union}
\pi_{L/K}^{-1}(x)-\pi_{\widehat{L^{\mathrm{alg}}}/L}(\sigma_{\widehat{L^{\mathrm{alg}}}
/\wKa}(\pi_{\wKa/K}^{-1}(x)))\;.
\end{equation}
is a (possibly empty) virtual open disk.\footnote{Indeed, 
if $D$ is a virtual open disk which is a connected component of 
\eqref{eq: disks disjoint union dwork gen} (with $\Omega=\widehat{L^{\mathrm{alg}}}$), then 
its image in $\pi_{L/K}^{-1}(x)$ is either a point in 
$\pi_{\widehat{L^{\mathrm{alg}}}/L}(\sigma_{\widehat{L^{\mathrm{alg}}}/\wKa}(\pi_{\wKa/K}^{-1}(x)))$, 
or it is a virtual open disk which is a connected component of \eqref{eq : dw gen dsk disj union}.}

\begin{remark}[Dwork generic disks]
Let $x$ is a point of type $2$, $3$ or $4$. Then 
Dwork generic disks are precisely the $L$-rational open disk in 
$\pi_{L/K}^{-1}(x)$ which are connected components of 
\eqref{eq : dw gen dsk disj union} (i.e. the disk 
must contain a Dwork generic point, cf. Definition \ref{Def : Dw gen pt}). 

Again, the set $M:=\pi_{\widehat{L^{\mathrm{alg}}}/L}(\sigma_{\widehat{L^{\mathrm{alg}}}
/\wKa}(\pi_{\wKa/K}^{-1}(x)))$ can be 
characterized as the set of maximal semi-norms in $\pi_{L/K}^{-1}(x)$.
In particular, Dwork generic disks are precisely the connected component 
of  $\pi^{-1}_{L/K}(x)-M$ containing an $L$-rational point. 

Equivalently, a Dwork generic disk is any $L$-rational open 
disk in $\pi^{-1}_{L/K}(x)$ which is maximal with respect to the inclusion 
relation as in Definition \ref{Def : Dwork gen disk}.
\end{remark}

\begin{proposition}[\protect{\cite{NP-II}}]\label{Def : S_L}
Let $S\subset X$ be a weak triangulation of $X$, and let $L/K$ be a 
complete valued field extension. Then 
\begin{equation}\label{eq : S_L}
S_L\;:=\;\pi_{\widehat{L^{\mathrm{alg}}}/L}(\sigma_{\widehat{L^{\mathrm{alg}}}/\wKa}(\pi_{\wKa/K}^{-1}(S)))
\end{equation}
is a weak triangulation of $X_L$.
\end{proposition}

\subsection{Maximal disks}
Recall the choice of $\Omega$ (cf. Setting \ref{Setting : Omega}). 
\begin{definition}\label{Def : can extensions}
Let $Z\subseteq X$ be any non empty subset such that $X-Z$ is a 
disjoint union of virtual open disks or annuli. 
For all $x\in X$ we define 
\begin{equation}\label{eq : def of D(x,Gamma) rpllll}
D(x,Z)\;\subseteq\; X_\Omega-\sigma_\Omega(\pi_{\wKa/K}^{-1}(Z))
\end{equation}
as the largest $\Omega$-rational open disk containing 
$t_x$ and contained in $X_\Omega-\sigma_\Omega(\pi_{\wKa/K}^{-1}(Z))$. 
\end{definition}

\begin{remark}\label{Radius=spectral radius at a point of Gamma}
In several situations we have $D(x,Z)=D(x)$. This is indeed the case if
\begin{enumerate}
\item $x\in X$ is a point of type $1$, because $D(x)=\{t_x\}$;
\item $x\in Z$;
\item $Z=X$.
\end{enumerate}
Notice moreover that $Z$ necessarily contains the analytic skeleton of 
$X$, which is the set of point without neighborhoods isomorphic to virtual 
open disks.
\end{remark}

\begin{definition}[Maximal disks]\label{Maximal disksksks}
Let $x\in X$. We call \emph{maximal disk of $x$} (with respect to the 
weak triangulation $S$) the disk (cf. Definition \ref{Def : can extensions})
\begin{equation}
D(x,S)\;:=\;D(x,\Gamma_S)\;\subseteq\; X_\Omega\;.
\end{equation}
\end{definition}
By definition for all $x\in X$ we have 
\begin{equation}
t_x\;\in\; D(x)\;\subseteq\; D(x,S)\;.
\end{equation}
As for the Dwork generic disks, the choice of the 
maximal disk is irrelevant. Indeed, two different choices of $t_x$ 
produce isomorphic maximal disks.\footnote{Indeed, 
from Proposition \ref{Prop: Omega acts transitively}, 
it is not hard to deduce that 
the set of open disks in $X_\Omega$ 
containing an $\Omega$-rational point in 
$\pi^{-1}_{\Omega/K}(x)$ 
is globally stable under the action of 
$\mathrm{Gal}^{\mathrm{cont}}(\Omega/K)$.} 

\begin{remark}\label{Remark : D(x,S) cases}
There are two possibilities for $D(x,S)$.
If $x$ belongs to the skeleton $\Gamma_S$,
then $D(x,S)$ is contained in $\pi_\Omega^{-1}(x)$, 
and $D(x,S)=D(x)$ 
(cf. Remark \ref{Radius=spectral radius at a point of Gamma}). 
Otherwise, if $x\notin\Gamma_S$, then $D(x,S)$ 
contains $\pi_\Omega^{-1}(x)$, 
and it is strictly larger than the generic disk. 
In this case $D(x,S)$ has a $K^{\mathrm{alg}}$-rational point, and 
its image in $X$ is a virtual open disk containing $x$ 
(cf. \cite[Lemma 1.5]{NP-I}), the image 
is indeed the connected component of 
$X-\Gamma_S$ containing $x$. In this case, with an abuse, 
we identify $D(x,S)$ with its image in $X$, if no confusion is possible. 
\end{remark}

\subsection{Multiradius.}\label{multiradius}
Assume that $\Omega/K$ is algebraically closed and spherically
complete. Choose 
an isomorphism $D(x,S)\simto D^-_\Omega(0,R)$ sending $t_x$ to 
$0$. 
Since $\Omega$ is spherically complete, by a result of M. Lazard (cf. \cite{Lazard} and 
\cite[Ch.II, Section 4.4]{Christol-Book}), the restriction 
$\widetilde{\Fs}$ of $\Fs$ to $D^-_\Omega(0,R)$ is free of rank $r=\mathrm{rank}(\Fs_x)$, 
and it is hence given by a differential module over $\O(D^-_\Omega(0,R))$. 
Denote by 
\begin{equation}\label{eq: R^F DV-Balda}
\R_{S,i}^{\widetilde{\Fs}}(x)\;>\;0
\end{equation}
the radius of the maximal open disk centered at $0$ and contained in 
$D^-_\Omega(0,R)$ on which the connection of $\widetilde{\Fs}$ 
admits at least $r-i+1$ horizontal sections that are 
linearly independent over $\Omega$. 
\begin{definition}[Multiradius]\label{def: Multiradius}
For every $i=1,\ldots,r$, we call \emph{$i$-th radius of convergence} of $\Fs$ at~$x$ relative to $S$ the number
\begin{equation}\label{eq : R_S,i(x,F)+R_i^F(x)/R}
\R_{S,i}(x,\Fs)\;:=\;\R_{S,i}^{\widetilde{\Fs}}(x)/R \;\in\; ]0,1]\;.
\end{equation}
For $i=1,\ldots,r$, we call \emph{$i$-th partial height} 
the function on $X$ defined by the product
\begin{equation}\label{eq : partial height}
H_{S,i}(x,\Fs)\;:=\;\prod_{j=1}^i\R_{S,j}(x,\Fs)\;.
\end{equation}
We call $S$-multiradius of $\Fs$ at $x$ the tuple 
\begin{equation}\label{eq : multiradius}
\RR_{S}(x,\Fs)\;:=\;(\R_{S,1}(x,\Fs),\ldots,\R_{S,r}(x,\Fs))\;.
\end{equation} 
\end{definition}
The definition of the radii and partial heights only depends on $x$ and $(\Fs,\nabla)$ and not on the choice of  $t_x$ (cf. \cite[Section 2.3]{NP-II}). 
Each $\R_{S,i}(x,\Fs)$ is the inverse of the \emph{modulus} 
(cf. Definition \ref{def:modulusannulus}) of a well 
defined sub-disk $D_{S,i}(x,\Fs)\subseteq D(x,S)$, 
centered at $t_x$:
\begin{equation}\label{disks}
\emptyset\neq D_{S,1}(x,\Fs)\;\subseteq\;D_{S,2}(x,\Fs)\;\subseteq\;
\cdots\;\subseteq\; D_{S,r}(x,\Fs)\;\subseteq\;D(x,S)\;.
\end{equation}
For every $0<R\leq 1$, we denote respectively by 
\begin{eqnarray}
D_S(x,R)&\;\subseteq\;&D(x,S)\\
D(x,R)&\;\subseteq\;&D(x)\label{D(x,R)}
\end{eqnarray}
the open sub-disks centered at $t_x$ of modulus equal to $1/R$.
With this convention one has
\begin{equation}\label{eq : D_S,i and D(x, R_S,i)}
D_{S,i}(x,\Fs)\;:=\;D_S(x,\R_{S,i}(x,\Fs))\;.
\end{equation}

The following proposition follows from Section \ref{Generic disks} and the 
definition of the radii.
\begin{proposition}[\protect{\cite{NP-II}}]
\label{Prop: insensitive to scalar ext}
The radii are insensitive to arbitrary scalar extension of $K$. 
Namely, let $L/K$ be a complete valued field extension, then for all 
$i=1,\ldots,\mathrm{rank}(\Fs)$ all $x\in X$ and all 
$y\in\pi_{L/K}^{-1}(x)\in X_L$, we have (cf. \eqref{eq : S_L})
\begin{equation}\label{eq : insensitive to scalar ext}
\phantom{.}\qquad\qquad
\R_{S,i}(x,\Fs)\;=\;\R_{S_L,i}(y,\Fs_L)\;.\qquad\qquad\qed
\end{equation}
\end{proposition}
\begin{remark}\label{changing triang}
Let $S,S'$ be two weak triangulations of $X$. 
If $\Gamma_S=\Gamma_{S'}$, then 
$\RR_{S}(-,\Fs)=\RR_{S'}(-,\Fs)$. 
Indeed the disk $D(x,S)$ only depends on 
$\Gamma_S$.
\end{remark}

\begin{definition}[Convergence Newton polygon]\label{Def : CNP}
We call \emph{convergence Newton polygon} of $\Fs$ at 
$x\in X$ the epigraph of the unique continuous convex 
function $NP_S(x):[0,+\infty[\to\mathbb{R}_{\geq 0}$ 
satisfying
\begin{enumerate}
\item $NP_S(x)(0)=0$, and 
\begin{equation}
NP_S(x)(i)\;=\;
\sum_{j=1}^i\log\R_{S,j}(x,\Fs)
\;=\;
\log\Bigl(H_{S,i}(x,\Fs)\Bigr)\;;
\end{equation}
\item 
For all $i=1,\ldots,r$ the function $t\mapsto NP_S(x)(t)$ 
is affine over $[i-1,i]$, and constant on $[r,\infty[$.
\end{enumerate}
\end{definition}
In other words $NP_S(x):[0,+\infty[\to\mathbb{R}$ is 
the polygon whose slopes are 
$\log\R_{S,1}(x,\Fs)\leq\cdots\leq
\log\R_{S,r}(x,\Fs)$.
Since the radii are by definition in $]0,1]$, the resulting 
polygon is a non-positive convex function defined in 
$[0,+\infty[$, whose initial value is $0$. We may think of 
the convergence Newton polygon $NP_S$ 
as a real-valued function defined on 
$X\times\mathbb{R}_{\geq 0}$ or as a polygon moving 
along $X$. What is interesting to us is the behavior, with respect to $x$, of its 
partial heights $H_{S,i}(x,\Fs)$ (for fixed $i$). Their properties, such as 
continuity and finiteness for instance, correspond to 
analogous properties of the radii 
\cite{NP-I, NP-II, NP-IV}. However, the partial heights enjoy the following integrality property that they do not share with the radii.
\begin{proposition}[Integrality of the partial heights]
\label{Prop : Integrality of the Partial heights}
Let $x\in X$, and let $b$ be a germ of segment out of 
$x$. For all $i=1,\ldots,r$, the slopes $\partial_bH_{S,i}(x,\Fs)$ of $H_{S,i}(-,\Fs)$ 
along $b$ belong to the set
\begin{equation}\label{eq : integrality}
\mathbb{Z}\cup\frac{1}{2}\mathbb{Z}\cup\cdots\cup
\frac{1}{r}\mathbb{Z}\;.
\end{equation} 
Moreover, if $i=r$ or if $\R_{S,i}(x,\Fs) < 
\R_{S,i+1}(x,\Fs)$, 
(i.e. if $i$ is a vertex of the convergence Newton polygon), then 
\begin{equation}
\partial_bH_{S,i}(x,\Fs)\;\in\;\mathbb{Z}\;.
\end{equation}
\end{proposition}
\begin{proof}
The definition of the radii is invariant by base change of 
$K$, so without loss of generality we may assume $K$ 
algebraically closed, spherically complete and 
$|K|=\mathbb{R}_{\geq 0}$.  
Hence, there exists a segment $]x,y[$ representing $b$ 
which is the skeleton of an open annulus $C_b$ 
in $X$ (cf. Definition \ref{Def. skeleton}). 
The localization at $C_b$ truncates the values of some  
radii and Proposition \ref{Prop : Irr local sopp} 
shows that the slope of the partial heights changes by an integer.
\footnote{For expository reasons, we 
placed this statement before Proposition \ref{Prop : Irr local sopp}. 
However the reader may verify that there is no circularity} 
Therefore, we may assume $X=C_b$ and the statement follows 
from \cite[Theorem 3.9]{NP-I}.
\end{proof}

\subsection{Controlling graphs}\label{section controlling graphs}

In \cite{NP-I} and \cite{NP-II} we obtained the following 
result.
\begin{theorem}[\protect{\cite{NP-I},\cite{NP-II}}]
\label{Thm : finiteness}
For all $i=1,\ldots,r$ the functions $x\mapsto\R_{S,i}(x,\Fs)$ are 
continuous. Moreover there exists a locally finite graph 
$\Gamma\subseteq X$ such that for all $i$ the radius 
$\R_{S,i}(-,\Fs)$ is constant on every connected component of 
$X-\Gamma$.
%
\end{theorem}

If $S=\emptyset$, then, by definition $X$ is a virtual open disk or 
annulus (recall that $X$ is connected by Setting 
\ref{Hypothesis : X connected}) and $\Gamma_S=\Gamma_{\emptyset}$ is its skeleton. 
On the other hand $\Gamma_S=\emptyset$ if, and only if, 
$X$ is a virtual open disk with empty weak triangulation.  
This is the unique case in which there is no retraction $X\to\Gamma_S$. Any other connected curve $X$ admits a canonical retraction 
\begin{equation}\label{eq : retraction fdyhtdssfkj}
\delta_{\Gamma_S}\;:\;X\to\Gamma_S\;.
\end{equation}
\if{
As a matter of notations, in order to treat at the same time the case of 
a disk with empty triangulation and the general case, we proceed as 
follows. 
If $X=D$ is an open disk with empty weak triangulation we identify $D$ with $D_K(0,R)$, for some $R>0$, 
and we consider the  topological closure $\overline{D}=D\cup\{x_{0,R}\}$ of $D$ in  $\mathbb{A}^{1,\mathrm{an}}_K$. 
We set $\overline{S}=\{x_{0,R}\}$, and we extend $\RR_S(x,\Fs)$ by continuity to a function on the whole $\overline{D}$ with values in 
$\mathbb{R}_{\geq 0}$ (note that $\R_{S,i}(x,\Fs)>0$ for all $x\in D$). 

FALSO !!! NON SI PUO' SEMPRE ESTENDERE PER CONTINUITA'
In this case $\overline{D}$ admits a retraction on 
\begin{equation}
\Gamma_{\overline{S}}\;=\;\{x_{0,R}\}\;.
\end{equation}
}\fi
The map $\delta_{\Gamma_S}$ is the identity on $\Gamma_S$, and it 
associates to $x\notin\Gamma_S$ the boundary in $\Gamma_S$ 
of the maximal disk $D(x,\Gamma_S)=D(x,S)$ (cf. Definition \ref{Maximal disksksks} and Remark \ref{Remark : D(x,S) cases}). This makes sense since 
$X-\Gamma_S$ is disjoint union of virtual open disks.

\begin{definition}\label{Def. :: controlling graphs}
Let $\mathcal{T}$ be a set, and let 
$f:X\to\mathcal{T}$ be a function.
We call $S$-\emph{controlling graph} (or $S$-\emph{skeleton}) 
of $f$ the set $\Gamma_{S}(f)$ of points $x\in X$ that admit no 
neighborhoods\footnote{Note that 
$D(x)$ is not a neighborhood of $x$ in $X$.} $D$ in $X$ such that
\begin{enumerate}
\item $D$ is a virtual open disk;
\item $f$ is constant on $D$;
\item $D\cap \Gamma_S=\emptyset$ (or equivalently 
$D\cap S=\emptyset$).
\end{enumerate}
In particular $\Gamma_S\subseteq\Gamma_S(f)$.
\end{definition}
\begin{remark}\label{Rk : locus of constancy neq graph}
The graph $\Gamma_{S}(f)$ is different from the locus 
defined as the complement of the union of the open subsets  
of $X$ on which $f$ is constant. 
Indeed $f$ can be constant along some segments in 
$\Gamma_{S}(f)$, and hence on the corresponding annulus in $X$. 
This is due to the fact that the definition of $\Gamma(f)$ 
involves only disks on which 
$f$ is constant, and not arbitrary subsets. 
\end{remark}
We denote by $\Gamma_{S,i}(\Fs)$ the controlling graph of the 
function $\R_{S,i}(-,\Fs)$. By definition 
\begin{equation}\label{eq : Gamma_S subset Gamma(F)}
\Gamma_S\;\subseteq\;\Gamma_{S,i}(\Fs)\;
\end{equation} 
and $X-\Gamma_{S,i}(\Fs)$ is a disjoint union of virtual 
open disks on which $\R_{S,i}(-,\Fs)$ is constant. 
If $X=D$ is a virtual open disk with empty weak triangulation, 
and if $\R_{S,i}(-,\Fs)$ is constant on $D$, then 
$\Gamma_{S,i}(\Fs)=\Gamma_S=\emptyset$. 
In all other cases $\Gamma_{S,i}(\Fs)$ is not empty, 
and there is a canonical continuous retraction
\begin{equation}
\delta_{\Gamma_{S,i}(\Fs)}\;:\;X\to \Gamma_{S,i}(\Fs)\;.
\end{equation} 
The \emph{controlling graph} $\Gamma_S(\Fs)$ of $(\Fs,\nabla)$ is 
by definition the union of all the $\Gamma_{S,i}(\Fs)$:
\begin{equation}\label{eq : total controlling graph}
\Gamma_S(\Fs)\;:=\;\bigcup_{i=1}^r\Gamma_{S,i}(\Fs)\;.
\end{equation}
One has $\Gamma_{S}(\Fs)=\emptyset$ if, and only if, $X=D$ is a 
virtual disk with empty weak triangulation, and the multi-radius function $\RR_{S}(-,\Fs)$ (cf. \eqref{eq : multiradius}) 
is a constant function on $D$.
In all other cases $\Gamma_S(\Fs)\neq\emptyset$ 
is the smallest graph containing 
$\Gamma_S$ on which $\RR_S(-,\Fs)$ factors by the canonical 
continuous retraction 
\begin{equation}\label{eq : retraction global}
\delta_{\Gamma_{S}(\Fs)}\;:\;X\to \Gamma_{S}(\Fs)\;.
\end{equation} 
An operative description of the controlling graphs will be given in \cite[Section 2]{NP-IV}.

\subsection{Filtered space of solutions}
We now define the space of solutions of $\Fs$ at a point 
$x\in X$, and its filtration by the radii.
\begin{definition}\label{scale}
We say that a tuple $(d_1,\ldots,d_r)$ is a \emph{scale} if
\begin{enumerate}
\item for all $i\in\{1,\ldots,r\}$ one has $d_i\in\{1,\ldots,r\}$;
\item $d_1=r$ and $d_1\geq d_2\geq \cdots\geq d_r$;
\item If $d_i\neq d_{i-1}$, then $d_i=r-i+1$.
\end{enumerate}
\end{definition}
It follows from the definition that  $ r-i+1\leq d_i\leq r$, 
for all $i=1,\ldots,r$. Below is represented a typical scale 
\begin{equation}
\begin{picture}(100,110)
\put(10,0){\vector(0,1){110}}
\put(150,-3){$i$}
\put(0,10){\vector(1,0){150}}
\put(-3,110){$d_i$}
\qbezier[50](10,100)(55,55)(100,10)
\put(22.5,82.5){$\bullet$}
\put(32.5,82.5){$\bullet$}
\put(42.5,82.5){$\bullet$}
\put(52.5,52.5){$\bullet$}
\put(62.5,52.5){$\bullet$}
\put(72.5,32.5){$\bullet$}
\put(82.5,32.5){$\bullet$}
\put(92.5,32.5){$\bullet$}
\put(102.5,32.5){$\bullet$}
\end{picture}
\end{equation}

\begin{definition}\label{Def: i separates the filtration}
Let $V$ be a vector space of dimension $r$ over a field $\Omega$. 
A \emph{filtration by the radii} of $V$ is a totally ordered family of $r$ 
sub-spaces 
\begin{equation}
V_r\;\subseteq\; V_{r-1}\;\subseteq\;\cdots\;\subseteq\;
V_1\;=\;V
\end{equation}
such that the sequence $(d_1,\ldots,d_r):=(\mathrm{dim}\;V_1,\ldots,
\mathrm{dim}\;V_r)$ is a scale. Note that for $i\in\{2,\ldots,r\}$
\begin{equation}\label{rank = r-i+1}
V_{i}\;\neq\; V_{i-1} \quad 
\textrm{ if, and only if, }\quad \mathrm{dim}\,V_{i}\;=\;r-i+1\;.
\end{equation}
If $i=1$, or if $i$ is an index such that \eqref{rank = r-i+1} holds, 
we say that \emph{the index $i$ separates the filtration}. 
\end{definition}

\subsubsection{Filtration by the radii of the solutions.}
Let $x\in X$, and let $t_x$ and $\Omega/\H(x)$ be as in 
Section \ref{Generic disks.-1}. 
We call \emph{solution of $\Fs$ at $x$} any element in 
the stalk $\Fs_{t_x}:=\Fs_x\otimes_{\O_{X,x}}\O_{X_\Omega,t_x}$ 
which is killed by the connection $\nabla$. Recall that $t_x$ is an $\Omega$-rational point of $X$ and hence $\O_{X_\Omega,t_x}$ is isomorphic to a field of convergent power series $\Omega\{\{T-t_x\}\}:=\{f=\sum_{n\geq 0}a_n(T-t_x)^n,a_n\in\Omega,\exists \varepsilon_f>0, \lim_n|a_n|\varepsilon_f^n=0\}$ and $\Fs_{t_x}\cong\Omega\{\{T-t_x\}\}^r$ is a free $\Omega\{\{T-t_x\}\}$-module with connection.
The kernel of $\nabla$ acting on $\Fs_{t_x}$ will be denoted by
\begin{equation}\label{omega(x,F)}
\Hdr^0(x,\Fs)\;:=\;\Hdr^0(\Fs_{t_x},\O_{X_\Omega,t_x})\;.
\end{equation}
Let $D\subseteq D(x,S)$ be an open disk containing $t_x$. We usually avoid to indicate explicitly the base 
change to $\Omega$. For instance, we set 
$\Hdr^0(D,\Fs)\;:=\;
\Hdr^0(D,\Fs\widehat{\otimes}_K\Omega)$ and 
$\Fs(D):=(\Fs\widehat{\otimes}_K\Omega)(D)$.
We may identify
\begin{equation}\label{eq : omega(D,Fs)}
\Hdr^0(D,\Fs)
\end{equation}
with its image in $\Hdr^0(x,\Fs)$. 
The rule $\Fs\mapsto\Hdr^0(D,\Fs)$ is a left exact functor (cf. Lemma 
\ref{Lemma :devisage}). We then define (cf. \eqref{disks})
\begin{equation}\label{eq : omeaga_S,i and D_S,i}
\omega_{S,i}(x,\Fs)\;:=\;\Hdr^0(D_{S,i}(x,\Fs),\Fs)\;.
\end{equation} 
The $\Omega$-vector space $\Hdr^0(x,\Fs)$ admits a filtration:
\begin{equation}\label{eq : filtration of omega}
0\;\neq \;\omega_{S,r}(x,\Fs)\;\subseteq\;
\omega_{S,r-1}(x,\Fs)\;\subseteq\;\cdots\;\subseteq\;
\omega_{S,1}(x,\Fs)\;=\;\Hdr^0(x,\Fs)\;.
\end{equation}
The following lemma follows directly from the 
definition of the radii $\R_{S,i}(x,\Fs)$.
\begin{lemma}
The filtration \eqref{eq : filtration of omega} is a filtration by the radii.
\hfill$\qed$
\end{lemma}
Condition \eqref{rank = r-i+1} is then expressed by the following

\begin{definition}\label{definition : i separates the radii}
Let $i\in\{1,\ldots,r\}$. 
We say that the index $i$ \emph{separates the radii of 
$\Fs$ at $x\in X$} if either $i=1$ or if 
one of the following equivalent conditions holds:
\begin{enumerate}
\item $\R_{S,i-1}(x,\Fs)<\R_{S,i}(x,\Fs)$;
\item $\omega_{S,i-1}(x,\Fs)\supset\omega_{S,i}(x,\Fs)$.
\end{enumerate}
We say that $i$ \emph{separates the radii of 
$\Fs$} if it separates the radii of $\Fs$ at all $x\in X$. 
\end{definition}

\begin{remark}
The index $i$ separates the radii of $\Fs$ at~$x$ if, 
and only if, $\mathrm{dim}_\Omega\;\omega_{S,i}(x,\Fs) = r-i+1$ (cf. \eqref{rank = r-i+1}).

%
\end{remark}

\subsection{Spectral, solvable, and over-solvable radii.}
\begin{definition}\label{Def: spectral index}
We say that the $i$-th radius $\R_{S,i}(x,\Fs)$ is 
\begin{equation}\label{eq : spectral - solvable - oversolvable}
\left\{
\begin{array}{lcl}
\textrm{\emph{spectral}}&\textrm{ if }&D_{S,i}(x,\Fs)\;
\subseteq \;D(x)\;,\\
\textrm{\emph{solvable}}&\textrm{ if }&D_{S,i}(x,\Fs)\;=\;D(x)\;,\\
\textrm{\emph{over-solvable}}&\textrm{ if }&D_{S,i}(x,\Fs)\;
\supset\;D(x)\;.
\end{array}
\right.
\end{equation}
Solvable radii are spectral by definition. 
We also say that \emph{the index $i$}, or the \emph{$i$-th step of the 
filtration $\omega_{S,i}(x,\Fs)$}, or \emph{the disk $D_{S,i}(x,\Fs)$}, 
is spectral, solvable, over-solvable.
\end{definition}
\begin{definition}\label{eq : over-solvable cutoff}
\label{eq : def of m(x).}
We denote by $0\leq i_x^{\mathrm{sp}}\leq \is{x}\leq 
r$ the indexes such that
\begin{enumerate}
\item $\R_{S,i}(x,\Fs)$ is spectral non solvable for 
$i\leq i_x^{\mathrm{sp}}$,

\item $\R_{S,i}(x,\Fs)$ is solvable for 
$i_x^{\mathrm{sp}}< i \leq \is{x}$,

\item $\R_{S,i}(x,\Fs)$ is over-solvable for 
$\is{x}< i$.
\end{enumerate}
We call $i_x^{\mathrm{sp}}$ and $\is{x}$ 
the \emph{spectral} and \emph{over-solvable cutoffs} respectively. 
\end{definition}

If $\is{x}=0$ (resp. $\is{x}=r$), then all 
the radii are over-solvable (resp. spectral). 
If $i_x^{\mathrm{sp}}=0$ (resp. $i_x^{\mathrm{sp}}=r$), 
then all the radii are solvable or over-solvable (resp. spectral non 
solvable).
If $i_x^{\mathrm{sp}}=\is{x}$, then $\Fs$ has no 
solvable radii.

\begin{remark}\label{rk : solvable not in gamma}
\label{Remark : R_i solvable on Gamma_i}
If $x\in\Gamma_{S,i}(\Fs)$, then the indexes $1,\ldots,i$ 
are all spectral at $x$.
Indeed, if $i$ is over-solvable at $x$, by definition the image of 
$D_{S,i}(x,S)\subseteq X_\Omega$ in $X$ by the canonical 
projection $\pi_{\Omega/K}$ is a virtual open disk $D_i$ containing $x$. 
Since the radii are insensitive to scalar extension of $K$, then $D_i$ is a 
neighborhood of $x$ in $X$ on which $\R_{S,i}(-,\Fs)$ is constant 
(cf. \eqref{eq : D^c} below). 
Hence, when $i$ is oversolvable at $x$ we have 
$x\notin\Gamma_{S,i}(\Fs)$. 
The case where $i$ is solvable at $x$ is also particular: 
it is related to the end-points of the controlling graphs 
(cf. \cite[Lemma 6.5]{NP-I} and \cite[Section 2]{NP-VI}).
\end{remark}

\subsubsection{Constancy disks.}\label{Section constancy disk}
We now provide a criterion to test whether a point lies in 
$\Gamma_{S,i}(\Fs)$. 
Define the $i$-th \emph{constancy disk} of $\Fs$ at $x$ as 
the maximal open disk $D^c_{S,i}(x,\Fs)$ centered at $t_x$ 
and contained in $D(x,S)$ on which $\R_{S,i}(-,\Fs)$ is a constant function. 
With the notations of Definition \ref{Def : can extensions} one has 
$D^c_{S,i}(x,\Fs)=D(x,\Gamma_{S,i}(\Fs))$. Then $D(x)$ and $D_{S,i}(x,\Fs)$ 
are both contained in $D_{S,i}^c(x,\Fs)$:\footnote{By definition the radii are 
stable by localization to $D(x,S)$, therefore we may assume that $X$ is the affine line, in which case the inclusion follows from 
\cite[Equation (2.2)]{NP-I}.}
\begin{equation}\label{eq : D^c}
D(x)\;\cup\;D_{S,i}(x,\Fs)\;\subseteq\;D_{S,i}^c(x,\Fs)\;
\subseteq\;D(x,S)
\end{equation}
The following proposition follows immediately from 
Definition \ref{Def. :: controlling graphs} 
(cf. \cite[Proposition 2.5, iv)]{NP-I}):
\begin{proposition}\label{Prop : D^c}
A point $x\in X$ lies in $\Gamma_{S,i}(\Fs)$ if, and only 
if, $D(x)=D_{S,i}^c(x,\Fs)$. Moreover 
\begin{equation}
\Gamma_{S,i}(\Fs)\;=\;
X-\bigcup_{x\in X_{[1]}} D_{S,i}^c(x,\Fs)\;,
\end{equation}
where $X_{[1]}\subset X$ is the subset of points of type 
$1$. \hfill$\qed$
\end{proposition}

\subsection{Change of triangulation}\label{Change of triangulation}
Here, we discuss how the radii depend on the triangulation. 
Let $S$ and $S'$ be two triangulations of~$X$. From Definition 
\ref{def: Multiradius}, for all $x\in X$, one has
\begin{equation}\label{eq: beh omega by change of tri}
D_{S',i}(x,\Fs)\cap D(x,S)\;=\;D_{S,i}(x,\Fs)\cap D(x,S')\;.
\end{equation}
Indeed, we have either $D(x,S)\subseteq D(x,S')$, or 
$D(x,S')\subseteq D(x,S)$, because they are both disks 
centered at $t_x$. Therefore, we have two possibilities: either the left hand side of 
\eqref{eq: beh omega by change of tri} equals 
$D_{S',i}(x,\Fs)$, or the right hand side equals 
$D_{S,i}(x,\Fs)$ (both conditions can occur simultaneously).
If $\Gamma_S\subseteq\Gamma_{S'}$, then 
$D(x,S')\subseteq D(x,S)$ for all $x\in X$, and we have the following 

\begin{proposition}[\protect{\cite[Equation (2.1)]{NP-II}}]
\label{Prop: A-4 fgh}
Let $S$, $S'$ be two triangulations such that 
$\Gamma_S\subseteq\Gamma_{S'}$. Then, for all $x\in X$ and all 
$i=1,\ldots,r$ one has 
\begin{equation}\label{Prop: A-4 fgh -eq}
D_{S',i}(x,\Fs)\;=\;D_{S,i}(x,\Fs)\cap D(x,S')\;.
\end{equation}
In particular $D_{S',i}(x,\Fs)=D_{S,i}(x,\Fs)$ if $D(x,S)=D(x,S')$. 
Hence for all $i=1,\ldots,r$ one has
\begin{equation}
\R_{S',i}(x,\Fs)\;=\;\min\Bigl(\;1\;,\; 
 f_{S,S'}(x)\cdot \R_{S,i}(x,\Fs)\;
\Bigr)\;,\label{eq : R_S' VS R_S}
\end{equation}
where $f_{S,S'}:X\to [1,+\infty[$ is the function 
associating to $x$ the modulus $f_{S,S'}(x)\geq 1$ of the inclusion of 
disks $D(x,S')\subseteq D(x,S)$ (cf. Definition \ref{def:modulusannulus}).
\hfill $\qed$
\end{proposition}
\begin{proof}
Specializing \eqref{eq: beh omega by change of tri} to the case where 
$\Gamma_S \subset \Gamma_{S'}$ gives the proof.
\end{proof}

\begin{proposition}[\protect{\cite[Lemma 3.14]{NP-II}}]
\label{Prop : A-5 ecp}
Let $S$, $S'$ be two triangulations such that 
$\Gamma_S\subseteq\Gamma_{S'}$. Then  
\begin{equation}\label{eq: change of tri eq 2.32}
\Gamma_{S',i}(\Fs)\;=\; 
\Gamma_{S'}\cup\Gamma_{S,i}(\Fs)\;.
\end{equation}
\end{proposition}
\begin{proof}
The function $f_{S,S'}$ is determined by the following properties: 
\begin{enumerate}
\item $f_{S,S'}$ is constant on each connected component 
of $X-\Gamma_{S'}$ (which is necessarily a virtual open disk);
\item $f_{S,S'}(x)=1$ for all $x\in\Gamma_S\subseteq\Gamma_{S'}$;
\item Let $x\in\Gamma_{S'}-\Gamma_S$. 
 Let $R$ be the radius of $D(x,S)$ in a given coordinate. 
Then 
$f_{S,S'}(x)=R/r(x)$, where $r(x)\leq R$ 
is the radius of the point $x$ with respect to the chosen coordinate on 
$D(x,S)$. 
\end{enumerate}

Let $D$ be a virtual disk such that 
$D\cap(\Gamma_{S'}\cup\Gamma_{S,i}(\Fs))=\emptyset$. Since 
$D\cap\Gamma_{S,i}(\Fs)=\emptyset$, $\R_{S,i}(-,\Fs)$ is 
constant on $D$. Since $D\cap\Gamma_{S'}=\emptyset$ then 
$f_{S,S'}$ is constant on $D$ too. 
Hence, by \eqref{eq : R_S' VS R_S}, the function 
$\R_{S',i}(-,\Fs)$ is constant on $D$. This proves that 
$\Gamma_{S',i}
(\Fs)\subseteq(\Gamma_{S'}\cup\Gamma_{S,i}(\Fs))$.

Conversely 
$\Gamma_S\subseteq\Gamma_{S'}\subseteq
\Gamma_{S',i}(\Fs)$. 
So it is enough to prove that 
$\Gamma_{S,i}(\Fs)\subseteq\Gamma_{S',i}(\Fs)$.
By Proposition \ref{Prop : D^c}, this amounts to proving that for all 
point $x$ of type $1$ one has 
$D_{S',i}^c(x,\Fs)\subseteq D_{S,i}^c(x,\Fs)$.
By definition $D_{S',i}^c(x,\Fs)\subseteq D(x,S')$, 
so we have to prove that if $\R_{S,i}(-,\Fs)$ is constant on a disk 
$D\subseteq D(x,S')$, then so does  $\R_{S',i}(-,\Fs)$. This follows 
from \eqref{eq : R_S' VS R_S} since $f_{S,S'}$ is constant on $D$.
\if{
So it is enough to show that if 
$x\notin\Gamma_{S'}$ then 
$D(x,S')\cap \Gamma_{S',i}(\Fs)=
D(x,S')\cap\Gamma_{S,i}(\Fs)$. We already have the inclusion 
$\subseteq$, so we can assume that 
$\Gamma_{S,i}\cap D(x,S')\neq\emptyset$. Under this condition 
$\R_{S,i}^{\widetilde{\Fs}}$, and also $\R_{S',i}^{\widetilde{\Fs}}$, 
are stable by localization to $D(x,S')$ (this follows from the definition 
of $\R_{S,i}^{\widetilde{\Fs}}$). 

So we are reduced to prove that 
$\R_{S,i}^{\widetilde{\Fs}}=\R_{S',i}^{\widetilde{\Fs}}$ over 
$D(x,S')$. 
Consider a coordinate in the disk $D(x,S)$ and consider on it the two 
functions $\R^{\widetilde{\Fs}}_{S,i}(x)$ and 
$\R^{\widetilde{\Fs}}_{S',i}(x)$. If $R'$ is the radius of the 
sub-disk $D(x,S')\subseteq D(x,S)$ in the chosen coordinate, 
it is clear that for all $y\in D(x,S)$ one has
\begin{equation}\label{eq : refskjkvvkccdhkhd}
\R^{\widetilde{\Fs}}_{S',i}(y)\;=\;
\min(R',\R^{\widetilde{\Fs}}_{S,i}(y))\;.
\end{equation} 
If for some $y\in D(x,S')$ one has 
$\R^{\widetilde{\Fs}}_{S,i}(y)\geq R'$, then both 
functions $\R^{\widetilde{\Fs}}_{S,i}$ and 
$\R^{\widetilde{\Fs}}_{S',i}$ are constant on $D(x,S')$, 
(because $\R^{\widetilde{\Fs}}_{S,i}$ is constant on 
$D_{S,i}(x,\Fs)\subseteq D^c_{S,i}(x,\Fs)$). 
Otherwise for all $y\in D(x,S')$ one has 
$\R^{\widetilde{\Fs}}_{S,i}(y)=\R^{\widetilde{\Fs}}_{S',i}(y)<R'$ 
by \eqref{eq : refskjkvvkccdhkhd}
(see \cite[Proposition 5.10, and Section 8]{NP-I} for more details). 
%
%
}\fi
\end{proof}

\subsection{Localization}
\label{Localization}
One of the major differences between spectral and over-solvable cases 
is that the spectral terms of the filtration 
are preserved by localization, while 
over-solvable ones result truncated.

The pre-image $\pi_\Omega^{-1}(x)$ is independent on $S$ and $X$ 
(cf. Proof of Proposition \ref{Lemma pts of type 4} and Section \ref{Local nature of spectral radii.}), as well as the generic disk 
$D(x)$, the separating index $\is{x}$, and all the disks 
$D_{S,i}(x,\Fs)$ for $i\leq \is{x}$. 
From this remark we obtain 
\begin{proposition}\label{Prop: Localization}
Let $U\subset X$ be an analytic domain. And let $S_U$ be a 
weak-triangulation of $U$. For all $i=1,\ldots,r$ and all $x\in U$ 
one has
\begin{equation}\label{Localization of sol-1}
D_{S_U,i}(x,\Fs_{|U})\cap D(x,S)\;=\;D_{S,i}(x,\Fs)\cap D(x,S_U)
\;.
\end{equation}
In particular $D_{S_U,i}(x,\Fs_{|U})=D_{S,i}(x,\Fs)$ if 
$D(x,S_U)=D(x,S)$, or if $i\leq \is{x}$.
Then for all $x\in U$ and all 
$i=1,\ldots,\is{x}$ one has
\begin{equation}
\omega_{S,i}(x,\Fs)\;=\;
\omega_{S_U,i}(x,\Fs_{|U})\;.
\end{equation}
For all $i=\is{x}+1,\ldots,r$ one has 
\begin{equation}\label{Localization of sol}
\omega_{S_U,i}(x,\Fs_U)\;=\;\left\{
\begin{array}{rcl}
\omega_{S,i}(x,\Fs)&
\textrm{ if }&D_{S,i}(x,\Fs)\subseteq \;D(x,S_U)\\
\Hdr^0(D(x,S_U),\Fs)&\textrm{ if }&D_{S,i}(x,
\Fs)\supset \;D(x,S_U)\;.
\end{array}
\right.\quad\qed
\end{equation}
\end{proposition}
The proof of the following proposition is similar to those of section 
\ref{Change of triangulation}.
\begin{proposition}\label{Prop : immersion}
Let $Y\subseteq X$ be an analytic domain. 
Let $S_X$ and $S_Y$ be triangulations of $X$ and $Y$ 
respectively such that 
$(\Gamma_{S_X}\cap Y)\subseteq\Gamma_{S_Y}$. 
If $y\in Y$, then for all $i=1,\ldots,r$ we have
\begin{equation}\label{Localization of sol-3}
D_{S_Y,i}(x,\Fs_{|Y})\;=\;
D_{S_X,i}(x,\Fs)\cap D(x,S_Y)\;.
\end{equation}
and 
\begin{equation}\label{Feq : Radii truncated by S to S'}
\R_{S_Y,i}(y,\Fs_{|Y})\;=\;\min\Bigl(\;1\;,\; 
 f_{S_X,S_Y}(y)\cdot \R_{S_X,i}(y,\Fs)\;
\Bigr)\;,
\end{equation}
where $f_{S_X,S_Y}:Y\to [1,+\infty[$ is the function 
associating to $y\in Y$ the modulus $ f_{S_X,S_Y}(y)\geq 1$ of the inclusion 
$D(y,S_Y)\subseteq D(y,S_X)$. Hence
\begin{equation}\label{eq : Gamma loc trunc}
\Gamma_{S_Y,i}(\Fs_{|Y})\;=\;(\Gamma_{S_X,i}(\Fs)
\cap Y)\cup\Gamma_{S_Y}\;.
\end{equation}
In particular, if $\Gamma_{S_X}\cap Y=\Gamma_{S_Y}$, then for all $i=1,\ldots,r$ equality \eqref{Feq : Radii truncated by S to S'} becomes 
\begin{equation}\label{eq : Res to Y equal radii f}
\R_{S_Y,i}(y,\Fs_{|Y})\;=\;\R_{S_X,i}(y,\Fs)
\end{equation}
and \eqref{eq : Gamma loc trunc} becomes 
\begin{equation}\label{eq : Res to Y equal graphs f}
\phantom{.}\qquad\qquad
\Gamma_{S_Y,i}(\Fs_{|Y})\;=\;\Gamma_{S_X,i}(\Fs)
\cap Y\;.\qquad\qquad\qed
\end{equation}
\end{proposition}
\if{\begin{proof}
Let $y\in Y$, and let $D:=D(y,S_X)\subseteq X$ be its maximal disk in 
$X$. Since $(\Gamma_{S_X}\cap Y)\subseteq\Gamma_{S_Y}$ both 
radii are stable by localization on $D(y,S)$. So we can assume $X=D$. 
Let $R$ be the radius of $D(y,S_Y)$ in a given coordinate of $D$. 
Then $\R_{}$

\begin{equation}
D(y,S_Y)\;=\; D(y,S_V)\;\subseteq\; 
D(y,S_V^{\mathrm{min}})\;\subseteq\; 
D(y,S_X)\;=\;D\;.
\end{equation}
If $S_V$ is its minimal triangulation $\Gamma_{V}$ denotes its minimal triangulation, 
then $\Gamma_{S_Y,i}(\Fs_{|Y})\;=\;(\Gamma_{S_X,i}(\Fs)
\cap Y)\cup\Gamma_{S_Y}$

Let $D\subseteq X$ be an open disk such that 
$D\cap \Gamma_{S_X}=\emptyset$, and such 
that its boundary $y$ lies in $\Gamma_{S_X}$. 
Then $D\cap Y$
\end{proof}
}\fi

\subsubsection{Localization of the partial heights.}
\label{Sec : Localization of the partial heights.}
Let~$X$ be a quasi-smooth $K$-analytic curve. Let $(\Fs,\nabla)$ be a differential equation on~$X$. The slopes of the partial heights are related to the local 
irregularities (cf. \cite{NP-V}).
Here, we investigate how the Laplacian is modified 
when we localize. 

Recall that, if $K$ is algebraically closed, and if $x$ is a 
point of type $2$, then 
$\H(x)$ is a valued field whose residual field 
$\widetilde{\H(x)}$ 
has transcendence degree $1$ over the residual field 
$\widetilde{K}$ of $K$ and it is the field of functions of a 
projective curve $\mathcal{C}_x$ over $\widetilde{K}$ (cf. \cite[3.3.5.2]{Duc}). 
The genus of the curve $g(\mathcal{C}_x)$ is called the 
genus of the point and it is denoted by $g(x)$. We 
extend this definition to all point of $x$ by $g(x)=0$ if 
the type of $x$ is not $2$.
If $x\notin\partial X$, the $\widetilde{K}$-rational points 
of $\mathcal{C}_x$ are in bijection with the set of 
germs of segments out of $x$. 
In general, the directions out of $x$ describe a Zariski 
open subset of $\mathcal{C}_x$.
In this context, rigid cohomology usually 
considers liftings in $X$ of Zariski open subsets of 
$\mathcal{C}_x$: we call these liftings elementary tubes 
centered at $x$ (see below for the definition). 
The coefficients of rigid cohomology are differential 
equations defined on some unspecified neighborhood in 
$X$ of such a tube and whose radii at $x$ are all 
solvable.

\begin{notation}
Let~$x\in X$ be a point of type~$2$ or~$3$. 

\begin{enumerate}
\item 
Assume that $x\in X-\Gamma_{S}$. Since $X-\Gamma_S$ is a 
disjoint union of virtual open disks, $x$ belongs to one of them and there exists a 
unique virtual closed disk inside $X-\Gamma_{S}$ with 
boundary~$x$. We denote it by~$D_{x}$. There exists 
a unique germ of segment out of~$x$ that does not 
belong to~$D_{x}$. We denote it by~$b_{x,\infty}$, or 
by~$b_{\infty}$ if no confusion may arise. 

\item Let $b$ be a germ of segment out of~$x$.
\begin{enumerate}
\item If $x\in\Gamma_S$, assume that~$b$ does not belong 
to~$\Gamma_{S}$. 
\item If $x\notin\Gamma_S$, assume 
that $b\neq b_\infty$.
\end{enumerate}
Then, the connected component of~$X-\{x\}$ 
containing~$b$ is a virtual open disk. We denote it 
by~$D_{x,b}$, or~$D_{b}$ if no confusion may arise. 
\end{enumerate}
\end{notation}

\begin{definition}\label{Def : tube}
Let~$x\in X$ be a point of type~$2$ or~$3$. 

A connected analytic domain~$V$ 
of~$X$ 
is said to be an \emph{elementary tube centered 
at~$x$} if $V-\{x\}$ is a disjoint union of virtual open 
disks.\footnote{For instance, if $x\in X-\Gamma_{S}$, then $D_{x}$ is an 
elementary tube centered at~$x$.}

It is said to be an elementary tube \emph{adapted 
to~$\Fs$} if, moreover, the radii of~$\Fs$ are constant 
on each of these virtual open disks.
\end{definition}

\begin{notation}\label{Def : D_b -2}
Let~$x\in X$ be a point of type~$2$ or~$3$. Let $x_1,\ldots,x_n$ be the preimages of $x$ in~$X_{\wKa}$. 

\begin{enumerate}

\item For each $i=1,\dotsc,n$, let $g(x_i)$ denote the genus of the point $x_i$. For all $i,j$, we have $g(x_j)=g(x_i)$.
We set $g(x):=ng(x_i)$.

\item Let $V$ be an elementary tube centered at $x\in X$. We denote by $N_V(x)$ the number of germs of segments out of the points $x_1,\ldots,x_n$ that do not belong to $V_{\wKa}$. (For 
instance, if $V$ is a closed virtual annulus with inner radius equal to outer radius and $\partial V=\{x\}$, 
and if 
$V_{\wKa}$ is disjoint union of $n$ flat closed annuli, then 
$N_V(x)=2n$.)



\item If $x\in 
\Gamma_S$, we denote by $N_S(x)$ the number of 
directions out of $x_1,\ldots,x_n$ belonging to 
$\Gamma_{S_{\wKa}}$.

\item Let $V$ be an elementary tube centered at $x\in X$. Let $C_1,\ldots,C_t$ be virtual open annuli such that
\begin{enumerate}
\item for each $j = 1,\dotsc,t$, $C_{j}$ has two distinct boundary points in~$X$, one of which is~$x$ ;
\item $U := V \cup \bigcup_{j=1}^t C_{j}$ is a neighborhood of~$x$ in~$X$.
\end{enumerate}
We set 
\begin{equation}
\chi_c(V)\;:=\;2-2g(x)-N_V(x)=2-2(ng(x_i))-nt\;.
\end{equation}

\item 
For $Y\subseteq X$ any analytic domain in $X$, we set
\begin{equation}\label{eq: hidr little}
\hdr^i(Y,\Fs):=\mathrm{dim}_K\Hdr^i(Y,\Fs_{|Y})\;.
\end{equation}
If $V$ is an elementary tube in $X$, centered at $x\in X$, 
we set 
\begin{equation}
\Hdr^i(V^\dag,\Fs)\;:=\;\varinjlim_{V\subset U}\Hdr^i(U,\Fs_{|U})\;,
\end{equation}
where $U$ runs through the set of open 
neighbourhoods of $V$ in $X$. We set
$\hdr^i(V^\dag,\Fs):=\mathrm{dim}_K
\Hdr^i(V^\dag,\Fs)$.
 
 (For instance, if $x \in X - \Gamma_{S}$, then $\hdr^0(D_{x}^\dag,\Fs)$
is the dimension of the $K$-vector spaces of 
overconvergent solutions of~$\Fs$ on~$D_{x}$, and 
$\Hdr^0(D_x^\dag,\Fs)$ coincides with 
$\varinjlim_{D_x\subset D}\Hdr^0(D,\Fs)$, where $D\subseteq X$ 
runs though the family of virtual open disks containing $D_x$.) 
\end{enumerate}
\end{notation}

\begin{proposition}\label{Prop : Irr local sopp}
Let $x\in X$ be a point of type~2 or~3. Set $r := \rk(\Fs_{x})$. Let $b$ be a germ of segment out of $x$ and let $C_b$ 
be a virtual open annulus whose skeleton represents  $b$. Then, the following equalities hold.
\begin{enumerate}
\item If $x\in\Gamma_{S}$ and $b\subseteq \Gamma_{S}$, then we have 
$\partial_bH_{\emptyset,r}(x,\Fs_{|C_b})=
\partial_bH_{S,r}(x,\Fs)$;
\item If $x\in\Gamma_{S}$ and $b\nsubseteq\Gamma_S$ or if $x\notin\Gamma_{S}$ and $b\ne b_{\infty}$, then we have
\begin{equation}\label{eq : changing radii sptbq99}
\partial_bH_{\emptyset,r}(x,\Fs_{|C_b})\;=\;
\partial_bH_{S,r}(x,\Fs)-\hdr^0(D_b,\Fs)+\deg(b) r\;;
\end{equation}
\item If $x\notin \Gamma_{S}$ and $b = b_{\infty}$, then, we have
\begin{equation}
\partial_{b}H_{\emptyset,r}(x,\Fs_{|C_b})=
\partial_{b}H_{S,r}(x,\Fs)+\hdr^0(D_x^\dag,\Fs)-\deg(b) r\;.
\end{equation}
\end{enumerate}
\end{proposition}
\begin{proof}
All the statements are deduced from 
Proposition \ref{Prop : immersion}.

i) It is immediate. 

ii) Restricting from~$X$ to~$D_b$ leaves the slopes of the radii along~$b$ unchanged. Restricting to~$C_b$ causes the slopes of the radii $\R_{S,i}(-,\Fs)$ that are spectral on~$b$ to increase by~1, whereas it leaves unchanged (and equal to~0) the slopes of the radii that are over-solvable on~$b$. Since the over-solvable radii correspond to the solutions of~$\Fs$ on~$D_{b}$, the result follows.

iii) As above, restricting from~$X$ to~$D_x^\dag$ leaves the slopes of the radii along~$b$ unchanged. The argument continues as before except that the slopes of the spectral radii decrease by~1 this time.
\end{proof}


\begin{corollary}\label{Prop: localization of ddc H to a tube}
Let $x$ be a point of type $2$ or $3$ and let~$V$ be an 
elementary tube centered at~$x$ that is adapted 
to~$\Fs$. Let $S:=\{x\}$ be the weak 
triangulation of $V$. Set $r := \rk(\Fs_{x})$. 
Let $U$ be as in item (iv) of notation 
\ref{Def : D_b -2}. 
Then the following results hold.
\begin{enumerate}
\item If $x\notin\Gamma_S$, then we have
\begin{equation}
dd^cH_{\{x\},r}(x,\Fs_{|U})\;=\;
dd^cH_{S,r}(x,\Fs) 
- r\cdot \chi_{c}(V)
+ \hdr^0(D_x^\dag,\Fs) 
- \sum_{
\sm{b\not\subseteq V\\
b\neq b_\infty}} \hdr^0(D_b,\Fs)\;.  
\end{equation}

\item If $x\in\Gamma_S$, then we have
\begin{equation}
dd^cH_{\{x\},r}(x,\Fs_{|U})\;=\;
dd^cH_{S,r}(x,\Fs) 
+ r\cdot (N_V(x)-N_S(x))
- \sum_{
\sm{b\not\subseteq V\\
b\nsubseteq\Gamma_S}} \hdr^0(D_b,\Fs)\;.
\end{equation}
\end{enumerate}
\hfill$\qed$
\end{corollary}

%

\subsection{Local nature of spectral radii.}
\label{Local nature of spectral radii.}
Let $x\in X$ be a point of type $2$, $3$ or $4$. 
By section \ref{Generic disks.-1}, the pull-back of an 
element $f\in \H(x)$ is a bounded analytic function on $D(x)$. 
More precisely, the real valued semi-norm function $|f(t)|$ is constant on $D(x)$ with value $|f(x)|$ and it 
coincides with the sup-norm $\|f\|_{D(x)}$. This means that for all $\rho\in]0,1]$ the inclusion
\begin{equation}\label{eq : H(x) in B(D(x))}
\H(x)\;\subset\; \mathcal{B}_\Omega(D(x,\rho))
\end{equation}
obtained in this way is isometric (cf. \cite[Lemme 3.1]{Angie}). If we consider the normoid structure (cf. \cite{Banachoid}) of $\O_\Omega(D(x,\rho))$ defined by the family of sup-norms $\{|.|_{t_x,\rho'}\}_{\rho'<\rho}$ (cf. \eqref{norm crho}), then the inclusion
\begin{equation}
\H(x)\subset\O_\Omega(D(x,\rho))
\end{equation}
is an isometry of normoid spaces too: $\forall f\in\H(x)$, $\forall\rho'<\rho$,
\begin{equation}\label{eq : isometry H(x) O}
|f|(x)\;=\;|f|_{t_x,\rho'}\;.
\end{equation}

Let $\M$ be a differential module over $\O_{X,x}$ or $\H(x)$. %
Let $\M_{|D(x)}$ be the pull-back of $\M$ on $D(x)$. Since $\Omega$ is spherically complete (cf. Setting \ref{Setting : Omega}) then, by 
a result of Lazard, $\M_{|D(x)}$ is a free $\O(D(x))$-module with connection (cf. \cite{Lazard}). 
More precisely, one can 
consider $D(x)$ as an $\Omega$-analytic curve with 
empty weak triangulation and $\M_{|D(x)}$ as a differential equation on $D(x)$. Notice that 
the maximal disk here is $D(t_x,\emptyset)=D(x)$.
Definition \ref{def: Multiradius} then applies, 
and it make sense to attribute to $\M_{|D(x)}$ the multiradius 
\begin{equation}
\RR_{\emptyset}(t_x,\M_{|D(x)})\;.
\end{equation}
The vector space of solution $\Hdr^0(t_x,\M_{|D(x)})$ 
is then filtered by 
the sub-spaces $\omega_{\emptyset,i}(t_x,\M_{|D(x)})$, 
corresponding to the disks 
$D_{\emptyset,i}(t_x,\M_{|D(x)})\subseteq D(t_x,\emptyset)=D(x)$ (cf. \eqref{eq : omeaga_S,i and D_S,i}).
The $i$-th radius $\R_{\emptyset,i}(t_x,\M_{|D(x)})$ is the 
inverse of the \emph{modulus} of the inclusion $D_{\emptyset,i}(t_x,
\M_{|D(x)})\subseteq D(x)$.

We now come back to our global sheaf $\Fs$ on $X$.
Since $\Fs_{|D(x)}=\Fs(x)\otimes_{\H(x)}\O(D(x))$, 
the space of convergent solutions around 
$t_x$ only depends on $\Fs(x)$ (cf. \eqref{omega(x,F)}):
\begin{equation}
\Hdr^0(t_x,\Fs_{|D(x)})\;=\;\Hdr^0(x,\Fs)\;.
\end{equation}

By Proposition \ref{Prop: Localization},
 the \emph{spectral} steps of the 
filtration $\omega_{S,i}(x,\Fs)$, $i\leq \is{x}$, 
are intrinsically attached to $x$, 
so that $\Hdr^0(t_x,\Fs_{|D(x)})$ 
carries the \emph{spectral part} of the 
filtration coming from the global definition of the multiradius. More 
precisely 
\if{one can 
consider $D(x)$ as an $\Omega$-rational analytic curve with empty 
triangulation, so that it make sense to attribute to $\Fs(x)$ the 
multiradius 
\begin{equation}
\RR(x,\Fs(x))\;:=\;\RR_\emptyset(t_x,
\Fs(x)\widehat{\otimes}_{\H(x)}\O(D(x)))
\end{equation}
The corresponding $i$-th radius $\R_i(x,\Fs(x))$ is the inverse of the  
modulus of the inclusion $D_{i}(x,\Fs(x))\subseteq D(x)$, where
}\fi
\begin{equation}\label{eq : truncation of disks}
D_{\emptyset,i}(t_x,\Fs_{|D(x)})\;=\;D_{S,i}(x,\Fs)\cap D(x)\;,
\end{equation}
so that the corresponding $i$-th radius results truncated 
(cf. \eqref{eq : D_S,i and D(x, R_S,i)}).
The vector space $\Hdr^0(t_x,\Fs_{|D(x)})$ then results filtered by
\begin{equation}\label{eq : truncation of omega}
\omega_{\emptyset,i}(t_x,\Fs_{|D(x)})\;:=\;\left\{
\begin{array}{rcl}
\omega_{S,i}(x,\Fs)&\textrm{ if }&i\leq \is{x}\\
\omega_{S,\is{x}}(x,\Fs)&\textrm{ if }&i>\is{x}\;.
\end{array}
\right.
\end{equation}
More precisely, for all $i$ we have 
(cf. \eqref{Feq : Radii truncated by S to S'})
\begin{equation}\label{eq : radii from local to global}
\R_{\emptyset,i}(t_x,\Fs_{|D(x)})\;=\;\min\Bigl(\;1\;,\; 
 f_{S_\Omega,\emptyset}(t_x)\cdot \R_{S_\Omega,i}(t_x,\Fs_\Omega)\;
\Bigr)\;.
\end{equation}
and $\R_{S_\Omega,i}(t_x,\Fs_\Omega)=\R_{S,i}(x,\Fs)$ 
(cf. \eqref{eq : insensitive to scalar ext}). 

\begin{remark}\label{Remark : localization explicit}
Let $i=1,\ldots,\mathrm{rank}(\Fs)$. 
Concretely, let us fix a coordinate on the maximal disk $D(x,S)$ and let 
\begin{itemize}
\item[$\bullet$] $R$ be the radius of $D(x,S)$; 
\item[$\bullet$] $R'\leq R$ be the radius of the generic disk 
$D(x)\subseteq D(x,S)$;
\item[$\bullet$] $R_i\leq R$ be the radius of $D_{S,i}(x,\Fs)\subseteq D(x,S)$ (hence $\R_{S,i}(x,\Fs)=R_i/R$);
\item[$\bullet$] $R_i'\leq R'$ be the radius of $D_{\emptyset,i}(t_x,\Fs_{|D(x)})=
D_{S,i}(x,\Fs)\cap D(x)$ (hence $\R_{\emptyset,i}(t_x,\Fs_{|D(x)})=R_i'/R'$);
\end{itemize}
then 
\begin{equation}
R_i'\;=\;\min(R_i,R')\;,\qquad  
f_{S_\Omega,\emptyset}(t_x)\;=\;\frac{R}{R'}
\end{equation}
and
\begin{equation}
\R_{\emptyset,i}(t_x,\Fs_{|D(x)})\;=\;\frac{R_i'}{R'}\;=\;
\min\Bigl(\;1\;,\; 
 \frac{R}{R'}\cdot \frac{R_i}{R}\;
\Bigr)\;.
\;=\;\min\Bigl(\;1\;,\; 
 f_{S_\Omega,\emptyset}(t_x)\cdot \R_{S_\Omega,i}(t_x,\Fs_\Omega)\;
\Bigr)\;.
\end{equation}
In particular, the following hold 
(cf. Definition \ref{Def: spectral index} and 
\eqref{eq : truncation of omega}).
\begin{itemize}
\item[$\bullet$] 
If $i$ is a solvable or an over-solvable index at $x$ for $\Fs$ with respect to $S$, then $\R_{\emptyset,i}(t_x,\Fs_{|D(x)})=1$;
\item[$\bullet$] If $i$ is a spectral non solvable index at $x$ for $\Fs$ with 
respect to $S$, then $\R_{\emptyset,i}(t_x,\Fs_{|D(x)})=\frac{R}{R'}\R_{S,i}(x,\Fs)<1$;
\item[$\bullet$] If $i$ is a spectral (possibly solvable) index at $x$ 
for $\Fs$ with respect to $S$, then the following are equivalent 
\begin{itemize}
\item $i$ separates the $S$-radii of $\Fs$ at $x$;
\item $i$ separates the $S_\Omega$-radii of $\Fs_\Omega$ at $t_x$;
\item $i$ separates the $\emptyset$-radii of $\Fs_{|D(x)}$ at $t_x$. 
\end{itemize}
\end{itemize}
\end{remark}
\begin{corollary}\label{Cor: spectral index separ = intrinsic}
The spectral indexes $i\leq i^{\mathrm{sol}}$ separating the 
$S$-radii of $\Fs$ at $x$ are insensitive to change of weak-triangulation 
nor to localization.\hfill$\qed$
\end{corollary}

\if{
In the remaining of this section we consider a differential module $\M$ 
over $\H(x)$, and we generalize to all points of type $2$, $3$, or $4$, 
of $X$ the decomposition results of Robba (cf. \cite{Ro-I}). Namely we 
decompose $\M$ (\emph{over $\H(x)$}) with respect to the filtration 
$\omega_{\emptyset,i}(t_x,\M_{D(x)})$. 
}\fi

\subsection{Some results about morphisms and duality}
\label{section : morphisms and duality}

\begin{theorem}\label{thm : spectral is exact}
Let $E:0\to\Fs' \to\Fs\to\Fs''\to 0$ be an exact sequence of differential 
modules. Let $D$ be an open disk contained in the generic disk 
$D(x)$. Then the sequence (cf. \eqref{eq : omega(D,Fs)})
\begin{equation}
\Hdr^0(D,E)\;:\;
0\to\Hdr^0(D,\Fs')\to\Hdr^0(D,\Fs)\to\Hdr^0(D,\Fs'')\to 0
\end{equation}
is exact.
\end{theorem}
\begin{proof}
The theorem is placed here for expository reasons, see Proposition
\ref{exactness} for the proof.
\end{proof}
\begin{remark}\label{Remark : an explicit quote to counterexample}
Theorem \ref{thm : spectral is exact} may fails if $D(x)\subset D$ 
(over-solvable case). See Section \ref{section counterexample 
to duality and exactness} for an explicit example.

In general, if $E:0\to\Fs' \to\Fs\to\Fs''$ is an exact sequence and 
$D$ is any open disc in $X_\Omega$, the sequence 
$0\to\Hdr^0(D,\Fs')\to\Hdr^0(D,\Fs)\to\Hdr^0(D,\Fs'')$ 
is exact (cf. Lemma \ref{Lemma :devisage}).

In particular, this means that if $\Fs\to\Fs''$ is any morphism of 
differential equations and if we denote by $\Fs'$ its kernel, 
\emph{the radii of $\Fs$ may only increase in $\Fs''$}. 
More precisely, if $s\in\Hdr^0(D,\Fs)$ is a solution of $\Fs$ over $D$, 
the image of $s$ in 
$\Hdr^0(D,\Fs'')$ is zero if, and only if, $s\in\Hdr^0(D,\Fs')$. 
If we work with differential modules and, in a basis of $\Fs$ over $D$, 
we express $s$ as a vector with coordinates in $\O(D)$, 
then when $s$ is not a solution of $\Fs'$, the image of $s$ in 
$\Fs''$ is a \emph{non zero} vector in $\Hdr^0(D,\Fs'')\subset\Fs''$. 
This means that the image of $s$ in $\Fs''$ has entries with a radius of 
convergence which is larger than or equal to the radius of $D$. 
\end{remark}

\begin{lemma}\label{Lemma : injective then appears with mult}
Let $\Fs'\to\Fs$ be an injective morphism of differential modules. 
Then the radii of $\Fs'$ appears among the radii of $\Fs$. 
Moreover if the radius $R'$ appears with multiplicity $m'$ among the 
radii of $\Fs'$, then it appears with multiplicity $m\geq m'$
among the radii of $\Fs$. 

In particular, the radii of convergence $\R_{S,i}(x,\Fs)$ are invariant by 
isomorphisms of differential equations.
\end{lemma}
\begin{proof} 
For every disk $D\subseteq D(x,S)$ we have 
$\Hdr^0(D,\Fs')\subset\Hdr^0(D,\Fs)$ and the multiplicity of the radii are the dimension of the spaces for $D=D_{S,i}(x,\Fs)$.
\end{proof}

\begin{lemma}\label{Lemma : Hom=0}
Let $\Fs,\Gs$ be two differential equations of ranks $r$, $r'$ respectively. 
Assume that there is a point $x\in X$ such that $\R_{S,1}(x,\Gs)>\R_{S,r}(x,\Fs)$. Then the only morphism $\Gs\to\Fs$ is the zero morphism : 
$\mathrm{Hom}(\Gs,\Fs)=0$.
\end{lemma}
\begin{proof}
Let $\alpha:\Gs\to\Fs$ be a morphism. 
Denote by $\Fs'$ the image of $\alpha$. Then $\R_{S,1}(x,\Fs')\geq
\R_{S,1}(x,\Gs)$. Indeed, $\Gs$ is trivial over the disk $D_{S,1}(x,\Gs)$ and so 
is every of its sub-quotients (cf. item \eqref{ii Lemma : Trivial iff solutions} 
of Lemma \ref{Lemma : Trivial iff solutions}). By Lemma \ref{Lemma : injective then appears with mult} we must have $\Fs'=0$.
\end{proof}

\begin{lemma}\label{Lemma : equality of V_i'+k}
Let $\Fs'\to\Fs$ be an injective morphism of differential modules of 
ranks $r'$ and $r$ respectively. 
Assume that for some $j',j$ one has 
$\omega_{S,j'}(x,\Fs')=\omega_{S,j}(x,\Fs)$. 
Let $i'$ and $i$ be the largest indexes separating the 
filtrations that are smaller than or equal to $j'$ and $j$ respectively (cf. Definition \ref{Def: i separates the filtration}).
Then $r-i=r'-i'$ and for all $k=1,\ldots,r-i$ one has
\begin{equation}
\R_{S,i'+k}(x,\Fs')\;=\;\R_{S,i+k}(x,\Fs)\;,\qquad
\omega_{S,i'+k}(x,\Fs')\;=\;\omega_{S,i+k}(x,\Fs)\;.
\end{equation}
\end{lemma}
\begin{proof}
One has $\omega_{S,i'}(x,\Fs')=\omega_{S,j'}(x,\Fs')=
\omega_{S,j}(x,\Fs)=\omega_{S,i}(x,\Fs)$. 
By \eqref{rank = r-i+1} they satisfy $r'-i'=r-i$. 
A function in this space 
has a well defined radius of convergence which is independent on the 
differential equation of which it is the solution. So the two filtrations of 
$\omega_{S,i'}(x,\Fs')$ coincide.
\end{proof}

\begin{proposition}
\label{Prop. Exact sequence prop separate radii in the right order}
Let $E:0\to\Fs' \to\Fs\to\Fs''\to 0$ be an exact sequence of 
differential equations. Let $r'$, $r$, $r''$ be their respective ranks. 
The following conditions are equivalent:
\begin{enumerate}
\item\label{prop: separates the radii, item i} $\R_{S,r''}(x,\Fs)<\R_{S,1}(x,\Fs')$;
\item\label{prop: separates the radii, item ii} $\R_{S,r''+1}(x,\Fs)=\R_{S,1}(x,\Fs')$ and 
$\R_{S,r''}(x,\Fs)<\R_{S,r''+1}(x,\Fs)$;
\item\label{prop: separates the radii, item iii} $\Hdr^0(x,\Fs')=\omega_{S,r''+1}(x,\Fs)$;
\item\label{prop: separates the radii, item iv} $\R_{S,r''}(x,\Fs'')<\R_{S,1}(x,\Fs')$;
\item\label{prop: separates the radii, item v} $\R_{S,j}(x,\Fs)=\R_{S,j}(x,\Fs'')$ for all $j=1,\ldots,r''$, and 
$\R_{S,r''}(x,\Fs)<\R_{S,r''+1}(x,\Fs)$.
\end{enumerate} 
If one of them holds, the multi-radius of 
$\Fs$ is given by:
\begin{equation}\label{eq : radii ext sep}
\RR_{S}(x,\Fs)\;=\;
(\R_{S,1}(x,\Fs''),\ldots,\R_{S,r''}(x,\Fs''),\R_{S,1}(x,\Fs'),\ldots,
\R_{S,r'}(x,\Fs'))\;.
\end{equation}
Moreover for all disks $D\subseteq D(x,S)$ 
containing $t_x$, the sequence $\Hdr^0(D,E)$ is exact.
\end{proposition}
\begin{proof}
%
(\ref{prop: separates the radii, item i}) $\Rightarrow$ (\ref{prop: separates the radii, item ii}). By Lemma \ref{Lemma : injective then appears with mult} 
the radii $\R_{S,1}(x,\Fs'),\ldots,\R_{S,r'}(x,\Fs')$ of 
$\Fs'$ all appear among those of $\Fs$ and by i) they 
are all larger than or equal to
$\R_{S,r''+1}(x,\Fs)$. Since $r'$ is precisely 
the number of radii of $\Fs$ that are larger than 
or equal to  $\R_{S,r''+1}(x,\Fs)$, we must have
$\R_{S,j}(x,\Fs')=\R_{S,r''+j}(x,\Fs)$, for all $j=1,\ldots,r'$. In particular, $r''+1$ separates the radii of $\Fs$.

(\ref{prop: separates the radii, item ii}) $\Rightarrow$ (\ref{prop: separates the radii, item iii}). The index $r''+1$ separates the radii of $\Fs$.
So $\mathrm{dim}\,\omega_{S,r''+1}(x,\Fs)=r-(r''+1)+1=r'$ by 
\eqref{rank = r-i+1}. 
Let $D:=D_{S,1}(x,\Fs')=D_{S,r''+1}(x,\Fs)$. Then 
$\Hdr^0(x,\Fs')=
\Hdr^0(D,\Fs')\subseteq
\Hdr^0(D,\Fs)=\omega_{S,r''+1}(x,\Fs)$. Since they 
have the same dimension, they coincide.

(\ref{prop: separates the radii, item iii}) $\Rightarrow$ (\ref{prop: separates the radii, item iv}).
One has $\mathrm{dim}\;\omega_{S,r''+1}(x,\Fs)=r'=r-(r''+1)+1$, 
so the index $r''+1$ separates the radii of $\Fs$, since the dimensions 
of the terms of the filtration by the radii form a scale 
(cf. \eqref{rank = r-i+1}).  
So $\R_{S,r''}(x,\Fs)<\R_{S,r''+1}(x,\Fs)=\R_{S,1}(x,\Fs')$. 
Now $\Fs'$ is trivialized by $D:=D_{S,1}(x,\Fs')$, and  
$\Hdr^0(D,\Fs')=\Hdr^0(D,\Fs)$. Then by Lemma
\ref{Lemma :devisage} one has $\Hdr^0(D,\Fs'')=0$, hence 
$\R_{S,r''}(x,\Fs'')<\R_{S,1}(x,\Fs')$.

(\ref{prop: separates the radii, item iv}) $\Rightarrow$ (\ref{prop: separates the radii, item v}) and (\ref{prop: separates the radii, item i}). 
We claim that (\ref{prop: separates the radii, item iv}) implies that for all open disk 
$D\subseteq D(x,S)$ containing $t_x$ the 
sequence $\Hdr^0(D,E)$ is exact. 
Indeed, we have two possible situations : if $D$ strictly 
contains $D_{S,r''}(x,\Fs'')$ one has $\Hdr^0(D,\Fs'')=0$, and hence 
$\Hdr^0(D,\Fs')=\Hdr^0(D,\Fs)$ by left exactness; on the other hand, 
for all $D\subseteq D_{S,1}(x,\Fs')$, the equation $\Fs'$ is trivialized 
by $\O(D)$, hence the exactness follows in this case 
from point  iii) of Lemma \ref{Lemma :devisage}. 
The exactness shows that (\ref{prop: separates the radii, item iv}) implies \eqref{eq : radii ext sep}, hence (\ref{prop: separates the radii, item i}) and (\ref{prop: separates the radii, item v}).

(\ref{prop: separates the radii, item v}) $\Rightarrow$ (\ref{prop: separates the radii, item i}). Let $D$ be an open disk such that 
$D_{S,r''}(x,\Fs)\subset D\subset D_{S,r''+1}(x,\Fs)$ (the inclusions being strict). 
The sequence 
$0\to\Hdr^0(D,\Fs')\to\Hdr^0(D,\Fs)\to\Hdr^0(D,\Fs'')$
verifies $\Hdr^0(D,\Fs'')=0$, hence $\Hdr^0(D,\Fs')=\Hdr^0(D,\Fs)$.
Since $r''+1$ separates the radii of $\Fs$ at $x$, then 
$\mathrm{dim}_\Omega\,\Hdr^0(D,\Fs)=r-(r''+1)+1=r'$ because the 
dimensions of the terms of the filtration by the radii form a scale (cf. \eqref{rank = r-i+1}).
This proves that $\Hdr^0(D,\Fs)=\Hdr^0(x,\Fs')$, so 
$D\subseteq D_{S,1}(x,\Fs')$ and i) holds.
\end{proof}

\begin{proposition}\label{prop : radii of the direct sum}
Let $\Fs=\Fs'\oplus\Fs''$. Then for all $x\in X$  
the set of radii $\R_{S,i}(x,\Fs)$ of $\Fs$ at $x$ with multiplicities 
is the union with multiplicities of the radii of $\Fs'$ and $\Fs''$ at $x$.
In other words if the value $R$ appears $m'$-times in 
$\RR_{S}(x,\Fs')$ and $m''$-times in $\RR_{S}(x,\Fs'')$, then $R$ 
appears $(m'+m'')$-times in $\RR_{S}(x,\Fs)$.
\qed
\end{proposition}
%

Recall that $X$ is connected (cf. Setting 
\ref{Hypothesis : X connected}).
\begin{proposition}\label{Prop : extended by continuity}
Assume that the index $i$ separates the radii of $\Fs$ (at all point of $X$, cf. Definition \ref{definition : i separates the radii}). 
Let $(\Fs',\nabla)\subseteq(\Fs,\nabla)$ 
be a sub-object such that there exists a point $x\in X$ 
satisfying 
\begin{equation}
\omega_{S,1}(x,\Fs')\;=\;\omega_{S,i}(x,\Fs)\;.
\end{equation} 
Then, the rank of $\Fs'$ is $r-i+1$ 
and for all $y\in X$ and all $j=1,\ldots,r-i+1$ we have
\begin{eqnarray}\label{eq : equal at a point implies everywhere}
\omega_{S,j}(y,\Fs')&\;=\;&\omega_{S,j+i-1}(y,\Fs)\;,
\\
\R_{S,j}(y,\Fs')&\;=\;&\R_{S,j+i-1}(y,\Fs)\;.
\label{eq : equal at a point implies everywhere-2}
\end{eqnarray}
\end{proposition}
\begin{proof}
From \eqref{rank = r-i+1} it follows that the rank of the 
locally free sheaf $\Fs'$ is $r-i+1$. 
By Lemma \ref{Lemma : equality of V_i'+k}  
it is enough to prove that
$\omega_{S,1}(y,\Fs')=\omega_{S,i}(y,\Fs)$ for all $y\in X$. By item ii) of
Proposition 
\ref{Prop. Exact sequence prop separate radii in the right order} this is  
equivalent to $\R_{S,1}(y,\Fs')=\R_{S,i}(y,\Fs)$ for all $y\in X$. 
Let $\mathcal{L}\subseteq X$ be the locus on which the 
equality holds. By assumption $\mathcal{L}$ is not empty because $x\in\mathcal{L}$. By continuity of the radii it is a closed subset of 
$X$. By item i) of Proposition 
\ref{Prop. Exact sequence prop separate radii in the right order} 
the condition $\R_{S,1}(y,\Fs')=\R_{S,i}(y,\Fs)$ is 
equivalent to $\R_{S,1}(y,\Fs')>\R_{S,i-1}(y,\Fs)$, hence 
$\mathcal{L}$ is open. Since $X$ is connected, we deduce that 
$\mathcal{L}=X$.
\end{proof}

\begin{proposition}\label{Prop : 1-th radius and dual}
For all $x\in X$ one has $\R_{S,1}(x,\Fs)=\R_{S,1}(x,\Fs^*)$. \footnote{This holds even if $\R_{S,1}(x,\Fs)$ is over-solvable.} The differential equations $\Fs$ and $\Fs^*$ share the same spectral cutoff $i_x^{\mathrm{sp}}$ (cf. Definition \ref{eq : over-solvable cutoff}).
Moreover for all $i=1,\ldots,i_x^{\mathrm{sp}}$ (spectral non solvable case) one has 
$\R_{S,i}(x,\Fs)=\R_{S,i}(x,\Fs^*)$.
\end{proposition}
\begin{proof}
The assertion about $\R_{S,1}$ is equivalent to ``\emph{$\Fs$ is 
trivial over the generic disk $D(x,\rho)$ if, and only if, $\Fs^*$ is}'', 
which is clearly true. 
The second assertion is well known if $x$ is a point of type $2$ 
or~$3$ of the affine line, and it 
will follow  from Section \ref{Robba-deco} in the general case (cf. \eqref{eq : compatibility with duals}).
\end{proof}

\begin{remark}\label{rk : quote counterex for duality}
The statement of Proposition \ref{Prop : 1-th radius and dual} holds even if $\R_{S,1}(x,\Fs)$ is over-solvable. However, it may fail for solvable and over-solvable higher radii. 
In section \ref{An explicit counterexample.} we 
give an example of equation whose over-solvable radii 
are not stable by duality (they become solvable by duality). 
\end{remark}

\subsection{Notes}
If $i=1$, the radius $\R_{S,1}(x,\Fs)$ admits the following classical 
description. Let $D\subseteq X_\Omega$ be a 
maximal disk. Let $\M$ be the restriction of $\Fs$ to $D$. Fix an 
isomorphism $D\simto D_\Omega^-(0,R)$, and consider the empty 
triangulation on it. The disk $D$ is the maximal disk of all its points, and 
for all $x\in D$ one has $\R_{S,i}(x,\Fs)=\R_{\emptyset,i}(x,\M)$. 
By \eqref{eq : R_S,i(x,F)+R_i^F(x)/R} the function 
$\R_{\emptyset,i}(-,\M)$ is determined by the function 
$\R^{\M}_{\emptyset,i}(x)$ appearing in \eqref{eq: R^F DV-Balda}. 
If $i=1$ this last has the following interpretation involving Taylor 
solutions. Since $\Omega$ is spherically complete,  $\M$ is free over 
$\O(D)$. Let $G\in\M_n(\O(D))$ be the matrix of $\nabla:\M\to\M$ 
with respect to some basis of $\M$. 
Namely the columns of $G$ are the images of the elements of 
a basis of $\M$ by $\nabla$. 
If $x\in D$, then a basis of solutions of $\M$ at $t_x$ is 
given by
\begin{equation}\label{eq: Y(T,t_x)}
Y(T,t_x)\;:=\;\sum_{n\geq 0}G_n(t_x)\frac{(T-t_x)^n}{n!}\;.
\end{equation}
where $G_0=\mathrm{Id}$, $G_1=G$, and inductively 
$G_{n+1}=G_n'+G_nG$. 
For all $x\in D$ we set $\R^Y(x):=\liminf_n(|G_n|(x)/|n!|)^{-1/n}$, 
then
\begin{equation}\label{eq : R^Y}
\R_{\emptyset,1}^{\M}(x)\;=\;\min(R,\R^Y(x))\;,\qquad
\R_{\emptyset,1}(x,\M)=\R_{S,1}(x,\Fs)=\min(1,\R^Y(x)/R)\;.
\end{equation}
Exploiting this formula one can prove the continuity of the individual 
radius $\R_{S,1}(-,\Fs)$. This process is used in \cite{Dv-Balda} for 
affinoid domains of the affine line, in \cite{Balda-Inventiones} (compact curves), 
\cite{Potentiel} (curves without boundary).

For $i\geq 2$ such an explicit expression of  $\R_{S,i}(-,\Fs)$ 
is missing. In the practice $\R_{S,i}(-,\Fs)$ behave as a first radius only 
outside 
$\Gamma_{S,1}(\Fs)\cup
\Gamma_{S,2}(\Fs)\cup\cdots\cup\Gamma_{S,i-1}(\Fs)$ (cf. \cite[Proofs in Section 6]{NP-I}, see also \cite[Remark 
2.1.6]{NP-IV}).

\section{Robba's decomposition by the spectral radii over $\H(x)$}
\label{Robba-deco}
In Sections \ref{Robba-deco} and 
\ref{Dwork-Robba (local) decomposition}, we generalize to curves 
Robba's \cite{Ro-I}, and Dwork-Robba's \cite{Dw-Robba} theorems of 
decomposition by spectral radii.\footnote{Notice that for us $K$ is a general complete valued field, while \cite{Ro-I, Dw-Robba} assume the residual field of positive characteristic. In particular, we allow $K$ to be trivially valued.} 
Such decompositions are proved with different methods 
in \cite{Kedlaya-draft}. Methods of \cite{Kedlaya-draft} and 
\cite{Kedlaya-book-2} make a systematic use of spectral norm of the 
connection, which permits to separate the radii thank to the Hensel 
factorization \cite{Robba-Hensel} and \cite[Lemme 1.4]{Ch-Dw}. 
We show here that the original proofs of 
\cite{Ro-I} and \cite{Dw-Robba} can 
be generalized quite smoothly to curves, up to minor 
implementations. The reason of our choice is that 
the point of view of Dwork's generic disks is more adapted to our 
global definition of radii. 
Working with spectral norms would oblige us to translate 
the terminologies, and losing the evocative image of generic disks.

\begin{hypothesis}
\label{Hyp : K alg clo}
\label{Hyp : K alg clos}
In sections \ref{Robba-deco} and
\ref{Dwork-Robba (local) decomposition}, 
until section \ref{Sec : Descent}, 
we assume that $K$ is algebraically closed.
\end{hypothesis}

In Lemma \ref{Lema : descent of M_xrho} 
we will descend the obtained decomposition to a possibly non algebraically closed field $K$.
The hypothesis is 
due to our use of Theorem \ref{thm:bonvois} below, 
and it is unnecessary if $x$ has a neighborhood which is isomorphic to 
an affinoid domain of the affine line. 

The statements as well as the proofs are similar to  
Dwork and Robba's original ones 
(cf. \cite{Dw-II}, \cite{Ro-I}, and \cite{Dw-Robba}), 
as exposed in the recent book of Christol (cf. \cite[Section 5]{Christol-Book}, 
\cite[Section 5.3]{Ch}). 
For the convenience of the reader, and to permit a complete 
understanding of Section~\ref{Dwork-Robba (local) decomposition}, 
we provide a complete set of proofs. This makes the paper self 
contained.

In Sections \ref{Robba-deco} and 
\ref{Dwork-Robba (local) decomposition}, the following notations will be systematically used :
\begin{itemize}
\item $x\in X$ is a point of type 
$2$, $3$ or $4$;
\item  $\M_x$ is a differential module over $\O_{X,x}$;  
\item  $\M$ is a differential module over $\H(x)$.  
\end{itemize}

\subsection{\'Etale maps.}
\label{Section : a description of Ducros etale pam}

We will define \'etale maps from the curve~$X$ to the 
affine line that are adapted to our needs. To achieve 
this, in the following sections, we assume that~$K$ is 
algebraically closed (cf. Hypothesis \ref{Hyp : K alg clos}). 
\if{
Let $b$ be a branch out of a point $x\in X$. We say 
that the skeleton of an annulus $C$ represents $b$ if $C$ is an open or  
semi-open annulus, not containing $x$, whose closure in $X$ contains $x$, and such that the skeleton of $C$ 
represents $b$. We also assume that the closure of $C$ in $X$ is simply connected (no loops).
}\fi
%
%
Recall that $p$ denotes the characteristic exponent of the residual field of $K$ ($p$ equals either~1 or a prime number, cf. Section \ref{Secn1}). We say that a property holds for almost every element of a set if it holds for all
but finitely many of them. 

\begin{theorem}[\protect{\cite[Theorem 3.12]{NP-II}}]\label{thm:bonvois}
Let~$x$ be a point of~$X$ of type~2. Let~$b_{1},\dotsc,b_{t}$ be 
distinct branches out of~$x$ 
(cf. Notation \ref{Notation branch - germ - direction}). 
There exists a star-shaped affinoid 
neighbourhood~$\wt{Y}$ of~$x$ in~$X$ (cf. Definition \ref{Def : star-shaped}), 
an affinoid domain~$W$ of~$\mathbb{P}^{1,\mathrm{an}}_{K}$ 
and a finite \'etale map $\psi \colon \wt{Y} \to W$ such that
\begin{enumerate}
\item\label{i:deg} the degree $[\H(x):\H(\psi(x))]$ is prime to~$p$;
\item\label{i:pointx} $\psi^{-1}(\psi(x))=\{x\}$;
\item\label{i:compx} almost every connected component of 
$\wt{Y}-\{x\}$ is an open unit disk with boundary~$\{x\}$;
\item\label{i:compfx}  almost every connected component of 
$W-\{\psi(x)\}$ is an open unit disk with 
boundary~$\{\psi(x)\}$;
\item\label{i:compiso} for almost every connected 
component~$\wt{C}$ of $\wt{Y}-\{x\}$, the induced 
morphism $\wt{C} \to \psi(\wt{C})$ is an isomorphism;
\item\label{i:isobi} for every $i = 1,\dotsc,t$, the morphism~$\psi$ 
induces an isomorphism between an annulus $C_{b_i}$ representing~$b_{i}$ and an annulus $C_{\psi(b_i)}$ representing~$\psi(b_{i})$ (cf. Notation \ref{Notation branch - germ - direction}).
\item\label{iii : omegaonefree} The sheaf of continuous differentials $\widehat{\Omega}^1_{Y/K}$ on $Y$ is free.\hfill$\Box$
\end{enumerate}
\end{theorem}
\begin{remark}
Item \eqref{iii : omegaonefree} of the above Theorem is not explicitely stated in \cite{NP-II}. However, it is clear that it can be easily obtained by shrinking $Y$, since the other properties remain unchanged by the shrinking.
\end{remark}


Let~$\Cs$ be the set of connected components~$\wt{C}$ of 
$\wt{Y}-\{x\}$ such that~$\psi$ induces an isomorphism 
$\wt{C} \simto \psi(\wt{C})$. Let~$\Bs$ be the set of branches out 
of~$x$ that have no representative in~$\Cs$. Let~$b \in \Bs$. If there 
exists an open annulus $C_b$ whose skeleton 
represents~$b$ such that~$\psi$ induces an 
isomorphism between $C_b$ and its image, then we may shrink $C_b$ and replace it by a semi-open annulus~$\wt{S}_{b}$ with the same properties. 
Otherwise, if $\psi$ does not realize an isomorphism between $C_b$ and its image, then we remove it and 
we set $\wt{S}_{b} = \emptyset$. Now, 
define
\begin{equation}\label{eq ; vV}
\wt{V} = \bigcup_{\wt{C}\in\Cs} \wt{C} \cup \bigcup_{b\in\Bs} 
\wt{S}_{b} \cup \{x\}.
\end{equation} 
It is an affinoid domain of~$\wt{Y}$ that contains 
almost every branch of $X$ out of~$x$. 

Now, we want to cover an entire affinoid neighborhood of $x$ in $X$ with such affinoid domains.
By property~(\ref{i:isobi}) of Theorem \ref{thm:bonvois}, there exists a finite family 
$\wt{Y}_1,\dotsc,\wt{Y}_n$ of affinoid neighborhoods of~$x$,
$\wt{V}_1,\dotsc,\wt{V}_n$ as in \eqref{eq ; vV}, and 
for every $j=1,\ldots,n$, there are \'etale maps 
\begin{equation}\label{RRRRR}
 V_j\;\subseteq\; \wt{Y}_j\xrightarrow[]{\;\;\;\psi_j\;\;\;} W_j \;\subset\;\mathbb{P}^{1,\mathrm{an}}_K
\end{equation} 
as above such that 
\begin{enumerate}
\item[a)] $\wt{V} = \bigcup_{j=1}^n \wt{V}_j$ is an affinoid neighborhood of $x$ in $X$;
\item[b)] for every $j=1,\ldots,n$ the map $\psi_j$ induces an isomorphism between every connected component of $\wt{V}_{j} - \{x\}$ and its image in $W_j$.
\end{enumerate}
Up to shrinking the $V_j$'s we may also assume that
\begin{enumerate}
\item[c)] For every $i,j=1,\ldots,n$ we have $V_j\subset Y_i$.
\end{enumerate}
In the following lemma we generalize the above 
construction to any point of $x\in X$. 
\begin{lemma}\label{Lemma : H(x) stable by d/dt}
Let $x\in X$. 
\if{Let $n$, $\{Y_j\}_{j=1,\ldots,n}, \{W_j\}_{j=1,\ldots,n}, \{V_j\}_{j=1,\ldots,n}, \{\psi_j:V_j\to W_j\}_{j=1,\ldots,n}$ as above.}\fi 
There exists an integer~$n$ and, 
for every $j=1,\dotsc,n$, an affinoid 
domain~$V_{j}$ of~$X$ 
containing~$x$, an affinoid domain~$W_{j}$ 
of~$\mathbb{P}^{1,\mathrm{an}}_K$, 
and a finite \'etale map $\psi_{j} \colon \wt{V}_{j} \to W_{j}$ such that
\begin{enumerate}
\item\label{i:starshapedY} the union
\begin{equation}
V \;=\; \bigcup_{j=1}^n V_{j}
\end{equation}
is a 
star-shaped affinoid neighborhood of~$x$ in~$X$ 
(cf. Definition \ref{Def : star-shaped}). Moreover,  
if $x$ has type~$2$, then with the notation of item c) after \eqref{RRRRR}, for every $i=1,\ldots,n$ we have 
\begin{equation}
V\subseteq Y_i\;;
\end{equation}
\item\label{i:isopsij} for every $y\in V$, there exists~$j$ such that 
$\psi_{j,\Omega}$ induces an isomorphism 
between the generic disks 
\begin{equation}\label{eq : iso of generic disks}
\psi_{j,\Omega} \;\colon\; D(y)\; \xrightarrow[]{\;\;\;\sim\;\;\;}\; D(\psi_j(y)).
\end{equation}
\item\label{iii:isopsij} in case $y=x$, \eqref{eq : iso of generic disks} is an isomorphism for every $j=1,\dotsc,n$.
\end{enumerate}

\end{lemma}
\begin{proof}
If~$x$ is of type~1, 3 or~4, then it has a neighborhood that is 
isomorphic to an affinoid domain of 
$\mathbb{P}^{1,\mathrm{an}}_K$ and the result is obvious.

Let us assume that~$x$ is of type~2 and proceed as we did at the 
beginning of the section. Property~(\ref{i:starshapedY}) holds by 
construction. Thanks to item b) after \eqref{RRRRR}, 
property~(\ref{i:isopsij}) holds for every 
$y \in V-\{x\}$. For the point~$x$ itself, use property (\ref{i:deg}) of 
Theorem \ref{thm:bonvois} and conclude 
by \cite[Lemma 3.22]{NP-II}.
\end{proof}

\begin{remark}\label{local derivation}
For technical reasons in some proofs we need to work 
with the explicit derivation $d/dT$, where $T$ is a 
coordinate of the generic disk $D(y)$. 
We will need to choose $T$ in order that 
$d/dT:\O_\Omega(D(y))\to\O_\Omega(D(y))$ 
stabilizes the sub-rings $\O_{X,y}$, and $\H(y)$. 
Such a particular coordinate is given by the pull-back by $\psi_j$ 
of a $K$-rational coordinate on $W_j$.
More precisely, for all $j=1,\ldots,n$, we
fix a coordinate~$T_{j}$ on~$W_{j}$ and we call 
$\omega_j$ the image of $1\otimes dT_j$  in 
$\widehat{\Omega}^1_{Y_j/K}(Y_j)$ by the canonical 
isomorphism (recall that 
$\widehat{\Omega}^1_{Y_j/K}$ is free by item 
\eqref{iii : omegaonefree} of Theorem \ref{thm:bonvois})
\begin{equation}
\O_{Y_j} \widehat{\otimes}_{\O_{W_j}} 
\widehat{\Omega}^1_{W_j/K} \;\xrightarrow[]{\;\;\;\sim\;\;\;} \;
\widehat{\Omega}^1_{Y_j/K}\;. 
\end{equation}
Then, we have an isomorphism
\begin{equation}
\nu_j\;:\;\widehat{\Omega}^1_{Y_j/K}
\xrightarrow{\;\;\;\sim\;\;\;}\O_{Y_j}\cdot\omega_j
\xrightarrow{\;\;\;\sim\;\;\;}\O_{Y_j}
\end{equation}
identifying a differential form 
$f\omega_j\in\O_{Y_j}\omega_j$ with $f\in\O_{Y_j}$.
We define~$d_j$ 
as the derivation on~$\O(Y_j)$ defined by the 
composite map $d_j:\O(Y_j)\xrightarrow{\;\;d\;\;}
\widehat{\Omega}^1_{Y_j/K}(Y_j)
\xrightarrow{\;\;\nu_j\;\;}\O(Y_j)$, where $d$ is the 
canonical derivation. 

If $y\in Y_j$ is a point of type $2$, 
$3$, or $4$, we may repeat this construction for 
$\O_{X,y}$, $\H(y)$, $\b_{\Omega}(D(y,\rho))$, and $\O_{\Omega}(D(y,\rho))$, for all  $\rho\in]0,1]$ (cf. \eqref{D(x,R)}). 
We indicate by the same symbol $d_j$ the 
derivation obtained, as before, 
to be the pull-back of $d/dT_j$. 
We then have the following commutative diagram
\begin{equation}\label{eq : d_j acts}
\xymatrix{
\O(Y_j)\ar[d]^{d_j}\ar[r]&
\O_{X,y}\ar[d]^{d_j}\ar[r]&
\H(y)\ar[d]^{d_j}\ar[r]&
\b_\Omega(D(y,\rho))\ar[d]^{d_j}\ar[r]&
\O_\Omega(D(y,\rho))\ar[d]^{d_j}\;\phantom{,}\\
\O(Y_j)\ar[r]&\O_{X,y}\ar[r]&\H(y)\ar[r]&\b_\Omega(D(y,\rho))\ar[r]&\O_\Omega(D(y,\rho))\;.}
\end{equation}
Now, if moreover
\begin{equation}
y\in V_j
\end{equation}
then, thanks to \eqref{eq : iso of generic disks}, 
we can identify  the disks 
$D(y,\rho)$ and $D(\psi_j(y),\rho)$ and $T_j$ with 
$T_j\circ\psi_{\Omega,j}$. With this identification, 
we have
\begin{equation}
d_j\;=\;d/dT_j
\end{equation}
as an operator on $\b_\Omega(D(y,\rho))$ and $\O_\Omega(D(y,\rho))$.

Analogously, in the situation of property
(\ref{iii:isopsij}) of Lemma \ref{Lemma : H(x) stable by d/dt}, let 
$Y\subseteq V$ be an affinoid neighborhood of $x$ in $X$. 
Then, for all $j=1,\ldots,n$, $\omega_j$ is a generator of $\widehat{\Omega}^1_{Y/K}$.
In particular, the rings $\O_\Omega(D(x,\rho))$, $\mathcal{B}_\Omega(D(x),\rho)$, $\O_{X,x}$, 
$\H(x)$, and $\O(Y)$ are stable under all the derivations $d_1,\ldots,d_n$. 
\end{remark}
\begin{definition} \label{def : elementary neigh}
We say that an affinoid neighborhood $Y$ of $x$ in $X$ 
is $\psi$-\emph{admissible} if $Y\subseteq V$, where 
$V$ is the star-shaped affinoid neighborhood of $x$ obtained in Lemma \ref{Lemma : H(x) stable by d/dt}.
\end{definition}

\subsubsection{Derivations.}
\label{Section 3.1.1 : derivations}
Recall that we only consider derivations of 
$\b_\Omega(D(x,\rho))$ 
coming from $\O(D(x,\rho))$ because the space of 
bounded derivations of $\b_\Omega(D(x,\rho))$ is 
possibly too large (cf. Section 
\ref{Section: Bounded Omega 1}).

Endow $\O_L(D(x,\rho))$ 
with the Normoid structure defined by the 
family of norms $\{|.|_{t,\rho'}\}_{\rho'<\rho}$. In the 
framework of Normoid spaces it makes sense to speak of 
bounded operators (cf. \cite{Banachoid}).
We know that any bounded derivation of 
$\O_\Omega(D(x,\rho))$ 
stabilizes $\b_\Omega(D(x,\rho))$ 
and induces on it a bounded derivation (cf. Proposition 
\ref{Prop: cont der on b}). In particular, any two 
coordinates $T,\widetilde{T}$ of $D(x,\rho)$ produce 
bounded $\Omega$-linear
derivations $d/dT$ and $d/d\widetilde{T}$ of 
$\b_\Omega(D(x,\rho))$ such that 
\begin{equation}
\b_\Omega(D(x,\rho))\lr{\frac{d}{dT}}\;=\;
\b_\Omega(D(x,\rho))\lr{\frac{d}{d\widetilde{T}}}\;.
\end{equation}
Analogously, it follows that if $d,\widetilde{d}$ are two derivations 
generating the space of bounded 
$\Omega$-derivations of 
$\O_\Omega(D(x,\rho))$ we have
$\b_\Omega(D(x,\rho))\lr{d}=
\b_\Omega(D(x,\rho))\lr{\widetilde{d}}$. We now prove a similar property for the fields $\O_{X,x}$ and $\H(x)$.
\begin{lemma}\label{Lemma : change of coordinate}
Let $A$ be one of the fields
$\O_{X,x}$, $\H(x)$. 
Let $T$ be a coordinate of $D(x,\rho)$ such that 
$d/dT$ stabilizes $A$ (such a coordinate exists by 
Remark \ref{local derivation}). 
Then, the following hold
\begin{enumerate}
\item\label{i Lemma : change of coordinate} 
 $d/dT$ generates the 
$A$-module of bounded $K$-derivations of $A$.
\item\label{ii Lemma : change of coordinate} 
For any two coordinates $T,\widetilde{T}$ 
of $D(x,\rho)$ with the above property we have $A\lr{\frac{d}{dT}}=A\lr{\frac{d}{d\widetilde{T}}}$.
\item\label{iii Lemma : change of coordinate} 
Let $d:\O_\Omega(D(x,\rho))\to \O_\Omega(D(x,\rho))$ 
be a bounded derivation stabilizing $A$. 
Then $d$ has the form $d=f\frac{d}{dT}$ with $f\in A$.
Moreover, $d$ generates 
the $A$-module of bounded $K$-derivations if, and 
only if, $f$ is non zero. In this case, we have $A\lr{d}=A\lr{\frac{d}{dT}}$.
\end{enumerate}
\end{lemma}
\begin{proof}
\eqref{i Lemma : change of coordinate}
Each derivation 
$d_j=d/dT_j$ of Remark \ref{local derivation} generates
the space of bounded $K$-derivations of $A$, 
which implies $d/dT=hd_j$, 
for some non zero $h\in A$. Since $A$ is 
a field, $h$ is invertible and hence $d/dT$ generates as 
well the space of derivations. 

Item \eqref{ii Lemma : change of coordinate} follows 
immediately from \eqref{i Lemma : change of coordinate}.

\eqref{iii Lemma : change of coordinate} Since $d$ is a 
bounded endomorphism of $\O_\Omega(D(x,\rho))$, its 
restriction to $\b_\Omega(D(x,\rho))$ is bounded too (cf. 
item \eqref{ii Prop: cont der on b} of Proposition 
\ref{Prop: cont der on b}). Since $d$ stabilizes globally 
$\H(x)$ and the map 
$\H(x)\to\b_\Omega(D(x,\rho))$ is an 
isometry, then $d$ is a bounded $K$-derivation of 
$\H(x)$. 
By item \eqref{i Lemma : change of coordinate}, we 
have $d=f\frac{d}{dT}$, for some $f\in\H(x)$. Item 
\eqref{iii Lemma : change of coordinate} for 
$A=\H(x)$ follows. If $A=\O_{X,x}$, we may  
write $d=f\frac{d}{dT}=fh d_j$, where 
$h\in\O_{X,x}$ is the function we have found in item 
\eqref{i Lemma : change of coordinate}. Since $T_j\in A$, we can write $d_j(T_j)=1$ and 
$f h d_j(T_j)=fh\in \O_{X,x}$. Since $h$ is invertible in $\O_{X,x}$, it follows that $f\in\O_{X,x}$ and the claim follows.
\end{proof}

\subsection{Norms on differential operators.}
\begin{definition}\label{eq : op norm on a group}
Let $(G,\|.\|)$ be an abelian 
group endowed with an ultrametric norm $\|.\|$. 
If $\varphi:G\to G$ is a linear map we set 
$\|\varphi\|_{op,G}:=\sup_{g\in G-\{0\}}\frac{\|
\varphi(g)\|}{\|g\|}$.
\end{definition}

The following notation will be adopted from now on 
in Section \ref{Robba-deco}.
\begin{notation}\label{Notation: T=coordinate on D(x)}
We fix a 
\begin{equation}
\rho\in]0,1]\;.
\end{equation}
Recall that we denote by $D(x,\rho)$ the sub-disk of $D(x)$ with 
modulus $\rho^{-1}$ (cf. \eqref{D(x,R)}). Recall that $\b_\Omega(D(x,\rho))$ 
is endowed with the sup-norm $\|.\|_{D(x,\rho)}$ (cf. 
\eqref{eq : sup-norm}).
Let $T$ be a coordinate on $D(x)$ and 
\begin{equation}
d\;:=\;d/dT
\end{equation}
the corresponding derivation. 
We choose the coordinate $T$ such that 
$d/dT$ stabilizes (globally) 
the field $\H(x)\subset\b_\Omega(D(x))$, 
this is possible by Section 
\ref{Section : a description of Ducros etale pam}.

Let $P$ be an element of the Weyl algebra $\b_\Omega(D(x,\rho))\langle d\rangle$ (cf. Section \ref{section : Differential  modules and trivial submodules}). For all $\rho \in]0, 1]$ we denote the operator norm of $P$ as an endomorphism of 
$\b_{\Omega}(D(x,\rho))$ by
\begin{equation}\label{eq: op norm P bDxrho}
\|P\|_{op,D(x,\rho)}\;:=\;
\|P\|_{op,\b_\Omega(D(x,\rho))}\;:=\;
\sup_{0\neq f \in 
\mathcal{B}_\Omega(D(x,\rho))}
\frac{\|P(f)\|_{D(x,\rho)}}{\|f\|_{D(x,\rho)}}\;.
\end{equation}
This definition only depends on $P$ as an endomorphism 
of $\b_{\Omega}(D(x,\rho))$ and not on the chosen 
coordinate, nor on the derivation $d$ 
used to represent $P$. Section \ref{Section 3.1.1 : derivations} 
shows that the choice of another
coordinate gives rise to the same Weyl algebra 
which is, 
in this sense, an intrinsic sub-ring of 
$\mathrm{End}(\b_\Omega(D(x,\rho)))$. 

For all $\rho'\leq \rho$, we may as well identify 
$P$ to an endomorphism of $\b_\Omega(D(x,\rho'))$ 
and consider $\|P\|_{op,D(x,\rho')}$. 
\end{notation}

\begin{remark}\label{Remark : d/dT coord on D(x)}
\label{remark : def op norm d on bounded}
\label{Remark : decreasing norm even though d_1}
Let $r(x)$ be the radius of $D(x)$ with respect to the 
coordinate $T$. Then, the radius of $D(x,\rho)$ is $\rho\cdot r(x)$, 
and the norm of $(d/dT)^k$ is given by 
\begin{equation}
\|(d/dT)^k\|_{op,D(x,\rho)}\;=\;|k!|(\rho\cdot r(x))^{-k}\;=\;
|k!|\cdot\|d/dT\|^{k}_{op,D(x,\rho)}\;.
\end{equation}
More generally, if $P=\sum_{k=0}^nf_k\cdot (d/dT)^k$, with 
$f_k\in\b_\Omega(D(x,\rho))$, then we may proceed for instance as in \cite[Proposition 4.3.1]{Ch} to prove that
\begin{equation}\label{eq : explicit norm operator}
\Bigl\|\sum_{k=0}^nf_k\cdot (d/dT)^k\Bigr\|_{op,D(x,\rho)}\;=\;
\max_{k=0,\ldots,n}|k!|\cdot\|f_k\|_{D(x,\rho)}\cdot 
(\rho\cdot r(x))^{-k}\;.
\end{equation}
This shows in particular 
that the operator norm $\|P\|_{op,D(x,\rho)}$ 
increases when $\rho\in]0,1]$ decreases. This 
increasing property doesn't follow directly from the 
definition \eqref{eq: op norm P bDxrho} since, for 
$\rho'<\rho$, the restriction map $\b_{\Omega}(D(x,
\rho))\to\b_{\Omega}(D(x,\rho'))$ is not isometric.
In the sequel, we use  coordinates 
to obtain explicit computations of intrinsic notions, 
as those just mentioned (cf. \eqref{eq : explicit norm operator}).
\end{remark}

\subsection{Topologies on differential modules}
Denote by  
$\mathcal{T}_{\rho}$ the topology on the Weyl algebra
$\H(x)\langle d\rangle$ induced by the operator norm
$\|.\|_{op,D(x,\rho)}$.
For all $q\geq 1$, we consider $\H(x)\langle d\rangle^q$ as endowed 
with the norm : 
\begin{equation}
\|(P_1,\ldots,P_q)\|_{op,D(x,\rho)} \;:=\;
\max_{i=1,\ldots,q}\|P_i\|_{op,D(x,\rho)}\;.
\end{equation}
Each differential module $\M$ over $\H(x)$ then 
acquires automatically a canonical topology $\mathcal{T}_{\rho}(\M)$  
as follows. If 
\begin{equation}
\pres:\H(x)\langle d\rangle^q\to\M\to 0
\end{equation}
is a presentation of $\M$ (i.e. a surjective morphism of 
$\H(x)\langle d\rangle$-modules), then we endow $\M$ with the 
quotient semi-norm 
\begin{equation}\label{eq: norm-Psi}
\|m\|_{\Psi}:=
\inf_{\Psi(P_1,\ldots,P_q)=m}\|(P_1,\ldots,P_q)\|_{op,D(x,\rho)}\;=\;
\inf_{\Psi(P_1,\ldots,P_q)=m}\;\;
\max_{i=1,\ldots,q}\|P_i\|_{op,D(x,\rho)}\;.
\end{equation}
\begin{proposition}[\protect{\cite[Proposition 7.7]{Christol-Book}}]
The semi-norms on $\M$ relative to two presentations are always 
equivalent. In particular, the topology $\mathcal{T}_{\rho}(\M)$ induced by $\|.\|_\Psi$ is independent on 
$\Psi$. 
\end{proposition}
\begin{proof}
Let $\pres_1:\H(x)\langle d\rangle^{q_1}\to\M\to 0$ and 
$\pres_2:\H(x)\langle d\rangle^{q_2}\to\M\to 0$ be two presentations of  $\M $. Since $\H(x)\langle d\rangle^{q_1}$ is free, we may find an 
$\H(x)\langle d\rangle$-linear map 
$\Psi_{1,2}:\H(x)\langle d\rangle^{q_1}\to\H(x)\langle d\rangle^{q_2}$
such that $\Psi_2\circ\Psi_{1,2}=\Psi_{1}$. The map $\Psi_{1,2}$ is the multiplication by a matrix with coefficients in $\H(x)\langle d\rangle$ and it is therefore bounded, with norm $N_{1,2}$. Hence, for every $m\in \M$ we have
\begin{eqnarray}
\|m\|_{\Psi_2}&\;:=\;&\inf_{\Psi_2(v_2)=m}\|v_2\|_{op,D(x,\rho)} \\
&\leq&
\inf_{\sm{\Psi_2(v_2)=m\\v_2=\Psi_{1,2}(v_1)}}\quad\|\Psi_{1,2}(v_1)\|_{op,D(x,\rho)}\;=\;
\inf_{\Psi_1(v_1)=m}\|\Psi_{1,2}(v_1)\|_{op,D(x,\rho)}
\\
&\leq&\inf_{\Psi_1(v_1)=m}N_{1,2}\|(v_1)\|_{op,D(x,\rho)}\;=\;
N_{1,2}\|m\|_{\Psi_1}
\end{eqnarray}
By exchanging the roles of $\Psi_1$ and $\Psi_2$, we obtain an inequality in the other direction, and thus the equivalence of the two semi-norms.
\end{proof}

\subsection{Decomposition over $\H(x)$ by the bounded solutions on 
$D(x,\rho)$}
The topology $\mathcal{T}_{\rho}(\M)$ on $\M$ is 
possibly not Hausdorff. Denote by $\overline{\M}$ 
(resp. $\M^{[1]}$) 
the Hausdorff quotient of $\M$ (resp. the closure of $0\in\M$). 
The sub-module $\M^{[1]}\subseteq\M$ is a differential sub-module 
because $\|P\cdot m\|_{\Psi}\leq\|P\|_{op,D(x,\rho)}\cdot 
\|m\|_{\Psi}$, 
for all $P\in \H(x)\langle d\rangle$, $m\in \M$. 
One has an exact sequence of differential modules
\begin{equation}\label{eq: Gamma}
0\to\M^{[1]}\to\M\xrightarrow{\gamma}\overline{\M}\to0\;.
\end{equation}

\begin{remark}\label{iterated topologies ... non trivial}
Let 
\begin{equation}
\|\overline{m}\|_{\Psi}^-\;:=\;\min_{\gamma(m)=
\overline{m}}\|m\|_{\Psi}
\end{equation}
be the quotient norm on $\overline{\M}$ 
induced by $\|.\|_\Psi$.
The corresponding topology $\overline{\mathcal{T}_{\rho}(\M)}$ on 
$\overline{\M}$ is the canonical one :
$\overline{\mathcal{T}_{\rho}(\M)}=
\mathcal{T}_{\rho}(\overline{\M})$. This follows
 from the independence of the presentation by choosing 
$\gamma\circ\Psi$ as a presentation of $\overline{\M}$.
On the other hand, the canonical topology $\mathcal{T}_{\rho}(\M^{[1]})$ 
of $\M^{[1]}$ is often \emph{non trivial}, so it differs from the 
(trivial) topology induced  by $\mathcal{T}_{\rho}(\M)$  and 
$\M^{[1]}$ can be different from $(\M^{[1]})^{[1]}$ 
(cf. section \ref{Decomposition over H(x) by the analytic solutions.}). 
This is related to the fact that the functor 
$\M\mapsto\M^{[1]}$ is not exact.
\end{remark}
We also have a dual notion.
Let $\M^*$ be the dual module of $\M$. Define
\begin{equation}
\M^{b}\;:=\;(\overline{\M^*})^*\;,\qquad\quad\M_{[1]}\;:=\;
((\M^*)^{[1]})^*\;,
\end{equation}
in order to have another exact sequence of differential modules
\begin{equation}
0\to\M^{b}\to\M\to\M_{[1]}\to0\;.
\end{equation}

The following proposition shows that the module $\M^b$ 
controls the bounded solutions of $\M$ on $D(x,\rho)$, while $\overline{M}$ controls those of $\M^*$.
\begin{proposition}\label{Prop: M^b and M^bar trivialized by Bounded}
The following hold 
\begin{enumerate}
\item\label{i Prop: M^b and M^bar trivialized by Bounded} \emph{$\M^b$ and $\overline{\M}$ are 
both trivialized by $\b_\Omega(D(x,\rho))$};
\item We have 
\begin{eqnarray}
\Hdr^0(\M,\mathcal{B}_\Omega(D(x,\rho)))&\;=\;&
\Hdr^0(\M^b,\mathcal{B}_\Omega(D(x,\rho)))\;,\label{eq: Sol M = Sol Mbar}\\
\mathrm{Hom}_{\H(x)\langle d\rangle}(\M,\b_\Omega(D(x,\rho)))
&\;=\;&
\mathrm{Hom}_{\H(x)\langle d\rangle}(\overline{\M},
\b_\Omega(D(x,\rho)))\;.\label{eq: Sol M = Sol Mbar-2}
\end{eqnarray} 
\end{enumerate}
\end{proposition}
\begin{proof} 
The assertions \eqref{eq: Sol M = Sol Mbar} and 
\eqref{eq: Sol M = Sol Mbar-2} are equivalent by duality 
\eqref{eq : omega = hom}. We prove \eqref{eq: Sol M = Sol Mbar-2}.
Let $e_1,\ldots,e_q$ denote the canonical basis of $\H(x)\lr{d}^q$. 

If $m=\Psi(\sum_iP_i(d)e_i)\in\M^{[1]}$ and if 
$s\in\mathrm{Hom}_{\H(x)\lr{d}}(\M,\b_\Omega(D(x,\rho)))$ 
one has
\begin{eqnarray}
\|s(m)\|_{D(x,\rho)}&\;=\;&\|s\circ\Psi(\sum P_i(d)e_i)\|_{D(x,\rho)}
\;=\;\|\sum P_i(d)s\circ\Psi(e_i)\|_{D(x,\rho)}\\
&\;\leq\;&
(\max_i\|P_i\|_{op,D(x,\rho)})
\cdot(\max_i\|s\circ \Psi(e_i)\|_{D(x,\rho)}).
\nonumber
\end{eqnarray}
Since 
$0=\|m\|_{\Psi}=\inf_{\Psi(\sum_iP_ie_i)=m}
\max_i\|P_i\|_{op,D(x,\rho)}$ 
one obtains $\|s(m)\|_{D(x,\rho)}=0$,  hence $s(m)=0$. 
So \eqref{eq: Sol M = Sol Mbar-2} holds. 

Now, let us prove that 
$\overline{\M}$ is trivialized by $\b_\Omega(D(x,\rho))$. 
To this aim, let us choose a presentation $\Psi$ such that 
$\{\Psi(e_i)\}_{i=1,\ldots,q}$ is a basis of $\M$ as 
$\H(x)$-vector space and 
the image $\overline{\bs{m}}:=
\{\overline{m}_1,\ldots,\overline{m}_r\}$ in 
$\overline{\M}$ of the first $r$ vectors  
$\{\Psi(e_{1}),\ldots,\Psi(e_r)\}$ is a basis of $\overline{\M}$.
Since $\overline{\M}$ is a vector space 
over the complete valued field $\H(x)$, 
all norms on $\overline{\M}$ are equivalent.  
Therefore, there exists a constant $C\in\mathbb{R}$ such that the sup-norm 
\begin{equation}
\|\sum_if_i\overline{m}_i\|_{\overline{\bs{m}}}
\;:=\;
\max_i|f_i(x)|
\end{equation} 
satisfies 
\begin{equation}
\|\cdot\|_{\overline{\bs{m}}} \;\leq\;C\cdot \|\cdot\|^-_\Psi\;.
\end{equation}
Let us test this inequality on the derivatives of $\overline{\bs{m}}$, by \eqref{eq : explicit norm operator} one finds
\begin{eqnarray}
\Bigl\|\frac{1}{n!}(\frac{d}{dT})^n(\overline{m}_i)\Bigr\|_{\overline{\bs{m}}} &\;\leq\;&
C\cdot\Bigl\|\frac{1}{n!}(\frac{d}{dT})^n(\overline{m}_i)\Bigr\|^-_\Psi\\
&\;=\;&
C\cdot\Bigl(\,\inf_{\gamma(\Psi(\sum_jP_j e_j))=
\frac{1}{n!}(\frac{d}{dT})^n(\overline{m}_i)}\;\;\max_j\|P_j\|_{op,D(x,\rho)}\Bigr)\\
&\;\leq\;& \label{eq : bound for d/dT)^n}
C\cdot\Bigl\|\frac{1}{n!}(\frac{d}{dT})^n\Bigr\|_{op,D(x,\rho)}\;=\;
C\cdot (\rho\cdot r(x))^{-n} \;,
\end{eqnarray}
where the last inequality follows by 
choosing a particular choice of the tuple $(P_k)_k$ as
$P_i=\frac{1}{n!}(\frac{d}{dT})^n$ 
and $P_j=0$ for $j\neq i$. 

Let $\nabla:\overline{\M}\to\overline{\M}$ be the multiplication by $d/dT$ in $\overline{\M}$. 
Let $G_n=(g_{n;i,j})_{i,j}$ be the $r\times r$ square matrix with entries in $\H(x)$ 
whose columns are the coordinates of the images by $\nabla^n$ of the basis 
$\overline{m}_1,\ldots,\overline{m}_r$. 
Then $\|G_n\|_{D(x,\rho)}=
\max_{i,j}\|g_{n;i,j}\|_{D(x,\rho)}$. 
The bound \eqref{eq : bound for d/dT)^n} implies
\begin{equation}
\frac{\|G_n\|_{D(x,\rho)}}{|n!|}=
\max_{i=1,\ldots,r}
\left\|\frac{\nabla^n}{n!}(\overline{m}_i)\right\|_{\overline{\bs{m}}}
=\max_{i=1,\ldots,r}\left\|\frac{1}{n!}(\frac{d}{dT})^n(\overline{m}_i)\right\|_{\overline{\bs{m}}}
\leq C(\rho\cdot r(x))^{-n}\;.
\end{equation} 
Since the entries $g_{n;i,j}$ of $G_n$ lie in $\H(x)$ one has 
$\|g_{n;i,j}\|_{D(x,\rho)}=|g_{n;i,j}(x)|=|g_{n;i,j}(t_x)|_\Omega$.
This proves that the Taylor solution 
$Y(T):=\sum_{n\geq 0}G_n(t_x)(T-t_x)^n/n!\in \Omega[\![T-t_x]\!]$ 
of the dual module $\overline{\M}^*$ belongs to 
$\b_\Omega(D(x,\rho))$. Equivalently $\overline{\M}^*$ is 
trivialized by $\b_\Omega(D(x,\rho))$, and then so is
$\overline{\M}$.
\end{proof}

\subsection{Decomposition over $\H(x)$ by the analytic solutions}
\label{Decomposition over H(x) by the analytic solutions.}
As observed in Remark \ref{iterated topologies ... non trivial} the 
canonical topology of $\M^{[1]}$ is often non-trivial, 
so that one can 
repeat the construction and define inductively 
\begin{eqnarray}
\M^{[i+1]}&:=&(\M^{[i]})^{[1]}\\
\M_{[i+1]}&:=&(\M_{[i]})_{[1]} \\
(\M_{[i]})^*&=&(\M^*)^{[i]}\;.
\end{eqnarray}
The process ends because $\M$ is finite dimensional.
One obtains a finite sequence of surjective maps 
\begin{equation}\label{Sequence of M_[i]}
\M:=\M_{[0]}\to\M_{[1]}\to\cdots \to\M_{[k-1]}\to\M_{[k]}\;,
\end{equation}
where $\M_{[k+1]}=\M_{[k]}$, and $\M_{[i+1]}\neq\M_{[i]}$ for 
all $i=0,\ldots,k-1$. \footnote{It can be shown that  
$\M_{[i]}$ takes into account the solutions 
with \emph{logarithmic growth of order $i-1$} of $\M$, so that it follows that all solutions of 
$\M$ have at most logarithmic growth of order $r=\mathrm{dim}_{\H(x)}\M$ (cf. \cite{Ch}).}

\begin{definition}
Let $k\geq 1$ be the smallest natural number 
such that 
$\M_{[k+1]}=\M_{[k]}$.
Denote by $\M^{< \rho}$ the module 
$\M_{[k]}$, and set
 $\M^{\geq \rho}:=\mathrm{Ker}(\M\to\M_{[k]})$. We have an exact sequence of differential modules
\begin{equation}\label{eq : Robba deco}
0\to\M^{\geq \rho}\to\M\to\M^{< \rho}\to 0\;.
\end{equation}
\end{definition}
\begin{remark}\label{remark : M^rho has no sol in b(t,rho)}
Since $\M_{[k]}=\M_{[k+1]}$, we have $\M^{b}_{[k]}=0$, hence (cf. \eqref{eq: Sol M = Sol Mbar})
\begin{equation}
\Hdr^0(\M_{[k]},\b_\Omega(D(x,\rho)))\;=\;\Hdr^0(\M_{[k]}^b,
\b_\Omega(D(x,\rho)))\;=\;0\;.
\end{equation}
In other words the module $\M^{< \rho}$ 
has no non trivial solutions with values in $\b_\Omega(D(x,\rho))$.
\end{remark}
The following proposition shows that the 
module $\M^{\geq \rho}$ takes 
into account the solutions of $\M$ with values in 
$\O_\Omega(D(x,\rho))$:

\begin{proposition}\label{Prop. deco analytic}
We have 
\begin{enumerate}
\item\label{i Prop. deco analytic} $\Hdr^0(\M,
\O_\Omega(D(x,\rho)))=\Hdr^0(\M^{\geq\rho},
\O_\Omega(D(x,\rho)))$;
\item\label{ii Prop. deco analytic} $\Hdr^0(\M^{<\rho},\O_\Omega(D(x,\rho)))=0$;
\item $\M^{\geq\rho}$ is trivialized by $\O_\Omega(D(x,\rho))$.\label{iii Prop. deco analytic}
\end{enumerate}
\end{proposition}
\begin{proof}
Let us begin with item \eqref{iii Prop. deco analytic}. 
The kernels $K_i:=\mathrm{Ker}(\M\to\M_{[i]})$ 
of the sequence \eqref{Sequence of M_[i]} define a 
filtration $0\subset K_{1}\subset K_{2}\subset\cdots
\subset K_{k}=\M^{\geq \rho}$ of $\M$ 
where every sub-quotient $K_i/K_{i-1}$ is trivialized by 
$\b_\Omega(D(x,\rho))$, hence also by $\O_\Omega(D(x,\rho))$ 
(cf. Proposition \ref{Prop: M^b and M^bar trivialized by Bounded}). 
Item \eqref{iii Lemma :devisage}  of Lemma \ref{Lemma :devisage} then implies that 
$\M^{\geq \rho}$ is trivialized by $\O_\Omega(D(x,\rho))$. 
Indeed, the assumptions of Lemma \ref{Lemma :devisage} are fulfilled 
because $d=d/dT$ is surjective on 
$\O_\Omega(D(x,\rho))$ 
(cf. Notation \ref{Notation: T=coordinate on D(x)}).

\eqref{ii Prop. deco analytic} By Remark \ref{remark : M^rho has no sol in b(t,rho)}, 
$\M^{<\rho}$ has no nontrivial solutions in 
$\b_\Omega(D(x,\rho))$. %
Then, \eqref{ii Prop. deco analytic} follows from Proposition \ref{crucial point} below. 

\eqref{i Prop. deco analytic} now follows from \eqref{ii Prop. deco analytic} using also \eqref{iii Prop. deco analytic} together with 
items \eqref{i Lemma :devisage} and \eqref{ii Lemma :devisage} of Lemma \ref{Lemma :devisage}.
\end{proof}

\begin{remark}\label{Rk :  sol in O imply sol in B .hy}
The following proposition asserts that 
\emph{a differential module having some non trivial analytic 
solutions in $\O_\Omega(D(x,\rho))$ must have at least a non trivial 
bounded solution in $\b_\Omega(D(x,\rho))$}. 
This is a crucial point of the theory, and it is originally due to Dwork 
\cite{Dw-II}.
\end{remark}

\begin{proposition}\label{crucial point}
The following statements hold:
\begin{enumerate}
\item\label{i crucial point} If $\M_{[1]}=\M$, then 
$\Hdr^0(\M,\O_\Omega(D(x,\rho)))=0$.
\item\label{ii crucial point} If $\M^{[1]}=\M$, then 
$\mathrm{Hom}_{\H(x)\lr{d}}(\M,\O_\Omega(D(x,\rho)))=0$.
\end{enumerate}
\end{proposition}
\begin{proof}
The two assertions are equivalent by duality 
(see \eqref{eq : omega = hom}). 
We prove \eqref{ii crucial point}.
Let $m\in\M$, and $s\in \mathrm{Hom}_{\H(x)\lr{d}}(\M,
\O_\Omega(D(x,\rho)))$. 
The assumption implies that 
$0=\|m\|_\Psi=\inf_{\Psi(\sum_iP_ie_i)=m}
\max_i\|P_i\|_{op,D(x,\rho)}$. Let  $P_1,\ldots,P_q$ be such that 
$\max_i\|P_i\|_{op,D(x,\rho)}<1$. 
Now, for all $\rho'<\rho$, 
the restriction of $s$ belongs to 
$\mathrm{Hom}_{\H(x)\lr{d}}(\M,\b_\Omega(D(x,\rho')))$, so 
\begin{eqnarray}
\|s(m)\|_{D(x,\rho')}&\;=\;&
\|\sum_iP_is(\Psi(e_i))\|_{D(x,\rho')}\;\leq\;
\max_i\|P_i\|_{op,D(x,\rho')}\|s(\Psi(e_i))\|_{D(x,\rho')}\\
&\;\leq\;&
\max_i\|P_i\|_{op,D(x,\rho')}\cdot \max_i\|s(\Psi(e_i))\|_{D(x,\rho')}\label{eq : 2.10}
\end{eqnarray}
By \eqref{eq : explicit norm operator} each map 
$\rho'\mapsto\|P_i\|_{op,D(x,\rho')}$ is continuous, 
hence there exists $\rho'<\rho$ such that for all $\rho''\in[\rho',\rho[$ we have
\begin{equation}\label{eq : dempq}
\max_i\|P_i\|_{op,D(x,\rho'')}\;<\;1\;.
\end{equation}
Assume, by contrapositive, that $s\neq 0$. In this case, for all $0\neq \rho'\leq\rho$,
we have $\max_i\|s(\Psi(e_i))\|_{D(x,\rho')}>0$.
In this case, \eqref{eq : dempq} implies that \eqref{eq : 2.10} is a strict inequality. 
This applies in particular to the case where $m$ is equal to the generators $\{\Psi(e_i)\}_{i=1,\ldots,q}$.
For all $i$, there exists $\rho_i<\rho$ such that for all $\rho'\in[\rho_i,\rho[$ we have 
$\|s(\Psi(e_i))\|_{D(x,\rho')}<\max_i\|s(\Psi(e_i))\|_{D(x,\rho')}$. Therefore, we find for 
$\rho_m:=\max_i\rho_i$ the strict inequality
\begin{equation}
\|s(\Psi(e_i))\|_{D(x,\rho_m)} \;<\;
\max_i\|s(\Psi(e_i))\|_{D(x,\rho_m)}\;.
\end{equation}  
This is absurd and hence $s=0$.
\end{proof}

\subsection{Exactness, compatibility with dual, and direct 
sum decomposition}
\label{Exactness, compatibility with dual, and direct 
sum decomposition.}
\begin{proposition}[\protect{\cite{Ro-I}}]\label{exactness}
The functors associating to an $\H(x)$-differential module $\M$ the 
$\Omega$-vector spaces 
$\mathrm{Hom}_{\H(x)\langle d\rangle}(\M,\O_\Omega(D(x,\rho)))$
and $\Hdr^0(D(x,\rho),\M)$ respectively are exact 
for all $\rho\in ]0,1]$.
\end{proposition}
\begin{proof}
The two assertions are equivalent by duality 
\eqref{eq : omega = hom}. To prove the first one, 
it is enough to show that, for all $\H(x)$-differential modules 
$\M$, one has $\mathrm{Ext}^1_{\H(x)\lr{d}}(\M,
\O_\Omega(D(x,\rho)))=0$. Since $\M^{\geq \rho}$ is trivialized by 
$\O_\Omega(D(x,\rho))$ one has (cf. \cite[Section 6.7]{Christol-Book})
\begin{eqnarray}
\mathrm{Ext}^1_{\H(x)\lr{d}}(\M^{\geq\rho},\O_\Omega(D(x,\rho))) 
&\;=\;&
\mathrm{Ext}^1_{\O_\Omega(D(x,\rho))\lr{d}}(\O_\Omega(D(x,
\rho))\otimes_{\H(x)}\M^{\geq\rho},\O_\Omega(D(x,\rho)))\qquad\\
&=&
\mathrm{Ext}^1_{\O_\Omega(D(x,\rho))\lr{d}}(\O_\Omega(D(x,
\rho)),\O_\Omega(D(x,\rho)))^{\mathrm{dim}\M^{\geq 
\rho}}\;=\;0\;,\qquad
\end{eqnarray}
where the last equality follows from Lemmas 
\ref{Lemma : Ext^1=H^1} and \ref{Lemma :devisage}.
Writing 
$0\to\M^{\geq\rho}\to\M\to\M^{<\rho}\to0$ 
we are reduced to proving that 
\begin{equation}\label{...EE...uiguh}
\mathrm{Ext}^1_{\H(x)\lr{d}}(\M^{<\rho},\O_\Omega(D(x,
\rho)))\;=\;0\;. 
\end{equation}
So we may 
assume $\M=\M^{<\rho}$. Let $L\in\H(x)\lr{d}$ be such that 
$\M\cong\M_L$ (cf. Section 
\ref{Filtrations of cyclic modules, and factorization of operators.}). By  
\cite[Sections 6.6 and 6.7]{Christol-Book}, in order to prove 
\eqref{...EE...uiguh} it is enough to prove that $L$ 
is surjective as an operator on $\O_\Omega(D(x,\rho))$. This follows 
from Lemma \ref{QL-1 small} below.
\end{proof}

\begin{remark}\label{Rk : finite number of steps}
For all $\rho'<\rho$ close enough to $\rho$ 
one has $\M^{\geq\rho'}=\M^{\geq \rho}$ and 
$\M^{<\rho'}=\M^{<\rho}$. 
Indeed $\M$ is a finite dimensional 
vector space so the filtration $\{\M^{\geq \rho}\}_{\rho}$ 
has a finite number of steps, 
whose dimensions equal those of the family 
$\{\Hdr^0(\M,\O_\Omega(D(x,\rho)))\}_\rho$ (cf. Proposition 
\ref{Prop. deco analytic}).
\end{remark}

\begin{lemma}\label{QL-1 small}
Let $L\in\H(x)\lr{d}$, and let $\M_L$ be the corresponding 
cyclic differential module over $\H(x)$. 
Assume that $\M_L=\M_L^{<\rho}$ or, equivalently, that $L$ is 
injective on $\O_\Omega(D(x,\rho))$. 
Then for all $\varepsilon>0$ 
there exists $Q_\varepsilon\in\O_{X,x}\lr{d}$ such that 
\begin{enumerate}
\item\label{i QL-1 small} $\|Q_\varepsilon 
L-1\|_{op,\b_\Omega(D(x,\rho))}<\varepsilon$;
\item\label{ii QL-1 small} $Q_\varepsilon$ and $L$ are bijective as endomorphisms of the rings 
$\H(x)$, $\b_\Omega(D(x,\rho))$ and  $\O_\Omega(D(x,\rho))$.
\end{enumerate}
The same same operator $Q_\varepsilon$ satisfies items 
\eqref{i QL-1 small} and \eqref{ii QL-1 small} with 
respect to all $\rho'<\rho$ close enough to $\rho$.
\end{lemma}
\begin{proof}
Let $\varepsilon>0$, we may assume that $\varepsilon<1$. 
Consider the presentation $\Psi:\H(x)\lr{d}\to\M_L\to 0$ 
with kernel $\H(x)\lr{d}L$. Since $\M_L=\M_L^{[1]}$, we have $\|\Psi(-1)\|_\Psi=0$ (cf. \eqref{eq: norm-Psi}). Therefore, there exists $\widetilde{P}\in \H(x)\lr{d}$ such that 
$\|\widetilde{P}\|_{op,\b_\Omega(D(x,\rho))}<\varepsilon$ and $\Psi(\widetilde{P})=\Psi(-1)$. We have $\Psi(\widetilde{P}+1)=0$, so $\widetilde{P}+1\in\H(x)\langle d\rangle L$, that is 
$\widetilde{P}=-1+\widetilde{Q}_\varepsilon L$ for some 
$\widetilde{Q}_\varepsilon\in\H(x)\lr{d}$. 
Let $n$ be the order of $\widetilde{Q}_\varepsilon$, and let 
$\H(x)\lr{d}^{\leq n}$ be the $\H(x)$-vector space of differential 
polynomials of order  $\leq n$. 
Since all norms on $\H(x)\lr{d}^{\leq n}$ are equivalent,  
$\|.\|_{op,\b_{\Omega}(D(x,\rho))}$ is equivalent to the sup-norm 
with respect to the basis $1,d,\ldots,d^n$ : 
$\|\sum_{i=0}^ng_id^i\|_{\H(x),d}:=\max_{i}|g_i|(x)$. We deduce 
that $\O_{X,x}\lr{d}^{\leq n}$ is dense in $\H(x)\lr{d}^{\leq n}$ with 
respect to $\|.\|_{op,\b_{\Omega}(D(x,\rho))}$. 
Hence, there exists $Q_\varepsilon\in\O_{X,x}\lr{d}^{\leq n}$ such 
that $\|Q_\varepsilon L-1\|_{op,\b_\Omega(D(x,\rho))}<\varepsilon$.
Since $\b_\Omega(D(x,\rho))$ is complete, 
$Q_\varepsilon L=1+P$ is invertible 
as an endomorphism of $\b_\Omega(D(x,\rho))$ with inverse 
$U:=\sum_{i\geq 0}(-1)^i P^i$. 
It follows that  $Q_\varepsilon LU=1$ and $UQ_\varepsilon L=1$. In particular,
$Q_\varepsilon$ is surjective (and $L$ is injective) 
as an operator on $\b_\Omega(D(x,\rho))$. 
By Lemma \ref{Lemma : QL and U invertible} below, $Q_\varepsilon$ is invertible as operators on 
$\b_\Omega(D(x,\rho))$ and injective also on $\O_\Omega(D(x,\rho))$ and $\H(x)$. 
We deduce that 
$L$ is invertible too as operator on 
$\b_\Omega(D(x,\rho))$ and again by 
Lemma \ref{Lemma : QL and U invertible}, $L$ is injective as operator on 
$\O_\Omega(D(x,\rho))$ (which is our assumption) and $\H(x)$. 

To deduce that $L$ is also surjective 
on $\O_\Omega(D(x,\rho))$ we may also argue as follows. One considers $\rho_0<\rho$ such that $\M_L^{<\rho_0}=\M_L$ (cf. Remark 
\ref{Rk : finite number of steps}). 
Then $L$ is bijective on 
$\b_\Omega(D(x,\rho'))$ for all $\rho_0\leq\rho'<\rho$. 
If $f\in\O_\Omega(D(x,\rho))$, for all $\rho'$, 
we may find $g_{\rho'}\in \b_\Omega(D(x,\rho'))$ such 
that $L(g_{\rho'})=f$. By injectivity, for all $\rho_0\leq \rho'<\rho''<\rho$ we must have $g_{\rho'}=g_{\rho''}$, 
which shows that $g_{\rho'}$ and $g_{\rho''}$ are both 
restrictions of the same analytic function 
$g\in\O_\Omega(D(x,\rho))=\bigcap_{\rho'<\rho}\b_\Omega(D(x,\rho'))$ satisfying $L(g)=f$. This shows that $L$ is surjective as an endomorphism of $\O_\Omega(D(x,\rho))$.

Notice that, by Lemma \ref{Lemma : QL and U invertible} below, $Q_\varepsilon$ is injective as an 
operator on $\O(D(x,\rho))$, so we can reproduce the same argument to prove the bijectivity of 
$Q_\varepsilon$ on $\O_\Omega(D(x,\rho))$.

Now, the inclusion 
$\H(x)\subset\b_\Omega(D(x,\rho))$ is isometric, so 
$\|.\|_{op,\H(x)}\leq\|.\|_{op,\b_\Omega(D(x,\rho))}$. Hence, as 
above, $Q_\varepsilon L$ is invertible as an endomorphism of $\H(x)$ 
with inverse $U$. As above, this implies that
$Q_\varepsilon$ is surjective (and $L$ is injective) 
as endomorphism on $\H(x)$.
Moreover, we know that 
$Q_\varepsilon$ is injective on $\H(x)$ 
(because it is injective on $\b_\Omega(D(x,\rho))$). 
So both $L$ and $Q_\varepsilon$ are bijective as 
endomorphisms of $\H(x)$.

Now to prove that the same holds for $\rho'<\rho$ close to $\rho$, we 
observe that %
$L$ remains injective in $\O(D(x,\rho'))$ by Remark 
\ref{Rk : finite number of steps}, 
and we can reproduce the proof for $\rho'$. 
Note that $\rho\mapsto
\|Q_\varepsilon L-1\|_{op,\b_\Omega(D(x,\rho))}$ is a continuous 
function of $\rho$ (cf. \eqref{eq : explicit norm operator}), 
so for $\rho'<\rho$ close enough to $\rho$ 
the inequality \eqref{i QL-1 small} is preserved, and the rest of the proof works 
identically.
\end{proof}  

\begin{lemma}[cf. \protect{\cite[15.4]{Christol-Book}}]
\label{Lemma : QL and U invertible}
If $Q\in\H(x)\lr{d}$ is surjective as an endomorphism of 
$\b_\Omega(D(x,\rho))$, then it is injective as an endomorphism of 
$\O_\Omega(D(x,\rho))$, hence also of $\b_\Omega(D(x,\rho))$ and 
of $\H(x)$. 
\end{lemma}
\begin{proof}
By contrapositive if $Q$ is not injective on $\O_\Omega(D(x,\rho))$, 
then, by Proposition \ref{crucial point} 
(cf. Remark \ref{Rk :  sol in O imply sol in B .hy}), 
it is not on $\b_\Omega(D(x,\rho))$ either. 
Let $u\in\b_\Omega(D(x,\rho))$ be such that $Q(u)=0$.
{Notation: T=coordinate on D(x)}
It is well known that $d$ has an infinite dimensional 
cokernel as an operator on $\b_\Omega(D(x,\rho))$ 
(roughly speaking there are lots of bounded functions 
whose primitive is not bounded).
More precisely, by \cite[15.1]{Christol-Book}, 
one may find an infinite dimensional 
$\Omega$-sub-vector space $V\subseteq\O(D(x,\rho))$ 
such that $V\cap\b_\Omega(D(x,\rho))=0$ and 
$d(V)\subset\b_\Omega(D(x,\rho))$. 
Now, the vector space $uV$ satisfies  
\begin{equation}\label{eq : hh d s tuo ,}
uV\cap\b_\Omega(D(x,\rho))\;=\;0\;,\qquad 
Q(uV)\;\subset\;\b_\Omega(D(x,\rho))\;.
\end{equation}
Indeed, if $uf\in\b_\Omega(D(x,\rho))$, then 
$f\in\b_\Omega(D(x,\rho))$ because 
we may fix $\rho_0<\rho$ and see that, for all 
$\rho_0\leq \rho'<\rho$, one has 
$\|f\|_{D(x,\rho')}=\frac{\|uf\|_{D(x,\rho')}}{\|u\|_{D(x,\rho')}}
\leq\frac{\|uf\|_{D(x,\rho)}}{\|u\|_{D(x,\rho_0)}}$, which shows that $f$ is bounded too. 
We deduce that 
$uV\cap\b_\Omega(D(x,\rho))=0$. 
The inclusion $Q(uV)\subset\b_\Omega(D(x,\rho))$ follows from the 
fact that $d^n(uf)=d^n(u)f+\sum_{i=1}^nb_id^i(f)$ with 
$b_i=\tbinom{n}{i}d^{n-i}(u)$, which implies
\begin{equation}\label{eq: RDF}
Q(uf)\;=\;Q(u)f+Rd(f)\;=\;Rd(f)\;,
\end{equation} 
with $R\in\b_\Omega(D(x,\rho))\lr{d}$. 
This shows that if $f\in V$, by definition 
$d(f)\in\b_\Omega(D(x,\rho))$ and by \eqref{eq: RDF} we have 
$Q(uf)\in\b_\Omega(D(x,\rho))$. 
This proves \eqref{eq : hh d s tuo ,}. 
As a consequence, the dimension of 
$\mathrm{coker}(Q,\b_\Omega(D(x,\rho)))$ is 
infinite, which contradicts the fact that $Q$ is surjective.  
\end{proof}

\if{
If $\varphi:\M\to\N$ is a morphism of $\H(x)$-differential modules, 
and if $s:\N\to\b_\Omega(t,\rho)$ is a bounded solution, 
then $s\phi$ is a solution of $\M$, and hence it annihilates 
$\M^{<\rho}$. 
So $\varphi(\M^{<\rho})$ is annihilated by all bounded solution of 
$\N$ and hence it is contained in 
$\N^{<\rho}$. This proves that $\M\mapsto\M^{\geq \rho}$ is a 
functor.
}\fi

\begin{corollary}\label{Cor : exactness of M-->M^geqrho}
$\M\mapsto \M^{\geq \rho}$ is an additive exact functor.
\end{corollary}
\begin{proof}
Additivity is clear. Let $\mathcal{F}(D(x,\rho))$ be the fraction field of 
$\O_\Omega(D(x,\rho))$. Let $E:0\to\N\to\M\to\P\to0$ be an exact 
sequence. By faithfully flatness of $\mathcal{F}(D(x,\rho))/\H(x)$, 
it is enough to prove that 
$E^{\geq \rho}\otimes_{\H(x)}\mathcal{F}(D(x,\rho))$ is exact.
This follows from the exactness of the sequence 
$E^{\geq \rho}\otimes_{\H(x)}\O_\Omega(D(x,\rho))$. 
Indeed $E^{\geq\rho}$ is constituted by modules trivialized by 
$\O_\Omega(D(x,\rho))$, so the sequence 
$E^{\geq \rho}\otimes_{\H(x)}\O_\Omega(D(x,\rho))$ 
is exact as soon as the sequence
\begin{equation}
\Hdr^0(E^{\geq \rho}\otimes_{\H(x)}\O_\Omega(D(x,\rho)),
\O_\Omega(D(x,\rho)))\;=\;
\Hdr^0(E^{\geq \rho},\O_\Omega(D(x,\rho)))
\;=\;\Hdr^0(E,\O_\Omega(D(x,\rho)))
\end{equation} 
is exact (cf. Lemmas \ref{Lemma :devisage} and 
\ref{Lemma : Trivial iff solutions}). 
The exactness of $\Hdr^0(E,\O_\Omega(D(x,\rho)))$ 
follows from Proposition \ref{exactness}.
\if{
 composed by trivial modules

If $E:0\to\N\to\M\to\P\to0$ is an exact sequence, then 
\begin{equation}
\Hdr^0(E,\O_\Omega(D(x,\rho)))
\;=\;
\Hdr^0(E^{\geq \rho},\O_\Omega(D(x,\rho)))
\;=\;
\Hdr^0(E^{\geq \rho}\otimes_{\H(x)}\O_\Omega(D(x,\rho)),
\O_\Omega(D(x,\rho)))
\end{equation} 
is exact by Prop. \ref{exactness}. 
The sequence $E^{\geq \rho}\otimes_{\H(x)}
\O_\Omega(D(x,\rho))$ is constituted by trivial modules

If $\mathcal{F}(D(x,\rho))$ is the fraction field of 
$\O_\Omega(D(x,\rho))$, 
the sequences $E^{\geq \rho}\otimes_{\H(x)}
\O_\Omega(D(x,\rho))$ and 
$E^{\geq \rho}\otimes_{\H(x)}\mathcal{F}(D(x,\rho))$ are 
constituted by trivial $\O_\Omega(D(x,\rho))$-differential modules, so 
they are exact 
hence they are exact by Lemma \ref{Lemma :devisage}. 
Since $\mathcal{F}(D(x,\rho))/\H(x)$ is faithfully flat 
$E^{\geq \rho}$ is exact.
}\fi
\end{proof}

\begin{lemma}\label{<rho stable by sub}
All sub-quotients $\mathrm{S}$  of $\M^{<\rho}$ satisfy 
$\mathrm{S}=\mathrm{S}^{<\rho}$ 
(i.e. $\mathrm{S}^{\geq \rho}=0$).
All sub-quotients of $\M^{\geq \rho}$ are trivialized by 
$\O_\Omega(D(x,\rho))$.
\end{lemma}
\begin{proof}
By exactness each sub-module $\N$ and each quotient $\Q$ of 
$\M^{<\rho}$ satisfy $\N^{\geq \rho}=\Q^{\geq \rho}=0$. 
So $\N=\N^{<\rho}$ and $\Q=\Q^{<\rho}$. 
The second assertion follows from Lemma 
\ref{Lemma : Trivial iff solutions}.
\end{proof}

\begin{proposition}[Compatibility with duals]\label{dual}
The composite map 
\begin{equation}
c:(\M^*)^{\geq \rho}\to\M^{*}\to 
(\M^{\geq \rho})^*
\end{equation} 
is an isomorphism.
\end{proposition}
\begin{proof}
Since $\M^{\geq \rho}$ is trivialized by $\O_\Omega(D(x,\rho))$ 
then so does its dual, hence 
$(\M^{\geq \rho})^*=((\M^{\geq \rho})^*)^{\geq \rho}$.
By exactness of the functor $\N\mapsto \N^{\geq \rho}$ 
(cf. Proposition \ref{Cor : exactness of M-->M^geqrho}), the exact sequence $E : 0 \to (\M^{< \rho})^* \to \M^* \to 
(\M^{\geq\rho})^* \to 0$ produces an exact sequence 
\begin{equation}
E^{\geq\rho}:0 \to ((\M^{<\rho})^*)^{\geq \rho} \to 
(\M^*)^{\geq\rho} \stackrel{c}{\to} (\M^{\geq\rho})^* \to 0\;.
\end{equation}
Therefore, in order to prove that the map $c:(\M^*)^{\geq\rho} \to (\M^{\geq\rho})^*$ so obtained is an isomorphism it is enough to prove that 
$((\M^{<\rho})^*)^{\geq\rho}=0$. This module is a 
submodule of $(\M^{<\rho})^*$ trivialized by 
$\O_\Omega(D(x,\rho))$. 
Hence, its dual is a quotient 
$S$ of $(\M^{<\rho})^{**}\cong\M^{<\rho}$ which is trivialized by 
$\O_\Omega(D(x,\rho))$ too. 
This implies that $S=S^{\geq \rho}$ 
and by Lemma \ref{<rho stable by sub} we obtain 
$S=0$.
\end{proof}

\begin{proposition}\label{Prop : direct summand}
For all $\rho\in]0,1]$ one has $\M=\M^{<\rho}\oplus
\M^{\geq \rho}$.
\end{proposition}
\begin{proof}
By Proposition \ref{dual} the composite map 
$\M^{\geq\rho}\subseteq\M\to((\M^*)^{\geq\rho})^*$ 
is an isomorphism. Composing with the inverse of this isomorphism, we obtain a map $\M\to((\M^*)^{\geq \rho})^*\simto\M^{\geq \rho}$ which is the identity on the submodule $\M^{\geq \rho}\subseteq\M$.
In other words, the sequence \eqref{eq : Robba deco} splits.
\if{
and by Lemma 
\ref{<rho stable by sub} one has 
$((\M^*)^{\geq\rho})^*\cap\M^{<\rho}=0$ since 
$((\M^*)^{\geq\rho})^*$ is trivialized by 
$\O_\Omega(D(x,\rho))$.
}\fi
\end{proof}

\begin{corollary}\label{corollary : uniqueness at H(x)}
The module $\M^{\geq \rho}$ is the union of all differential sub-modules of $\M$ 
trivialized by $\O_\Omega(D(x,\rho))$. 
Moreover one has a unique decomposition 
\begin{equation}\label{eq : deco at x Robba eq}
\M\;:=\;\bigoplus_{0<\rho\leq 1}\M^{\rho}
\end{equation}
with the following properties:
\begin{enumerate}
\item\label{i corollary : uniqueness at H(x)} 
$\M^{\rho}$ is trivialized by $\O_\Omega(D(x,\rho))$,
\item\label{ii corollary : uniqueness at H(x)}  
$(\M^{\rho})^{\geq\rho'}=0$, for all 
$\rho<\rho'\leq 1$.
\end{enumerate} 
One has moreover the following properties: 
\begin{enumerate}
\item[iii)] For all $\rho\in]0,1]$ one has  
\begin{equation}\label{eq : compatibility with duals}
(\M^{\rho})^*\;\cong\;(\M^*)^\rho\;.
\end{equation}
\item[iv)] The module $\M^\rho$ satisfies (cf. \eqref{omega(x,F)})
\begin{equation}\label{hom Mrho,Orho'}
\mathrm{Hom}_{\H(x)\lr{d}}(\M^{\rho},
\O_\Omega(D(x,\rho')))\;\cong\;\left\{
\begin{array}{rcl}
\Hdr^0(x,\M^{\rho})^*&\textrm{ if }&\rho'\leq\rho\\
0&\textrm{ if }&\rho'>\rho.
\end{array}
\right.
\end{equation}
\item[v)] If $\M,\N$ are differential modules over $\H(x)$, 
then for all $\rho\neq\rho'$, one has 
\begin{equation}\label{eq : Hom (Mrho,Mrho')=0}
\mathrm{Hom}_{\H(x)\lr{d}}(\M^\rho,\N^{\rho'})\;=\;0\;.
\end{equation}
\end{enumerate}
\end{corollary}
\begin{proof}
If $\N\subseteq\M$ is trivialized by $\O_\Omega(D(x,\rho))$, 
then so is 
$\N+\M^{\geq \rho}$ by item \eqref{iii Lemma :devisage}  of Lemma \ref{Lemma :devisage}. 
By 
Lemma \ref{<rho stable by sub} we have 
$(\N+\M^{\geq \rho})\cap\M^{<\rho}=0$ and 
therefore 
$\N\subseteq\M^{\geq \rho}$ by a dimension's argument. This proves 
that $\M^{\geq\rho}$ is the 
union of all differential sub-modules of $\M$ 
trivialized by $\O_\Omega(D(x,\rho))$.

The existence of decomposition \eqref{eq : deco at x Robba eq} 
follows from Proposition 
\ref{Prop : direct summand} by setting
\begin{equation}\label{eq : Mrho def as a quotient}
\M^{\rho}\;:=\;\frac{\M^{\geq \rho}}{\bigcup_{\rho'>\rho}\M^{\geq \rho'}}\;.
\end{equation}

Compatibility with duals \eqref{eq : compatibility with duals} 
then follows from Proposition \ref{dual}. 

By the usual duality argument \eqref{eq : omega = hom}, property \eqref{hom Mrho,Orho'} is 
 equivalent, 
to 
\begin{equation}\label{eq : qsdfghjuzertv}
\Hdr^0((\M^{\rho})^*,\O_\Omega(D(x,\rho')))\;\cong\;
\left\{
\begin{array}{rcl}
\Hdr^0(x,\M^{\rho})^*&\textrm{ if }&\rho'\leq\rho\\
0&\textrm{ if }&\rho'>\rho.
\end{array}
\right.
\end{equation}
By \eqref{eq : compatibility with duals} one has 
$\Hdr^0((\M^{\rho})^*,\O_\Omega(D(x,\rho')))
=\Hdr^0((\M^{*})^\rho,\O_\Omega(D(x,\rho')))$. 
So \eqref{eq : qsdfghjuzertv} follows from 
Proposition \ref{Prop. deco analytic}, since $\Hdr^0(x,\M^\rho)^*$ 
may be identified to $\Hdr^0(x,(\M^\rho)^*)$, 
and hence to $\Hdr^0(x,(\M^*)^{\rho})$. 

We now prove \eqref{eq : Hom (Mrho,Mrho')=0}. 
Let $\alpha:\M^{\rho}\to\N^{\rho'}$ be an 
$\H(x)$-linear morphism commuting with the 
connections. 
By \eqref{hom Mrho,Orho'}, if $\rho<\rho'$, 
then $s\circ\alpha=0$ 
for all 
$s\in\mathrm{Hom}_{\H(x)\lr{d}}(\N^{\rho'},
\O_\Omega(D(x,\rho')))$. 
Since $\N^{\rho'}$ is trivialized by $\O_\Omega(D(x,\rho'))$, 
then $\bigcap_s\mathrm{Ker}(s)=0$, where $s$ runs in 
$\mathrm{Hom}_{\H(x)\lr{d}}(\N^{\rho'},\O_\Omega(D(x,\rho')))$. 
Hence $\alpha=0$. If $\rho>\rho'$ the same argument proves, by 
duality \eqref{eq : compatibility with duals}, that the dual 
of $\alpha$ is zero, and hence also $\alpha=0$.

Assume now that $\M=\oplus_{\eta\leq 1}\widetilde{\M}^\eta$ 
is another decomposition satisfying \eqref{i corollary : uniqueness at H(x)}  and \eqref{ii corollary : uniqueness at H(x)}. 
By the fact that this graduation has a finite number of 
non zero terms, and by \eqref{ii corollary : uniqueness at H(x)}, one has 
\begin{equation}
\M^{\rho}
\;=\;\frac{\M^{\geq \rho}}{\cup_{\rho'>\rho}\M^{\geq \rho'}}
\;=\;
\frac{(\oplus_{\eta\leq 1}\widetilde{\M}^\eta)^{\geq \rho}}{\cup_{\rho'>\rho}(\oplus_{\eta\leq 1}\widetilde{\M}^\eta)^{\geq \rho'}}
\;=\;
\frac{\oplus_{\eta\in[\rho,1]}\,\widetilde{\M}^\eta}{\cup_{\rho'>\rho}(\oplus_{\eta\in[\rho',1]} \widetilde{\M}^\eta)}
\;=\;\widetilde{\M}^{\rho}\;,
\end{equation}
which concludes the proof.\end{proof}

\section{Dwork-Robba's decomposition by the spectral radii 
over $\O_{X,x}$} 
\label{Dwork-Robba (local) decomposition}
In this section we prove that, if $x$ is a point of type 
$2$, $3$, or $4$, and if $\M_x$ is a differential module 
over $\O_{X,x}$, then Robba's decomposition 
of $\M_x\otimes_{\O_{X,x}}\H(x)$ by the spectral radii 
(cf. \eqref{corollary : uniqueness at H(x)}) 
descends from $\H(x)$ to $\O_{X,x}$ into a 
decomposition of $\M_x$. 
Such a decomposition was obtained by Dwork and Robba 
in \cite[First Thm. of section 4]{Dw-Robba} for a point 
$x\in\mathbb{A}_K^{1,\mathrm{an}}$ of type $2$ and 
a base field with residual field of positive characteristic.

In this section, as in Section \ref{Robba-deco}, 
$K$ is algebraically closed (cf. Hypothesis \ref{Hyp : K alg clo}). We will remove this assumption in the next 
sections (cf.  Lemma \ref{Lema : descent of M_xrho}).

Along Section \ref{Dwork-Robba (local) decomposition} we will need to work with the \'etale 
maps obtained in Section \ref{Section : a description of Ducros etale pam}, we maintain the notations of that 
section and in particular those established in Lemma 
\ref{Lemma : H(x) stable by d/dt} and Remark \ref{local derivation}.

\subsection{Statement of Dwork-Robba's decomposition}
Let $x\in X$ be a point of type $2$, $3$, or $4$. Let $\M_x$ be a 
differential module over $\O_{X,x}$.
\begin{theorem}[\protect{\cite[Theorem in Section 4.1]{Dw-Robba}}]\label{Dw-Robba}
There exists a unique decomposition
\begin{equation}\label{eq : deco of M_x .gdyu}
\M_x\;=\;\bigoplus_{0<\rho\leq 1}\M_x^{\rho}
\end{equation}
such that for all $0<\rho\leq 1$ one has
\begin{equation}\label{eq: M_x^rhooH=M^rho}
\M_x^{\rho}\otimes_{\O_{X,x}}\H(x) \;=\;
(\M_x\otimes_{\O_{X,x}}\H(x))^\rho\;.
\end{equation}
Moreover, if 
\begin{equation}\label{eq : def M^geq rho,sp}
\M_x^{\geq\rho}\;:=\;\oplus_{\rho\leq\rho'}\M_x^{\rho'}\;,
\end{equation}
and if $d:\O_{X,x}\to\O_{X,x}$ is one of the derivations $d_1,\ldots,d_n$ obtained in Remark \ref{local derivation}, then
\begin{enumerate}
\item\label{i eq : deco of M_x .gdyu} The analogous of 
\eqref{hom Mrho,Orho'} holds for $\M_x^\rho$. 

\item\label{ii eq : deco of M_x .gdyu} The canonical composite map 
$(\M_x^*)^{\geq\rho}\to\M_x^*\to(\M_x^{\geq\rho})^*$ 
is an isomorphism, in particular 
$(\M_x^\rho)^*\cong(\M_x^*)^{\rho}$.
\item\label{iii eq : deco of M_x .gdyu} If $\M_x$ and $\N_x$ are differential modules over $\O_{X,x}$, 
and if $\rho\neq \rho'$,  then 
$\mathrm{Hom}_{\O_{X,x}\langle d\rangle}
(\M_x^\rho,\N_x^{\rho'})=0$.
\end{enumerate}
\end{theorem}
\begin{proof}
We firstly prove the uniqueness. 
Let $\M_x=\oplus_{\eta\in]0,1]}\widetilde{\M}_x^\eta$ 
be another decomposition. For all $0<\rho\leq 1$ we have $\M_x^\rho\otimes_{\O_{X,x}}\H(x)=
\widetilde{\M}_x^\rho\otimes_{\O_{X,x}}\H(x)$. 
Then, by \eqref{eq : Hom (Mrho,Mrho')=0}, for all 
$\rho\neq\rho'$ the composite morphism 
$\alpha:\M^\rho_x\subseteq\M_x\to \widetilde{\M}^{\rho'}_x$ is zero 
after scalar extension to $\H(x)$. It follows that  
$\alpha=0$, because $\O_{X,x}\subset\H(x)$ is an inclusion of fields. 
This proves that 
$\M^\rho\subseteq\widetilde{\M}^{\rho}$, and  
the reverse of this argument gives the equality.

To prove the existence of \eqref{eq : deco of M_x .gdyu} it is enough to descend the filtration 
$(\M^{\geq\rho})_{\rho\in]0,1]}$ of $\M$. In other words, it is enough to show the existence, for all $0<\rho\leq 1$, of an $\O_{X,x}$-lattice 
$\M_x^{\geq \rho}$ of $\M^{\geq \rho}$ such that :
\begin{equation}\label{M_x^geqrho}
\textrm{The inclusion 
$\M^{\geq\rho}\subseteq\M$ is the
scalar extension of an 
inclusion of $\O_{X,x}$-lattices $\M_x^{\geq\rho}\subseteq\M_x$.}
\end{equation}	
Claim \eqref{M_x^geqrho} requires some work and 
is proved in Theorem 
\ref{Dwork-Robba} below. 
Let us show that it implies \eqref{eq: M_x^rhooH=M^rho} and the other properties. 

Let us begin by item \eqref{ii eq : deco of M_x .gdyu}.
From \eqref{M_x^geqrho} one shows that
the scalar extension of the map 
$(\M_x^*)^{\geq\rho}\to\M_x^*\to(\M_x^{\geq\rho})^*$ 
 to $\H(x)$ has a non zero 
determinant, so the map itself has a non zero determinant in $\O_{X,x}$. 
This proves the compatibility with duals, and hence also 
that $\M_x^{\geq \rho}$ is a direct summand of 
$\M_x$, as in the proof of Proposition \ref{Prop : direct summand}. Item \eqref{ii eq : deco of M_x .gdyu} follows.

In order to prove \eqref{eq: M_x^rhooH=M^rho} we define 
$\M_x^{\rho} := \M_x^{\geq\rho} / 
(\cup_{\rho' > \rho} \M_x^{\geq \rho'})$, and it is clear that  
$\M_x=\oplus_{\rho\in ]0,1]}\M_x^\rho$, and that $\M_x^\rho$ is a 
lattice of $\M^\rho$.

Item \eqref{i eq : deco of M_x .gdyu} now follows 
easily from equality \eqref{eq: M_x^rhooH=M^rho}.

Finally, to prove \eqref{iii eq : deco of M_x .gdyu} it is enough to observe that the existence of a non 
zero map $\mathrm{Hom}_{\O_{X,x}\langle d\rangle}
(\M_x^\rho,\N_x^{\rho'})$ produces a non zero map in 
$\mathrm{Hom}_{\H(x)\langle d\rangle}
(\M^\rho,\N^{\rho'})$ 
which contradicts \eqref{eq : Hom (Mrho,Mrho')=0}.
\end{proof}

\if{
\begin{remark}
It is possible to derive the proof of Dwork-Robba's decomposition 
\ref{Dw-Robba} over $X$ from the knowledge of the decomposition 
over the affine line as follows. 
Let $Y$ be an affinoid neighborhood of $x$ 
in $X$, and let  $f:Y\to\mathbb{P}^{1,\mathrm{an}}_K$ be an 
àtale map such that $[\H(x):\H(f(x))]$ is prime to the residual characteristic $p$.  
Let $\M_x$ be a differential module over $\O_{X,x}$, and let 
$\M:=\M_x\otimes_{\O_{X,x}}\H(x)$. 
Then the spectral radii of $f_*f^*\M_x$ can be easily computed 
using  \cite[Lemma \ref{lem:tameradii}]{NP-II}. 
By Corollary \ref{corollary : uniqueness at H(x)}, v), the maps 
$\M_x\to f^*f_*\M_x$ and $\M\to f^*f_*\M$ preserve the radii. 
From this it is possible to deduce the decomposition of $\M_x$ by 
expressing it as $\M_x=f^*f_*\M_x\cap\M$ inside $f^*f_*\M_x$. 
Then one shows that the decomposition of $\M$ and of $f^*f_*\M$ 
are compatible, and gives the Dwork-Robba's 
decomposition of $\M_x$.

\comment{Je crois que la on raconte des conneries ! 
La flèche canonique elle va dans l'autre sens $f^*f_*\M\to\M$ non ?}

The fact is that the original proof of Dwork-Raobba \cite{Dw-Robba} 
works on both the affine line, and on $X$ (up to minor 
implementations). 
In order to provide a complete set of proofs, in this section we 
provide the entire proof of Dwork-Robba's theorem directly on $X$.
\end{remark}
}\fi

\subsection{Norms on differential operators}
\label{Norms on differential operators}
Let $x$ be a point of type $2$, $3$, or $4$ of $X$. Following the notations established in Section \ref{Section : a description of Ducros etale pam}, 
we fix a coordinate $T_j$ of $W_j$. In particular, this allow us to consider the radius $r(y)$ of a point $y\in W_j$ (cf. Definition \ref{Radius of a point.}). Recall that we have a nice set of \'etale coordinates 
$\{\psi_j:Y_j\to W_j\}_{j=1,\ldots,n}$ around $x$ (cf. Lemma \ref{Lemma : H(x) stable by d/dt}).
In the sequel $d$ will denote one of the derivations 
$d_1,\ldots,d_n$ corresponding to $\psi_1,\ldots,\psi_n$ 
as in Remark \ref{local derivation}. 
As mentioned in Lemma \ref{Lemma : change of coordinate} and 
Remark \ref{Remark : d/dT coord on D(x)} 
any choice of $d$ produces the same Weyl 
algebras $\H(x)\langle d\rangle$ and 
$\O_{X,x}\langle d\rangle$ which are intrinsic in this 
sense.

\subsubsection{Approximation of operator norms.}
\label{Approximation of operator norms.}
A \emph{star} around $x$ is a subset $\U$ of $X$ such 
that  
\begin{enumerate}
\item For each germ of segment $b$ out of $x$, there 
exists a segment $[x,y]$ contained in $\U$ representing 
$b$;\footnote{Recall that a segment $[x,y]$ is homeomorphic  to a closed interval of $\mathbb{R}$.}
\item\label{ii: star}  
There exists $\varepsilon>0$ such that for all $z\in\U$, and all  
$j=1,\ldots, n$, one has $r(\psi_j(z))>\varepsilon$.
\end{enumerate}

Let $Y$ be a $\psi$-\emph{admissible} neighborhood of $x$ (cf. Definition \ref{def : elementary neigh}) 
and let $\U_Y\subset Y$ be a star around $x$ 
such that the Shilov boundary of $Y$ is contained in 
$\U_Y$. 
For every $\rho\in]0,1]$, we define the 
product of Banach algebras 
\begin{equation}
\Prod(\U_Y,\rho)\;:=\;
\prod_{y\in\U_Y}\b_\Omega(D(y,\rho))
\end{equation}
as the set of tuples 
$(f_y)_{y\in\U_Y}$, such that 
\begin{enumerate}
\item for all $y\in \U_Y$ we 
have $f_y\in \b_\Omega(D(y,\rho))$;
\item the tuple $(f_y)_{y\in\U_Y}$ is bounded in the sense that 
\begin{equation}
\|(f_y)_y\|_{\Prod(\U_Y,\rho)}\;:=\;
\sup_{y\in\U_Y}\|f_y\|_{D(y,\rho)}\;<\;+\infty\;.
\end{equation}
\end{enumerate}
The algebra $\Prod(\U_Y,\rho)$ together with 
the sup-norm $\|.\|_{\Prod(\U_Y,\rho)}$ is a Banach 
algebra. 
It is the product in the category of $K$-Banach 
spaces with bounded $K$-linear maps whose norm is 
bounded by $1$ (i.e. the maps decrease the norms).

Denote by the analogous symbol  
$\Prod(\U_Y):=\prod_{y\in\U_Y}\H(y)$ 
the Banach algebra of bounded 
tuples $(f_y)_{y\in \U_Y}$, where $f_y\in\H(y)$ for every $y\in\U_Y$.
Since $\U_Y$ contains the Shilov boundary of $Y$, the 
natural maps
\begin{equation}\label{eq: isometric O(Y)-H(y)-P}
\O(Y)\;\to\;\Prod(\U_Y)\;\to\;
\Prod(\U_Y,\rho)
\end{equation}
associating to $f\in\O(Y)$ the tuples $(f(y))_{y\in\U_Y}$ 
and $(f_{|D(y,\rho)})_{y\in\U_Y}$ are isometric (cf. 
\eqref{eq : H(x) in B(D(x))}).

Let $d$ be one of the derivations $d_1,\ldots,d_n$ described in Remark
\ref{local derivation} and assume that 
(cf. Lemma \ref{Lemma : H(x) stable by d/dt})
\begin{equation}
Y\;\subseteq\; V\;.
\end{equation}
Recall that $d$ acts on $\H(y)$, $\O_{X,y}$ and 
$\b_\Omega(D(y,\rho))$ for all $y\in\U_Y$ (cf. \eqref{eq : d_j acts}).

A differential operator in 
$\Prod(\U_Y,\rho)\langle d\rangle$ is equivalent to the datum of  a tuple 
\begin{equation}\label{eq : P=(P_y)_y}
(P_y)_{y\in\U_Y}
\end{equation}
of differential
operators $P_y\in\b_\Omega(D(y,\rho))\langle d\rangle$ satisfying the following conditions:
\begin{enumerate}
\item The exists an integer $s\geq 0$ such that, for all 
$y\in\U_Y$, the order of $P_y$ is less than or equal to 
$s$;
\item For all $y\in\U_Y$ write $P_y=\sum_{k=0}^sf_{y,k}d^k\in\b_\Omega(D(x,\rho))\langle d\rangle$. Then, for all 
$k=0,\ldots, s$, the family $(f_{y,k})_{y\in\U_Y}$ belongs to $\Prod(\U_Y,\rho)$.
\end{enumerate}

\begin{lemma}\label{Lemma: op UY=sup op B}
We have (cf. Definition \ref{eq : op norm on a group})
\begin{equation}\label{eq : norm is the sup on generic disks}
\|(P_y)_{y\in\U_Y}\|_{op,\Prod(\U_Y,\rho)}\;=\;
\sup_{y\in\U_Y}\|P_y\|_{op,\b_\Omega(D(y,\rho))}\;. 
\end{equation}
\end{lemma}
\begin{proof}
By definition we have 
$\|(P_y)_{y\in\U_Y}\|_{op,\Prod(\U_Y,\rho)}=\sup_{0\neq(f_y)_{y\in\U_Y}}\frac{
\sup_{y\in\U_Y}\|P_y(f_y)\|_{D(y,\rho)}}{\sup_{y\in\U_Y}\|f_y\|_{D(y,\rho)}}$. For all $y\in\U_Y$ we may consider the set $S_y\subset\Prod(\U_Y,\rho)$ of tuples $\mathbf{f}=(f_{z})_{z\in\U_Y}$ such that $f_z=0$ if $z\neq y$. We see that 
\begin{eqnarray}
\sup_{0\neq\mathbf{f}\in\Prod(\U_Y,\rho)}
\frac{
\sup_{z\in\U_Y}\|P_z(f_z)\|_{D(z,\rho)}}{\sup_{z\in\U_Y}\|f_z\|_{D(z,\rho)}}&\;\geq\;&
\sup_{0\neq\mathbf{f}\in S_y}\frac{
\sup_{z\in\U_Y}\|P_z(f_z)\|_{D(z,\rho)}}{\sup_{z\in\U_Y}\|f_z\|_{D(z,\rho)}}\\
&\;=\;&
\sup_{0\neq f_y\in\b_{\Omega}(D(y,\rho))}\frac{
\|P_y(f_y)\|_{D(y,\rho)}}{\|f_y\|_{D(y,\rho)}}
\;=\;\|P_y\|_{op,\b_\Omega(D(y,\rho))}\qquad
\end{eqnarray}
Which proves the inequality $\|(P_y)_{y\in\U_Y}\|_{op,\Prod(\U_Y,\rho)}\geq\sup_{y\in\U_Y} \|P_y\|_{op,\b_\Omega(D(y,\rho))}$. 

On the other hand, we have
\begin{eqnarray}
\sup_{0\neq\mathbf{f}\in\Prod(\U_Y,\rho)}
\frac{
\sup_{z\in\U_Y}\|P_z(f_z)\|_{D(z,\rho)}}{\sup_{z\in\U_Y}\|f_z\|_{D(z,\rho)}}&\;\leq\;&
\sup_{0\neq\mathbf{f}\in\Prod(\U_Y,\rho)}
\frac{
\sup_{z\in\U_Y}\bigl(\|P_z\|_{op,\b_\Omega(D(y,\rho))}
\cdot\|f_z\|_{D(z,\rho)}\bigr)}{\sup_{z\in\U_Y}\|f_z\|_{D(z,\rho)}}\\
&\;\leq\;&
\sup_{0\neq\mathbf{f}\in\Prod(\U_Y,\rho)}
\frac{
\Bigl(\sup_{z\in\U_Y}\|P_z\|_{op,\b_\Omega(D(y,\rho))}\Bigr)\cdot\Bigl(
\sup_{z\in\U_Y}\|f_z\|_{D(z,\rho)}\Bigr)}{\sup_{z\in\U_Y}\|f_z\|_{D(z,\rho)}}\quad\qquad
\end{eqnarray}
which shows the converse inequality.
\end{proof}

If $P\in\O(Y)\lr{d}$, the isometric inclusions of
\eqref{eq: isometric O(Y)-H(y)-P} imply
\begin{equation}\label{eq : inequalities of norms}
\|P\|_{op,\O(Y)}\;\leq\;\|P\|_{op,\Prod(\U_Y)}
\;\leq\;
\|P\|_{op,\Prod(\U_Y,\rho)}\;=\;\sup_{y\in\U_Y}\|P_y\|_{op,\b_\Omega(D(y,\rho))}\;.
\end{equation}

We now prove that we may 
shrink $Y$ around the point $x\in X$, in order to
 improve this estimate. 
\begin{remark}
Let $Y'\subset Y$ be two affinoid neighborhoods of $x$ 
in $X$. When restricting from $Y$ to $Y'$, 
we should pay attention 
to the fact that the restriction $\O(Y)\to\O(Y')$ is not an 
isometry and that it does not have a dense image in 
general. 
Therefore, the operators norms $\|P\|_{op,\O(Y)}$ and 
$\|P\|_{op,\O(Y')}$ are relatively unrelated each other.
\end{remark}
\begin{proposition}\label{proposition : comparison o norms}
Let $Y$ be an affinoid neighborhood of $x$ in $X$ such that 
$\widehat{\Omega}^1_{Y/K}$ is free, and let $d:\O(Y)\to\O(Y)$ 
be a derivation corresponding to a generator of $\widehat{\Omega}^1_{Y/K}$.  Let $\U_Y\subset Y$ 
be a star around $x$ containing the Shilov boundary of $Y$.

For all $P\in\O(Y)\lr{d}$ and $\varepsilon>0$, 
there exists a basis of $\psi$-\emph{admissible} affinoid neighborhoods 
$Y_{P,\varepsilon}$ of $x$ in $X$ such that, for all $\rho\in]0,1]$,
one has 
\begin{equation}\label{eq : estimation of norms}
\|P\|_{op,\O(Y_{P,\varepsilon})}\;\leq\;\|P\|_{op,
\b_\Omega(D(x,\rho))}+\varepsilon\;.
\end{equation}
More precisely, for all $\rho\in]0,1]$, there exists a $\psi$-\emph{admissible} affinoid 
neighborhood $Y_{P,\varepsilon,\rho}$ of $x$ in $X$ and 
a star $\U_{Y_{P,\varepsilon,\rho}}\subset  Y_{P,\varepsilon,\rho}\cap\U_Y$ containing the Shilov boundary of $Y_{P,\varepsilon,\rho}$,
such that 
\begin{equation}\label{eq : estimation of norms-2}
\|P\|_{op,\O(Y_{P,\varepsilon,\rho})}\;\leq\;
\|P\|_{op,\Prod(\U_{Y_{P,\varepsilon,\rho}},\rho)}
\;\leq\;\|P\|_{op,
\b_\Omega(D(x,\rho))}+\varepsilon\;.
\end{equation}
\end{proposition}
\begin{proof}
The bound \eqref{eq : estimation of norms} follows from 
\eqref{eq : estimation of norms-2}, with $Y_{P,\varepsilon}:=Y_{P,\varepsilon,1}$. Indeed,
the function $\rho \mapsto \|P\|_{op,
\b_\Omega(D(x,\rho))}$ is non increasing by 
\eqref{eq : explicit norm operator}.

Assume that $Y$ is a $\psi$-\emph{admissible}, neighborhood of $x$ 
(cf. Definition \ref{def : elementary neigh}). 
Let $[x,y]\subset\U_Y$ be a segment. 
By Lemma \ref{Lemma : H(x) stable by d/dt}, the segment $[x,y]$ is 
contained in some $V_j$, and the map 
$\psi_j:Y\to W_j$ gives the identification 
$\psi_j:D(z)\simto D(\psi_j(z))$ for all $z\in Y\cap V_j$. 
Let $r_j(z)$ be the 
radius of $D(\psi_j(z))$ with respect to a coordinate $T_j$ on 
$W_j$. The radius of $\psi_j (D(z,\rho))$ is then $\rho\cdot r_j(z)$. 
Let $d_j$ be the corresponding derivation of 
$\O(Y)$ (cf. Remark \ref{local derivation}). 
Then for all $z\in\U_Y\cap V_j$ one has (cf. Remark 
\ref{remark : def op norm d on bounded})
\begin{equation}
\|d_j\|_{op,\b_\Omega( D(z,\rho) )} \;=\; (\rho\cdot r_j(z))^{-1}\;.
\end{equation} 
As explained in Remark \ref{local derivation}, $P$ can be written as 
$P=\sum f_id_j^i$, with $f_i\in\O(Y)$. 
Together with \eqref{eq : explicit norm operator}, this proves that 
$z\mapsto \|P\|_{op,\b_\Omega( D(z,\rho) )}$ is a 
continuous function of $z$ along $[x,y]$. 
Up to restrict $[x,y]$ we may assume that 
$\|P\|_{op,\b_\Omega( D(z,\rho) )}\leq 
\|P\|_{op,\b_\Omega( D(x,\rho) )}+\varepsilon$ for all $z\in[x,y]$. 
As explained in Remark \ref{Remark : decreasing norm even though d_1}, the operator norm is 
intrinsic, 
therefore this bound is independent on the chosen 
maps $\psi_j$ and on the chosen derivation. 

Now, for all germ of segment $b$ out of $x$ chose $[x,y]\in b$ with 
this property, and define the star $\U$ as the union of all those 
segments. Then,
$\sup_{y\in\U}\|P\|_{op,\b_\Omega(D(y,\rho))}
\leq 
\|P\|_{op,\b_\Omega( D(x,\rho) )}+\varepsilon$. 
It follows that for all star-shaped neighborhood 
$Y_{P,\varepsilon,\rho}\subseteq Y$ of $x$, with 
Shilov boundary in $\U$, estimation 
\eqref{eq : estimation of norms-2} follows from 
 \eqref{eq : norm is the sup on generic disks} and 
\eqref{eq : inequalities of norms} with $\U_{Y_{P,\varepsilon,\rho}}:=\U\cap Y_{P,\varepsilon}$.
\end{proof}

\subsubsection{Norms on $A\lr{d}^{\leq n}$.}
Let $(A,\|.\|_A)$ be a normed algebra with a derivation $d:A\to A$. 
Let
\begin{equation}\label{def : norme sup on the coeff}
\|(f_0,\ldots,f_n)\|_A\;:=\;\max_i\|f_i\|_A\;
\end{equation}
be the sup-norm on $A^n$. Denote by
 $A\langle d\rangle^{\leq n}$ be the subset 
of differential operators $\sum_{i=0}^ng_id^i$ of order at most $n$. 
We can identify $A\lr{d}^{\leq n}$ to $A^{n+1}$ by associating to 
$\sum_{i=0}^n f_id^i$ the tuple $(f_0,\ldots,f_n)$. We denote the 
 resulting norm on $A\lr{d}^{\leq n}$ by
\begin{equation}\label{eq : def of sup norm on Ad}
\|\sum f_id^i\|_{A,d}\;:=\;\|(f_0,\ldots,f_n)\|_A\;=\;
\max_i\|f_i\|_A\;.
\end{equation}
\begin{lemma}\label{Lemma : norm sup equiv}
If $d_1$, $d_2$ are two derivations of $A$ generating $A\lr{d}$ (i.e. $A\lr{d_1}= A\lr{d_2}=A\lr{d}$). Then, 
$\|.\|_{A,d_1}$ and $\|.\|_{A,d_2}$ are
equivalent on $A\lr{d}^{\leq n}$.
\end{lemma}
\begin{proof}
Equality $A\lr{d_1}= A\lr{d_2}=A\lr{d}$ implies 
$d=ad_1=bd_2$ for some invertible elements $a,b\in A$.

The $A$-module $A\lr{d}^{\leq n}$ is finite and free.  The subsets
$\{1,d_1,\ldots,d_1^n\}$ and $\{1,d_2,\ldots,d_2^n\}$ 
are two basis. If $U\in GL_n(A)$ is the base change matrix, then 
$\|U^{-1}\|^{-1}_{A}\cdot\|.\|_{A,d_2}\leq\|.\|_{A,d_1}\leq 
\|U\|_A \cdot\|.\|_{A,d_1}$, where
$\|(u_{i,j})\|_A:=\max_{i,j}\|u_{i,j}\|_A$.
\end{proof}
\begin{proposition}\label{Prop : equivalent norms on O(Y)}
Let $0<\rho\leq 1$. The following claims hold:

\begin{enumerate}
\item\label{i Prop : equivalent norms on O(Y)}  Let $d:\b_\Omega(D(x,\rho))\to\b_\Omega(D(x,\rho))$ 
be a derivation of the form $d/dT$ where $T$ is a coordinate of $D(x)$ (cf. Sections \ref{Section: Bounded Omega 1} and \ref{Section 3.1.1 : derivations}).
Let $A$ be a normed ring, \emph{isometrically} 
included in 
$\b_\Omega(D(x,\rho))$ and stable under $d$. 
Then $\|.\|_{op,\b_\Omega(D(x,\rho))}$ and 
$\|.\|_{A,d}$ are equivalent as norms on $A\lr{d}^{\leq n}$.
\item\label{ii Prop : equivalent norms on O(Y)}
Let $Y$ be a star-shaped $\psi$-\emph{admissible}, affinoid neighborhood 
of $x$ (cf. Definitions \ref{def:starshapedopen} and  \ref{def : elementary neigh}), and let $d:\O(Y)\to\O(Y)$ 
be a derivation corresponding to a generator of 
$\widehat{\Omega}^1_{Y/K}$.\footnote{Recall that $\widehat{\Omega}^1_{Y/K}$ is free by item \eqref{iii : omegaonefree} of Theorem \ref{thm:bonvois}.}   
Let $\U_Y\subset Y$ be a star around $x$ containing the Shilov 
boundary of $Y$. Then 
$\|.\|_{op,\Prod(\U_Y,\rho)}$ and $\|.\|_{\O(Y),d}$
are equivalent as norms on $\O(Y)\lr{d}^{\leq n}$.
\end{enumerate}
\end{proposition}
\begin{proof}
In \eqref{i Prop : equivalent norms on O(Y)} we can assume $A=\b_\Omega(D(x,\rho))$. 
The radius of $D(x,\rho)$ with respect to $T$ 
is $\rho\cdot r(x)$, where $r(x)$ is the radius of $D(x)$ in this 
coordinate. By \eqref{eq : explicit norm operator} one has 
\begin{equation}
\Bigl\|\sum_{k=0}^nf_k\cdot (d/dT)^k
\Bigr\|_{op,\b_{\Omega}(D(x,\rho))}\;=\;
\max_{k=0,\ldots,n}|k!|\cdot 
|f_k|(x)\cdot 
(\rho\cdot r(x))^{-k}\;.
\end{equation}
Let $C_1:=\min_k|k!|(\rho\cdot r(x))^{-k}$ and 
$C_2:=\max_k|k!|(\rho\cdot r(x))^{-k}$, then
\begin{equation}
C_1\|.\|_{A,d/dT}\;\leq\;\|.\|_{op,\b_{\Omega}(D(x,\rho))}\;\leq\;
C_2\|.\|_{A,d/dT}\;.
\end{equation}
We now prove \eqref{ii Prop : equivalent norms on O(Y)}. 
With the notations of Remark \ref{local derivation} one has 
$Y\subseteq\bigcup_j V_j$, hence (cf. Lemma \ref{Lemma: op UY=sup op B})
\begin{eqnarray}
\|.\|_{op,\Prod(\U_Y,\rho)}&\;=\;&
\max_j\|.\|_{op,\Prod(\U_Y\cap V_j,\rho)}\;,\\
\|.\|_{\O(Y),d}&\;=\;&\max_j\|.\|_{\O(V_j\cap Y),d}\;.
\end{eqnarray}
Therefore, it is enough to prove that, for all $j$, 
the norms $\|.\|_{op,\Prod(\U_Y\cap V_j,\rho)}$ and 
$\|.\|_{\O(V_j\cap Y),d}$ are equivalent.
Now, by Lemma \ref{Lemma : norm sup equiv} we can replace $d$ by 
the derivation $d_j=d/dT_j$ of Remark \ref{local derivation}. 
The reason of this choice is that the map
$\psi_j:V_j\to W_j$ identifies $D(y,\rho)$ with 
$D(\psi_j(y),\rho)$ for all $y\in V_j$. 
So the coordinate $T_j$ 
on $W_j$ is then \emph{simultaneously} a coordinate on $D(y,\rho)$, 
for all $y\in\U_Y\cap V_j$. 
In particular, if $P=\sum_{k=0}^nf_{j,k} d_j^k
\in\O(Y\cap V_j)\lr{d_j}^{\leq n}$, with $f_{j,k}\in\O(Y\cap 
V_j)$, then 
for all $y\in \U_Y\cap V_j$ the image of $P$ in $\b_\Omega(D(y,\rho))
\lr{d_j}^{\leq n}$ is given by (cf. \eqref{eq : P=(P_y)_y})
\begin{equation}
P_y\;=\;\sum_{k=0}^nf_{j,k} \Bigl(\frac{d}{dT_j}\Bigr)^k\;\in\;\b_\Omega(D(y,\rho))
\lr{d_j}^{\leq n}\;.
\end{equation}
As a consequence, we have the explicit formula (cf. Lemma \ref{Lemma: op UY=sup op B} and Remark \ref{Remark : d/dT coord on D(x)})
\begin{equation}\label{eq : explicit formula}
\|\sum_{k=0}^nf_{k,j}d_j^k
\|_{op,\Prod(\U_Y\cap V_j,\rho)} 
\;=\;\sup_{y\in\U_Y\cap V_j}\;\;
\max_{k=0,\ldots,n}|k!|\cdot|f_{k,j}|(y)\cdot
(\rho \cdot r_j(y))^{-k}\;,
\end{equation}
where $r_j(y)$ is the radius of $D(\psi_j(y))$ 
with respect to $T_j$. 
Moreover
\begin{equation}\label{eq : 4.14reh;o}
\sup_{y\in\U_Y\cap V_j}
\max_{k=0,\ldots,n}|k!|\cdot|f_{k,j}|(y)\cdot(\rho \cdot r_j(y))^{-k}
\;\leq\;
\max_{k=0,\ldots,n}|k!|\cdot\|f_{k,j}\|_{V_j\cap Y}
\Bigl(\inf_{y\in\U_Y\cap V_j}\rho\cdot r_j(y)\Bigr)^{-k}\;.
\end{equation}
Now, by definition of star 
(cf. item \eqref{ii: star} at the beginning of Section 
\ref{Approximation of operator norms.}) one has  
$\inf_{y\in\U_Y\cap V_j}r_j(y)>0$. 
If $r^-_j:=\inf_{y\in\U_Y\cap V_j}\rho\cdot r_j(y)$, 
then from \eqref{eq : explicit formula} and \eqref{eq : 4.14reh;o} 
one gets
\begin{equation}
\|\sum_{k=0}^nf_{k,j}d_j^k\|_{op,\Prod(\U_Y\cap V_j,\rho)} 
\;\leq\;
\max_{k=0,\ldots,n}|k!|\cdot\|f_{k,j}\|_{V_j\cap Y}
\cdot (r_j^-)^{-k}\;.
\end{equation}
If $C_2:=\max_{k}|k!|\cdot(r_j^-)^{-k}$, then
$\|.\|_{op,\Prod(\U_Y\cap V_j,\rho)}\leq C_2
\cdot\|.\|_{\O(Y\cap V_j),d_j}$.

On the other hand, by the definition of $V_j$, the Shilov boundary of $V_j\cap Y$ is contained in the union of $\{x\}$ with that of $Y$ (cf. section \ref{Section : a description of Ducros etale pam}). Hence  $\U_Y\cap V_j$ contains the Shilov boundary $S_j$ of $Y\cap V_j$. 
Therefore, by \eqref{eq : explicit formula}, 
it follows that
\begin{eqnarray}
\|\sum_{k=0}^nf_{k,j}d_j^k
\|_{op,\Prod(\U_Y\cap V_j,\rho)} 
&\;\geq\;&
\max_{k=0,\ldots,n}|k!|\cdot
\sup_{y\in S_j}
\Bigl(|f_{k,j}|(y)\cdot
(\rho \cdot r_j(y))^{-k}\Bigr)\;.\label{efds}
\end{eqnarray}
Let $r_j^+:=\sup_{z\in W_j}r(z)$, where $r(z)$ is 
the radius of the point (cf. Definition~\ref{Radius of a point.}).
Then $r_j^+<+\infty$, because $W_j$ is an affinoid domain of 
$\mathbb{A}^{1,\mathrm{an}}_K$. For every 
$y\in V_j$ we have $r(y)\leq r_j^+$, 
because $r(y)=r(\psi_j(y))$. 
Inequality \eqref{efds} then implies
\begin{eqnarray}
\|\sum_{k=0}^nf_{k,j}d_j^k
\|_{op,\Prod(\U_Y\cap V_j,\rho)} 
&\;\geq\;&
\max_{k=0,\ldots,n}|k!|(r_j^+)^{-k}\cdot
\sup_{y\in S_j}|f_{k,j}|(y)\;=\;
\max_{k=0,\ldots,n}|k!|
\|f_{k,j}\|_{V_j\cap Y}(r_j^+)^{-k}\;.\qquad\;\;
\end{eqnarray}
If $C_1:=\min_k|k!|(r_j^+)^{-k}>0$, then  one has 
$\|.\|_{op,\Prod(\U_Y\cap V_j,\rho)}\geq C_1\cdot
\|.\|_{\O(V_j\cap Y),d_j}$ as required.
\end{proof}

\subsection{Descending Robba's 
decomposition to $\O_{X,x}$}\label{D-R liftingsection}
Let $x\in X$ be a point of type $2$, $3$, or $4$. 
Let $d:\O_{X,x}\to\O_{X,x}$ be a derivation 
generating the $\O_{X,x}$-module of $K$-linear 
continuous derivations of $\O_{X,x}$. 
Since $\O_{X,x}$ is a field, all 
differential modules $\M_x$ 
are cyclic. By Section 
\ref{Filtrations of cyclic modules, and factorization of 
operators.}, we may translate the assertion 
\eqref{M_x^geqrho} in the proof of Theorem \ref{Dw-Robba} in terms of differential operators as 
follows.
\begin{theorem}[Dwork-Robba's decomposition]\label{Dwork-Robba}
Let $\rho\leq 1$. Let $P \in 
\O_{X,x}\langle d\rangle$ be a monic differential 
polynomials corresponding to $\M$. 
Let $P=P^{\geq\rho}\cdot P^{<\rho}$ be a 
factorization in $\H(x)\lr{d}$ 
corresponding to the Robba's decomposition 
\eqref{eq : Robba deco}
\begin{equation}
0\to\M^{\geq \rho}\to \M\to\M^{<\rho}\to 0\;.
\end{equation}
That is, following the notations of 
\eqref{eq : notation M_P}, we have $\M\cong\M_P$, $\M^{\geq\rho}\cong\M_{P^{\geq \rho}}$ and
$\M^{<\rho}\cong\M_{P^{< \rho}}$.
Then 
$P^{\geq \rho}$ and $P^{<\rho}$ also belongs to 
$\O_{X,x}\langle d\rangle$.
\end{theorem}
\begin{proof}
\if{By remark \ref{Rk : finite number of steps} we can assume that 
(this will be used in Lemma
\ref{Lemma : zeta works well} and Prop. \ref{Lemma contractive})
\begin{equation}\label{eq: essumption rho<1}
\rho\;<\;1\;.
\end{equation}
}\fi

Write $P=d^n+\sum_{k=0}^{n-1}f_kd^k\in\O_{X,x}\lr{d}$. 
Let $Y$ be a star-shaped $\psi$-\emph{admissible}, affinoid neighborhood of $x$ 
(cf. Definitions \ref{def:starshapedopen} and  \ref{def : elementary neigh}) such that the coefficients $f_k$ of $P$ 
lies in $\O(Y)$. Recall that, by item \eqref{iii : omegaonefree} of Theorem \ref{thm:bonvois}, $\widehat{\Omega}^1_{Y/K}$ 
is free. Choose a derivation $d:\O(Y)\to\O(Y)$ corresponding to a 
generator of $\widehat{\Omega}^1_{Y/K}$. 
Let $A$ be one of the rings 
$\O(Y)$, $\H(x)$, $\b_{\Omega}(D(x,\rho))$. 
Consider the injective maps 
\begin{equation}
\beta_n\;:\;A^n\; \hookrightarrow\;
A\langle d\rangle^{\leq n}\;,\qquad\quad
\widetilde{\beta}_n\;:\;A^n\; \hookrightarrow\;
A\langle d\rangle^{\leq n}
\end{equation}
defined by 
\begin{equation}
\widetilde{\beta}_n(g_0,\ldots,g_{n-1}) 
\;:=\;\sum_{i=0}^{n-1}g_id^i\;,
\end{equation}
and 
\begin{equation}
\beta_n(g_0, \ldots,g_{n-1})
\;:=\;
d^n+\widetilde{\beta}_n(g_0,\ldots,g_{n-1})\;.
\end{equation}
Let $n_{1}$, $n_{2}$, $n=n_1+n_2$ be the orders of 
$P^{\geq \rho}$, $P^{< \rho}$, $P$ respectively. 
The multiplication in $A\langle d\rangle$ provides a diagram
\begin{equation}\label{eq : diagram Dw-Ro}
\xymatrix{
&A\langle d\rangle^{\leq n_1}\times 
A\langle d\rangle^{\leq n_2}\ar[rr]^-{\mathfrak{m}}\ar@{}[rrd]|-{\circlearrowleft}&&A\langle d\rangle^{\leq n}\\
A^n\ar@/_{15pt}/[rrr]_{G_A}
\ar[ur]^-{\gamma}\ar[r]_{\delta}^{\sim}&A^{n_1}\times 
A^{n_2} \ar@{-->}[rr]
\ar@{^(->}[u]_-{\beta_{n_1}\times\beta_{n_2}}
&&A^n\ar@{^(->}[u]_-{\widetilde{\beta}_n}
}
\end{equation}
where 
\begin{enumerate}
\item one has 
\begin{equation}
\mathfrak{m}(Q_1,Q_2)\;:=\;Q_1Q_2-P\;,
\end{equation}
\item $\delta$ identifies $A^n$ with $A^{n_1}\times A^{n_2}$ by 
$(v_1,\ldots,v_n)\mapsto ((v_1,\ldots,v_{n_1}),
(v_{n_1+1},\ldots,v_{n_1+n_2}))$,
\item 
$\gamma:=(\beta_{n_1}\times\beta_{n_2})\circ\delta$ 
sends $(v_1,\ldots,v_n)$ onto the pair 
$(d^{n_1}+\sum_{k=0}^{n_1-1}v_{k+1}d^k,d^{n_2}+
\sum_{k=0}^{n_2-1}v_{n_1+k+1}d^k)$.
\end{enumerate}
The image of $\mathfrak{m}\circ\gamma$ is contained 
in that of $\widetilde{\beta}_n$, because $P$ is monic.
We are then allowed to set  
\begin{equation}
G_A\;:=\;\widetilde{\beta}_n^{-1}\circ \mathfrak{m}\circ\gamma
\end{equation}
the map $A^{n}\to A^n$ obtained in this way.

Let  $\bs{u}\in\H(x)^n$ be such that  
$\gamma(\bs{u})=(P^{\geq \rho},P^{<\rho}) 
\in \H(x)\lr{d}^{\leq n_1}\times \H(x)\lr{d}^{\leq n_2}$. Then 
\begin{equation}\label{eq : G_A(u)=0}
G_{\H(x)}(\bs{u})\;=\;0\;.
\end{equation}
This is actually a non linear system of differential equations on the entries 
of $\bs{u}$ (cf. \eqref{eq : def G_A} below).
We have to prove that there exists an affinoid neighborhood 
$Y$ of $x$ in $X$ such that 
$\bs{u}\in \O(Y)^{n}$ (i.e. the coefficients of $P^{\geq\rho}$ and $P^{<\rho}$ lie 
in $\O(Y)$).
For all $\bs{v}=(v_1,\ldots,v_n)\in A^n$, the vector 
$G_A(\bs{v})\in A^n$ is a finite sum of the form
\begin{equation}\label{eq : def G_A}
G_A(\bs{v})\;=\;\sum_{
\begin{smallmatrix}i,j=1,\ldots,n\\
k= 0,\ldots,n_1
\end{smallmatrix}}\bs{C}_{i,j,k}\cdot v_i \cdot 
d^k(v_j)-\bs{\ell}\;,
\end{equation}
where $\bs{C}_{i,j,k}\in\mathbb{N}^n$ are tuples of natural numbers and 
the vector $\bs{\ell}$ lies in 
$\O(Y)^n$ for some affinoid neighborhood $Y$ of $X$ 
(indeed $\bs{\ell}$ is the tuple associated with the coefficients of $P$). 

We linearize the problem as follows. 
Let 
$\{X_i^{(k)}\}_{i=1,\ldots,n;\;k= 0,\ldots,n_1}$ 
be a family of indeterminates. 
For all $0\leq k\leq n_1$, set 
$\bs{X}^{(k)}:=(X_1^{(k)},\ldots,X_n^{(k)})$ and 
\begin{equation}
\bs{X}\;:=\;(X_1^{(0)},\ldots,X_n^{(0)},X_1^{(1)},\ldots,X_n^{(1)},
\ldots,X_1^{(n_1)},\ldots,X_n^{(n_1)})\;.
\end{equation}
With the notations of \eqref{eq : def G_A} let
\begin{equation}
F(\bs{X})\;=\;F(\bs{X}^{(0)},\ldots,
\bs{X}^{(n_1)})\;:=\;
\sum_{i,j,k}\bs{C}_{i,j,k}\cdot X_i^{(0)} \cdot 
X_j^{(k)}-\bs{\ell}\;.
\end{equation}
This is now a polynomial with coefficients in $\O(Y)^n$, and it defines a map
\begin{equation}
F\;:\;A^{n(n_1+1)}\xrightarrow[]{\qquad}A^n\;.
\end{equation}
Denote by 
\begin{equation}
\zeta\;:\;A^n\;\xrightarrow[]{\quad}\; A^{n\cdot (n_1+1)}
\end{equation}
the $\mathbb{Z}$-linear 
map 
\begin{equation}
\zeta(\bs{v})\;:=\;(\bs{v},d(\bs{v}),\ldots,d^{n_1}(\bs{v}))\;.
\end{equation}
Then, for every $\bs{v}\in A^n$ we have a commutative diagram
\begin{equation}
\xymatrix{
A^n\ar[d]_-{G_A}\ar[r]^-{\zeta}&A^{n(n_1+1)}\ar[dl]^-{F}\\
A^n&
}
\end{equation}
in which
\begin{equation}\label{eq : G_A=F}
G_A(\bs{v})\;=\;F(\zeta(\bs{v}))\;=\;F(\bs{X}^{(0)},\bs{X}^{(1)},
\ldots,
\bs{X}^{(n_1)})_{|X_i^{(k)}=d^k(v_i)}\;.
\end{equation}
The above process of linearization corresponds 
to working on the jet-space 
(tangent space if $n_1=1$).
Since $F$ is a polynomial expression in $\bs{X}$,  
the Taylor formula gives 
\begin{equation}\label{eq : Taylor empq}
F(\bs{X}+\bs{Y})\;=\;
F(\bs{X})+dF_{\bs{X}}({\bs{Y}})+N_{\bs{X}}(\bs{Y})\;.
\end{equation}
where
\begin{equation}\label{eq : trqsdfpv}
dF_{\bs{X}}(\bs{Y}^{(0)},\ldots,\bs{Y}^{(n_1)})\;=\;\sum_{
\begin{smallmatrix}i=1,\ldots,n\\
k= 0,\ldots,n_1
\end{smallmatrix}
}\frac{\partial F}{\partial X_i^{(k)}}(\bs{X})\cdot Y_i^{(k)}\;,
\end{equation}
is the linear part, and $N_{\bs{X}}(\bs{Y})$ is the non linear part. So 
\begin{equation}\label{eq : G=G+L+N}
G_A(\bs{v}+\bs{\xi})\;=\;G_A(\bs{v})+dG_{A,\bs{v}}
(\bs{\xi})+N_{\zeta(\bs{v})}(\zeta(\bs{\xi}))\;,
\end{equation}
where 
\begin{equation}
dG_{A,\bs{v}}
(\bs{\xi})\;=\;
dF_{\zeta(\bs{v})}
(\zeta(\bs{\xi}))\;.
\end{equation} 
Applied to $\bs{u}\in\H(x)^n$ (cf. \eqref{eq : G_A(u)=0}) 
this gives
\begin{equation}\label{eq : explicit fixed point}
0\;=\;G_{\H(x)}(\bs{u})\;=\; 
G_{\H(x)}(\bs{v}) + 
dG_{\H(x),\bs{v}}(\bs{u}-\bs{v})+
N_{\zeta(\bs{v})}(\zeta(\bs{u}-\bs{v}))\;.
\end{equation}
Once we fix $\bs{v}\in A^n$, the condition in this form can be interpreted as an 
equation in the indeterminate $\bs{\xi}=\bs{u}-\bs{v}\in A^n$.
More precisely, in the following sections \ref{section : computation of df}, 
\ref{section : Some estimations.}, \ref{bijectivity}, 
\ref{section : contractive..}, we prove that if $Y$ is a small enough neighborhood of $x\in X$ and if the image of 
$\bs{v}\in\O(Y)^n$ in 
$\H(x)^n$ is close to $\bs{u}$, then  the maps 
$dG_{\H(x),\bs{v}}:\H(x)^n\to \H(x)^n$ and $dG_{\O(Y),\bs{v}}:\O(Y)^n\to \O(Y)^n$ are both bijectives. We then fix such $\bs{v}\in\O(Y)^n$ and prove that, for both 
$A=\H(x)$ and $A=\O(Y)$, the map
\begin{equation}
\phi_{\bs{v},A}(\bs{\xi})\;:=\;-dG_{A,\bs{v}}^{-1}
\Bigl(G_A(\bs{v})+N_{\zeta(\bs{v})}(\zeta(\bs{\xi}))\Bigr)
\end{equation}
is a contraction of a poly-disk 
$\bs{D}_{A}^+(0,q)=\{\bs{\xi}\in A^n\textrm{ such that }\|
\bs{\xi}\|_{A}\leq q\}$, whose radius $q$ satisfies 
$\|\bs{v}-\bs{u}\|_{\H(x)}<q$.
This implies that $\phi_{\bs{v},A}$ has a unique fixed point $\bs{\xi}_{\phi}$ in $\bs{D}_{A}^+(0,q)$. 
Since it is compatible with the inclusion 
$\bs{D}_{\O(Y)}^+(0,q)\subset \bs{D}_{\H(x)}^+(0,q)$, 
the fixed point of 
$\phi_{\bs{v},\H(x)}$ coincides with
that of $\phi_{\bs{v},\O(Y)}$. 
If $A=\H(x)$ the fixed 
point is $\bs{u}-\bs{v}$ by \eqref{eq : explicit fixed point}. 
This proves that $\bs{u}-\bs{v}\in \bs{D}_{\O(Y)}^+(0,q)$. 
In particular, this shows that $\bs{u}\in\O(Y)^n$ as desired.
\end{proof}

\subsubsection{Computation of $dG_{A,\bs{v}}$.}
\label{section : computation of df}
\if{
Let $A$ be equal to $\H(x)$, $\O(Y)$, or $\b_\Omega(D(x,\rho))$.
Consider $(A\lr{d}^{\leq n},\|.\|_{A,d})$ as a normed 
$K$-vector space. It is isomorphic to $(A^n,\|.\|_{A})$ by 
\eqref{def : norme sup on the coeff}.
Define as usual the differential of a map $\varphi:A^n\to A^m$ at 
$\bs{v}_0\in A^n$ as a continuous linear map 
$d\varphi_{\bs{v}_0}:A^n\to A^m$ satisfying 
\begin{equation}
\lim_{\bs{v}\to\bs{v}_0}\frac{\|\varphi(\bs{v})-\varphi(\bs{v}_0)-
d\varphi_{\bs{v}_0}(\bs{v})\|_A}{\|\bs{v}-\bs{v}_0\|_A}\;=\;0\;.
\end{equation}

\framebox{Semplificare tutto qui}

With this definition the 
differential $dG_{A,\bs{v}}$ of $G_A$ at $\bs{v}$ is equal to 
(cf. \eqref{eq : G=G+L+N}) 
\begin{equation}
dG_{A,\bs{v}}(\bs{\xi})\;=\;dF_{\zeta(\bs{v})}(\zeta(\bs{\xi}))
\end{equation} 
because $\zeta$ is $K$-linear and continuous.
We now compute $dG_{A,\bs{v}}$ in another way. 
}\fi
We consider $(A\lr{d}^{\leq n},\|.\|_{A,d})$ as a normed 
$K$-vector space. It is isomorphic to $(A^n,\|.\|_{A})$ by 
\eqref{def : norme sup on the coeff}. 
The definition of differential of a function 
$A^n\to A^m$ then has a meaning.

If $A=\O(Y)$ (resp. $A=\H(x)$, $\b_\Omega(D(x,\rho))$), the product 
$\mathfrak{m}:A\langle d\rangle^{\leq n_1}\times 
A\langle d\rangle^{\leq n_2}\to A\langle d\rangle^{\leq n}$ 
is a $K$-bilinear map which is continuous with respect to 
$\|.\|_{op,\Prod(\U_Y,\rho)}$ (resp. 
$\|.\|_{op,\b_\Omega(D(x,\rho))}$), where $\U_Y$ is a star around $x$ containing the Shilov boundary of $Y$. 
By Proposition \ref{Prop : equivalent norms on O(Y)} the same is true 
with respect to the equivalent norm $\|.\|_A$ on $A^n$.
The differential $dG_{A,\bs{v}}$ is then %
given by 
\begin{equation}\label{eq : L_v(G_A) with mult}
dG_{A,\bs{v}}\;=\;d\mathfrak{m}_{(P_{1,\bs{v}},P_{2,\bs{v}})}
\circ 
(\widetilde{\beta}_{n_1}\times\widetilde{\beta}_{n_2})\circ\delta\;,
\end{equation}
where $(P_{1,\bs{v}},P_{2,\bs{v}}):=\gamma(\bs{v})$. 
Notice here that $\widetilde{\beta}_s=(d\beta_{s})_{\bs{v}}$ and 
$d\delta_{\bs{v}}=\delta$ for all $\bs{v}\in A$, and all $s\geq 1$. 

Now, since 
the multiplication in $A\lr{d}$ is a continuous 
$K$-bilinear map, the differential %
of $\mathfrak{m}$ is
\begin{equation}
d\mathfrak{m}_{(P_{1,\bs{v}},P_{2,\bs{v}})}
(Q_1,Q_2)\;=\;Q_1P_{2,\bs{v}}+P_{1,\bs{v}}Q_2\;.
\end{equation}
\if{
\begin{remark}\label{Rk : image of beta x beta}
By \eqref{eq : L_v(G_A) with mult} the differential  $dG_{A,\bs{v}}$ 
can be interpreted as a map defined in the image of 
$(\widetilde{\beta}_{n_1}\times\widetilde{\beta}_{n_2})$. 
Following this identification one has
\begin{equation}
dG_{A,\bs{v}}(P_1,P_2)\;=\;P_1Q_{2,\bs{v}}+Q_{1,\bs{v}}P_2
\;.
\end{equation}
where $P_i$ is now a differential operator in $A\lr{d}^{\leq n_i-1}$ of 
degree \emph{strictly} less than $n_i$.
\end{remark}
}\fi

\subsubsection{Some estimations.}\label{section : Some estimations.}
We maintain the above notations. Let 
\begin{equation}
\bs{D}^+_A(0,r):=\{\bs{x}\in A^{n\cdot (n_1+1)}\;,\|\bs{x}\|_A
\leq r\}\;,
\end{equation}
 and let 
\begin{equation}
\|F\|_{\bs{D}_A^+(0,r)}\;:=\;\max_{\bs{x}\in A^{n(n_1+1)}\;,\;\|\bs{x}\|_A\leq r}\|F(\bs{x})\|_A\;.
\end{equation}
On the other hand, the map $\zeta:A^{n}\to A^{n(n_1+1)}$ is linear 
and bounded. We denote its operator norm by $\|\zeta\|_{op,A}$, for all $\bs{x}\in A^n$ we have
\begin{equation}\label{zeta opn}
\|\zeta(\bs{x})\|_{A}\;\leq\;\|\zeta\|_{op,A}\cdot\|\bs{x}\|_A\;.
\end{equation}
\begin{lemma}
Let $A$ be one of the rings $\H(x),\O(Y),\b_\Omega(D(x,\rho))$. 
Let $r>0$. The following hold:
\begin{enumerate}
\item\label{LS : i} Let $\bs{u}\in \H(x)^n$ be the zero of $G_{\H(x)}$ 
(cf. \eqref{eq : G_A(u)=0}), and let 
$\bs{v}\in\b_\Omega(D(x,\rho))^n$. 
If 
\begin{equation}
\|\bs{v}\|_{\b_\Omega(D(x,\rho))}\;,\;
\|\bs{u}\|_{\b_\Omega(D(x,\rho))}\;\leq\; 
r\|\zeta\|^{-1}_{op,\b_{\Omega}(D(x,\rho))}\;,
\end{equation} 
then 
\begin{equation}\label{eq : estim G(v)}
\|G_{\b_\Omega(D(x,\rho))}(\bs{v})\|_{\b_\Omega(D(x,\rho))}\;\leq\;
r^{-1}\cdot\|F\|_{\bs{D}_{\b_\Omega(D(x,\rho))}^+(0,r)}\cdot
\|\zeta\|_{op,\b_\Omega(D(x,\rho))}\cdot
\|\bs{v}-\bs{u}\|_{\b_\Omega(D(x,\rho))}\;;
\end{equation}
\item\label{LS : ii} If $\bs{v},\bs{\xi}\in A^n$ and if  $\|\bs{v}\|_A\leq r \|\zeta\|_{op,A}^{-1}$, then 
\begin{equation}
\|dG_{A,\bs{v}}(\bs{\xi})\|_A\;\leq\; 
r^{-1}\cdot 
\|F\|_{\bs{D}_A^+(0,r)}\cdot \|\zeta\|_{op,A}\cdot\|\bs{\xi}\|_A\;;
\end{equation}
\item \label{LS : iii} If $\bs{v}_1,\bs{v}_2,\bs{\xi}\in A^n$ and if 
$\|\bs{v}_1\|_A,\|\bs{v}_2\|_A\leq r\|\zeta\|^{-1}_{op,A}$, then
\begin{equation}\label{eq : norm dG_v1-dG_v2}
\|(dG_{\bs{v}_1}-dG_{\bs{v}_2})(\bs{\xi})\|_A\;\leq\; r^{-2}\cdot
\|F\|_{\bs{D}_A^+(0,r)}\cdot\|\zeta\|_{op,A}^2\cdot
\|\bs{v}_1-\bs{v}_2\|_A
\cdot\|\bs{\xi}\|_A\;;
\end{equation}
\item \label{LS : iv} If $\bs{v},\bs{\xi}\in A^n$ and if  $\|\bs{v}\|_A,\|\bs{\xi}\|_{A}\leq r\|\zeta\|_{op,A}^{-1}$, then 
\begin{equation}\label{eq : estim N_v(xi)}
\|N_{\zeta(\bs{v})}(\zeta(\bs{\xi}))\|_A\leq r^{-2}\cdot
\|F\|_{\bs{D}_A^+(0,r)}\cdot
\|\zeta\|_{op,A}^2\cdot 
\|\bs{\xi}\|_A^2\;;
\end{equation}
\item \label{LS : v} If $\bs{v},\bs{\xi}_1,\bs{\xi}_2\in A^n$ and if $\|\bs{v}\|_A,\|\bs{\xi}_1\|_A,\|\bs{\xi}_2\|_A\leq 
r\|\zeta\|_{op,A}^{-1}$, then
\begin{equation}\label{eq : norm of N_v(xi_1)-N_v(xi_2)}
\!\!\!\!\!
\|N_{\zeta(\bs{v})}(\zeta(\bs{\xi}_1))-
N_{\zeta(\bs{v})}(\zeta(\bs{\xi}_2))\|_A\;\leq\; 
r^{-2}\cdot
\|F\|_{\bs{D}_A^+(0,r)}\cdot
\|\zeta\|_{op,A}^2\cdot
\|\bs{\xi}_1-\bs{\xi}_2\|_A\cdot
\max(\|\bs{\xi}_1\|_A,\|\bs{\xi}_2\|_A)\;.
\end{equation}
\end{enumerate}
\end{lemma}
\begin{proof}
Write Taylor's formula as 
\begin{equation}
F(\bs{X}+\bs{Y})\;=\;
\sum_{\alpha\in\mathbb{N}^{n(n_1+1)}}
D_\alpha(F)(\bs{X})\bs{Y}^\alpha\;,
\end{equation}
where for $\alpha=(\alpha_{i,k})_{i\in\{1,\ldots,n\},k\in\{0,\ldots,n_1\}}
\in\mathbb{N}^{n(n_1+1)}$, we set
\begin{equation}
\bs{Y}^\alpha\;:=\;\prod_{i,k}(Y_i^{(k)})^{\alpha_{i,k}}
\end{equation}
and 
\begin{equation}
D_\alpha\;:=\;\prod_{i=1,\ldots,n;\;k=0,\ldots,n_1}
\frac{1}{\alpha_{i,k}!}\cdot 
\Bigl(\frac{d}{dX_{i}^{(k)}}\Bigr)^{\alpha_{i,k}}\;.
\end{equation}
The assumptions we are going to make will imply that 
$\bs{X},\bs{Y}\in\bs{D}^+_A(0,r)$ 
(i.e. $\|\bs{X}\|_{A},\|\bs{Y}\|_A\leq r$) for some convenient $r$. 
We recall that, if 
\begin{equation}
|\alpha|\;=\;\sum_{i=1,\ldots,n;\;k=0,\ldots,n_1}\alpha_{i,k}\;,
\end{equation}
then for all $\bs{X}\in \bs{D}^+_A(0,r)$ one has 
\begin{equation}\label{eq : bound of the derivation multiv}
\|D_\alpha(F)(\bs{X})\|_A\;
\leq \;
\frac{\|F\|_{\bs{D}_A^+(0,r)}}{r^{|\alpha|}}\;.
\end{equation}

Let us begin with item \eqref{LS : iv}. 
One has $N_{\bs{X}}(\bs{Y})=
\sum_{|\alpha|\geq 2}D_\alpha(F)(\bs{X})\bs{Y}^\alpha$. Therefore,  for 
$\|\bs{X}\|_A,\|\bs{Y}\|_A\leq r$ one finds
\begin{equation}\label{eq : bound for N_X.ert}
\|N_{\bs{X}}(\bs{Y})\|_A\leq
\sup_{|\alpha|\geq 2}\|D_\alpha(F)(\bs{X})\|_A\|\bs{Y}^{\alpha}\|_A
\leq
\sup_{|\alpha|\geq 2}\|F\|_{\bs{D}^+_A(0,r)}
\Bigl(\frac{\|\bs{Y}\|_A}{r}\Bigr)^{|\alpha|}=
\|F\|_{\bs{D}^+_A(0,r)}\Bigl(\frac{\|\bs{Y}\|_A}{r}\Bigr)^2\;.
\end{equation}
Replacing $\bs{X}$ and $\bs{Y}$ by $\zeta(\bs{v})$ and $\zeta(\bs{\xi})$ respectively yields the desired inequality by \eqref{zeta opn}.

Let us prove \eqref{LS : i}. We start from 
\eqref{eq : Taylor empq} and perform the substitution 
$\bs{X}=\zeta(\bs{u})$, and 
$\bs{Y}=\zeta(\bs{v}-\bs{u})$. The claim then follows from the fact 
$F(\zeta(\bs{u}))=G_A(\bs{u})=0$, together with the bound 
\eqref{eq : bound for N_X.ert}, and the following inequality
(cf. \eqref{eq : bound of the derivation multiv})
\begin{eqnarray}\label{eq: frsqmpeeee}
\|dF_{\bs{X}}(\bs{Y})\|_{\bs{D}^+_{A}(0,r)}&\;\leq\;&
r^{-1}\|F\|_{\bs{D}^+_{A}(0,r)}\|\bs{Y}\|_A\;.
\end{eqnarray}

Similarly, \eqref{LS : ii} follows  from \eqref{eq : trqsdfpv}, and 
$\|\frac{\partial F}{\partial X_i^{(k)}}\|_{\bs{D}^+_{A}(0,r)}
\leq r^{-1}\|F\|_{\bs{D}_A^+(0,r)}$.

Let us prove \eqref{LS : iii}. Let $H:\bs{D}^+_A(0,r)\to A^n$ be any power series converging 
on $\bs{D}^+_A(0,r)$. The Taylor expansion of $H(\bs{X}_1)$ around 
$\bs{X}_2$ gives for $\bs{Z}:=\bs{X}_1-\bs{X}_2$
\begin{equation}
H(\bs{X}_1)-H(\bs{X}_2)=H(\bs{X}_2+\bs{Z})-
H(\bs{X}_2)=\sum_{|\alpha|\geq 1}
D_\alpha(H)(\bs{X}_2)\bs{Z}^\alpha\;.
\end{equation}
So for $\|\bs{X}_1\|_A, \|\bs{X}_2\|_A, \|\bs{Z}\|_A\leq r$ one obtains
\begin{equation}
\|H(\bs{X}_1)-H(\bs{X}_2)\|_A\leq 
\sup_{|\alpha|\geq 1}\|D_\alpha(H)(\bs{X}_2)\|_A\cdot
\|\bs{Z}^\alpha\|_A
\leq 
\sup_{|\alpha|\geq 1}\|H\|_{\bs{D}^+_A(0,r)}\cdot
\Bigl(\frac{\|\bs{Z}\|_A}{r}\Bigr)^{|\alpha|}=
\|H\|_{\bs{D}^+_A(0,r)}
\frac{\|\bs{Z}\|_A}{r}
\end{equation}
We apply this to $H(\bs{X}):=dF_{\bs{X}}(\bs{Y})$, together with 
\eqref{eq: frsqmpeeee}, to obtain 
$\|(dF_{\bs{X}_1}-dF_{\bs{X}_2})(\bs{Y})\|_A\leq 
r^{-2}\cdot \|F\|_{\bs{D}^+_A(0,r)}\|\bs{X}_1-\bs{X}_2\|_A
\|\bs{Y}\|_{A}$.

Let us prove \eqref{LS : v} Taylor formula gives
$N_{\bs{X}}(\bs{Y}_1)-N_{\bs{X}}(\bs{Y}_2)=
\sum_{|\alpha|\geq 2}D_\alpha(F)(\bs{X})
(\bs{Y}_1^\alpha-\bs{Y}_2^{\alpha})$. 
Now v) follows as in the above cases using
the inequality
$\|\bs{Y}_1^\alpha-\bs{Y}_2^{\alpha}\|_A\leq
\|\bs{Y}_1-\bs{Y}_2\|_A\cdot
\max(\|\bs{Y}_1\|,\|\bs{Y}_2\|)^{|\alpha|-1}$. 
Indeed
\begin{equation}
\bs{Y}_1^\alpha-\bs{Y}_2^{\alpha}\;=\;
\sum_{|\beta|>1}\tbinom{\alpha}{\beta}\bs{Y}_1^{\alpha-\beta}
(\bs{Y}_1-\bs{Y}_2)^\beta\;=\;
\sum_{i=1}^{n}(Y_{1,i}-Y_{2,i})\cdot 
Q_{i,\alpha}(\bs{Y}_1,\bs{Y}_2)\;.
\end{equation}
This is a sum of monomials of degree $|\alpha|$ so $Q_{i,\alpha}$ is a 
sum of monomials of degree $|\alpha|-1$ with integer coefficients, 
and hence $\|Q_{i,\alpha}(\bs{Y}_1,\bs{Y}_2)\|_A\leq 
\max(\|\bs{Y}_1\|,\|\bs{Y}_2\|)^{|\alpha|-1}$. This proves that 
$\|N_{\bs{X}}(\bs{Y}_1)-N_{\bs{X}}(\bs{Y}_2)\|_A\leq 
r^{-2}\|F\|_{\bs{D}^+_A(0,r)}   \|\bs{Y}_1-\bs{Y}_2\|_A 
\max(\|\bs{Y}_1\|_A,\|\bs{Y}_2\|_A)$.
\end{proof}

\if{
\begin{lemma}\label{Lemma : zeta works well}
One can choose the derivation $d$ so that 
$\|d\|_{op,\b_\Omega(D(x,\rho))}>1$. 
With such a choice one finds 
\begin{equation}
\|\zeta\|_{op,\b_\Omega(D(x,\rho))}\;=\;
\max_{k=1,\ldots,n_1}\|d^k\|_{op,\b_\Omega(D(x,\rho))}\;
\end{equation}
(note that $k\neq 0$ in the $\max$). 
Moreover $\rho\mapsto\|\zeta\|_{op,\b_\Omega(D(x,\rho))}$ is a 
\emph{strictly} decreasing function.
\end{lemma}
\begin{proof}
Recall that $\rho<1$ by \eqref{eq: essumption rho<1}.
A possible choice of $d$ is the derivation $d_j=d/dT_j$ of 
Remark \ref{local derivation}. Since $\rho<1$ we have 
$\|d^k\|_{op,\b_\Omega(D(x,\rho))}=
|k!|\cdot (\rho r(\psi_j(x)))^{-k}$. 
Up to replace $\psi_j$ with $\psi_j\circ(\cdot c)$, where $(\cdot c)$ is 
the multiplication by a constant of $K$, we can assume that 
$r(\psi_j(x))\leq 1$.
\end{proof}
}\fi

\subsubsection{Bijectivity of $dG_{\O(Y),\bs{v}}$.}
\label{bijectivity}

\begin{notation}\label{Not : 4.8}
Let $A$ be one of the rings $\O(Y)$, $\O_{X,x}$, $\H(x)$, 
$\b_\Omega(D(x,\rho))$, and let $\bs{v}\in A^n$. 
The action of $dG_{\O(Y),\bs{v}}$ on $A^n$ is represented by a 
$n\times n$ matrix 
\begin{equation}
L_{\bs{v}}
\end{equation}
with coefficients in $A\lr{d}$. 
Remark  that $L_{\bs{v}}$ acts also on each $A\lr{d}$-module 
of the form $B^n$, where $B$ is any $A\lr{d}$-module.
\end{notation}

\begin{proposition}\label{Prop : L_u bijk}
Let $\bs{u}$ be the zero of $G_{\H(x)}$ of \eqref{eq : G_A(u)=0}.
Then
\begin{equation}
dG_{\H(x),\bs{u}}\;:\;B^n\xrightarrow{\;\;\sim\;\;} B^n
\end{equation}
is bijective, for 
$B=\H(x)$, $\b_\Omega(D(x,\rho))$, $\O_\Omega(D(x,\rho))$.
\end{proposition}
\begin{proof}
\emph{Injectivity.} 
Consider diagram \eqref{eq : diagram Dw-Ro} 
replacing $A$ by the ring $B$. The map 
$d(\beta_{n_1}\times\beta_{n_2})_{\bs{u}}=\widetilde{\beta}_{n_1}\times\widetilde{\beta}_{n_2}$ 
is injective, and $\delta$ is bijective, so by 
\eqref{eq : L_v(G_A) with mult} we have to prove 
that the restriction of 
$d\mathfrak{m}_{(P^{\geq \rho},P^{<\rho})}$ to the image
$B\lr{d}^{\leq n_1-1}\times B\lr{d}^{\leq n_2-1}$ 
of 
$\widetilde{\beta}_{n_1}\times\widetilde{\beta}_{n_2}$ is injective. 
Let $(Q_1,Q_2)\in B\lr{d}^{\leq n_1-1}\times
B\lr{d}^{\leq n_2-1}$ 
be such that 
$d\mathfrak{m}_{(P^{\geq \rho},P^{<\rho})}(Q_1,Q_2)=0$. 
This means that 
\begin{equation}\label{eq : PQ_2=Q_1P}
P^{\geq \rho}Q_2\;=\;-Q_1P^{<\rho}\;.
\end{equation}
Assume, by contradiction, that $(Q_1,Q_2)\ne (0,0)$. 
Then $Q_1$ and $Q_2$ are both non-zero and we may assume that 
they are monic.
With the notations of Section \ref{Filtrations of cyclic modules, and 
factorization of operators.} one has two exact sequences 
$0\to\M_{P^{\geq \rho}}\to\M_{P^{\geq \rho}Q_2}\to\M_{Q_2}
\to 0$ and 
$0\to\M_{Q_1}\to\M_{Q_1P^{<\rho}}\to\M_{P^{<\rho}}
\to 0$. Now $\Hdr^0(\M_{P^{<\rho}},\O_\Omega(D(x,\rho)))=0$, so 
$\Hdr^0(\M_{Q_1},\O_\Omega((x,\rho)))\cong
\Hdr^0(\M_{Q_1P^{<\rho}},\O_\Omega(D(x,\rho)))$. 
Moreover $\M_{P^{\geq \rho}Q_2}\cong\M_{Q_1P^{<\rho}}$ by 
\eqref{eq : PQ_2=Q_1P}. Applying the functor 
$\Hdr^0(-,\O(D(x,\rho)))$ one finds
\begin{eqnarray}
\Hdr^0(\M_{P^{\geq\rho}},\O_\Omega(D(x,\rho)))\;\subseteq\;
\Hdr^0(\M_{P^{\geq \rho}Q_2},\O_\Omega(D(x,\rho)))&\;\cong\;&
\Hdr^0(\M_{Q_1P^{<\rho}},\O_\Omega(D(x,\rho)))\nonumber\\
&\;\cong\;&
\Hdr^0(\M_{Q_1},\O_\Omega(D(x,\rho)))\;.\qquad\qquad
\end{eqnarray}
This is a contradiction because $\M_{P^{\geq \rho}}$ is trivialized by 
$\O_{\Omega}(D(x,\rho))$ and hence
\begin{equation}
\mathrm{dim}_\Omega\;\Hdr^0(\M_{P^{\geq\rho}},
\O_\Omega(D(x,\rho)))\;=\;\mathrm{rank}\;
\M_{P^{\geq \rho}}\;>\;
\mathrm{rank}\;\M_{Q_1}\;\geq\;\mathrm{dim}_\Omega\;
\Hdr^0(\M_{Q_1},\O_\Omega(D(x,\rho)))\;.
\end{equation}
\emph{Surjectivity} will follow from Lemma 
\ref{Lemma : matrix QL-1<e} below.
\end{proof}
The following lemma is a generalization of Lemma \ref{QL-1 small}.
If $M=(m_{i,j})$ is a matrix with coefficients in 
$\b_{\Omega}(D(x,\rho))$, we 
denote by $\|M\|_{op,\b_{\Omega}(D(x,\rho))}$ 
the maximum of the norms 
$\|m_{i,j}\|_{op,\b_{\Omega}(D(x,\rho))}$ of the 
coefficients of $M$.
\begin{lemma}\label{Lemma : matrix QL-1<e}
Let $L$ be a square $n\times n$ matrix with coefficients in 
$\H(x)\lr{d}$. 
Assume that $L$ is injective as an endomorphism of 
$\O_\Omega(D(x,\rho))^n$. Then, for all $\varepsilon>0$ 
there exists a $n\times n$ matrix $Q_\varepsilon$ with coefficients in 
$\O_{X,x}\lr{d}$ such that 
\begin{enumerate}
\item \label{Lemma : matrix QL-1<e i}
$\|Q_\varepsilon L-1
\|_{op,\b_{\Omega}(D(x,\rho))}<\varepsilon$.
\item \label{Lemma : matrix QL-1<e ii}
$\Q_\varepsilon$ and $L$ are bijective as endomorphisms of 
$\H(x)$, of $\b_\Omega(D(x,\rho))$, and of $\O_\Omega(D(x,\rho))$.
\end{enumerate}
The same matrix $Q_\varepsilon$ satisfies 
the same properties with respect to every $\rho'$ less 
than or equal to $\rho$ and close enough to $\rho$.
\end{lemma}
\begin{proof}
Since $\H(x)$ is a field, the ring $\H(x)\lr{d}$ is 
(left and right) Euclidean. 
In particular it is a (left and right) principal ideal domain. 
So by the theory of elementary divisors, there exist invertible 
matrices $U,V\in GL_n(\H(x)\lr{d})$ 
such that $\widetilde{L}=ULV$ is diagonal 
$\widetilde{L}=\mathrm{diag}(\widetilde{L}_1,\ldots,
\widetilde{L}_n)$. 
By assumption $L$ is injective on $\O_\Omega(D(x,\rho))^n$, 
then so is $\widetilde{L}$, 
and hence each $\widetilde{L}_i$ is injective on 
$\O_\Omega(D(x,\rho))$. By Lemma \ref{QL-1 small} 
$\widetilde{L}_i$ is bijective on $B$, with  
$B\in\{\H(x),\; \b_\Omega(D(x,\rho)),\;\O_\Omega(D(x,\rho))\}$. Therefore 
$\widetilde{L}$ is bijective on $B^n$, and so is $L$. 
Now by Lemma \ref{QL-1 small}, for all $\varepsilon'>0$, there exists 
$\widetilde{Q}_i\in\H(x)\lr{d}$ such that the matrix 
$\widetilde{Q}:=\mathrm{diag}(\widetilde{Q}_1,\ldots,
\widetilde{Q}_n)$ verifies 
$\|\widetilde{Q}\widetilde{L}-1\|_{op,\b_{\Omega}(
D(x,\rho))}<\varepsilon'$. The matrix $Q_1:=V\widetilde{Q}U$ 
verifies $Q_1L-1=V(\widetilde{Q}\widetilde{L}-1)V^{-1}$. So if 
$\varepsilon'<\varepsilon/(\|V\|_{op,\b_\Omega(D(x,\rho))}\cdot
\|V^{-1}\|_{op,\b_\Omega(D(x,\rho))})$, then $Q_1$ verifies item \eqref{Lemma : matrix QL-1<e i}. 

As in the proof of Lemma \ref{QL-1 small}, 
all the norms are equivalent on $\H(x)\lr{d}^{\leq d}$ and 
$\O_{X,x}\lr{d}^{\leq d}\subset
\H(x)\lr{d}^{\leq d}$ is dense for the norm $\|.\|_{op,
\b_\Omega(D(x,\rho))}$. Therefore, we may find a $n\times n$ matrix
$Q_\varepsilon$  with coefficients in 
$\O_{X,x}\lr{d}$ such that 
$\|Q_{\varepsilon}-Q_1\|_{op,\b_\Omega(D(x,\rho))}$ 
is strictly smaller than 
$\|L\|_{op,\b_\Omega(D(x,\rho))}^{-1}\|Q_1L-1\|_{op,\b_\Omega(D(x,\rho))}$.  
This implies that 
\begin{equation}
\|Q_\varepsilon L-1\|_{op,\b_\Omega(D(x,\rho))}\;=\;
\|Q_1L-1-(Q_1-Q_\varepsilon)L\|_{op,\b_\Omega(D(x,\rho))}\;=\;
\|Q_1L-1\|_{op,\b_\Omega(D(x,\rho))}\;<\;\varepsilon\;.
\end{equation}
Hence $Q_\varepsilon$ verifies item \eqref{Lemma : matrix QL-1<e i}. This implies that $Q_\varepsilon L$ is invertible 
as an endomorphisms of $\b_\Omega(D(x,\rho))^n$, so 
$Q_\varepsilon$ is surjective on it. 

Let now $A,B\in GL_n(\O_{X,x}\lr{d})$ such that 
$AQ_\varepsilon B= 
\mathrm{diag}(Q_{\varepsilon,1}, \ldots,Q_{\varepsilon,n})$. Each 
$Q_{\varepsilon,i}$ is surjective on $\b_\Omega(D(x,\rho))$, so by 
Lemma \ref{Lemma : QL and U invertible} it is injective on 
$\O_\Omega(D(x,\rho))$. This proves that
$Q_\varepsilon$ is injective on $\O_\Omega(D(x,\rho))^n$. 
So we can reproduce the first part of 
this proof replacing $L$ by $Q_\varepsilon$, to prove that 
$Q_\varepsilon$ is bijective  as endomorphisms of 
$\H(x)$, of $\b_\Omega(D(x,\rho))$, and of $\O_\Omega(D(x,\rho))$.

To prove the last statement we notice, 
in analogy with the proof of Lemma \ref{QL-1 small}, 
that the kernel of $L$ on 
$\O(D(x,\rho'))^n$ is a finite dimensional vector space over 
$\Omega$ (this is true if $n=1$, and the case $n>1$ reduces as above to the case
$n=1$ by the theory of elementary divisors). 
The kernel has a filtration by the radii 
of convergence of the entries of its vectors, and we conclude as in 
Lemma \ref{QL-1 small}. 
\end{proof}

Recall that $L_{\bs{v}}$ is the matrix associated with 
 the differential of $G$ at 
$\bs{v}$ (cf. Notation \ref{Not : 4.8}).

\begin{corollary}\label{Lemma : bij at v}
\label{Corollary : cor norm of dG=Q}
Let $\bs{u}\in\H(x)^n$ be the zero of $G_{\H(x)}$ of \eqref{eq : G_A(u)=0}. 
There exists a radius $w>0$ such that
\begin{enumerate}
\item\label{Lemma : bij at v-i} If $\bs{v}\in \b_\Omega(D(x,\rho))^n$ satisfies 
$\|\bs{v}-\bs{u}\|_{\b_\Omega(D(x,\rho))}<w$, then 
$L_{\bs{v}}$ is invertible as endomorphism of 
$\b_\Omega(D(x,\rho))^n$ and of $\O_\Omega(D(x,\rho))^n$.
\item\label{Lemma : bij at v-ii} If moreover $\bs{v}\in\H(x)^n$, then 
$L_{\bs{v}}$ is also invertible as endomorphism of $\H(x)^n$.
\item\label{Lemma : bij at v-iii} If $\bs{v}\in\O_{X,x}^n$, and if it satisfies $\|\bs{v}-\bs{u}
\|_{\b_\Omega(D(x,\rho))}<w$,  
then there exists a basis of $\psi$-\emph{admissible}, 
affinoid neighborhoods $Y$ of 
$x$ in $X$ (cf. Definition \ref{def : elementary neigh}), such that 
$L_{\bs{v}}$ is invertible as endomorphism of each $\O(Y)^n$.
\end{enumerate}

In the situation of \eqref{Lemma : bij at v-i} and 
\eqref{Lemma : bij at v-ii} there exists a square matrix 
$Q_\varepsilon$ with coefficients in $\O_{X,x}\lr{d}$ such that, 
for all $\bs{v}$ verifying 
$\|\bs{v}-\bs{u}\|_{\b_\Omega(D(x,\rho))}<w$, one has 
$\|Q_\varepsilon L_{\bs{v}}-1
\|_{op,\b_{\Omega}(D(x,\rho))}<\varepsilon$,
and also
\begin{equation}\label{eq : norm of dG is that of Q_e}
\|dG_{\b_{\Omega}(D(x,\rho)),\bs{v}}^{-1}
\|_{op,\b_\Omega(D(x,\rho))}
\;=\;
\|Q_\varepsilon\|_{op,\b_\Omega(D(x,\rho))}\;.
\end{equation}
In particular this norm is independent of $\bs{v}$. 

If moreover $\bs{v}\in\O_{X,x}^n$, 
then for all $\varepsilon>0$ there exists a basis of  $\psi$-\emph{admissible}, 
neighborhoods of $x$ in $X$ 
(cf. Definition \ref{def : elementary neigh}) satisfying
\begin{equation}\label{eq : norm of inverse of dg is that of Q_e}
\|dG_{\O(Y),\bs{v}}^{-1}\|_{op,\O(Y)}\;\leq \;
\|Q_\varepsilon\|_{op,\b_\Omega(D(x,\rho))}+\varepsilon\;.
\end{equation}
\end{corollary}
\begin{proof}
Let us prove item \eqref{Lemma : bij at v-i}. By Proposition \ref{Prop : L_u bijk}, 
$L_{\bs{u}}$ is injective on $\O(D(x,\rho))^n$ and the assumptions of 
Lemma \ref{Lemma : matrix QL-1<e} are satisfied. 
Let $\varepsilon>0$, and let $Q_{\varepsilon}$ be the matrix of 
Lemma \ref{Lemma : matrix QL-1<e} such that 
$\|Q_\varepsilon L_{\bs{u}}-1\|_{op,
\b_\Omega(D(x,\rho))^n}<\varepsilon/2$. 
Write 
$Q_\varepsilon L_{\bs{v}}-1=
Q_\varepsilon L_{\bs{u}}-1+Q_\varepsilon(L_{\bs{v}}-L_{\bs{u}})$. 
From the bound \eqref{eq : norm dG_v1-dG_v2} it follows that 
there exists $w$ such that if 
$\|\bs{v}-\bs{u}\|_{\b_\Omega(D(x,\rho))}<w$, then 
$\|Q_\varepsilon L_{\bs{v}}-1\|_{op,\b_{\Omega}(D(x,\rho))}
\;<\;\varepsilon/2$. By \eqref{eq : explicit norm operator}, the map 
$\rho\mapsto 
\|Q_\varepsilon L_{\bs{v}}-1\|_{op,\b_{\Omega}(D(x,\rho))}$ is 
continuous, and for $\rho'$ close enough to $\rho$ we have $\|Q_\varepsilon L_{\bs{v}}-1\|_{op,\b_{\Omega}(D(x,\rho'))}<1$.
Since $\b_\Omega(D(x,\rho'))^n$ is complete, 
$Q_\varepsilon L_{\bs{v}}=1+P_{\bs{v}}$ is invertible 
as an endomorphism of $\b_\Omega(D(x,\rho'))^n$ with inverse 
$U:=\sum_{i\geq 0}(-1)^i P_{\bs{v}}^i$. 
Since $Q_\varepsilon$ is an isomorphism too, so is 
$L_{\bs{v}}$. Since this holds for all $\rho'\leq \rho$ close enough to 
$\rho$, then arguing as in the proof of Lemma \ref{QL-1 small}, we deduce that $L_{\bs{v}}$ is an isomorphism on 
$\O(D(x,\rho))=\cap_{\rho'\leq \rho}\b_\Omega(D(x,\rho'))$. 

Item \eqref{Lemma : bij at v-ii} follows from Lemma \ref{Lemma : matrix QL-1<e}. 

Let us prove Item \eqref{Lemma : bij at v-iii}. 
Assume now that $\bs{v}\in\O_{X,x}^n$. 
By Proposition \ref{proposition : comparison o norms} there is a basis of 
$\psi$-\emph{admissible}, neighborhoods $Y$ of $x$ in $X$ 
(cf. Definition \ref{def : elementary neigh}) such that  
\begin{itemize}
\item[$\bullet$] $\bs{v}\in\O(Y)^n$, 
\item[$\bullet$] the coefficients of $Q_\varepsilon$ lie in $\O(Y)$, 
\item[$\bullet$] $\|Q_\varepsilon L_{\bs{v}}-1\|_{op,\O(Y)}<\varepsilon$.
\end{itemize}
Then $Q_\varepsilon L_{\bs{v}}=1-P_{\bs{v}}$ 
is invertible on $\O(Y)^n$ with inverse 
$\sum_{i\geq 0}P_{\bs{v}}^i$. This implies that $Q_\varepsilon$ 
is surjective on $\O(Y)^n$. 
It is also injective, because it is so as an operator on $\H(x)^n$. Finally,  
$Q_\varepsilon$ is bijective on $\O(Y)^n$, and so is $L_{\bs{v}}$.

Equality \eqref{eq : norm of dG is that of Q_e} follows from the fact that
$L_{\bs{v}}Q_\varepsilon=1-P_{\bs{v}}$ and  
$\|P_{\bs{v}}\|_{op,\b_\Omega(D(x,\rho))}<1$. Indeed, $dG_{\b_{\Omega}(D(x,\rho)),\bs{v}}^{-1}=L_{\bs{v}}^{-1}=Q_\varepsilon(1-P_{\bs{v}})^{-1}$ 
and hence 
\begin{equation}
\|dG_{\b_{\Omega}(D(x,\rho)),\bs{v}}^{-1}
\|_{op,\b_\Omega(D(x,\rho))}=
\|Q_\varepsilon(1-P_{\bs{v}})^{-1}\|_{op,\b_\Omega(D(x,\rho))}
=\|Q_\varepsilon+\sum_{i\geq 1}Q_\varepsilon P_{\bs{v}}^i\|_{op,\b_\Omega(D(x,\rho))}
\end{equation}
 and for all $i\geq 1$ we have $\|Q_\varepsilon P_{\bs{v}}^i\|_{op,\b_\Omega(D(x,\rho))}\leq\|Q_\varepsilon\|_{op,\b_\Omega(D(x,\rho))}\|P_{\bs{v}}\|^i_{op,\b_\Omega(D(x,\rho))}<\|Q_\varepsilon\|_{op,\b_\Omega(D(x,\rho))}$. Equality \eqref{eq : norm of dG is that of Q_e} follows.

Inequality \eqref{eq : norm of inverse of dg is that of Q_e} follows 
from \eqref{eq : estimation of norms}. 
\end{proof}

\if{
\begin{corollary}\label{Corollary : cor norm of dG=Q}
Let $w>0$ be as in Corollary \ref{Lemma : bij at v}.
For all $\bs{v}\in\b_\Omega(D(x,\rho))^n$ such that 
$\|\bs{v}-\bs{u}\|_{\b_\Omega(D(x,\rho))}<w$ one has
\begin{equation}\label{eq : norm of dG is that of Q_e}
\|dG_{\b_{\Omega}(D(x,\rho)),\bs{v}}^{-1}
\|_{op,\b_\Omega(D(x,\rho))}
\;=\;
\|Q_\varepsilon\|_{op,\b_\Omega(D(x,\rho))}\;.
\end{equation}
In particular this norm is independent on $\bs{v}$. Moreover, if 
$\bs{v}\in\O_{X,x}^n$, 
then for all $\varepsilon>0$ there exists a basis formed by 
$\psi$-\emph{admissible}, affinoid 
neighborhoods $Y$ of $x$ in $X$ 
(cf. Def. \ref{def : elementary neigh}), satisfying
\begin{equation}\label{eq : norm of inverse of dg is that of Q_e}
\|dG_{\O(Y),\bs{v}}^{-1}\|_{op,\O(Y)}\;\leq \;
\|Q_\varepsilon\|_{op,\b_\Omega(D(x,\rho))}+\varepsilon\;.
\end{equation}
\end{corollary}
\begin{proof}
Equality \eqref{eq : norm of dG is that of Q_e} follows from 
$L_{\bs{v}}Q_\varepsilon=1-P_{\bs{v}}$ with 
$\|P_{\bs{v}}\|_{op,\b_\Omega(D(x,\rho))}<1$. 
Inequality \eqref{eq : norm of inverse of dg is that of Q_e} follows 
from \eqref{eq : estimation of norms}.
\end{proof}
}\fi
\subsubsection{Contractiveness of $\phi_{\bs{v},A}$.}
\label{section : contractive..}
Let $w>0$ be as in Corollary \ref{Lemma : bij at v}. 
We have proved that for all $\bs{v}\in\O_{X,x}^n$ satisfying
$\|\bs{v}-\bs{u}\|_{\H(x)}<w$, the map 
$dG_{\O(Y),\bs{v}}:\O(Y)^n\to\O(Y)^n$ is invertible for a basis of  affinoid neighborhoods $Y$ of $x$ in $X$. 
In particular, we may choose $Y$ small enough in order to have \eqref{eq : 
norm of inverse of dg is that of Q_e}. 
Then, for $A=\O(Y)$, or $A=\H(x)$, we are allowed to define, for all $\bs{\xi}\in A^n$, the map 
\begin{equation}
\phi_{\bs{v},A}(\bs{\xi})\;:=\; 
-dG^{-1}_{A,\bs{v}}\Bigl(G_{A}(\bs{v})+
N_{\zeta(\bs{v})}(\zeta(\bs{\xi}))\Bigr)\;.
\end{equation}
\begin{proposition}\label{Lemma contractive}
There exist an affinoid neighborhood $Y$ of $x$ in $X$, a 
$\bs{v}\in\O(Y)^n$, and a real number 
$q>\|\bs{v}-\bs{u}\|_{\H(x)}$ such that 
\begin{equation}
\phi_{\bs{v},A}\;:\; A^n\longrightarrow A^n
\end{equation}
is a contractive map on the disk $\bs{D}_A^+(0,q)$, for both 
$A=\H(x)$ and $A=\O(Y)$.
\if{
Up to chose $\bs{v}\in\O_{X,x}^n$ in order to restrict 
$\|\bs{v}-\bs{u}\|_{\H(x)}$, 
then there exist $q> \|\bs{v}-\bs{u}\|_{\H(x)}$ 
such that the map 
$\phi_{\bs{v},\H(x)}:\H(x)^n\to \H(x)^n$ is contractive on the disk 
$\bs{D}^+_{\H(x)}(0,q)$. Moreover with the same value of $q$ there 
exists an affinoid neighborhood $Y$ of $x$ in $X$ such that 
$\phi_{\bs{v},\O(Y)}:\O(Y)^n\to \O(Y)^n$ is contractive on the disk 
$\bs{D}^+_{\O(Y)}(0,q)$.
}\fi
\end{proposition}
\begin{proof}
Let $r>0$ be a real number such that   
$\|\bs{u}\|_{\H(x)}< r\|\zeta\|_{op,\H(x)}^{-1}$, in order that 
$\zeta(\bs{u})\in\bs{D}^-_{\H(x)}(0,r)$. Notice that the inclusion 
$\H(x)\subset\b_\Omega(D(x,\rho))$ is an isometry (cf. \eqref{eq : H(x) in B(D(x))}), 
therefore we also have
\begin{equation}\label{Eq: ass 4.59 usef}
\|\bs{u}\|_{\b_\Omega(D(x,\rho))}\;=\;
\|\bs{u}\|_{\H(x)}\;<\; 
r\|\zeta\|_{op,\H(x)}^{-1}\;\leq\; 
r\|\zeta\|_{op,\b_\Omega(D(x,\rho))}^{-1}\;.
\end{equation}
Let 
\begin{equation}
\kappa\;:=\; 
r\cdot   
\|\zeta\|_{op,\b_\Omega(D(x,\rho))}^{-1} \cdot 
\|F\|_{\bs{D}^+_{\b_\Omega(D(x,\rho))}(0,r)}^{-1} \cdot
\|Q_\varepsilon\|_{op,\b_\Omega(D(x,\rho))}^{-1}\;.
\end{equation}
We fix an arbitrary real number $0<C<1$, and we choose $q$ such that 
\begin{equation}\label{eq : assumption technn}
0\;<\;q\;<\;C\cdot r\cdot 
\|\zeta\|_{op,\b_\Omega(D(x,\rho))}^{-1}\cdot \min(\kappa,1)\;.
\end{equation}
Then, we choose $\bs{v}\in\O_{X,x}^n$ such that 
\begin{equation}\label{eq : assumption technn-2}
\|\bs{v}-\bs{u}\|_{\H(x)}\;<\;\min\Bigl(q,\;q\cdot \kappa,\;
r\|\zeta\|_{op,\H(x)}^{-1}\Bigr)\;.
\end{equation}
With these choices the following properties hold
\begin{enumerate}
\item\label{Lemma contractive -i} $\phi_{\bs{v},\H(x)}(\bs{D}^+_{\H(x)}(0,q))\subseteq 
\bs{D}^+_{\H(x)}(0,q)$; 
\item\label{Lemma contractive -ii} There exists $Y$ such that 
$\phi_{\bs{v},\O(Y)}(\bs{D}^+_{\O(Y)}(0,q))\subseteq 
\bs{D}^+_{\O(Y)}(0,q)$;
\item\label{Lemma contractive -iii} There exists $0<C''<1$ such that, for all 
$\bs{\xi}_1,\bs{\xi}_2\in \bs{D}^+_{\H(x)}(0,q)$ 
one has 
\begin{equation}
\|\phi_{\bs{v},\H(x)}(\bs{\xi}_1)-
\phi_{\bs{v},\H(x)}(\bs{\xi}_2)\|_{\H(x)}\;\leq\; C''\cdot
\|\bs{\xi}_1-\bs{\xi}_2\|_{\H(x)}\;;
\end{equation}
\item\label{Lemma contractive -iv} There exists $Y$ satisfying ii), and $0<C'<1$, such that 
for all $\bs{\xi}_1,\bs{\xi}_2\in 
\bs{D}^+_{\O(Y)}(0,q)$ 
one has 
\begin{equation}\label{eq : contc}
\|\phi_{\bs{v},\O(Y)}(\bs{\xi}_1)-
\phi_{\bs{v},\O(Y)}(\bs{\xi}_2)\|_{\O(Y)}\;\leq\; C'\cdot
\|\bs{\xi}_1-\bs{\xi}_2\|_{\O(Y)}\;.
\end{equation}
\end{enumerate}

\if{\comm{Controllare}

Therefore, we ship the proof of \eqref{Lemma contractive -i} and we  
only provide the proof of item \eqref{Lemma contractive -ii}.
By \eqref{eq : norm of inverse of dg is that of Q_e} 
we can assume, restricting $Y$, that 
$\|\bs{u}\|_{\O(Y)},
\|\bs{v}-\bs{u}\|_{\O(Y)}< r\|\zeta\|_{op,\O(Y)}^{-1}$, and 
that $\|\zeta\|_{op,\O(Y)}$, 
$\|G_{\O(Y)}(\bs{v})\|_{\O(Y)}$, 
and $\|dG_{\O(Y),\bs{v}}^{-1}\|_{op,\O(Y)}$ are close enough 
to $\|\zeta\|_{op,\b_\Omega(D(x,\rho))}$, 
$\|G_{\O(Y)}(\bs{v})\|_{\b_\Omega(D(x,\rho))}$, and 
$\|Q_\varepsilon\|_{op,\b_\Omega(D(x,\rho))}$ respectively. 

\comm{sbagliato}
}\fi

Let us begin with item \eqref{Lemma contractive -ii}.
First of all, by \eqref{eq : estimation of norms}, 
for all $\varepsilon>0$ we have a basis neighborhoods $Y$ of 
$x$ in $X$ such that
\begin{equation}
\|\zeta\|_{op,\O(Y)}\;\leq\;
\|\zeta\|_{op,\b_\Omega(D(x,\rho))}+\varepsilon\;.
\end{equation}
We may chose $\varepsilon>0$ small enough in order 
to have $q<r\|\zeta\|_{op,\O(Y)}^{-1}$ 
(cf. \eqref{eq : assumption technn}). This guarantee that if 
$\|\bs{\xi}\|_{\O(Y)}\leq q$, then 
$\zeta(\bs{\xi})\in\bs{D}^-_{\O(Y)}(0,r)$. 

By Corollary \ref{Corollary : cor norm of dG=Q} and Proposition 
\ref{proposition : comparison o norms}, for all 
$\bs{\xi}\in\bs{D}^+_{\O(Y)}(0,q)$ and all $\varepsilon_1,\varepsilon_2,\varepsilon_3>0$ we may further shrink $Y$ in order to have
\begin{eqnarray}
\|\phi_{\bs{v},\O(Y)}(\bs{\xi})\|_{\O(Y)}&\;=\;&
\Bigl\|-dG_{\O(Y),\bs{v}}^{-1}\Bigl(G_{\O(Y)}(\bs{v})+
N_{\zeta(\bs{v})}(\zeta(\bs{\xi}))\Bigr)
\Bigr\|_{\O(Y)}\\
&\leq& \|-dG_{\O(Y),\bs{v}}^{-1}\|_{op,\O(Y)}\cdot
\max\Bigl(\|G_{\O(Y)}(\bs{v})\|_{\O(Y)},
\|N_{\zeta(\bs{v})}(\zeta(\bs{\xi}))\|_{\O(Y)}\Bigr)\label{eq : temp -yg}\\
&\leq& 
(\|Q_\varepsilon\|_{op,\b_\Omega(D(x,\rho))}+\varepsilon_1)
\cdot \max\Bigl(\|G_{\O(Y)}(\bs{v})\|_{\b_\Omega(D(x,\rho))}+\varepsilon_2,
\|N_{\zeta(\bs{v})}(\zeta(\bs{\xi}))\|_{\b_\Omega(D(x,
\rho))}+\varepsilon_3\Bigr)\;.\nonumber
\end{eqnarray}
The assumption \eqref{Eq: ass 4.59 usef} ensures that 
\eqref{eq : estim G(v)} holds. By 
\eqref{eq : assumption technn-2}, we have $\|\bs{u}-\bs{v}\|_{\b_\Omega(D(x,\rho))}<q\kappa$, which implies 
\begin{equation}
\|G_{\O(Y)}(\bs{v})\|_{\b_\Omega(D(x,\rho))} \;<\; 
q\|Q_\varepsilon\|^{-1}_{op,\b_\Omega(D(x,\rho))}\;.
\end{equation}
Since this inequality is strict it is possible to chose 
$\varepsilon_1,\varepsilon_2$ in order that 
$(\|Q_\varepsilon\|_{op,\b_\Omega(D(x,\rho))}+\varepsilon_1)
\cdot 
(\|G(\bs{v})\|_{\b_\Omega(D(x,\rho))}
+\varepsilon_2)\leq q$. 

Analogously, if $\|\bs{\xi}\|_{\O(Y)}\leq q$, the inequality \eqref{eq : 
assumption technn} implies that \eqref{eq : estim N_v(xi)} holds, 
and moreover that $\|\bs{\xi}\|_{\O(Y)}^2\leq q^2<q\cdot r\cdot 
\|\zeta\|_{op,\b_\Omega(D(x,\rho))}^{-1}\cdot \kappa$. Therefore 
\begin{equation}
\|N_{\zeta(\bs{v})}(\zeta(\bs{\xi}))\|_{\b_\Omega(D(x,\rho))}\;<\;
q\|Q_\varepsilon\|^{-1}_{op,\b_\Omega(D(x,\rho))}\;.
\end{equation}
As above, we may further restrict $\varepsilon_1,\varepsilon_3$ to have 
$(\|Q_\varepsilon\|_{op,\b_\Omega(D(x,\rho))}+\varepsilon_1)
\cdot 
(\|N_{\zeta(\bs{v})}(\zeta(\bs{\xi}))\|_{\b_\Omega(D(x,\rho))}
+\varepsilon_3)\leq q$. 
We conclude that 
$\|\phi_{\bs{v},\O(Y)}(\bs{\xi})\|_{\O(Y)}\leq q$, and \eqref{Lemma contractive -ii} holds. 

The proof of item \eqref{Lemma contractive -i} follows line by line 
that of item \eqref{Lemma contractive -ii}. It involves the same 
ingredients (i.e. \eqref{eq : estim G(v)}, 
\eqref{eq : estim N_v(xi)} and Corollary 
\ref{Corollary : cor norm of dG=Q}) with the 
simplification that no $\varepsilon_i$ is necessary. 
We let the reader to complete the proof.

\if{
\eqref{eq : estim N_v(xi)} together with 
\eqref{eq : assumption technn} gives\footnote{Here one uses the fact 
that $\|F\|_{\bs{D}^+_{\O(Y)}(0,r)}\leq
\|F\|_{\bs{D}^+_{\b_{\Omega}(D(x,\rho))}(0,r)}$.} 
\begin{equation}\label{eq : result for N_v(xi) cruucial}
\|N_{\zeta(\bs{v})}(\zeta(\bs{\xi}))\|_{\b_\Omega(D(x,
\rho))}+\varepsilon_3\;\leq \;\frac{\|\zeta\|_{op,\O(Y)}^2}{
\|\zeta\|_{op,\b_{\Omega}(D(x,\rho))}^2}\cdot 
q\|Q_\varepsilon\|_{op,\b_\Omega(D(x,\rho))}^{-1} + 
\varepsilon_3\;.
\end{equation}
Now by \eqref{eq : norm of inverse of dg is that of Q_e} for all 
$\varepsilon_4$ one finds $Y$ such that $\|\zeta\|_{op,\O(Y)}\leq 
\|\zeta\|_{op,\b_\Omega(D(x))}+\varepsilon_4$. 
By choosing a convenient derivative $d$, Lemma 
\ref{Lemma : zeta works well} gives  for $\varepsilon_4$ 
small enough $\|\zeta\|_{op,\b_{\Omega}(D(x,\rho))}>
\|\zeta\|_{op,\b_{\Omega}(D(x))}+\varepsilon_4$ 
(indeed we have $\rho<1$, cf. \eqref{eq: essumption rho<1}).
So \eqref{eq : result for N_v(xi) cruucial} becomes 
$\|N_{\zeta(\bs{v})}(\zeta(\bs{\xi}))\|_{\b_\Omega(D(x,
\rho))}+\varepsilon_3\leq 
(q\|Q_\varepsilon\|_{op,\b_\Omega(D(x,\rho))}^{-1} + 
\varepsilon_3)$. So as above one can arrange $\varepsilon_1$ and 
$\varepsilon_3$ in order to obtain 
$\|\phi_{\bs{v},\O(Y)}(\bs{\xi})\|_{\O(Y)}\leq q$. 
}\fi

Let us prove item \eqref{Lemma contractive -iv}. 
Recall that our choice of $Y$ guarantee that if 
$\|\bs{\xi}_i\|_{\O(Y)}\leq q$, we have 
$\zeta(\bs{\xi}_i)\in\bs{D}^-_{\O(Y)}(0,r)$. Arguing as in \eqref{eq : temp -yg}, by Proposition 
\ref{proposition : comparison o norms} we obtain
\begin{eqnarray}
\|\phi_{\bs{v},\O(Y)}(\bs{\xi}_1)-
\phi_{\bs{v},\O(Y)}(\bs{\xi}_2)\|_{\O(Y)}
&\;=\;&\|-dG_{\O(Y),\bs{v}}^{-1}
(N_{\zeta(\bs{v})}(\zeta(\bs{\xi}_1))-
N_{\zeta(\bs{v})}(\zeta(\bs{\xi}_2)))\|_{\O(Y)}\\
&\leq&(\|Q_\varepsilon\|_{op,
\b_\Omega(D(x,\rho))}+\varepsilon_1)\cdot
\|N_{\zeta(\bs{v})}(\zeta(\bs{\xi}_1))-N_{\zeta(\bs{v})}
(\zeta(\bs{\xi}_2))\|_{\O(Y)}\;.\qquad\nonumber
\end{eqnarray}
Estimation \eqref{eq : norm of N_v(xi_1)-N_v(xi_2)} then furnishes
\begin{eqnarray}
\|N_{\zeta(\bs{v})}(\zeta(\bs{\xi}_1))-N_{\zeta(\bs{v})}
(\zeta(\bs{\xi}_2))\|_{\O(Y)}
&\;\leq\;&r^{-2}\|F\|_{\bs{D}^+_{\O(Y)}(0,r)}\cdot
\|\zeta\|^2_{op,\O(Y)}\cdot
\|\bs{\xi}_1-\bs{\xi}_2\|_{\O(Y)}\cdot
\max(\|\bs{\xi}_1\|_{\O(Y)},\|\bs{\xi}_2\|_{\O(Y)})\nonumber\\
&\leq&r^{-2}
(\|F\|_{\bs{D}^+_{\b_\Omega(D(x,\rho))}(0,r)}+\varepsilon_4)\cdot
(\|\zeta\|_{op,\b_\Omega(D(x,\rho))}+\varepsilon_5)^2\cdot
\|\bs{\xi}_1-\bs{\xi}_2\|_{\O(Y)}\cdot q\;.\nonumber
\end{eqnarray}
Now, by \eqref{eq : assumption technn}, we have 
\begin{eqnarray}
q&\;<\;&C\cdot r\cdot\|\zeta\|_{op,\b_\Omega(D(x,\rho))}^{-1}\cdot\kappa\\
&<&C\cdot r^2\cdot\|\zeta\|_{op,\b_\Omega(D(x,\rho))}^{-2}\cdot
\|F\|_{\bs{D}^+_{\b_\Omega(D(x,\rho))}(0,r)}^{-1}\cdot
\|Q_\varepsilon\|^{-1}_{op,\b_\Omega(D(x,\rho))}\;.
\end{eqnarray}
Therefore, inequality \eqref{eq : contc} holds with 
\begin{equation}
C'=C\cdot\Bigl(\frac{\|Q_\varepsilon\|_{op,
\b_\Omega(D(x,\rho))}+\varepsilon_1}{\|Q_\varepsilon\|_{op,
\b_\Omega(D(x,\rho))}}\Bigr)
\Bigl(\frac{
\|F\|_{\bs{D}^+_{\b_\Omega(D(x,\rho))}(0,r)}+\varepsilon_4}{
\|F\|_{\bs{D}^+_{\b_\Omega(D(x,\rho))}(0,r)}}\Bigr)
\Bigl(\frac{\|\zeta\|_{op,\b_\Omega(D(x,\rho))}+\varepsilon_5}{
\|\zeta\|_{op,\b_\Omega(D(x,\rho))}}\Bigr)^2
\end{equation}
Restricting $\varepsilon_1,\varepsilon_4,\varepsilon_5$ 
the last three factors can be made arbitrarily close to 
$1$. Hence, we may find $Y$ such that $0<C'<1$.

Item \eqref{Lemma contractive -iii} follows 
closely item \eqref{Lemma contractive -iv} and it is easier, 
since it does not requires the 
argument with the $\varepsilon_i$. We left it to the reader.

The Proposition follows.
\end{proof}

This concludes our proof of Theorem \ref{Dwork-Robba}.

\subsection{Lifting to $\O_{X,x}$ solutions of non linear differential 
equations with solutions in $\H(x)$.}

With the identical proof one proves the following result. We here do 
not assume that $K$ is algebraically closed.

Let $Y$ be a fixed affinoid neighborhood of $x\in X$.
Let $\bs{Z}:=(Z_1,\ldots,Z_{n\cdot s})$, let $m\geq n$, and let 
\begin{equation}
F\;:\;\O(Y)^{ns}\to\O(Y)^m\;,\qquad
F(\bs{Z})\;:=\;
\sum_{|\alpha|\geq 0} 
\bs{C}_\alpha \bs{Z}^{\alpha}\;,
\end{equation}
where 
$\alpha=(\alpha_1,\ldots,\alpha_{ns})\in\mathbb{N}^{ns}$, 
$|\alpha|=\sum_i\alpha_i$, 
$\bs{Z}^\alpha:=Z_1^{\alpha_1}\cdots Z_{ns}^{\alpha_{ns}}$, 
$\bs{C}_\alpha\in \O(Y)^m$ is a power 
series satisfying $\|\bs{C}_\alpha\|_{\O(Y)}\cdot r^{|\alpha|}\to 0$ as 
$|\alpha|\to\infty$, i.e. $F$ converges in a disk 
$\mathrm{D}^+_{\O(Y)}(0,r)\subseteq \O(Y)^{ns}$. 
For $\bs{X}=(X_1,\ldots,X_n)$ 
consider the non-linear differential equation
\begin{equation}\label{eq : more general non linear}
F(\bs{X},d(\bs{X}),\ldots,d^{s}(\bs{X}))\;=\;0\;.
\end{equation}
Denote in analogy with the above sections 
$\zeta:\O(Y)^n\to\O(Y)^{ns}$, the map 
$\zeta(\bs{x})=(\bs{x},d(\bs{x}),\ldots,d^s(\bs{x}))$.
\begin{theorem}[\protect{\cite[3.1.6]{Dw-Robba}}]
Let $\bs{x}\in\H(x)^n$ be a solution of 
\eqref{eq : more general non linear}.
If the linearized map
\begin{equation}\label{eq : Thm : inj}
dF_{\zeta(\bs{x})}\circ\zeta\;:\;\H(x)^{n}\to\H(x)^m
\end{equation}
is injective, then $\bs{x}\in\O_{X,x}^n$.
\end{theorem}
\begin{proof}
Assume that $K$ is algebraically closed.
As in the above sections $dF_{\zeta(\bs{x})}\circ\zeta$ 
acts by a matrix $L_{\bs{x}}\in M_{m\times n}(\O(Y)\lr{d})$.
By the theory of elementary divisors there exists square 
matrices $U,V$ with coefficients in $\O_{X,x}\lr{d}$ such that
the unique non zeros terms of $UL_{\bs{x}}V$ are in the diagonal.
The terms in the diagonal are either zero or differential 
polynomials in $\O_{X,x}\lr{d}$ that are injective as linear maps  
$\H(x)^n\to\H(x)^m$. If $m>n$ then we can discard $m-s$ 
components and reduce to the case $m=n$. 
It is now possible to reproduce the proof as in the above sections. The claim follows.

Consider now the case of a general field $K$.
To descend the theorem from $K^{\mathrm{alg}}$ to $K$, 
it is enough to notice that if 
$x_1,\ldots,x_n\in X_{\widehat{K^{\mathrm{alg}}}}$
 are the points with image $x\in X$, then for all $x_i$ one has 
$\O_{X,x}=(\H(x)\cap\O_{X_{\widehat{K^{\mathrm{alg}}}},x_i})
\subset\H(x_i)$. 
\end{proof}

\begin{remark}
To prove the Theorem \ref{Dw-Robba} 
of Dwork and Robba another 
possible strategy can be imagined. 
The condition $q^2=q$ in 
$\mathrm{End}_{\O_{X,x}}(\M_x)$ defines a projector. This condition 
is described by a non-linear differential system of the above type. 
Robba's theorem \ref{corollary : uniqueness at H(x)} 
provides a solution of such a system with 
coefficients in $\H(x)$. Dwork and Robba then 
relate the injectivity of \eqref{eq : Thm : inj} to the radii. 
\end{remark}

\subsection{Augmented Dwork-Robba decomposition}\label{sec:augmenteddecomposition}
We maintain the assumption that $K$ is algebraically closed and the notations of Section \ref{Radii and filtrations by the radii} ($\Fs$, $S$, 
$r=\mathrm{rank}(\Fs)$, etc...). 

In this section, we extend the Dwork-Robba 
decomposition (cf. Theorem \ref{Dw-Robba}) by taking into account 
\emph{over-solvable radii}. Recall that these are the radii of the solutions 
of $\Fs$ which converge in some over-solvable disk, that is a disk 
containing strictly the Dwork generic disk $D(x)$. For this reason, the 
over-solvable radii are not intrinsically associated with the localization 
$\Fs_x$ of $\Fs$, indeed base changes in $\Fs_x$ by matrices in $\O_{X,x}$ 
would truncate the oversolvable radii of the solutions (cf. Section 
\ref{Local nature of spectral radii.}). 
More specifically, we are going to add some further terms in the 
decomposition of $\Fs_x$ provided by Theorem \ref{Dw-Robba},
by adding a filtration of the differential sub-module 
$\Fs_x^{\geq 1}$ (cf. notation \eqref{eq : def M^geq rho,sp}). 
These new terms in the filtration of $\Fs_x$ are \emph{generated by the solutions} of $\Fs$ converging in some over-solvable disk. In other word, this augmented filtration of $\Fs_x$ remembers the fact that the local module 
$\Fs_x$ is restriction of a global differential equation $\Fs$ on $X$.

\begin{remark}
Unfortunately, over-solvable radii do not enjoy all the properties of spectral ones. Namely, they are not compatible with duality and items 
\eqref{i eq : deco of M_x .gdyu}, 
\eqref{ii eq : deco of M_x .gdyu}, 
\eqref{iii eq : deco of M_x .gdyu} of Theorem 
\ref{Dw-Robba} fail for over-solvable radii. 
We provide an explicit counterexample in Section 
\ref{An explicit counterexample.}. 
The properties satisfied by over-solvable radii 
are listed in Section \ref{section : morphisms and duality}.
\end{remark}

\begin{proposition}[Augmented decomposition]
\label{Prop : existence of (F_S,i)_x}
Let $x\in X$. 
Assume that $i$ is an index separating the $S$-radii of $\Fs$ 
at the individual point $x$ 
(cf. Definition \ref{definition : i separates the radii}). 
Then, there exists a unique differential sub-module 
 $(\Fs_{\geq i})_x\subseteq\Fs_x$ of rank $r-i+1$ over $\O_{X,x}$ 
such that (cf. notations \eqref{omega(x,F)}and \eqref{eq : omeaga_S,i and D_S,i})
\begin{equation}\label{eq : bijection local}
\Hdr^0(x,(\Fs_{\geq i})_x) \;=\; \omega_{S,i}(x,\Fs)\;.
\end{equation}
Let $1=i_1< \cdots < i_h$ be the indices 
separating the radii of $\Fs$ at $x$. Then, according to 
\eqref{eq : filtration of omega}, we have a corresponding 
decreasing sequence of differential submodules:
\begin{equation}
0\;\neq\;(\Fs_{\geq i_h})_x\;\subset\;(\Fs_{\geq i_{h-1}})_x\;
\subset\;
\cdots\;\subset\;(\Fs_{\geq i_1})_x\;=\;\Fs_x\;.
\end{equation}
\end{proposition}
\begin{proof}
Let $i$ be an over-solvable index separating the $S$-radii of $\Fs$, that is the Dwork generic disk $D(x)$ 
is strictly contained in the maximal disk $D_{S,i}(x,\Fs)\subset X_\Omega$. 
It follows that the image of  $D_{S,i}(x,\Fs)$ in comes by scalar 
extension from a 
open disk $D_i$ in $X$ containing $x$ (cf. Remark \ref{Remark : D(x,S) cases}).  
Since de Rham cohomology commutes with scalar extension of $K$, 
we have $\omega_{S,i}(x,\Fs):=\Hdr^0(D_{S,i}(x,\Fs),\Fs)=\Hdr^0(D_{i},\Fs)\otimes_K\Omega$. 
The natural inclusion $\Hdr^0(D_{i},\Fs)\subset\Hdr^0(\Fs_x,\O_{X,x})\subset\Fs_x$ generates, by \eqref{eq : j_M},
a trivial differential sub-module $(\Fs_{\geq i})_{x}$ of $\Fs_x$ 
which is by definition the unique sub-module of $\Fs$ 
whose solutions in $\O_{X,x}$ are identified with $\omega_{S,i}(x,\Fs)$ after 
scalar extension to $\Omega$ (cf. \eqref{def : M_A} and Corollary 
\ref{Corollary : M_A unique}). The claim follows.

Let us now consider spectral indexes, that is indexes $i$ such that 
$D_{S,i}(x,\Fs)\subseteq D(x)$. 
In Section \ref{Local nature of spectral radii.} 
(cf.  \eqref{eq : truncation of omega}) we showed 
the precise relation between the $S$-radii 
$\{\R_{S,i}(x,\Fs)\}_{i=1,\ldots,r}$ and the $\emptyset$-radii  $\{\R_{\emptyset,i}(t_x,\Fs_{|D(x)})\}_{i=1,\ldots,r}$ 
of the restriction 
$\Fs_{|D(x)}$ of $\Fs$ to the Dwork generic disk $D(x)$.
Solvable and over-solvable 
$S$-radii of $\Fs$ (cf. Definition \ref{Def: spectral index}) are truncated and collapse to the $\emptyset$-radius $1$ of $\Fs_{|D(x)}$, while  spectral non solvable $S$-radii of $\Fs$ correspond 
bijectively to spectral non solvable $\emptyset$-radii of $\Fs_{|D(x)}$ by a multiplicative factor.
With this correspondence in mind, we may translate the Dwork-Robba statement in term of indexes separating the radii. 
Namely, if $\rho\in]0,1]$ is given, and if 
$i_\rho$ is the smallest spectral index in $i\in \{1,\ldots,r\}$ such that 
$\rho\leq\R_{\emptyset,i}(t_x,\Fs_{|D(x)})$, then $i_\rho$ separates the (global) $S$-radii of $\Fs$ at $x$, as well as the (local) $\emptyset$-radii of 
$\Fs_{|D(x)}$ at $t_x$. By Dwork-Robba's decomposition 
(cf. Corollary~\ref{Lema : descent of M_xrho}), there exists a unique submodule
$(\Fs_x)^{\geq\rho}\;\subseteq\;\Fs_x$ with the required properties,
which we denote by 
\begin{equation}
(\Fs_{\geq i_\rho})_x\;:=\;(\Fs_x)^{\geq\rho}\;.
\end{equation} 
Reciprocally, for every spectral index $i$ separating the $S$-radii of $\Fs$ at $x$ 
there exists $\rho$ such that $i=i_\rho$. The result is then proved 
for spectral indexes separating spectral radii at $x$.

\end{proof}

The following Lemma will be useful to glue the local decompositions in view of the global decomposition Theorem \ref{MAIN Theorem}.
\begin{lemma}\label{Lemma : localization}
Let $x\in X$.
Let $i$ be an index separating the radii at the individual point $x$
(cf. Definition \ref{definition : i separates the radii}). 
Then there exist an open neighborhood $U_x$ of $x$, together with a 
weak triangulation $S_{U_x}$ of $U_x$ and a subdifferential equation 
$(\Fs_{\geq i})_{|U_x}$ of $\Fs_{|U_x}$ over $U_x$, such that 
\begin{enumerate}
\item\label{Lemma : localization - item i} 
The inclusion $(\Fs_{\geq i})_y\subseteq\Fs_y$ is the restriction of  the 
inclusion $(\Fs_{\geq i})_{|U_x}\subseteq\Fs_{|U_x}$;
\item\label{Lemma : localization - item ii} For all $y\in U_x$ the following conditions hold
\begin{enumerate}
\item\label{Lemma : localization - item ii-a} one has $\omega_{S,i}(y,\Fs)\subset\omega_{S,i-1}(y,\Fs)$ 
(i.e. $i$ separates the global $S$-radii of $\Fs$ over $U_x$);
\item\label{Lemma : localization - item ii-b} the restriction 
$\Hdr^0(y,\Fs)\to\Hdr(y,\Fs_{U_x})$ to $U_x$ induces the equalities
\begin{eqnarray}
\omega_{S,i}(y,\Fs)&\;=\;&\omega_{S_{U_x},i}(y,\Fs_{|U_x})\;, 
\label{eq: localization -1}\\
\omega_{S,i-1}(y,\Fs)&\;=\;&\omega_{S_{U_x},i-1}(y,\Fs_{|U_x})\;. 
\label{eq: localization -2}
\end{eqnarray}
\item\label{Lemma : localization - item ii-c} 
We have a strict inclusion of disks 
\begin{equation}\label{eq : this is what one wants for localization}
D_{S,i-1}(y,\Fs)\;\subsetneq\; D(y,S_{U_x})
\end{equation}
\end{enumerate}
\noindent \hspace{-0.9cm}In particular one has:
\item\label{Lemma : localization - item iii} $\R_{S_{U_x},i}(y,\Fs_{|U_x}) > 
\R_{S_{U_x},i-1}(y,\Fs_{|U_x})$ for all 
$y\in U_x$ (i.e. the index $i$ separates the radii of $\Fs$ 
after localization at $U_x$);
\item\label{Lemma : localization - item iv} $\R_{S_{U_x},1}(y,(\Fs_{\geq i})_{|U_x}) = 
\R_{S_{U_x},i}(y,\Fs_{|U_x})$ for all $y\in U_x$;
\item \label{Lemma : localization - item v}
$\omega_{S_{U_x},1}(y,(\Fs_{\geq i})_{|U_x})=
\omega_{S_{U_x},i}(y,\Fs_{|U_x})=
\omega_{S,i}(y,\Fs)$, for all $y\in U_x$.
\end{enumerate}
\end{lemma}
\begin{proof}
First observe that items \eqref{Lemma : localization - item iv} and \eqref{Lemma : localization - item v} 
are immediate consequences of 
\eqref{Lemma : localization - item i}, 
\eqref{Lemma : localization - item ii}, 
\eqref{Lemma : localization - item iii} by 
Proposition \ref{Prop : extended by continuity}. 
Now observe that \eqref{eq: localization -2} is a consequence of 
\eqref{eq: localization -1} by the rule \eqref{Localization of sol-1}. 
Moreover \eqref{Lemma : localization - item iii} 
is a consequence of \eqref{Lemma : localization - item ii} 
by definitions 
\eqref{eq : D_S,i and D(x, R_S,i)} and 
\eqref{eq : omeaga_S,i and D_S,i}. We now define $U_x$ and 
$S_{U_x}$ satisfying \eqref{Lemma : localization - item i} and \eqref{Lemma : localization - item ii}. 

Let $U$ be the largest connected open 
neighborhood of $x$ such that  the coefficients of the matrix 
of the connection of $(\Fs_{\geq i})_x$ lie in $\O(U)$. 
Then item \eqref{Lemma : localization - item i} holds for $U$. 
By continuity of the radii (cf. Theorem \ref{Thm : finiteness}) 
up to restricting $U$ one has 
$\R_{S,i}(y,\Fs)>\R_{S,i-1}(y,\Fs)$ for all $y\in U$,
hence \eqref{Lemma : localization - item ii-a} holds at each point $y\in U$. 
In order to find $U_x\subset U$ and a weak triangulation of $U_x$  
satisfying \eqref{Lemma : localization - item ii-b}, 
we now proceed as follows.
By Proposition \ref{Prop: Localization}, 
it is enough to define $U_x\subseteq U$ and $S_{U_x}$ satisfying  
\eqref{eq : this is what one wants for localization} for every $y\in U_x$.

First assume that the index $i$ 
is over-solvable at $x$, $D_{S,i}(x,\Fs)$ is a
virtual open disk in $X$ containing $x$. 
Then set $U_x:=D_{S,i}(x,\Fs)$ and 
$S_{U_x}=\emptyset$. 
Since $D_{S,i}(y,\Fs)=D_{S,i}(x,\Fs)=D(y,S_{U_x})$ for all 
$y\in D_{S,i}(x,\Fs)$, this proves 
\eqref{eq : this is what one wants for localization}. 
Indeed by contrapositive the function $\R_{S,i-1}(-,\Fs)$ is constant on 
$D_{S,i-1}(y,\Fs)$ (cf. \eqref{eq : D^c}), so the equality 
$D_{S,i-1}(y,\Fs)=D_{S,i}(x,\Fs)$ implies 
$D_{S,i-1}(x,\Fs)=D_{S,i-1}(y,\Fs)=D_{S,i}(x,\Fs)$ which is absurd 
because the radii of $\Fs$ are separated at $x$.

Assume now that the index $i$ is spectral at $x$. This may only happen 
if $x$ is of 
type $2$, $3$, or $4$. By continuity and finiteness of 
the radii (cf. Theorem \ref{Thm : finiteness}) we can choose 
$U_x\subseteq U$ such that 
\begin{enumerate}
\item[1)] $U_x$ is a star-shaped neighborhood of $x$ endowed 
with its canonical triangulation (cf. Definition~\ref{def:starshapedopen});
\item[2)] One has $\R_{S,i-1}(y,\Fs)<\R_{S,i}(y,\Fs)$, 
for all $y\in U_x$ (this is automatic since $U_x\subseteq U$);
\item[3)] The radius 
$\R_{S,i-1}(-,\Fs)$ remains spectral and \emph{non solvable} 
on the \emph{pointed skeleton} of $U_x$ (cf. Definition~\ref{def:starshapedopen});
\item[4)] Either $\Gamma_S(\Fs)\cap U_x$ is empty, or  
$x\in\Gamma_S(\Fs)$ and $\Gamma_S(\Fs)\cap U_x$ is 
included in the pointed skeleton of $U_x$.
\end{enumerate}
We claim that these properties imply 
\eqref{eq : this is what one wants for localization}. 
By 3) this is clear over the pointed skeleton of $U_x$ by the rule 
\eqref{Localization of sol} (indeed the spectral steps of the filtration of 
the solution are stable by localization, cf. also \eqref{eq : truncation of disks} and 
\eqref{eq : truncation of omega}). 
Now let $y\in U_x$ be outside the pointed skeleton of $U_x$. 
By contrapositive assume that 
$D(y,S_{U_x})\subseteq D_{S,i-1}(y,\Fs)$, 
then $i-1$ is solvable at $y$ with respect to $S_{U_x}$. 
By \eqref{eq : D^c} the radius $\R_{S_{U_x},i-1}(-,\Fs_{|U_x})$ 
is constant on $D(y,S_{U_x})$. 
Hence, by continuity, $\R_{S_{U_x},i-1}(-,\Fs_{|U_x})$
remains solvable  
at the topological boundary $z$ of $D(y,S_{U_x})$. 
Since $z$ lies in the pointed skeleton of $U_x$ this is a contradiction.
\end{proof}

\begin{remark}[Gluing property]\label{Rk: gling propty}
Let us maintain notations as in Lemma \ref{Lemma : localization}. 
For every $y\in U_x$, the sub-object 
$(\Fs_{\geq i})_y\subseteq\Fs_y$ provided by 
Proposition \ref{Prop : existence of (F_S,i)_x}
is the restriction of  the inclusion 
$(\Fs_{\geq i})_{|U_x}\subseteq\Fs_{|U_x}$.
Indeed, by Proposition \ref{Prop : existence of (F_S,i)_x}, 
$(\Fs_{\geq i})_y$ is the \emph{unique} sub-object of 
$\Fs_y$ satisfying \eqref{eq : bijection local} and, 
by Lemma \ref{Lemma : localization}, the localization of  
$(\Fs_{\geq i})_{|U_x}$ at $y$ has the same property 
\end{remark}

\subsection{Descent of the augmented Dwork-Robba filtration.}
\label{Sec : Descent}

We now come back to our general setting.
\begin{hypothesis}\label{hyp : K general}
From now on $K$ is again an arbitrary complete ultrametric valued 
field of characteristic $0$ as in our original setting \ref{Secn1}. 
\end{hypothesis}
In this section we descend the augmented Dwork-Robba decomposition 
(cf. Corollary \ref{Lema : descent of M_xrho}).
The inverse image in $X_{\wKa}$ of a point of type $2$ or $3$ of $X$ 
is always a finite set and the descent is standard. 
However, for points of type $1$ or $4$, the Galois orbit may be infinite. In 
order to reduce to the finite case, we will not work directly with the points 
themselves, but with neighborhoods, which have finite Galois orbits.
\begin{remark}
We will prove that if $x\in X$ is a point of type $2$ or $3$, then 
Robba's decomposition of differential modules over $\H(x)$ 
also descends to 
$K$ (cf. Proposition \ref{rk : descent of Robba, not type 4}). However, for 
points of type $4$, this remains an open problem 
(cf. Remark \ref{rk : descent of Robba, not type 4-}). 
More specifically, the orbit of $x$ might be an infinite 
compact set and the ring of functions over it is 
unknown to us and does not seem to satisfy 
Proposition \ref{Lemma : j_Vmono -1}.
We believe that, instead of performing a descent, another possible 
approach would be to rewrite the results of section 4 directly over $K$.
\end{remark}

\subsubsection{Galois action on local domains around the inverse image of a point.}
\label{subsection : Descent of Dwork-Robba filtration}

In this subsection we provide some results about the action of the Galois group on various ring of functions.

Denote by $\mathrm{G} := \mathrm{Gal}(K^{\mathrm{alg}}/K)$ the absolute Galois group of $K^{\mathrm{alg}}$ over $K$. Its action on $K^{\mathrm{alg}}$ extends to a continuous action on~$\wKa$.

Denote by $\pi \colon X_{\wKa} \to X$ the extension of scalars map. It 
is continuous proper and surjective. The Galois group~$\mathrm{G}$ 
acts continuously on~$X_{\wKa}$ and its action is transitive on each 
fiber of~$\pi$. Those statements are well-known. We point the reader 
to \cite[Corollary~5.6]{PoineauTurchettiVIASMI} for a proof in the 
case of the affine line, which easily extends. 
We begin with three 
lemmas about the action of $G$ on our rings of functions.

\begin{lemma}\label{lem:AKalgG}
Let $U$ be an affinoid domain of~$X$ with algebra~$A$. The Galois group~$\mathrm{G}$ acts on $A \hat\otimes_{K} \wKa$ by acting on the second factor and we have 
\begin{equation}
(A \hat\otimes_{K} \wKa)^\mathrm{G} = A.
\end{equation}
\end{lemma}
\begin{proof}
Let $\|\wc\|$ denote the norm on~$A$. By definition, $A$ is a quotient of a Tate algebra. It follows that it contains a dense subspace whose dimension is 
at most countable. By \cite[Lemma~1.3.8]{Kedlaya-book-2}, there exists a sequence $(m_{i})_{i\in \NN}$ such that any element $a$ in~$A$ may be written in a unique way as a convergent series
\begin{equation}\label{eq:sumaimi}
a = \sum_{i\in \NN} a_{i} \, m_{i},
\end{equation}
with the $a_{i}$'s in~$K$. Moreover, with notation as in \eqref{eq:sumaimi}, the norm on~$A$ defined by
\begin{equation}
\|\wc\|' \colon a = \sum_{i\in \NN} a_{i} \, m_{i} \longmapsto \sup_{i\in \NN} (|a_{i}| \, \|m_{i}\|) \in \ERRE_{\ge 0}
\end{equation} 
is equivalent to the given norm~$\|\wc\|$.

\if{
\comm{Malheureusement le lemme 1.3.8 de Kedlaya ne marche pas en 
valuation triviale. Il a ?crit le bouquin comme si tout marchait en 
valuation triviale (et effectivement je n'ai pas trouv? d'?rreurs avant le 
Theor?me 1.3.6 inclus) mais dans la preuve 
de 1.3.7 et 1.3.8 on voit clairement qu'il abandonne cette hypoth?se 
sans s'en apercevoir et sans le mentionner dans l'nonc?. Par ailleurs, 
Kedlaya dit que son Lemme 1.3.8 vient du livre de Schneider 
(Proposition 10.4) et Schneider fait l'hypoth?se non triviale dans tout le 
bouquin.

Je pense que dans le cas de valuation triviale la preuve devrait ?tre 
essentiellement la m?me. Il suffit de d?montrer l'existence d'une base 
comme celle du Lemme 1.3.8 de Kedlaya.

On pourrait ?galement travailler avec la cohomologie, il faudrait 
travailler avec la suite exacte longue de cohomologie Galoisienne, 
et utilisr le fait que pour les alg?bres de Tate $T$ on a la base du Lemme 1.3.8, ... 
mais le probl?me est l'id?al $I$ de d?finition $0\to I\to T\to A\to 0$ on peut par exemple prendre une presentation de notre alg?bre affinoide $A$ par 
des alg?bres de Tate $\cdots T_{n_k}\to T_{n_{k-1}}\to\cdots\to T_{n_2}\to T_{n_1}\to A\to 0$. C'est une suite exacte longue o? $T_{n_i}$ est 
exacte et v?rifie $\overline{T_{n_k}}^G=T_{n_k}$. Si par exemple on avait une suite courte $0\to T_{n_2}\to T_{n_1}\to A\to 0$ alors la cohomologie 
serait  $0\to T_{n_2}\to T_{n_2}\to \overline{A}^G\to H^1(\overline{T_{n_2}},G)\to H^1(\overline{T_{n_1}},G)\cdots$, 
et je pense qu'on a un Hilbert `90 pour les $T_n$ qui entraine $H^1(\overline{T_{n_2}},G)=0$. Je pense que ?a marche, mais c'est plus compliqu?... 
tu me disais que tu avais la d?scente dans Banachoid... c'est probablement plus simple...
}
}\fi

It follows that any element $b$ in $A \hat\otimes_{K} \wKa$ may be 
written in a unique way as a convergent series
\begin{equation}
b = \sum_{i\in \NN} b_{i} \, m_{i},
\end{equation}
with the $b_{i}$'s in~$\wKa$. 

Let $b = \sum_{i\in \NN} b_{i} \, m_{i} \in (A \hat\otimes_{K} \wKa)^\mathrm{G}$. For each $g \in \mathrm{G}$, we have 
\begin{equation}
g(b) = \sum_{i\in \NN} g(b_{i}) \, m_{i}.
\end{equation}
Since $g(b)=b$, we deduce that $g(b_{i}) = b_{i}$ for all $i\in \NN$, by uniqueness of the coefficients. It then follows from the theorem of Ax-Sen-Tate that  $b_{i} \in K$ for all $i\in \NN$, hence that $b \in A$, as required.
\end{proof}

For all $x\in X$ we set 
\begin{equation} 
\overline{\O}_{x} := \O_{X_{\wKa}}(\pi^{-1}(x)) = 
\varinjlim_{W \supset \pi^{-1}(x)} \O(W)\;=\;
\bigcup_{W \supset \pi^{-1}(x)} \O(W)\;,\footnote{Here and in the whole text, the inductive limits are taken in the category or rings.}
\end{equation}
where $W$ runs through the neighborhoods of~$\pi^{-1}(x)$ in~$X_{\wKa}$. Since~$\pi$ is proper, we obtain a basis of neighborhoods of~$\pi^{-1}(x)$ by taking the preimages of the elements of a basis of neighborhoods of~$x$.\footnote{Let $U_{0}$ be a compact neighborhood of~$x$. Then $\pi^{-1}(U_{0})$ is a compact neighborhood of~$\pi^{-1}(x)$. The complement $K$ of any open neighborhood $Z$ of $\pi^{-1}(x)$ in~$\pi^{-1}(U_{0})$ is compact. Its image $\pi(K)$ is a compact subset of~$U_{0}$ that does not contain~$x$. The set $U_{0} - \pi(K)$ is an open neighborhood of~$x$ whose preimage is contained in~$Z$, by construction.
}
As a consequence, we have 
\begin{equation} 
\overline{\O}_{x} = 
\varinjlim_{U \ni x } \O(\pi^{-1}(U)) =  
\varinjlim_{U \ni x } \big(\O(U) \hat\otimes_{K} \wKa\big),
\end{equation}
where $U$ runs through the affinoid neighborhoods of~$x$ in~$X$. 
In order to simplify notations we set
\begin{equation}\label{eq : OUbar}
\overline{\O(U)}\;:=\;\O(U) \hat\otimes_{K} \wKa\;.
\end{equation}
We deduce the following result from Lemma~\ref{lem:AKalgG}. 
\begin{lemma}\label{lem:barOxG}
We have
\begin{equation}
\overline{\O}_{x}^\mathrm{G}\; =\; 
\varinjlim_{U \ni x}\overline{\O(U)}^{\mathrm{G}}\;=\;
\varinjlim_{U \ni x}\O(U)\;=\;
\O_{x}, 
\end{equation}
where $U$ runs through the family of affinoid neighborhood of $x$ in $X$.
\qed
\end{lemma}
\begin{lemma}\label{lem:RW}
Let $y_{0}\in \pi^{-1}(x)$. For each neighborhood~$W_{0}$ of~$y_{0}$ in~$X_{\wKa}$, there exists an affinoid neighborhood~$W$ of~$y_{0}$ in~$W_{0}$ and a finite subset~$R_{W}$ of~$\mathrm{G}$ containing~$\mathrm{id}_{G}$ such that the $g(W)$ for $g\in R_{W}$ are disjoint and
\begin{equation}
 \bigsqcup_{g \in R_{W}} g(W) \supset \pi^{-1}(x).
\end{equation}
\end{lemma}
\begin{proof}
Let $W_{0}$ be a neighborhood of $y_{0}$ in $X_{\wKa}$.

\noindent $\bullet$ Assume that $x$ is of type~2 or~3. In this case, the fiber of~$\pi$ over~$x$ is finite: $\pi^{-1}(x) = \{y_{0},\dotsc,y_{t}\}$. For each $i\in \{1,\dotsc,t\}$, choose $g_{i} \in \mathrm{G}$ such that $g_{i}(y_{0}) = y_{i}$. Set $g_{0} := \mathrm{id}_{G}$.

Since~$X_{\wKa}$ is Hausdorff, there exist open subsets $U_{0},\dotsc,U_{t}$ of~$X$ that contain respectively $y_{0},\dotsc,y_{t}$ and are pariwise disjoint.

Any affinoid neighborhood~$W$ of~$y_{0}$ contained in
\begin{equation}
 W_{0} \cap \bigcap_{1\le i \le t} g_{i}^{-1}(U_{i})
\end{equation}
now satisfies the properties of the statement with $R_{W} = \{\mathrm{id}_{G},g_{1},\dotsc,g_{t}\}$.

\medbreak

\noindent $\bullet$ Assume that $x$ is of type~$1$ or $4$. 

By \cite[Th\'eor\`eme~4.5.4]{Duc}, the point~$x$ is contained in a virtual open disk, hence we may assume that~$X$ is a virtual open disk. In this case, $X_{\wKa}$ is a disjoint union of finitely many open disks that form an orbit under the action of~$\mathrm{G}$. Let~$D$ be the open disk containing~$y_{0}$. Then, there exists $h_{0},\dotsc,h_{s} \in \mathrm{G}$ with $h_{0} = \mathrm{id}_{G}$ such that 
\begin{equation}
X_{\wKa} \;=\; \bigsqcup_{0\le j\le s} h_{j}(D)\;.
\end{equation}

Let $\mathrm{S}_{D} := \{ g \in \mathrm{G} \,\mid\, g(D) = D\}$. This is a subgroup of~$G$ (of finite index). We claim that it is enough to prove that there exists an affinoid neighborhood~$W$ of~$y_{0}$ in~$D \cap W_{0}$ and $g_{0},\dotsc,g_{r} \in S_{D}$ such that  
\begin{equation}
\bigsqcup_{0\le i \le r} g_{i}(W) \;\supset\; \pi^{-1}(x) \cap D\;.
\end{equation} 
Indeed, the result of the statement then holds with~$W$ and $R_{W} = \{h_{j} g_{i},\ 0\le j\le s, 0\le i\le r\}$.

The point~$y_{0}$ has same type as $x$, hence, by 
\cite[Th\'eor\`eme~4.5.4]{Duc} again, it has a basis of neighborhoods 
consisting of disks. In particular, if $R$ is the radius of $D$ in a 
given coordinate, there exists an element of $z_{0}\in D(\wKa)$ 
which is algebraic over $K$ and $r\in \mathopen{]}0,R\mathclose{[}$ such 
that the closed subdisk $\overline{D}(z_{0},r)$ of~$D$ is contained in~$W_{0}$ and contains $y_0$. 

Set 
\begin{equation}
\mathrm{H}\; :=\; \{ g \in \mathrm{S}_{D} \,\mid\, |g(z_{0}) - z_{0}| \le r\}\;.
\end{equation}
Since the action of~$\mathrm{G}$ on~$\wKa$ is isometric, $\mathrm{H}$ is a subgroup of~$\mathrm{S}_{D}$. Moreover, it contains the stabilizer of $z_0$, which is an open subgroup of~$\mathrm{S}_{D}$, hence $\mathrm{H}$ has finite index. 

Let $g_{1},\dotsc,g_{t}$ be representatives of the non-trivial classes in $\mathrm{S}_{D}/\mathrm{H}$. Set $g_{0} := \mathrm{id}_{G}$.

Thanks to the fact that the action of~$\mathrm{G}$ is isometric, for each $i \in \{0,\dotsc,r\}$, we have $g_{i}(\overline{D}(z_{0},r)) = \overline{D}(g_{i}(z_{0}),r)$, and all the $\overline{D}(g_{i}(z_{0}),r)$ are disjoint, by definition of~$H$. The result follows.
\end{proof}

\subsubsection{Semi-linear representations around the inverse image of a point.}
Let $\overline A$ be a ring endowed with an action of~$\mathrm{G}$. Denote by $A := \overline{A}^{\mathrm{G}}$ the subring of invariant elements. Recall that a \emph{semi-linear action} of $\mathrm{G}$ on an $\overline{A}$-module 
$\overline{V}$ is a map 
\begin{equation}
\mu\colon (g,v) \in \mathrm{G}\times \overline{V}\longmapsto g(v) \in \overline{V}
\end{equation}
 such that, for all $g,g_1,g_2\in \mathrm{G}$, $v,v_1,v_2\in\overline{V}$, 
$a\in\overline{A}$, one has 
\begin{enumerate}
\item $g_1(g_2(v))=(g_1g_2)(v)$, and $\mathrm{id}_{\mathrm{G}}(v)=v$;
\item $g(v_1+v_2)=g(v_1)+g(v_2)$;
\item $g(av)=g(a)g(v)$.
\end{enumerate}
We also say that $\overline{V}$ is a \emph{semi-linear representation} of $\mathrm{G}$ over $\overline{A}$. 

We denote by 
\begin{equation}
\overline{V}^\mathrm{G}\;:=\;\{v\in\overline{V}\;|\; 
\forall\; g\in \mathrm{G},\; g(v)=v\}\;
\end{equation}
the set of elements fixed by $\mathrm{G}$. It is naturally an $A$-module.

For every semi-linear representation~$\overline V$ of $\mathrm{G}$ over $\overline A$, we have a 
natural action of~$\mathrm{G}$ on $\overline{V}^\mathrm{G}\otimes_{A}\overline{A}$ given by 
$g(v\otimes a)=v\otimes g(a)$, for all $v\in \overline{V}^\mathrm{G}$ and 
$a\in \overline{A}$. Moreover, we have a natural
 $\overline{A}$-linear map that commutes with the action of  
$\mathrm{G}$
\begin{equation}\label{eq : j_V def}
\begin{array}{cccccc}
j_{\overline{V}} & \colon & \overline{V}^{\mathrm{G}}\otimes_{A}\overline{A} & \longrightarrow &\overline{V} &.\\
&& v \otimes a & \longmapsto & av &
\end{array}
\end{equation}

\begin{definition}\label{Def : trivial semi-rep}
We say that the action of $\mathrm{G}$ 
is \emph{trivial} on $\overline{V}$ if $j_{\overline{V}}$ is an isomorphism and 
$\overline{V}^{\mathrm{G}}$ is free of finite rank over~$A$. 
\end{definition}
In other words, the action of $\mathrm{G}$ is \emph{trivial} if, and only if, 
there exists an integer~$n$ and an $\overline{A}$-linear isomorphism $\overline{V}\cong\overline{A}^n$ commuting with the action of $\mathrm{G}$, where $\mathrm{G}$ acts on $\overline{A}^n$ componentwise. \bigskip

Let $x\in X$. 
We now consider an affinoid neighborhood  $U_{0}$ of~$x$ in~$X$ and a 
semi-linear representation $\overline{V}$ of $\mathrm{G}$ over $\overline{\O(U_{0})}$. 
In analogy with \eqref{eq : OUbar}, for every affinoid neighborhood 
$U\subseteq U_0$ of $x$ we set 
\begin{equation}\label{eq : notations V, V_x}
\overline{V(U)} \;:=\; 
\overline{V} \otimes_{\overline{\O(U_{0})}} \overline{\O(U)}
\qquad(\textrm{resp.}\quad
\overline{V}_{x} := \overline{V} \otimes_{\overline{\O(U_{0})}} \overline{\O}_{x})\;,
\end{equation}
where the action of $\mathrm{G}$ is given by $g(v\otimes f):=g(v)\otimes g(f)$, $v\in\overline{V}$, $f\in \overline{\O(U)}$ (resp. $f\in \overline{\O}_x$). 
Both are semi-linear representations of $\mathrm{G}$ over $\overline{\O(U)}$ 
and $\overline{\O}_{x}$ respectively. 
Moreover, the action of $\mathrm{G}$ is compatible with the restrictions
$\overline{V(U_1)}\to\overline{V(U_2)}$, for $U_2\subset U_1$, and inductive limits commute with tensor product, hence by Lemma \ref{lem:barOxG} we have a natural identification
\begin{equation}\label{eq: V^G bar_x}
\overline{V}^{\mathrm{G}}_x\;=\;\varinjlim_{x\in U}\overline{V(U)}^{\mathrm{G}}\;.
\end{equation}

\begin{lemma} \label{Lemma: torfree- flat O_x}
Every torsion free $\O_x$-module is flat. 
\end{lemma}
\begin{proof}
The ring $\O_x$ is either a field (if $x$ is not rigid) or a discrete valuation ring (when $x$ is rigid). The claim then follows from \cite[Chapitre VI, \S3, n.6, Lemme1]{Bou-Comm-Alg}.
\end{proof}

\if{\comm{J'ai redigé cette preuve, à corriger soigneusement car je ne maitrise pas bien les choses}
\begin{lemma}\label{Lemma : j_Vmono -1.0}
For every connected affinoid neighborhood $U\subseteq U_0$ of $x$, the map 
\begin{equation}
j_{\overline{V(U)}}\;:\;\overline{V(U)}^{\mathrm{G}}\otimes_{\O(U)}\overline{\O(U)} \xrightarrow{\quad} \overline{V(U)}
\end{equation}
is injective. 
\end{lemma}
\begin{proof}
Assume, by contradiction, that there exists 
$a \in (\overline{V(U)}^{\mathrm{G}}\otimes_{\O(U)}\overline{\O(U)}) \setminus \{0\}$ such that $j_{\overline{V(U)}}(a) = 0$. 
Write $a = \sum_{i=1}^m v_i \otimes f_{i}$, with 
$m\ge 1$, 
$v_{i} \in \overline{V(U)}^\mathrm{G}$ and 
$f_{i} \in \overline{\O(U)}$. 
We may assume that $m$ is minimal among all possible choices of 
$a\ne 0$ in the kernel, and all possible choices of presentation of $a$ as 
above. 

If $U'\subset U$ is another connected affinoid domain, we have a commutative diagram in which the maps commute with the action of 
$\mathrm{G}$
\begin{equation}
\xymatrix{
\overline{V(U)}^{\mathrm{G}}\otimes_{\O(U)}\overline{\O(U)}\ar[d] \ar[rr]^-{j_{\overline{V(U)}}} &&\overline{V(U)}\ar[d]\;\;\\
\overline{V(U')}^{\mathrm{G}}\otimes_{\O(U')}\overline{\O(U')} \ar[rr]^-{j_{\overline{V(U')}}} &&\overline{V(U')}\;.
}
\end{equation}
The restriction map 
$\overline{\O(U)}\to\overline{\O(U')}$ is injective and flat (cf. 
\cite[Proposition 2.2.4]{Ber} 

\comm{verifier reference! trouver une 
reference pour l'injectivité=continuation alaytique}), 
therefore the vertical arrows are injective. This shows that the injectivity of 
$i_{\overline{V(U')}}$ implies the injectivity of $i_{\overline{V(U)}}$.

As a first step, we are going to replace $a$ by another element in 
the kernel, admitting a similar representation 
with same length $m$ and same $v_i$s in which at least one of the $f_i$ is 
invertible in some neighbourhood of $\pi^{-1}(x)$.

\end{proof}
}\fi


\begin{proposition}\label{Lemma : j_Vmono -1}
We maintain the notation as above. 
Then, the map 
\begin{equation}
j_{\overline{V}_{x}}\;:\;
\overline{V}_{x}^{\mathrm{G}}\otimes_{\O_x}\overline{\O}_{x} 
\;\xrightarrow{\quad}\; \overline{V}_{x}
\end{equation}
is injective. 
\end{proposition}
\begin{proof}
\textbf{Step 1.} 
We may assume that $\overline{V}_x$ is finitely generated over 
$\overline{\O}_x$. 

Indeed, if 
$a\in\overline{V}_x^{\mathrm{G}}\otimes_{\O_x}\overline{\O}_x-\{0\}$ is 
a non zero element in the kernel of $j_{\overline{V}_x}$ we may write 
$a=\sum_{i=1}^mv_i\otimes f_i$ with $v_i\in \overline{V}_x^{\mathrm{G}}$ 
and $f_i\in\overline{\O}_x$. Let 
$\overline{V}_x'=\sum_{i=1}^mv_i\overline{\O}_x$ be the 
sub-$\overline{\O}_x$-module of $\overline{V}_x$ generated by 
$\{v_1,\ldots,v_m\}$. It is clearly globally stable by $\mathrm{G}$. 
As $\overline{\O}_x$ is flat over $\O_x$ (cf. Lemma \ref{Lemma: torfree- flat O_x}), the map 
$(\overline{V}_x')^{\mathrm{G}}\otimes_{\O_x}\overline{\O}_x
\to \overline{V}_x^{\mathrm{G}}\otimes_{\O_x}\overline{\O}_x$ is injective. 
Therefore, $j_{\overline{V}_x'}(a)=0$ and $a\neq 0$ in $\overline{V}_x'$.
It follows that the non-injectivity of $j_{\overline{V}_x}$ implies the 
non-injectivity of $j_{\overline{V}_x'}$ for the finite module 
$\overline{V}_x'$. Therefore, it is enough to prove the injectivity of 
$j_{\overline{V}_x}$ for every finitely generated $\overline{V}_x$.

\textbf{Step 2.} We may assume moreover 
that $\overline{V}_x$ is either torsion free as $\O_x$-module or killed by 
a generator $P$ of the maximal ideal of $\O_x$.

If $x$ is not rigid, then $\O_x$ is a field, and the statement is trivial.

If $x$ is rigid, denote by $(\overline{V}_x)_{tor}$ the torsion 
sub-$\O_x$-module 
of $\overline{V}_x$. By definition $v\in (\overline{V}_x)_{tor}$ if, and only if,
$fv=0$, for some $f\in\O_x - \{0\}$. 
It is easy to see that $(\overline{V}_x)_{tor}$ is a 
sub-$\overline{\O}_x$-module stable by the action of 
$\mathrm{G}$ and that we 
have an exact sequence of semi-linear representations 
of $\mathrm{G}$ over $\overline{\O}_x$
\begin{equation}
0\to (\overline{V}_x)_{tor} \to \overline{V}_x \to \overline{V}_x/(\overline{V}_x)_{tor}\to 0
\end{equation}
where the quotient $\overline{V}_x/(\overline{V}_x)_{tor}$ is torsion free as 
$\O_x$-module. This produces a left exact sequence 
$0\to (\overline{V}_x)_{tor}^{\mathrm{G}} \to \overline{V}_x^{\mathrm{G}} \to (\overline{V}_x/(\overline{V}_x)_{tor})^{\mathrm{G}}$ of 
$\O_x$-modules. Since $\overline{\O}_x$ has no torsion as $\O_x$-module, 
it is flat over $\O_x$ (cf. Lemma \ref{Lemma: torfree- flat O_x}) 
and hence we have a left exact sequence 
$0\to (\overline{V}_x)_{tor}^{\mathrm{G}}\otimes_{\O_x}\overline{\O}_x \to 
\overline{V}_x^{\mathrm{G}}\otimes_{\O_x}\overline{\O}_x \to 
(\overline{V}_x/(\overline{V}_x)_{tor})^{\mathrm{G}}
\otimes_{\O_x}\overline{\O}_x$. 
It follows that, the injectivity of $j_{(\overline{V}_x)_{tor}}$ 
and $j_{\overline{V}_x/(\overline{V}_x)_{tor}}$ implies that of 
$j_{\overline{V}_x}$, because we have a left exact sequence of their kernels. 

Therefore, without loss of generality, if $x$ is rigid 
we may assume that $\overline{V}_x$ 
is either torsion, or torsion free as $\O_x$-module. 

If $\overline{V}_x$ is torsion over $\O_x$, by definition every element of 
$\overline{V}_x$ is killed by a power of $P$. If $P\cdot \overline{V}_x = 0$, we are done. Otherwise, since $\overline{V}_x$ is finitely generated 
over $\overline{\O}_x$ (cf. Step 1) there is a power of $P$ killing 
every element of a family of generators. The same power of $P$ 
is easily seen to kill every other element of $\overline{V}_x$. Moreover, 
$P\cdot \overline{V}_x$ is a $\mathrm{G}$-invariant sub-$\overline{\O}_x$-module 
which is killed by a strictly smaller power of $P$.
We have an exact sequence
\begin{equation}
0\to P\cdot\overline{V}_x\to \overline{V}_x\to 
(\overline{V}_x/P\cdot\overline{V}_x) \to 0\;
\end{equation}
in which the left and right terms are killed by strictly smaller powers of $P$. 
As before, the injectivity of $j_{P\overline{V}_x}$ and 
$j_{\overline{V}_x/P\overline{V}_x}$ imply the injectivity of 
$j_{\overline{V}_x}$. Therefore, we may proceed by induction on the 
power of $P$ killing $\overline{V}_x$. It is therefore enoug to prove the 
claim when $\overline{V}_x$ is killed by $P$.

\textbf{Step 3.} Assume, by contradiction, that there exists 
$a \in (\overline{V}_{x}^{\mathrm{G}}
\otimes_{\O_{x}}\overline{\O}_{x}) \setminus \{0\}$ such that 
$j_{\overline{V}_{x}}(a) = 0$. Write $a = \sum_{i=1}^m v_i \otimes f_{i}$, 
with $m\ge 1$, $v_{i} \in \overline{V}_{x}^\mathrm{G}$ and 
$f_{i} \in \overline \O_{x}$. We may assume that $m$ is minimal among all 
possible choices of $a\ne 0$ in the kernel, and all possible choices of 
presentation of $a$ as above. We are going to prove that the existence of 
$a$ leads to a contradiction. 

As a first step, we are going to replace $a$ by another element in the 
kernel, admitting a similar representation 
(with same length $m$ and same set $\{v_i\}_i$) 
for which at least one of the $f_i$ is invertible in some 
neighbourhood of $\pi^{-1}(x)$.

By \eqref{eq: V^G bar_x}, there exists an affinoid neighborhood~$U_{1}$ of~$x$ in~$U_{0}$ such that $v_{i}\in \overline{V(U_1)}^\mathrm{G}$, $f_i\in\overline{O(U_1)}$ and $\sum_{i=1}^m f_{i} v_i = 0$ in $\overline{V(U_1)}$. 
We have a commutative diagram
\begin{equation}
\xymatrix{
\overline{V}_x^{\mathrm{G}}\otimes_{\O_x}\overline{\O}_x\ar[rr]^-{j_{\overline{V}_x}}&&\overline{V}_x\\
\overline{V(U_1)}^{\mathrm{G}}\otimes_{\O(U_1)}\overline{\O(U_1)}\ar[rr]^-{j_{\overline{V(U_1)}}}\ar[u]&&\ar[u]\overline{V(U_1)}
}
\end{equation}
We now specialize this relation at the points of $\pi^{-1}(x)\subset 
\pi^{-1}(U_1)$. 

%
%
%

\begin{lemma}
If $x\in X$ is not rigid, there exists $y_0\in\pi^{-1}(x)$ and $i_{0} \in 
\{1,\dotsc,m\}$ such that the image of $f_{i_{0}}$ in $\O_{y_{0}}$ is 
invertible.

If $x\in X$ is rigid, the same holds up to replace $a$ by another relation 
with the same length and same $v_i$'s.
\end{lemma}
\begin{proof}
Let us first assume that $x$ is not rigid. 
In this case, we may proceed by contrapositive without changing $a$. 
If for all $i$ and all $y\in\pi^{-1}(x)$ the image of $f_i$ in $\O_{y}$ is zero, 
then for all $i$ we may find a neighborhood $U_y$ 
of every $y$ on which $f_i$ is zero. 
The union of those neighborhoods is a neighborhood of $\pi^{-1}(x)$ 
on which every $f_i$ is zero. This implies that the restriction of 
$\sum_iv_i\otimes f_i$ to  some  neighbourhood $U$ of $\pi^{-1}(x)$ 
is zero in $\overline{V(U)}^{\mathrm{G}}\otimes_{\O(U)}\overline{\O(U)}$, 
contradicting the fact that $a\neq 0$. Therefore, we may find $i_0$ and 
$y_0\in\pi^{-1}(x)$ such that the image of $f_{i_0}$ in $\O_{y_0}$ is not 
zero. 

If $x$ has type $2$, $3$, or $4$, then $\O_{y_0}$ is a field and the image of 
$f_{i_0}$ in $\O_{y_0}$ is invertible. The claim is proved in this case.

If $x$ has type $1$ but $x$ is not rigid, then 
the fiber $\pi^{-1}(x)$ is an infinite compact set on which $\mathrm{G}$ 
acts transitively. If $U_{y_0}$ is a connected affinoid neighborhood of $y_0$ 
on which $f_{i_0}$ is defined, then $\bigcup_{g\in\mathrm{G}}\;g(U_{y_0})$ 
covers $\pi^{-1}(x)$ and we may extract from it a finite covering of 
$\pi^{-1}(x)$. It follows that $\pi^{-1}(x)\cap U_{y_0}$ is an infinite set and 
since $f_{i_0}$ has a finite number of zeros on $U_{y_0}$ we can find 
$y_{1}\in U_{y_0}\cap\pi^{-1}(x)$ such that $f_{i_0}(y_1)$ is not zero in 
$\H(y_{1})\cong\wKa$. Hence, the image of $f_{i_0}$  in $\O_{y_1}$ is invertible. 
The claim follows in this case too.

Assume now that $x$ is rigid. In this case 
$\pi^{-1}(x)$ is a 
finite set and $\overline{\O}_{x}=\prod_{\pi(y)=x}\O_{y}$. For each $y \in \pi^{-1}(x)$, let $p_{y}$ be an element of $\overline{\O}_{x}$ whose image in~$\O_{y}$ generates the maximal ideal and whose image in $\O_{y'}$ with $y' \ne y$ is 0. Since $\wKa/K$ is unramified, $P$ generates the maximal ideal of each~$\O_{y}$, hence there exists an invertible element $u$ of $\overline{\O}_{x}$ such that $P = u \prod_{y \in \pi^{-1}(x)} p_{y}$.

Assume that $\overline{V}_x$ is torsion free over $\O_x$. If 
for all $i$ and all $y\in\pi^{-1}(x)$ the image of 
$f_i$ in $\O_{y}$ is not invertible, that is to say divisible by~$p_{y}$, then $P$ divides every $f_i$ in $\overline{\O}_x$. 
In this case, we may replace $a=\sum_i v_i\otimes f_i$ by 
$a':=\sum_i v_i\otimes (P^{-k}\cdot f_i)$, where $k\geq 0$ is 
the maximal natural number such that $P^k$ divides every $f_i$.
There exists $y\in\pi^{-1}(x)$ such that at least one of the $P^{-k}f_i$ 
has an image in $\O_y$ which is invertible. 
Moreover $j_{\overline{V}_x}(a')=P^{-k}\sum_i v_i\cdot f_i$ is zero because 
the multiplication by $P$ in $\overline{V}_x$ is injective by assumption. 
Hence, $a'$ is a relation with the same length $m$ and the same $v_i$'s.

Assume that $\overline{V}_x$ is killed by $P$. Then it is naturally a module 
over the ring $\overline{\O}_x/P\overline{\O}_x=
\prod_{y\in\pi^{-1}(x)}\O_y/p_{y}\O_y$. Note that for every $y\in\pi^{-1}(x)$, 
we have $\O_y/p_{y}\O_y=\wKa$ and that
$\overline{V}_x^{\mathrm{G}}$ is a vector space over the field 
$\O_x/P\O_x$. It follows that we have an isomorphism of $\O_x$-modules
\begin{equation}
\overline{V}_x^{\mathrm{G}}\otimes_{\O_x}\overline{\O}_x
\;\cong\;
\overline{V}_x^{\mathrm{G}}\otimes_{\O_x/P\O_x}\prod_{y\in\pi^{-1}(x)}
(\overline{\O}_x/p_{y}\overline{\O}_x)
\;\cong\;
\prod_{y\in\pi^{-1}(x)}
\overline{V}_x^{\mathrm{G}}\otimes_{\O_x/P\O_x}
(\overline{\O}_x/p_{y}\overline{\O}_x)
\end{equation}
By this isomorphism, 
the element $a=\sum_{i=1}^mv_i\otimes f_i$ can be written 
as $a=(a_y)_y$, where $a_y=\sum_{i=1}^mv_i\otimes f_{i,y}$ and 
$f_i=(f_{i,y})_{y\in\pi^{-1}(x)}\in \prod_{y\in\pi^{-1}(x)}
(\overline{\O}_x/p_{y}\overline{\O}_x)$. If $a\neq 0$, then at least 
one of the components $a_{y_0}$ is non-zero and therefore at least one of 
the $f_{i,y_0}\in \overline{\O}_x/p_{y_{0}}\overline{\O}_x$ is non-zero too. 
Since $\overline{\O}_x/p_{y_{0}}\overline{\O}_x = \wKa$ it follows that $f_{i,y_0}$ is 
invertible and the original element $f_{i}$ has an image in $\O_{y_0}$ 
which is not in the maximal ideal, hence invertible in $\O_{y_0}$. 
The claim follows.
\end{proof}
%
%
%
%
%
%
\emph{Continuation of Proof of Proposition \ref{Lemma : j_Vmono -1}.}

\textbf{Step 4.} Since the image of $f_{i_0}$ is invertible in $\O_{y_0}$, 
there exists a neighborhood~$W_{0}$ of~$y_{0}$ in $\pi^{-1}(U_{1})$ such 
that, for each $z\in W_{0}$, we have $f_{i_{0}}(z)\ne 0$. In particular, the 
image of $f_{i_{0}}$ in~$\O(W_{0})$ is invertible. 
By Lemma~\ref{lem:RW}, $W_{0}$ contains an affinoid neighborhood~$W$ of~$y_{0}$ for which there exists a finite subset~$R_{W}$ of~$\mathrm{G}$ containing~$\mathrm{id}_{G}$ such that $\bigsqcup_{g \in R_{W}} g(W) \supset \pi^{-1}(x)$.


Let $g\in R_{W}$. The action of~$g$ on~$\pi^{-1}(U_{0})$ induces an isomorphism of locally ringed spaces $W \simto g(W)$. We still denote by $g \colon \O(g(W)) \simto \O(W)$ the induced morphism on rings of functions. Applying its inverse to the restrictions~$f_{i,|W}$'s, we get a family $(g^{-1}(f_{1,|W}),\dotsc,g^{-1}(f_{m,|W}))$ of elements of~$\O(g(W))$ such that $g^{-1}(f_{i_{0},|W})$ is invertible in~$\O(g(W))$. Since the~$v_{i}$'s are $\mathrm{G}$-invariant, we have $\sum_{i=1}^m g^{-1}(f_{i,|W})v_{i} = 0$ in~$\O(g(W))$.

For $i \in \{1,\dotsc,m\}$, set
\begin{equation}\label{eq : f_i,R...}
f_{i,R} := (g^{-1}(f_{i,|W}))_{g\in R_{W}} \in \prod_{g\in R_{W}} \O(g(W)) = 
\O\big(  \bigsqcup_{g \in R_{W}} g(W) \big).
\end{equation}
The element $f_{i_{0},R}$ is invertible in $\O(  \sqcup_{g \in R_{W}} g(W))$ 
and we have $\sum_{i=1}^m f_{i,R}v_{i} = 0$ in $\O(  \sqcup_{g \in R_{W}} g(W))
$. Up to multiplying by the inverse of~$f_{i_{0},R}$, we may assume that $f_{i_{0},R} =1$. Note also that the image of $\sum_{i=1}^m v_i \otimes f_{i,R}$ in $\overline{V}_{x}^{\mathrm{G}}\otimes_{\O_{x}}\overline\O_{x}$ is non zero. Indeed, the projection of $\sum_{i=1}^m v_i \otimes f_{i,R}$ onto the factor $\O(W)$ (corresponding to $g = \mathrm{id}_{G}\in R_W$) is $\sum_{i=1}^m v_i \otimes f_{i}$, hence localizing at~$y_{0}$ gives the element $a_{y_{0}}$, which is non-zero by assumption.

Remark that there exists $j \in \{1,\dotsc,m\}$ such that the image of $f_{j,R}$ in~$\overline{\O}_{x}$ does not belong to~$\O_{x}$. Indeed, otherwise the element $\sum_{i=1}^m v_{i} \otimes f_{i,R}$ would belong to $\overline{V}_{x}^{\mathrm{G}}\otimes_{\O_{x}}\O_{x}$, which is sent isomorphically to $\overline{V}_{x}^{\mathrm{G}}$ by~$j_{\overline{V}_{x}}$. 
Since $\sum_{i=1}^m f_{i,R}v_{i} = 0$ and $\sum_{i=1}^m v_{i} \otimes f_{i,R} \ne 0$, we would get a contradiction. 

It now follows from Lemma~\ref{lem:barOxG} that there exists $h \in G$ such that $h(f_{j,R}) \ne f_{j,R}$.
%
%
%


%
This allows us to write
\begin{equation}
0=(h-\mathrm{id})\big(\sum_{i=1}^m f_{i,R} v_i \big)
=
\sum_{i=1}^m (h(f_{i,R})  h(v_i) - f_{i,R} v_{i})
=
\sum_{i=1}^m (h(f_{i,R}) - f_{i,R}) v_{i} \;.
\end{equation}
Since $h(f_{i_{0},R}) - f_{i_{0},R} = h(1) - 1 = 0$ and $h(f_{j,R}) - f_{j,R} \ne 0$, we get a non-trivial relation with at most $m-1$ terms. This contradicts the minimality of~$m$.


The claim follows.
\end{proof}

\begin{proposition}\label{Lema : sub rep trivial}
We maintain notation \eqref{eq : notations V, V_x}. Assume that $\overline{V}_{x}$ is a trivial semi-linear representation of~$\mathrm{G}$ over $\overline{\O}_{x}$. Let $\overline{W}$ be a sub-$\overline{\O(U_{0})}$-module of~$\overline V$ that is stable under~$\mathrm{G}$. Then $\overline{W}_x=\overline{W} \otimes_{\overline{\O(U)}} \overline{\O}_{x}$ is a trivial semi-linear representation of~$\mathrm{G}$ over $\overline{\O}_{x}$.
\end{proposition}
\begin{proof}
%
We have a short exact sequence of semi-linear representations of $\mathrm{G}$ over $\overline{\O(U_{0})}$~:
\begin{equation}
E:0\to\overline{W}\stackrel{i}{\to}\overline{V}\stackrel{p}{\to}\overline{Q}\to 0.
\end{equation}
For all affinoid neighborhood $U$ of $x$ contained in $U_0$ the algebra $\overline{\O(U)}$ is flat over $\O(U_0)$ (cf. \cite[Proposition 2.2.4]{Ber}). Moreover directed colimits are exact. 
We deduce a short exact sequence of semi-linear representations of 
$\mathrm{G}$ over $\overline{\O}_{x}$~:
\begin{equation}
E_{x}:0\to\overline{W}_{x}\stackrel{i_{x}}{\to}\overline{V}_{x}\stackrel{p_{x}}{\to}\overline{Q}_{x}\to 0.
\end{equation}
The sequence of invariants~$E_{x}^\mathrm{G}$ is easily seen to be left 
exact. 

The sequence 
$E_{x}^\mathrm{G}\otimes_{\O_{x}}\overline{\O}_{x}$ 
remains left exact because $\overline{\O}_x$ is flat over $\O_x$ (cf. Lemma 
\ref{Lemma: torfree- flat O_x}).
We then have a commutative diagram with exact rows
\begin{equation}\label{eq: diagram E_x^GoO_x to E_x}
\xymatrix{
0\ar[r]&\overline{W}_{x}^{\mathrm{G}}\otimes_{\O_x}\overline{\O_x}\ar[d]^{j_{\overline{W}_x}}\ar[r]^{i_{x}\otimes 1}&
\overline{V}_{x}^{\mathrm{G}}\otimes_{\O_x}\overline{\O_x}\ar[r]^{p_{x}\otimes 1}\ar[d]^{j_{\overline{V}_x}}&
\overline{Q}_{x}^{\mathrm{G}}\otimes_{\O_x}\overline{\O_x}
\ar[d]^{j_{\overline{Q}_x}} &\\
0\ar[r]&\overline{W}_{x}\ar[r]^{i_{x}}&\overline{V}_{x}\ar[r]^{p_{x}}&
\overline{Q}_{x}\ar[r] &0
}\end{equation}
By assumption, $\overline{V}_x$ is trivial, therefore $j_{\overline{V}_x}$ is an isomorphism. From the surjectivity of $p_{x}\circ j_{\overline{V}_{x}}$, we 
deduce the surjectivity of~$j_{\overline{Q}_{x}}$. Its injectivity follows from 
Proposition \ref{Lemma : j_Vmono -1}. 

This implies that $p_x\otimes1$ is surjective.
%
%
%
We may now apply the snake lemma to the diagram 
\eqref{eq: diagram E_x^GoO_x to E_x} to conclude that 
$j_{\overline{W}_{x}}$ is an 
isomorphism. To conclude we need to prove that 
$\overline{W}_{x}^{\mathrm{G}}$ is finite free over $\O_x$. 
We have an inclusion of $\overline{W}_{x}^{\mathrm{G}}$ in
$\overline{V}_{x}^{\mathrm{G}}$, which is finite free by assumption.  
For all $x\in X$, the ring $\O_x$ is a principal ideal domain, therefore
the claim follows by the structure theorem for finite 
modules over principal ideal domains. 
%
%
\end{proof}

%
%
%
%
%
%
%
%

\subsubsection{Descent of  the augmented Dwork-Robba decomposition.}\

The augmented Dwork-Robba decomposition holds 
on some neighborhood of every point of $X_{\wKa}$ 
(cf. Proposition \ref{Prop : existence of (F_S,i)_x}). 
If $x\in X$, we firstly obtain a general statement that is 
helpful to descend sub-objects defined around the stalk 
$\pi^{-1}(x)\subset X_{\wKa}$ (cf. Proposition 
\ref{Prop: descent of sub-objects}). Then, we prove how to glue the 
augmented Dwork-Robba decompositions at the points of $\pi^{-1}(x)$ to 
give a differential equation on a neighborhood of the stalk which fulfils
the conditions of that descent statement (cf. Corollary \ref{Lema : descent of M_xrho}). 

Let $S$ be a weak triangulation of~$X$. Its preimage $\overline{S} := 
\pi^{-1}(S)$ in $X_{\wKa}$ is a weak triangulation of $X_{\wKa}$ 
(cf. Proposition \ref{Def : S_L}). 
Let us set $\overline{X}:=X_{\wKa}$ and 
denote the pullback of~$\Fs$ to~$\overline{X}$ by 
$\overline{\Fs}=\Fs_{\wKa}=\pi^*\Fs$. 

\begin{lemma}\label{Lemma: galois preserves the radii}
For every differential equation $\Gs$ over $\overline{X}$, every 
$z\in \overline{X}$, 
every $i=1,\ldots,r=\mathrm{rank}(\Gs)$ and every $g\in \mathrm{G}$ 
we have
\begin{equation}
\R_{\overline{S},i}(g^{-1}(z),g^*(\Gs))\;=\;
\R_{\overline{S},i}(z,\Gs)\;.
\end{equation}
\end{lemma}
\begin{proof}
We may see $\wKa$ as a complete valued field extension of itself whose 
structural morphism is $g:\wKa\simto\wKa$. Thus 
$g:\overline{X}\to \overline{X}$ can be interpreted as the scalar extension 
morphism 
(cf. \eqref{eq : def of pi_L/K}). 
The claim is then a direct application of 
Proposition \ref{Prop: insensitive to scalar ext}.
\end{proof}
\if{\begin{proof}
Assume first that $\Gamma_S$ is not empty, 
that is $X$ is not a virtual open 
disk with empty weak weak triangulation 
(cf. Section \ref{section controlling graphs}). In this case, 
we may cut $X$ along $\Gamma_S$ and replace it 
by an affinoid domain with a weak triangulation without 
affecting the radii (cf. \eqref{eq : Res to Y equal radii f}). 
Without loss of generality, we assume that $X$ is affinoid.
Thus, we are now working with differential 
modules over the ring of global sections. 
Let $A$ be the affinoid algebra of $X$, and $\overline{A}=A\widehat{\otimes}_K\wKa$ be that of 
$X_{\wKa}$. 
Let us express $1\otimes g:\overline{A}\xrightarrow{\;\sim\;} \overline{A}$ 
as a composite $1\otimes g=\alpha_g\circ\phi_g$, 
where 
\begin{enumerate}
\item $\phi_g:\overline{A}\to \overline{A}\otimes_{\wKa,g}\wKa$ 
is the canonical  map $\overline{f}\mapsto \overline{f}\otimes 1$, where 
$\wKa$ is seen as an extension of complete valued fields via the structural 
morphism $g:\wKa\to\wKa$;
\item $\alpha_g:\overline{A}\otimes_{\wKa,g}\wKa\to 
\overline{A}$ is the map sending $f\widehat{\otimes} a\otimes b\in A\widehat{\otimes}_K\wKa\otimes_{\wKa,g}\wKa$ into 
$f\widehat{\otimes} g(a)b\in A\widehat{\otimes}_K\wKa$. 
\end{enumerate}
The first morphism $\phi_g$ is just a scalar extension of the ground field 
$\wKa$. Therefore $\phi_g(\overline{S})=\overline{S}$ 
(cf. \eqref{eq : S_L}) and the pull-back $\phi_g^*$ 
preserves the $\overline{S}$-radii. On the other hand, $\alpha_g$ is a $\wKa$-linear bounded 
isomorphism of Banach algebras and therefore it corresponds to a 
($\wKa$-linear) isomorphism of Berkovich curves. 
Moreover $\overline{S}$ is preserved by $\alpha_g$ because $g$ and 
$\phi_g$ do preserve it. It follows that the pull-back 
operation $\alpha_g^*$ also preserve the 
$\overline{S}$-radii. By composition, $g^*$ preserves the 
$\overline{S}$-radii and the claim follows.

The case where $X$ is a virtual open disk with empty weak triangulation 
is similar.
\end{proof}
}\fi

If $Q\subset X$ is a subset, we denote by $\nabla-\text{Germ}(Q)$ the category whose objects are differential equations defined on some unspecified analytic domain $U$ which is a neighborhood of $Q$, and the arrows are morphisms of differential equation defined on some unspecified $U'\subset U$ as above, where both the domain and the target of the arrow are defined. We call \emph{germ of differential equation on $Q$} an object in this category.
If $x\in X$, 
$\nabla-\text{Germ}(\{x\})$ and $\nabla-\text{Germ}(\pi^{-1}(x))$ 
are equivalent to the categories of differential modules over $\O_x$ and 
$\overline{\O}_x$ respectively.

Denote by $(\nabla,\mathrm{G})-\text{Germ}(\pi^{-1}(x))$ the category of germs of differential equations $\Fs$ on the stalk $\pi^{-1}(x)$ endowed with a semi-linear action of $\mathrm{G}$ commuting with $\nabla$, that is, for all $g\in \mathrm{G}$, the diagram
\begin{equation}\label{eq: action of G on connection}
\xymatrix{
\Fs\ar[d]_{g}\ar[r]^-{\nabla}&
\Fs\otimes\Omega^1_{\overline{X}}\ar[d]^-{g\otimes dg}\\
\Fs\ar[r]^-{\nabla}&
\Fs\otimes\Omega^1_{\overline{X}}
}
\end{equation}
commutes (recall that $\Omega^1_{\overline{X}}=\pi^*(\Omega^1_X)$ (cf. \cite[Proposition 3.3.3,(ii)]{bleu}). Morphisms in this category commute with both the connection and the action of $\mathrm{G}$.
\begin{proposition}
The category $\nabla-\text{Germ}(\{x\})$ is equivalent to the full 
subcategory of $(\nabla,\mathrm{G})-\text{Germ}(\pi^{-1}(x))$ formed 
by objects whose action of $\mathrm{G}$ is trivial. 
The pull-back functor $\pi^*$ and the fixed by Galois functor 
$\Fs\mapsto\pi_*(\Fs)^{\mathrm{G}}$ being quasi-inverse each other. 
The rank is preserved by this correspondence.\hfill $\qed$
\end{proposition}

By Proposition \ref{Lema : sub rep trivial} we deduce the following
\begin{proposition}[Descent of sub-objects]\label{Prop: descent of sub-objects}
Let $x\in X$. Let $\Fs\in\nabla-\text{Germ}(\{x\})$ be a germ of differential equation.
Let $\mathcal{G}\subset\pi^*\Fs$ be a sub-differential module globally 
stable by the action of $\mathrm{G}$ on $\pi^*\Fs$, then the action of 
$\mathrm{G}$ is trivial on $\mathcal{G}$ and there exists a sub-differential 
module $\mathcal{G}'\subset\Fs$ such that 
$\pi^*(\mathcal{G}')\cong\mathcal{G}$. 
\hfill$\qed$
\end{proposition}

\if{

....

....

....
 
Let us fix a differential equation $\Fs$ over $X$.
Set $\overline{X} := X_{\wKa}$ and denote by $\overline{\Fs}$ the pullback of~$
\Fs$ to~$\overline{X}$. Recall that $\overline{\Fs}$ is obtained from 
$\nabla:\Fs\to\Fs\otimes\Omega_X^1$ by scalar extension 
$\overline{\nabla}=\nabla\otimes1:\overline{\Fs}\to\overline{\Fs}\otimes_{\O_{\overline{X}}}\Omega^1_{\overline{X}}$ 
(recall that $\Omega^1_{\overline{X}}=\pi^*(\Omega^1_X)$ (cf. \cite[Proposition 3.3.3,(ii)]{bleu}).
In analogy with the previous notations, for every analytic domain $U$ in $X$ we set $\overline{\Fs(U)}:=\overline{\Fs}(\pi^{-1}(U))$ and for every $x\in X$, 
$\overline{\Fs}_x:=\varinjlim_{x\in U}\overline{\Fs(U)}$.
Let $S$ be a weak triangulation of~$X$. Its preimage $
\overline{S} := \pi^{-1}(S)$ in $\overline{X}$ is a weak triangulation of $
\overline{X}$ (cf. \cite{NP-II}). We have a natural semi-linear action of $
\mathrm{G}=\mathrm{Gal}(\wKa/K)$ on $\overline{\Fs}$ commuting with the 
connection, that is, for every $g\in \mathrm{G}$ the following 
diagram commutes
\begin{equation}\label{eq: action of G on connection}
\xymatrix{
\overline{\Fs}\ar[d]_{g}\ar[r]^-{\overline{\nabla}}&
\overline{\Fs}\otimes\Omega^1_{\overline{X}}\ar[d]^-{g\otimes dg}\\
\overline{\Fs}\ar[r]^-{\overline{\nabla}}&
\overline{\Fs}\otimes\Omega^1_{\overline{X}}
}
\end{equation}
For every $x\in X$ the resulting representation 
$\overline{\Fs}_x$ of $\mathrm{G}$ over $\overline{\O}_x$ is trivial.

\begin{proposition}[Descent datum]\label{Lemma : general descent-}
Let $\mathcal{T}$ be the 
category whose objects are triplets 
$(\overline{\Gs},\nabla_{\overline{\Gs}},\rho)$, where 
\begin{enumerate}
\item $\Gs$ is a finite free $\O_{\overline{U}}$-module defined on the inverse image $\pi^{-1}(U)$ for some unspecified affinoid neighborhood $U$ of $x$;
\item $\nabla:\Gs\to\Gs\otimes\Omega^1_{\overline{X}}$ 
is a connection on $\Gs$ over $\overline{U}$;
\item 
$\rho:\mathrm{G}\to Aut(\Gs)$ is a \emph{trivial} 
semi-linear action of $\mathrm{G}$ on $\Gs$ 
commuting with the connection $\nabla$.\footnote{In other words, 
for every affinoid neighborhood $V$ in $\overline{U}$ 
and every $g\in\mathrm{G}$, the action $\rho(g)(V):\Gs(V)\simto \Gs(V)$ commutes with the restrictions of the sheaf $\Gs$ and with the connection. Moreover, this action of $\mathrm{G}$ is trivial in the sense that there 
is $n\in\mathbb{N}$ and an $\O_{\overline{U}}$-linear isomorphism 
$i:\Gs\simto\O_{\overline{U}}^n$ commuting with $\mathrm{G}$, 
where $\mathrm{G}$ acts componentwise  on $\O_{\overline{U}}^n$. 
Notice that $i$ is not required to commute with the connection.}
\end{enumerate}
Morphisms of triplets in $\mathcal{T}$ 
are morphisms of $\overline{\O}_{x}$-modules (that is, morphisms of $\O_{\overline{V}}$-modules, for some unspecified affinoid neighborhood $V$ of $x$) commuting simultaneously with the connection and the action of 
$\mathrm{G}$. 

Then, the category of differential equations over $\O_x$ is equivalent to $\mathcal{T}$. The equivalence sends a differential equation $\Fs_x$ over $\O_x$, defined on some unspecified neighborhood $U$ of $x$, into the triplet
\begin{equation}
T(\Fs_x)\;:=\;(\overline{\Fs(U)},\overline{\nabla},\rho)
\end{equation}
where $\overline{\nabla}$ is the pull-back of the connection of $\Fs$ over $U$ and $\rho$ is the action just 
described (cf. \eqref{eq: action of G on connection}).
The inverse functor takes a triplet 
$(\Gs,\nabla,\rho)\in\mathcal{T}$ into the 
$\O_x$-module $\Fs_x:=(\Gs\otimes_{\overline{\O(U)}}
\overline{\O}_x)^{\mathrm{G}}=\varinjlim_{x\in U}\Gs(\overline{U})^G$ 
endowed with the restriction of $\nabla$.
\end{proposition}
\begin{proof} Let $(\Gs,\nabla,\rho)\in \mathcal{T}$ be a triplet. 
Since the action of $\mathrm{G}$ commutes with $\nabla$,  
$\Gs^{\mathrm{G}}$ carries the restriction of the connection.
Since the action of $\mathrm{G}$ is trivial on $\Gs$, the fixed point 
functor preserves the ranks. It is now straightforward to check that the 
two functors are inverse each other. 
\end{proof}

....

....

....

We now apply the results of the above section 
to descend sub-differential equations around $x$.
\begin{proposition}\label{Proposition : equivalence descent}
Let $\Fs_x$ be a differential module over $\O_x$ and let $U$ be an 
affinoid neighborhood of $x$ on which $\Fs_x$ extends as in \eqref{eq : M(U) choose}.
Let $(\overline{\Fs},\overline{\nabla},\rho)\in\mathcal{T}$ 
be the triplet associates to  $\Fs_x$ by the above proposition.
over $\overline{\O}_x$ furnished by the above equivalence.

Let $S\subseteq\overline{\M}_x$ be a sub-$\overline{\O}_x$-module such 
that 
\begin{itemize}
\item $\overline{\nabla}(S)\subseteq S$;
\item $g(S)\subseteq S$, for all $g\in \mathrm{G}$.
\item There exists a sub-$\overline{\O(U)}$-module $S_U\subseteq 
\overline{\M}_x$ endowed with a semi-linear action of 
$\mathrm{G}$ over $\overline{\O(U)}$ and an 
isomorphism of $\overline{\O}_x$-modules commuting with $\mathrm{G}$
\begin{equation}
S_U\otimes_{\overline{\O(U)}}\overline{\O}_x\;\simto\;S\;.
\end{equation}
\end{itemize}
Then, the action of $\mathrm{G}$ is trivial on $S$ and the triplet 
$(S,\overline{\nabla}_{|S},\rho_{|S})$ is an object of the category 
mentioned in Proposition \ref{Lemma : general descent-}. 
In particular, there exists a differential module 
$\N_x\subseteq\M_x$ such that $D(\N_x)=(S,\overline{\nabla}_{|S},\rho_{|S})$.
\end{proposition}
\begin{proof}
Apply Lemma \ref{Lema : sub rep trivial}.
\end{proof}

\comm{Permet moi de remprendre des notations locales. Dans tout ce qui precede on a fait la déscente pour les representations des espaces vectoriels et des modules sur $\O(U)$, $\O_x$ etc... et non pas des faisceaux. Il est logique d'utiliser des modules plutôt que des faisceaux. D'autre part, l'énoncé de Dwork-Robba est donné pour des modules. Je ferai le lien avec les faisceux à la fin.}
We now show how to apply the above descent result to differential modules defined around $x$. 

Let $U_0$ be an affinoid neighborhood of $x$ in $X$ such that 
$\Omega^1_{U_0/K}$ is free over $\O(U_0)$, generated by a continuous 
$1$-form $\omega\in \Omega^1_{X/K}(U_0)$. 
This choice corresponds to a continuous derivation 
$d:\O(U_0)\to \O(U_0)$ obtained as 
the composite map $\O(U_0)\to\Omega^1_{U_0/K}(U_0)=
\O(U_0)\omega\simto\O(U_0)$, where the first map is the canonical 
derivation, the second one is the $\O(U_0)$-linear isomorphism sending 
$\omega\in\Omega^1_{U_0/K}(U_0)$ onto $1\in\O(U_0)$.
We denote by the same symbol $d$ the restriction of $d$ to every affinoid 
neighborhood $U\subseteq U_0$ of $x$ and also the corresponding 
derivative $d:\O_{x}\to\O_x$. Let us set $\overline{d}:=d\otimes1:\O(U)
\widehat{\otimes}_K\wKa\to \O(U)\widehat{\otimes}_K\wKa$. It is the unique 
$\wKa$-linear derivation on $\overline{\O(U)}=\O(U)\widehat{\otimes}_K\wKa$ 
extending the $K$-linear derivation $d$ on $\O(U)$. Moreover, 
it commutes with the action of $\mathrm{G}$ on $\overline{\O(U)}$, which is 
given by $1\otimes g$, for $g\in\mathrm{G}$. 
The sheaf of differentials $\Omega^1_{X/K}$ is 
compatible with arbitrary complete valued extensions of $K$ (cf. \cite[Proposition 3.3.3,(ii)]{bleu}). 
In particular, for every affinoid neighborhood $U$ of $x\in X$ we have 
\begin{equation}
\Omega^1_{\pi^{-1}(U)/\wKa}\;=\;
\pi^*(\Omega^1_{U/K})\;.
\end{equation}
Therefore, the derivation $\overline{d}$ generates the space of continuous 
derivations of $\overline{\O(U)}$. We denote again by 
$\overline{d}:\overline{\O}_x\to\overline{\O}_x$ the corresponding 
derivation.
 
Let $\M_x$ be a differential module over $\O_x$ with respect to 
$d:\O_x\to\O_x$ and $\nabla:\M_x\to\M_x$ its connection. 
We may find an affinoid neighborhood $U$ of $x$ in $U_0$ and a 
differential module $\M(U)$ over $U$ such that 
$\M_x=\M(U)\otimes_{\O(U)}\O_{X,x}$.\footnote{Indeed, once we 
choose an $\O_x$-linear isomorphism $\alpha:\O_x^r\simto\M_x$, the 
pull-back of the connection of $\M_x$ is given by an operator of the form 
$\alpha^{-1}\circ\nabla\circ\alpha=d-G:\O_x^r\to\O_x^r$, where $G$ is a 
matrix with coefficients in $\O_x$ and $d$ denotes the endomorphism of 
$\O_x^r$ acting as the derivation $d$ component by component. 
The coefficients of the matrix $G$, as well as the matrix of $\alpha$ and its 
inverse, are analytic functions on some affinoid neighborhood 
$U\subseteq U_0$ of $x\in X$ (for which $\Omega^1_{U/K}$ is free). }
 Let us denote by 
\begin{equation}\label{eq : M(U) choose}
\overline{\M(U)}\;:=\;\M(U)\otimes_{\O(U)}\overline{\O(U)}\;,\qquad
\overline{\M}_x\;:=\;\M_x\otimes_{\O_x}\overline{\O}_x\;,
\end{equation}
endowed with the connections $\nabla\otimes 1-1\otimes \overline{d}$ (cf. \eqref{eq : tensor product connection}).

Since $\M_x$ is a finite dimensional vector space over $\O_x$, 
we have 
$\overline{\M}_x^G=\M_x$, because the action of 
$G$ on $\overline{\M}_x$ is trivial and $\O_x=\overline{\O}_x^G$. 
Moreover, the connection 
\begin{equation}
\overline{\nabla}\;:=\;\nabla\otimes 1+1\otimes \overline{d}
\end{equation}
on $\overline{\M}_x$ commutes with $G$.

\begin{proposition}[Descent datum]\label{Lemma : general descent-}
The category of differential modules over $\O_x$ is equivalent to the 
category whose objects are triplets $(D,\overline{\nabla},\rho)$, where 
\begin{itemize}
\item $D$ is a finite free $\overline{\O_x}$-module 
\item $\overline{\nabla}:D\to D$ 
is a connection with respect to $\overline{d}$
\item $\rho:\mathrm{G}\to Aut(D)$ is 
a \emph{trivial} semi-linear action of $\mathrm{G}$ over 
$\overline{\O}_x$ commuting with the connection $\overline{\nabla}$.
\end{itemize}
and where morphisms are $\overline{\O}_x$-linear maps commuting 
simultaneously with the connection and the action of $\mathrm{G}$. 

The equivalence sends a differential module $\M_x$ over $\O_x$ into the 
triplet 
\begin{equation}
D(\M_x)\;:=\;(\overline{\M}_x,\overline{\nabla},\rho)
\end{equation}
defined just before the proposition. The inverse functor takes a triplet 
$(D,\overline{\nabla},\rho)$ into the $\O_x$-module
$\M_x:=D^{\mathrm{G}}$ endowed with the restriction of 
$\overline{\nabla}$.
\end{proposition}
\begin{proof} Let $(D,\overline{\nabla},\rho)$ be an object of this category. Since the action of $\mathrm{G}$ commutes with $\overline{\nabla}$,  
$D^{\mathrm{G}}$ carries the restriction of the connection 
$\nabla:=\overline{\nabla}_{|\M_x}$.
Since the action of $\mathrm{G}$ is 
trivial on $D$, the dimension is preserved by the 
functor. It is now straightforward to check that the 
two functors are inverse each other. 
\end{proof}
We now apply the 
results of the above section to descend sub-differential modules.
\begin{proposition}\label{Proposition : equivalence descent}
Let $\M_x$ be a differential module over $\O_x$ and let $U$ be an 
affinoid neighborhood of $x$ on which $\M_x$ extends as in \eqref{eq : M(U) choose}.
Let $(\overline{\M}_x,\overline{\nabla},\mathrm{G})$ be the object over 
$\overline{\O}_x$ furnished by the above equivalence.

Let $S\subseteq\overline{\M}_x$ be a sub-$\overline{\O}_x$-module such 
that 
\begin{itemize}
\item $\overline{\nabla}(S)\subseteq S$;
\item $g(S)\subseteq S$, for all $g\in \mathrm{G}$.
\item There exists a sub-$\overline{\O(U)}$-module $S_U\subseteq 
\overline{\M}_x$ endowed with a semi-linear action of 
$\mathrm{G}$ over $\overline{\O(U)}$ and an 
isomorphism of $\overline{\O}_x$-modules commuting with $\mathrm{G}$
\begin{equation}
S_U\otimes_{\overline{\O(U)}}\overline{\O}_x\;\simto\;S\;.
\end{equation}
\end{itemize}
Then, the action of $\mathrm{G}$ is trivial on $S$ and the triplet 
$(S,\overline{\nabla}_{|S},\rho_{|S})$ is an object of the category 
mentioned in Proposition \ref{Lemma : general descent-}. 
In particular, there exists a differential module 
$\N_x\subseteq\M_x$ such that $D(\N_x)=(S,\overline{\nabla}_{|S},\rho_{|S})$.
\end{proposition}
\begin{proof}
Apply Lemma \ref{Lema : sub rep trivial}.
\end{proof}

---------

----------

----------
}\fi

\begin{corollary}[Descent of the augmented Dwork-Robba decomposition]\label{Lema : descent of M_xrho}
For all $x\in X$, the statements of Theorem \ref{Dw-Robba}, 
Proposition \ref{Prop : existence of (F_S,i)_x} and 
Lemma \ref{Lemma : localization} descend to $K$.
\end{corollary}
\begin{proof}

First of all, notice that we only need to prove the 
existence statement of Proposition \ref{Prop : existence of (F_S,i)_x}. 
Indeed, the other properties in the three statements 
involve radii and solutions which are 
easily seen to be true once the existence is proved.

As explained in the proof of Proposition \ref{Prop : existence of (F_S,i)_x}, 
if $i$ is an over-solvable index separating the radii of $\Fs$ at $x\in X$, 
the image of the disk $D_{S,i}(x,\Fs)\subset X_\Omega$ in $X$ is 
a virtual open disk $D_i$ in $X-\Gamma_{S}$ 
containing $x$ (cf. Remark \ref{rk : solvable not in gamma}). 
Since $(D_i)_{\wKa}=\pi^{-1}(D_i)$ contains 
the stalk $\pi^{-1}(x)\subset \overline{X}$ we may define 
$(\Fs_{\geq i})_x\cong \Hdr^0(D_i,\Fs)\otimes_{K}\O(D_i)$ as 
the trivial differential sub-module of $\Fs$ generated by the solutions over 
$D_i$ (cf. \eqref{def : M_A}). 
The existence of $(\Fs_{\geq i})_{x}$ 
is then proved for oversolvable radii. 

Let us assume that $i$ is a spectral index separating the $S$-radii of 
$\Fs$ at $x$. 
 By continuity of the radii, we may find an open 
neighborhood $U$ of $x$ such that the index $i$ 
separates the $S$-radii of $\Fs$ at every point $z\in U$. 
Since the radii are insensitive to scalar extension of $K$, the index $i$ 
separates also the radii of $\overline{\Fs}$ at every point of 
$\overline{U}=\pi^{-1}(U)$.  
We now apply Lemma \ref{Lemma : localization} to every 
point $y\in\pi^{-1}(x)$. We obtain an open covering $\{\overline{U}_y\}_{y\in\pi^{-1}(x)}$ of the fiber $\pi^{-1}(x)$ and a family of sub-objects 
$(\overline{\Fs}_{\geq i})_{|\overline{U}_y}\subseteq\overline{\Fs}_{|\overline{U}_y}$. 
By replacing $\overline{U}_y$ with $\overline{U}\cap \overline{U}_y$, we 
may assume that 
$\overline{U}=\bigcup_{y\in\pi^{-1}(x)} \overline{U}_y$.  
By local uniqueness (cf. Proposition \ref{Prop : existence of (F_S,i)_x}), 
the family $\{(\overline{\Fs}_{\geq i})_{|\overline{U}_y}\}_{y\in\pi^{-1}(x)}$ 
glues to a global sub-object $(\overline{\Fs}_{\geq i})_{|\overline{U}}$ (cf. Remark \ref{Rk: gling propty}). 
By Proposition \ref{Prop: descent of sub-objects}, it is enough to prove 
that the sub-object $(\overline{\Fs}_{\geq i})_{|\overline{U}}
\subseteq\overline{\Fs}_{|\overline{U}}$ is stable under the 
action of every $g\in\mathrm{G}$ : 
$g((\overline{\Fs}_{\geq i})_{|\overline{U}})\subseteq (\overline{\Fs}_{\geq i})_{|\overline{U}}$. 
The datum of the $\O_{\overline{U}}$-semi-linear 
automorphism $g$ of $\overline{\Fs}_{|\overline{U}}$ commuting with the 
connection is equivalent to an $\O_{\overline{U}}$-linear isomorphism of 
differential equations
\begin{equation}
\psi_g\;:\;g^*(\overline{\Fs}_{|\overline{U}})\xrightarrow{\;\sim\;}\overline{\Fs}_{|\overline{U}}\;.
\end{equation}
The statement then amounts 
to prove that for every $g\in\mathrm{G}$ the map $\psi_g$ 
realizes an isomorphism between the sub-modules 
$g^*((\overline{\Fs}_{\geq i})_{|\overline{U}})\subseteq 
g^*(\overline{\Fs}_{|\overline{U}})$ and 
$(\overline{\Fs}_{\geq i})_{|\overline{U}}
\subseteq
\overline{\Fs}_{|\overline{U}}$. 

We now want to deal with the radii of $(\Fs_{\geq i})_{|\overline{U}}$, 
which are defined only after localization to $\overline{U}$. 
Recall that, by Corollary 
\ref{Cor: spectral index separ = intrinsic}, the spectral indexes separating the radii are intrinsic. Therefore, we may assume $X=U$, and 
consider an arbitrary weak triangulation $S$ of it 
without affecting the \emph{spectral indexes} separating the 
$S$-radii of $\Fs$ at $x$ (and hence also the $\overline{S}$-radii of 
$\overline{\Fs}$ the points of $\pi^{-1}(x)$). 
By continuity of the radii, up to further shrinking $X$, we may 
also assume that the \emph{same} spectral indexes separate the 
$S$-radii of $\Fs$ at every point of $X$.\footnote{For instance, 
we may consider $X$ to be a star-shaped affinoid domain centered at $x$ 
(cf. Definition \ref{def:starshapedopen}) which is the 
union of an elementary tube $V$ centered at $x$ 
that is adapted to $\Fs$ (cf. Definition \ref{Def : tube}) with 
small annuli with boundary $x$ in the missing directions.}
Thus, we may simplify notation and set 
$\Fs,\overline{\Fs},\overline{\Fs}_{\geq i}$ instead of 
$\Fs_{|U},\overline{\Fs}_{\overline{U}},(\overline{\Fs}_{\geq i})_{|\overline{U}}$. 
By Remark \ref{Rk: gling propty}, for every $z\in \overline{X}$ its 
localization $(\overline{\Fs}_{\geq i})_z$ is the sub-object furnished by 
Proposition \ref{Prop : existence of (F_S,i)_x}.
\begin{lemma}\label{Lemma : radii const on Galois orbits}
For all $z\in \overline{U}$ and all $j=1,\ldots,r$ we have
\begin{equation}
\R_{\overline{S},k}(z,\overline{\Fs})\;=\;
\R_{\overline{S},k}(g(z),\overline{\Fs})\;.
\end{equation}
The same holds for $g^*(\overline{\Fs})$.
\end{lemma}
\begin{proof}
Since $\psi_g:g^*\overline{\Fs}\simto\overline{\Fs}$ 
is an isomorphism of differential equations, then 
$\overline{\Fs}$ and $g^*\overline{\Fs}$ have 
the same radii : for every $z\in \overline{X}$ and every $j=1,\ldots,r$ one has
$\R_{\overline{S},j}(z,g^*\overline{\Fs})=
\R_{\overline{S},j}(z,\overline{\Fs})=
\R_{\overline{S},j}(g^{-1}(z),g^*\overline{\Fs})$
(cf. Lemma \ref{Lemma: galois preserves the radii}). 
In particular, the first and the last term tells us that the radii of $g^*\overline{\Fs}$
(and hence also those of $\overline{\Fs}$) 
are constant on every Galois orbit.
\end{proof}

\emph{Continuation of Proof of Corollary \ref{Lema : descent of M_xrho}.}
Now, $\psi_g$ being an isomorphism,  
$g^*(\overline{\Fs}_{\geq i})$ is isomorphic to its image 
$\psi_g(g^*(\overline{\Fs}_{\geq i}))$. Therefore, they have the same radii. 
Hence, by Lemma \ref{Lemma: galois preserves the radii}, 
for every $z\in \overline{U}$ and every $j=1,\ldots,r-i+1$ we have
\begin{equation}
\R_{\overline{S},j}(z,\psi_g(g^*(\overline{\Fs}_{\geq i})))\;=\;
\R_{\overline{S},j}(z,g^*(\overline{\Fs}_{\geq i}))\;=\;
\R_{\overline{S},j}(g(z),\overline{\Fs}_{\geq i})\;=\;
\R_{\overline{S},j+i-1}(g(z),\overline{\Fs})\;.
\end{equation}
On the other hand, by Lemma \ref{Lemma : radii const on Galois orbits}, we obtain
\begin{equation}
\R_{\overline{S},j+i-1}(g(z),\overline{\Fs})\;=\;
\R_{\overline{S},j+i-1}(z,\overline{\Fs})\;=\;
\R_{\overline{S},j}(z,\overline{\Fs}_{\geq i})\;.
\end{equation}
In other words, $\psi_g(g^*(\overline{\Fs}_{\geq i}))$ and 
$\overline{\Fs}_{\geq i}$ have the same radii at every point of $\overline{X}$. 
By Proposition \ref{Prop. Exact sequence prop separate radii in the right order} 
the radii of the quotient $\overline{\Fs}/\overline{\Fs}_{\geq i}$ are those 
of  $\overline{\Fs}$ that are strictly smaller than 
$\R_{\overline{S},i}(z,\overline{\Fs})$ and by
Lemma \ref{Lemma : Hom=0} it follows that the image of 
$\psi_g(g^*(\overline{\Fs}_{\geq i}))$ is zero in the quotient. In other words,
$\psi_g(g^*(\overline{\Fs}_{\geq i}))\subseteq \overline{\Fs}_{\geq i}$. Since they have same ranks, they coincide and the claim follows.
\end{proof}

We conclude this section with a descent result for Robba's decomposition 
theorem.
\begin{proposition}\label{rk : descent of Robba, not type 4}
Let $x\in X$ be a point of type 2 or 3, then  Corollary 
\ref{corollary : uniqueness at H(x)} descends from $\wKa$ to $K$. 
\end{proposition}
\begin{proof}
The proofs follow closely those of this section and are actually simpler 
because the inverse image $\pi^{-1}(x)$ is a finite set and 
$\H(x)\widehat{\otimes}_K\wKa=\prod_{y\in\pi^{-1}(x)}\H(y)$.
\end{proof}
\begin{remark}\label{rk : descent of Robba, not type 4-}
Let $x\in X$ be a type 4 point. 
If $\pi^{-1}(x)$ is infinite we have not been able to obtain a descent of Robba's decomposition. 
In this case, the orbit $\pi^{-1}(x)$ is a infinite compact set. By 
\cite[Proposition 1.2.3]{Ber} and the Stone-Cech compactification theorem, 
there is a canonical map of the ring of functions $A=\H(x)\otimes_K\wKa$ 
over $\pi^{-1}(x)$ into the product of fields $B:=\prod_{y\in\pi^{-1}(x)}\H(y)$. 
Therefore, we may expected $A$ to be described as 
certain sequences $(f_y)_{y\in\pi^{-1}(x)}\in\prod_{y\in\pi^{-1}(x)}\H(y)$. 
However, we do not know whether the canonical map $A\to B$ is injective, 
and we do not have a description of its image. 
Specifically, the problem in the descent argument 
lies in the proof of Proposition \ref{Lemma : j_Vmono -1}, 
where the analogous of the element $f_{i,R}$ 
(cf. \eqref{eq : f_i,R...}) is supposed to correspond to an infinite sequence
$(f_y)_{y\in\pi^{-1}(x)}$ satisfying $f_{g(y)}=g(f_y)$, $g\in\mathrm{G}$, but 
we do not have arguments to prove that such a sequence belong to $A$.
\end{remark}

\section{
Global decomposition theorem.}\label{Proof of Main Th}

We now come back to the global setting. 
In Section \ref{sec:globaldecomposition} we obtain 
a global decomposition theorem. In the remaining sections we provide 
several properties and conditions satisfied by the decomposition. In 
particular, in Section \ref{Conditions to have a direct sum decomposition} 
we provide conditions ensuring that the decomposition is a direct sum.

Remember that $K$ is a general complete valued field 
(cf. Hypothesis \ref{hyp : K general}).

\if{
\begin{lemma}\label{lem:Fgeiunique}
Let $i \in \{1,\dotsc,r\}$. Assume that $i$ is a solvable index separating the radii of $\Fs$ at $x$. Then there exists at most one differential sub-module 
 $(\Fs_{\geq i})_x\subseteq\Fs_x$ of rank $r-i+1$ over $\O_{X,x}$ such that
 \begin{equation}
 \forall j =1,\dotsc,r-i+1,\ \R_{\emptyset,j}(x,\Fs_{\ge i,\vert D(x)}) = \R_{\emptyset,j+i-1}(x,\Fs_{\vert D(x)}).
 \end{equation}


\end{lemma}
\begin{proof}
\ 

\comment{Cette preuve devrait pouvoir se faire comme la preuve du th\'eor\`eme \ref{MAIN Theorem}, mais je ne la comprends pas. ll faudrait la reprendre et la pr\'eciser. Je la recopie ci-dessous.}

The uniqueness of $\Fs_{\geq i}$ can be seen as follows. 
Let $\Fs_{\geq i}'\subseteq\Fs$ be another sub-object with the same 
properties. Then, the composite map $\Fs_{\geq i}'\subseteq\Fs\to\Fs_{<i}$ must be zero because, by Proposition \ref{exactness} (cf. also item \eqref{ii Lemma :devisage} of Lemma \ref{Lemma :devisage}), the functor associating to a 
differential equation its solutions in a disk $D(x,\rho)$ is exact, therefore 
if $\Fs_{\geq i}'$ were not included in $\Fs_{\geq i}$, then $\Fs_{<i}$ would have some solution with great radius which contradicts the definition of $\Fs_{<i}$. This proves that $\Fs_{\geq i}'\subseteq\Fs_{\geq i}$ and the reverse inclusion follows from a symmetric argument. The claim follows.
\end{proof}

\begin{proposition}\label{prop:Fgeiexistence}
Let $i \in \{1,\dotsc,r\}$. Assume that $i$ is a solvable index separating the radii of $\Fs$ at $x$.
Then there exists a unique differential sub-module 
 $(\Fs_{\geq i})_x\subseteq\Fs_x$ of rank $r-i+1$ over $\O_{X,x}$ such that
  \begin{equation}\label{eq:Rjgei}
 \forall j =1,\dotsc,r-i+1,\ \R_{\emptyset,j}(x,\Fs_{\ge i,\vert D(x)}) = \R_{\emptyset,j+i-1}(x,\Fs_{\vert D(x)}).
 \end{equation}
\end{proposition}
\begin{proof}
As before, denote by $\pi \colon X_{\wKa} \to X$ the extension of scalars map. Set $\overline{\Fs} := \Fs_{\wKa}$. Let $y_{0} \in \pi^{-1}(x)$. 

Set $R_{i} := \R_{\emptyset,i}(x,\Fs_{\vert D(x)}) = \R_{\emptyset,i}(y_{0},\overline\Fs_{\vert D(y_{0})})$. By Dwork-Robba's decomposition (Theorem~\ref{Dw-Robba}), the differential submodule~$M_{i}$ of~$\overline\Fs_{y_{0}}$ defined by $M_i := \bigoplus_{\rho \ge R_{i}} \overline{\Fs}_{x}^\rho$ satisfies \eqref{eq:Rjgei}. 

\comment{Est-il bien clair que le th\'eor\`eme \ref{Dw-Robba} peut se r\'einterpr\'eter comme \c ca ?}

There exists an affinoid neighborhood~$Y_{0}$ of~$y_{0}$ in~$X_{\wKa}$ and a differential submodule~$\overline\Fs_i$ of~$\overline\Fs_{\vert Y_{0}}$ whose fiber at~$y_{0}$ identifies to~$M_{i}$. We may assume that $Y_{0}$ is connected. Using Lemma~\ref{lem:RW}, we may moreover assume that there exists an affinoid neighborhood~$U_{0}$ of~$x$ in~$X$ and a finite subset~$R_{0}$ of~$\mathrm{G}$ containing~$\mathrm{id}_{G}$ such that the $g(W_{0})$ for $g\in R_{0}$ are disjoint and
\[ \bigsqcup_{g \in R_{0}} g(Y_{0}) = \pi^{-1}(U_{0}).\] 

We extend the differential module~$\overline\Fs_{i}$ on~$Y_{0}$ to a differential module on~$\pi^{-1}(U_{0})$ by defining it as $(g^{-1})^\ast \overline\Fs_{i}$ on $g(Y_{0})$ for $g\in R_{0}$. We still denote the resulting module by~$\overline\Fs_{i}$. It is a differential submodule of~$\overline\Fs$.

It remains to prove that the germ of $\overline\Fs_{i}$ around~$\pi^{-1}(x)$ comes from a differential submodule~$\Fs_{\ge i}$ of~$\Fs_{x}$ by scalar extension. By Lemma~\ref{Lema : sub rep trivial}, it is enough to show that $\overline\Fs_{i}$ is stable under~$\mathrm{G}$. It is enough to show that $g^\ast \overline\Fs_{i} = \overline\Fs_{i}$ for $g$ in the stabilizer of~$W_{0}$. Consider such a~$g$. Set $\Gs_{i} :=  g^\ast \overline\Fs_{i} \cap \overline\Fs_{i}$. The quotient $\Fs_{i}/\Gs_{i}$ is a differential module. In particular, it is locally free. It follows from Lemma~\ref{lem:Fgeiunique} and the invariance of the radii by scalar extension that the stalks of~$\Gs_{i}$ and~$\Fs_{i}$ at~$x$ coincide. Since $W_{0}$ is connected, we deduce that $\Fs_{i}/\Gs_{i} = 0$. The result follows.
\end{proof}

\comment{Au final, on n'a le r\'esultat que pour un module local qui provient d'un module global par localisation. Je ne pense pas pouvoir faire mieux. Il faudrait peut-\^etre pr\'eciser \c ca au d\'ebut de la section.}

\comment{Je ne touche \`a rien ci-dessous et passe directement \`a la  section~\ref{sec:augmenteddecomposition}. }

\bigbreak

\comment{Je ne comprends pas bien ce qui se passe ci-dessous, jusqu'\`a la proposition \ref{Lemma : general descent}. On travaille ici avec une d\'erivation fix\'ee, non ? Et on peut prendre n'importe laquelle ? A-t-on vraiment besoin qu'elle engendre toutes les d\'erivations ? En tout cas, pour trouver \c ca, on peut prendre localement un morphisme \'etale vers $A^1$ et tirer en arri\`ere le $d/dT$ de la droite.}

Let us chose a derivation $\overline{d}:\overline{A}\to\overline{A}$ 
generating the 
$\overline{A}$-module of continuous $\wKa$-derivations of 
$\overline{A}$ and commuting with the action of $G$. Let $d$ be the 
restriction of $\overline{d}$ to $A$. Then, $d$ generates the 
$A$-module of continuous $K$-derivations of $A$.

\comm{Il fudrait d?montrer q'une telle derivation existe. Si il est vrai que  
$\overline{A}=A\widehat{\otimes}_K\wKa$ alors, il suffit de prendre 
$\overline{d}=d\otimes 1$, avec $d$ qui engendre les derivations de $A$. \bigskip\\

Mais pour $\O_{X,x}$ je ne crois pas qu'on ait tout ? fait cela, il est 
difficile de donner un sens ? $\widehat{\otimes}$ la compl?tion me 
d?range. 
Peut-?tre que dans les espaces Normoides, ?a a un sens ?\bigskip\\

Berkovich, dans 1.3.5, p.16 il dit qu'on a bien 
$\mathscr{M}(A\widehat{\otimes}_K\wKa)/G=\mathscr{M}(A)$, mais 
sans preuve.

}
If $\M$ is a differential module over $A$ we denote by 
$\overline{\M}:=\M\otimes_A\overline{A}$ its scalar extension 
to $\overline{A}$. $\overline{\M}$ carries a natural 
action of  $G$ given by $g(m\otimes a)=g(a)\otimes m$, for all $g\in G$, $m\in\M$, 
$a\in\overline{A}$. Since $\M$ is finite free over $A$, we have $
\overline{\M}^G=\M$, because the action of $G$ on $\overline{\M}$ is 
trivial and $A=\overline{A}^G$. 
Moreover, the connection (cf. \eqref{eq : tensor product connection})
\begin{equation}
\overline{\nabla}\;:=\;\nabla\otimes 1+1\otimes \overline{d}
\end{equation}
on $\overline{\M}$ commutes with $G$.

\begin{proposition}\label{Lemma : general descent}
The category of differential modules over $A$ is equivalent to the 
category of differential modules over 
$\overline{A}$ 
together with a \emph{trivial} action of $\mathrm{G}$ commuting 
with the connection, and morphisms commuting simultaneously with the connection 
and the action of $\mathrm{G}$. \hfill$\qed$
\end{proposition}

\comment{Ci-dessous, ne serait-il pas mieux d'\'enoncer de nouveau les r\'esultats~? \c Ca m'ennuie de ne pas les avoir sous forme finale dans le texte. Pour utiliser la d\'ecomposition sur un corps quelconque, \c ca oblige \`a citer deux r\'esultats.}

\begin{lemma}
Let $\rho \in \mathopen{]}0,1\mathclose{]}$. There exists a unique sub-$\overline{\O}_{x}$-module~$(\overline{M})^{\ge \rho}$ of~$\overline{M}$ such that, for each $y\in \pi^{-1}(x)$, we have $(\overline{M})^{\ge \rho}_{y} = (\overline{M}_{y})^{\ge \rho}$.
\end{lemma}
\begin{proof}
Uniqueness is clear. Let us prove existence. 
\end{proof}

\begin{lemma}\label{Lema : descent of M_xrho}
Robba's decomposition \eqref{eq : deco at x Robba eq} 
and Dwork-Robba's decompositions  \eqref{eq : deco of M_x .gdyu} 
by the spectral radii of $\overline{\M}$  
over $\widehat{K^{\mathrm{alg}}}$ descend to analogous 
decompositions of $\M$ by its spectral radii over $K$. 

Moreover the statements of Corollary 
\ref{corollary : uniqueness at H(x)}, and of
Theorem \ref{Dw-Robba}, also descend to $K$.
\end{lemma}
\begin{proof}
We use Proposition \ref{Lemma : general descent}. By Lemma \ref{Lema : sub rep trivial} every sub-representation of a trivial semi-
linear representation of $G$ is automatically trivial.
 Therefore, it is enough to prove that the differential submodule $(\overline{\M})^{\geq \rho}$ is stable under the action of $\mathrm{G}$. 

\comm{Erreur ici : $(\overline{\M})^{\geq \rho}$ n'est pas d\'efini, c'est un module d\'efini sur toute l'orbite $O(x)$. Si celle ci est finie, on a le droit de prendre point par point, mais si c'est infinie, c'est plus compliqu\'e... il faut prendre des disques autour du point de type 4...}

For all $g\in G$ we have a $G$-semi-linear bijection 
$g:\overline{\M}\simto\overline{\M}$. Define 
$g^*\overline{\M}:=\overline{\M}\otimes_{\overline{A},g}
\overline{A}$ as the
the scalar extension of 
$\overline{\M}$ by the ring homomorphism
$g:\overline{A}\to \overline{A}$, so 
that for all $g\in G$, $x,y\in\overline{A}$, $m\in\overline{\M}$ one has  
$x(m\otimes y)=(xm)\otimes y=m\otimes g(x)y \in g^*\overline{\M}$. 
There exists a unique $\overline{A}$-linear isomorphism $L_g:g^*\overline{\M}\simto\M$
extending the $G$-semi-linear isomorphism $g:\overline{\M}
\simto\overline{\M}$. 
That is, if $i:\overline{\M}\simto g^*\overline{\M}$ is the map $m\mapsto m\otimes 1$, then the following diagram is commutative
\begin{equation}\label{eq : diag linearization of g}
\xymatrix{g^*\overline{\M}\ar[r]^-{L_g}_{\sim}&\overline{\M}\\
\overline{\M}\ar[u]^-{i}_{\wr}\ar[ur]_-{g}&}
\end{equation}
The connection $\overline{\nabla}$ 
of $\overline{\M}$ extends to a connection 
$g^*\overline{\nabla}:=\overline{\nabla}\otimes 1+1\otimes\overline{d}$ 
on $g^*\overline{\M}$. Since $g:\overline{\M}\to\overline{\M}$ 
commutes with $\overline{\nabla}$, then $L_g\circ 
g^*\overline{\nabla}=\overline{\nabla}\circ L_g$.
Therefore, $L_g$ is an isomorphism of $\overline{A}$-differential 
modules.
It follows that, for all $y\in O(x)$, the radii of $g^*\overline{\M}$ and $\overline{\M}$ at 
$y$ coincide. By definition, the radii are insensitive to scalar extension 
of the ground field $K$, hence the radii of $\overline{\M}$ at $y$ are also 
the same as those of $\M$ at $x$. 

By base change to $X_\Omega$, the group $G$ acts transitively  
on the generic disks $D(y,\rho)$ (cf. Proposition 
\ref{Lemma pts of type 4} and \eqref{eq : piOmega/K}). 
For every $g\in G$ we have $g(D(y,\rho))=D(g(y),\rho)$ 
(cf. \cite{NP-II}). 

\smallcomment{Il faudrait expliquer mieux cette action de $G$ sur les disques g\'en\'eriques} 

We then have an action of $G$ on the disjoint union 
$D_\rho:=\cup_{y\in O(x)}D(y,\rho)$. 
We may restrict the diagram \eqref{eq : diag linearization of g} to $D_\rho$ 
\begin{equation}\label{eq : diag linearization of g-}
\xymatrix{g^*\overline{\M}_{|\O(D_\rho)}\ar[r]^-{L_g}_{\sim}&\overline{\M}_{|\O(D_\rho)}\\
\overline{\M}_{|\O(D_\rho)}\ar[u]^-{i}_{\wr}\ar[ur]_-{g}&}
\end{equation}
Now, Robba's (resp. Dwork-Robba's) theorems imply that $\overline{\M}^{\geq \rho}$ is the inverse image in 
$\overline{\M}$ of the maximal trivial differential sub-module of $\overline{\M}_{|\O(D_\rho)}$ via the restriction map $\overline{\M}\to\overline{\M}_{|\O(D_\rho)}$ (cf. item \eqref{iii Lemma : Trivial iff solutions} Lemma \ref{Lemma : Trivial iff solutions}). It follows that it is enough to show 
that the arrows in diagram \eqref{eq : diag linearization of g} 
preserve the maximal trivial sub-differential modules. 
The map $L_g$ is an isomorphism of differential modules, hence it 
induces an isomorphism on the maximal 
trivial sub-modules of $g^*\overline{\M}_{|\O(D_\rho)}$ and 
$\overline{\M}_{|\O(D_\rho)}$. It is then enough to prove that the map $i:\overline{\M}_{|\O(D_\rho)}\to g^*\overline{\M}_{|\O(D_\rho)}$ 
preserves the maximal trivial sub-modules. This is true because if 
$\overline{\nabla}(s)=0$ in $\overline{\M}$, then $g^*\overline{\nabla}(i(s))=\overline{\nabla}\otimes 1+1\otimes\overline{d}(s\otimes 1)=\overline{\nabla}(s)\otimes 1=0$.
This shows that the decompositions descend to $A$.

The other claims of the Robba and Dwork-Robba's statements 
descends as follows. 

Let us check the compatibility of the decomposition with respect to the 
duality. Let us consider the composite of the canonical morphisms 
$c:(\M^*)^{\geq\rho}\to\M^*\to(\M^{\geq\rho})^*$. By base change 
to $\overline{A}$, it gives the same canonical morphism $c:(\overline{\M}^*)^{\geq\rho}\to\overline{\M}^*\to(\overline{\M}^{\geq\rho})^*$  
 for $\overline{\M}$. 
By \eqref{eq : compatibility with duals} (resp. item \eqref{ii eq : deco of M_x .gdyu} of Theorem \ref{Dw-Robba}), 
it is an isomorphism over 
$\overline{A}$ and it commutes with $\mathrm{G}$. Therefore, 
$c$ itself is an isomorphism. 

Let us check the compatibility with morphisms.
Let $\alpha:\M^{\rho}\to\N^{\rho'}$ be 
a morphism with $\rho\neq\rho'$. By base change to $\overline{A}$, this produces a morphism of the corresponding $\overline{A}$-modules 
commuting with $\mathrm{G}$. By \eqref{eq : Hom (Mrho,Mrho')=0} (resp. item \eqref{iii eq : deco of M_x .gdyu} of Theorem \ref{Dw-Robba}), it is $0$, therefore $\alpha$ itself is $0$.
\end{proof}
}\fi

\subsection{Global decomposition theorem}\label{sec:globaldecomposition}

%
%
%
%
%

The following result is our main theorem.
Recall that $X$ is connected and $r=\mathrm{rank}(\Fs)$.

\begin{theorem}\label{MAIN Theorem}
Assume that the index $i$ separates the radii of $\Fs$ over $X$ 
(cf. Definition \ref{definition : i separates the radii}). 
Then $\Fs$ admits a unique sub-object 
$(\Fs_{\geq i},\nabla_{\geq i})\subset (\Fs,\nabla)$ such that for all 
$x\in X$ one has
\begin{enumerate}
\item\label{MAIN Theorem-i} $\mathrm{rank}\;\Fs_{\geq i}=
\mathrm{dim}_{\Omega}\;\omega_{S,i}(x,\Fs)=r-i+1$. 
\item\label{MAIN Theorem-ii} For all $j=1,\ldots,r-i+1$ the canonical inclusion 
$\Hdr^0(x,\Fs_{\geq i})\subset\Hdr^0(x,\Fs)$ identifies 
\begin{equation}
\omega_{S,j}(x,\Fs_{\geq i})\;=\;\omega_{S,j+i-1}(x,\Fs)\;.
\end{equation}
\end{enumerate}
Set $\Fs_{<i}:=\Fs/\Fs_{\geq i}$ and consider the exact sequence 
\begin{equation}\label{eq : sequence F>i --> F --> F<i}
0\to\Fs_{\geq i}\to\Fs\to\Fs_{<i}\to 0\;.
\end{equation}
Then, for all $x\in X$, one has 
\begin{equation}\label{eq : equality of radii of F_Si}
\R_{S,j}(x,\Fs)\;=\;\left\{\begin{array}{lcl}
\R_{S,j}(x,\Fs_{<i})&\textrm{ if }&j=1,\ldots,i-1\\
\R_{S,j-i+1}(x,\Fs_{\geq i})&\textrm{ if }&j=i,\ldots,r\;.\\
\end{array} \right.
\end{equation}
\end{theorem}
\begin{proof}
By Lemma \ref{Lemma : localization} (cf. Corollary 
\ref{Lema : descent of M_xrho}), for all point $x\in X$ there 
exists a neighborhood $U_x$ of $x$, and a 
unique sub-object  
\begin{equation}\label{eq : local sub-object}
((\Fs_{\geq i})_{|U_x},(\nabla_{\geq i})_{|U_x})\;\subset\;
 (\Fs_{|U_x},\nabla_{|U_x})
\end{equation} 
such that for all $y\in U_x$ one has 
\begin{enumerate}
\item[(A)] $\mathrm{rank}\;(\Fs_{\geq i})_{|U_x}\;=\;
\mathrm{dim}_\Omega\;\omega_{S,i}(y,\Fs)\;=\;r-i+1$;
\item[(B)] $\Hdr^0(y,(\Fs_{\geq i})_{|U_x})\;=\;
\omega_{S,i}(y,\Fs)$.
\end{enumerate}
Now, by local uniqueness (cf. Proposition \ref{Prop : existence of (F_S,i)_x}), 
the family 
$\{((\Fs_{\geq i})_{|U_x},(\nabla_{\geq i})_{|U_x})\}_x$ 
glues to a global sub-object $\Fs_{\geq i}$ (cf. Remark 
\ref{Rk: gling propty}). By \eqref{rank = r-i+1}
this global sub-object satisfies item \eqref{MAIN Theorem-i}. 
Item \eqref{MAIN Theorem-ii} follows from  Proposition 
\ref{Prop : extended by continuity}. 
The other claims follows from Proposition \ref{Prop. Exact 
sequence prop separate radii in the right order}. 

The uniqueness of $\Fs_{\geq i}$ can be seen as follows. 
Let $\Fs_{\geq i}'\subseteq\Fs$ be another sub-object with the same 
properties. Then, the composite map $\Fs_{\geq i}'\subseteq\Fs\to\Fs_{<i}$ 
must be zero by Lemma \ref{Lemma : Hom=0}. 
This proves that $\Fs_{\geq i}'\subseteq\Fs_{\geq i}$ and the reverse 
inclusion follows from a symmetric argument. The claim follows.
%
%
\end{proof}
\begin{remark}
In Section \ref{An explicit counterexample.} we provide an example 
in which \emph{$\Fs_{\geq i}$ is not a direct summand of $\Fs$}. 
In Section \ref{Conditions to have a direct sum decomposition} below
we provide criteria to guarantee that $\Fs_{\geq i}$ is a direct 
summand.
\end{remark}

\begin{proposition}[Independence of $S$]
\label{Prop. : independence on S}
Let $S$ and $S'$ be two weak triangulations. Assume that 
the index $i$ separates the radii of $\Fs$ with respect to both $S$ and 
$S'$ and denote by $\Fs_{S,\geq i}$ and $\Fs_{S',\geq i}$ the 
resulting sub-modules. Then 
$\Fs_{S,\geq i}=\Fs_{S',\geq i}$.
\end{proposition}
\begin{proof}
Since $i$ separates the radii in both cases one has 
$\omega_{S,j}(x,\Fs)=\omega_{S',j}(x,\Fs)$ for all $j\leq i$, and all 
$x\in X$ (this is a consequence of 
\eqref{eq: beh omega by change of tri}). 
By uniqueness of the augmented Dwork-Robba decomposition one has
$(\Fs_{S,\geq i})_x=(\Fs_{S',\geq i})_x$, for all $x\in X$. 
Hence the composite map 
$\Fs_{S',i}\subset\Fs\to\Fs_{S,<i}$ 
is locally the zero map, so it is globally zero 
by Proposition \ref{Prop. : iso on a point implies iso global}. 
Then $\Fs_{S',\geq i}\subseteq\Fs_{S,\geq i}$, and by a symmetric 
argument $\Fs_{S,\geq i}=\Fs_{S',\geq i}$.
\if{
Hence $\Fs_{S,\geq i}=\Fs_{S',\geq i}$ because 
the global decomposition is obtained by gluing the local 
decompositions. 
}\fi
\end{proof}
\begin{remark}\label{Remark : changing tr F_S,I}
Let $S,S'$ be two weak triangulations of $X$ such that 
$\Gamma_S\subseteq\Gamma_{S'}$. 
Passing from $S$ to $S'$ has the effect that over-solvable radii result 
truncated by the rule \eqref{eq : R_S' VS R_S}. 
Therefore, if the index $i$ 
separates the radii with respect to $S'$, then
 it also separates the radii with respect to $S$. 
This shows that, in order to fulfill the 
assumptions of Theorem \ref{MAIN Theorem}, 
the more convenient choice of triangulation is 
the weak triangulation $S$ with smaller graph $\Gamma_S$.
\end{remark}

\subsection{Conditions to have a direct sum decomposition}
\label{Conditions to have a direct sum decomposition}

In this section we provide criteria to test whether $\Fs_{\geq i}$ is a 
direct summand.

\begin{proposition}\label{Prop : radii= implies Graph=}
The following conditions are equivalent:
\begin{enumerate}
\item\label{i - Prop : radii= implies Graph=} $\Gamma_{S,i}(\Fs)=\Gamma_{S,i}(\Fs^*)$;
\item\label{ii - Prop : radii= implies Graph=} For all $x\in X$ one has 
$\R_{S,i}(x,\Fs)=\R_{S,i}(x,\Fs^*)$. 
\end{enumerate}
\end{proposition}
\begin{proof}
By the definition of controlling graph, 
clearly \eqref{ii - Prop : radii= implies Graph=} implies 
\eqref{i - Prop : radii= implies Graph=}. 
Assume then that \eqref{i - Prop : radii= implies Graph=} holds. 
Set 
$\Gamma:=\Gamma_{S,i}(\Fs)=\Gamma_{S,i}(\Fs^*)$. 

If $\Gamma=\emptyset$, then $X$ is a virtual open disk with empty 
triangulation, and both $\R_{S,i}(-,\Fs)$ and $\R_{S,i}(-,\Fs^*)$ are 
constant on $X$. If they are different, then one of them, say 
$\R_{S,i}(-,\Fs)$, is strictly less than $1$. 
Then, for $x$ approaching the boundary of the disk $X$, 
the radius $\R_{S,i}(x,\Fs)$ is spectral non-solvable. 
This contradicts Proposition \ref{dual}. So \eqref{ii - Prop : radii= implies Graph=} holds in this case.

Assume that $\Gamma\neq\emptyset$. 
Then $X-\Gamma$ is a disjoint union of virtual open disks on which
$\R_{S,i}(x,\Fs)$ and $\R_{S,i}(x,\Fs^*)$ are constant.
Therefore, it is enough to prove that 
$\R_{S,i}(x,\Fs)=\R_{S,i}(x,\Fs^*)$ for all $x\in\Gamma$. By 
Remark \ref{Remark : R_i solvable on Gamma_i}, if $x\in\Gamma$, 
then $\R_{S,i}(x,\Fs)$ and $\R_{S,i}(x,\Fs^*)$ are spectral. 
Hence they coincide by Proposition 
\ref{Prop : 1-th radius and dual}.
\end{proof}

\begin{remark}
See \cite[Proposition 2.4.4]{NP-IV} for another condition implying 
those of Proposition \ref{Prop : radii= implies Graph=}.
\end{remark}

\begin{remark}
Observe that for all $x\in\Gamma_S$ the radii are by definition 
spectral, and therefore $\R_{S,i}(x,\Fs)=\R_{S,i}(x,\Fs^*)$ for all 
$x\in\Gamma_S$ (cf. Propositions \ref{Prop : 1-th radius and dual})). 
Therefore, the locus where they may differ is the 
complement of $\Gamma_S$ in 
$X$, which is a disjoint union of maximal disks 
(cf. Definition \ref{Maximal disksksks}). 
Localization to a maximal disk $D(y,S)$ does not cut the radii 
(cf. Proposition \ref{Prop: Localization}), hence to study 
the differences between $\Fs$ and $\Fs^*$ we are often induced to 
work over an open disk with empty weak-triangulation, as in the 
example of Section \ref{optimality}.
\end{remark}

\begin{remark}\label{Remark : remove maximal disks to have ...}
Up to remove some disks from $X$, we can always obtain the 
equivalent conditions of Proposition \ref{Prop : radii= implies Graph=}. 
Let us show how.
By the main result of \cite{NP-I,NP-II}, both the graphs 
$\Gamma_{S,i}(\Fs)$ and $\Gamma_{S,i}(\Fs^*)$ are locally finite 
and contain $\Gamma_S$.  
On the other hand, $X-\Gamma_S$ is a disjoint 
union of virtual open disks, which are by definition maximal disks (cf. 
Definition \ref{Maximal disksksks} and 
Remark \ref{Remark : D(x,S) cases}). This implies that if we remove 
from $X$ some of them, the radii functions do not change (cf. Proposition \ref{Prop : immersion}). More precisely, let $Y$ be the union of the 
connected components $D$ of $X-\Gamma_S$ such that 
$\Gamma_{S,i}(\Fs)\cap D\neq\Gamma_{S,i}(\Fs')\cap D$. 
Set $X':=X-Y$. By construction we have $\Gamma_S\subset X'$. 
Moreover, by locally finiteness of the graphs, there exists a 
weak-triangulation $S'$ of $X'$ containing $S$ 
such that $\Gamma_{S'}=\Gamma_S$ (we may simply add to $S$ the 
points at the boundary of the disks that we have removed from $X$).
Then, by Proposition  \ref{Prop : immersion}, for all $x\in X'$ we have 
\begin{equation}
\R_{S',i}(x,\Fs_{|X'})\;=\;\R_{S,i}(x,\Fs)
\;=\;
\R_{S',i}(x,\Fs_{|X'}^*)\;=\;\R_{S,i}(x,\Fs^*)
\end{equation}
and 
\begin{equation}
\Gamma_{S',i}(\Fs_{|X'})\;=\;\Gamma_{S,i}(\Fs)\cap X' \;=\;
\Gamma_{S',i}(\Fs_{|X'}^*)\;=\;\Gamma_{S,i}(\Fs^*)\cap X' \;.
\end{equation}
\end{remark}

\begin{proposition}\label{Prop. i sep F and F^* then R=R}
If the index $i$ separates the radii of $\Fs$ and of $\Fs^*$ 
(at each $x\in X$), then for all $x\in X$ one has 
$\R_{S,i}(x,\Fs)=\R_{S,i}(x,\Fs^*)$ and 
$\Gamma_{S,i}(\Fs)=\Gamma_{S,i}(\Fs^*)$. 
\end{proposition}
\begin{proof} 
By definition of controlling graphs, the equality 
$\Gamma_{S,i}(\Fs)=\Gamma_{S,i}(\Fs^*)$ follows from the equality 
of the radii. Let $x\in X$. 
It is enough to prove that $D_{S,i}(x,\Fs)=D_{S,i}(x,\Fs^*)$ (cf. \eqref{eq : D_S,i and D(x, R_S,i)}). 
If $\R_{S,i}(x,\Fs)$ or $\R_{S,i}(x,\Fs^*)$ is spectral non 
solvable, then they are equal by Proposition 
\ref{Prop : 1-th radius and dual} and the claim follows. 

Assume then that both radii are solvable or over-solvable at $x$. 
By contradiction, assume that $D_{S,i}(x,\Fs)\neq D_{S,i}(x,\Fs^*)$. 
Exchanging the roles of $\Fs$ and $\Fs^*$, without loss of generality 
we may assume  $D_{S,i}(x,\Fs)\subset D_{S,i}(x,\Fs^*)$
(i.e. $\R_{S,i}(x,\Fs)<\R_{S,i}(x,\Fs^*)$). 
We now prove that this is absurd.

Since this  inequality of disks is strict and $\R_{S,i}(x,\Fs)$ is not 
spectral non-solvable at $x$, it follows that $\R_{S,i}(x,\Fs^*)$ is over-solvable at $x$. Therefore, 
there exists a virtual open disk $D\subseteq X$ whose base change to 
$\Omega$ contains $D_{S,i}(x,\Fs^*)$ as a connected component.
By \eqref{eq : D^c}, the radius $\R_{S,i}(-,\Fs^*)$ is constant over 
$D$, and over-solvable at each point of it. 
By compatibility with duals in the spectral case (cf. Proposition \ref{dual}) this implies that 
$\R_{S,i}(z,\Fs)$ is solvable or over-solvable too, for all $z\in D$.
Moreover, $\R_{S,i}(-,\Fs)$ is not constant on $D$. 
\if{Indeed, if $\R_{S,i}(-,\Fs)$ were constant on $D$, 
we would have the inclusion of disks
$D_{S,i}(x,\Fs)\subsetneq D_{S,i}(x,\Fs^*)\subseteq D_{S,i}^c(x,\Fs)\subseteq D(x,S)$ (cf. \eqref{eq : D^c}).
}\fi
Indeed, the inequality $\R_{S,i}(x,\Fs)<\R_{S,i}(x,\Fs^*)$ implies that 
the value of $\R_{S,i}(x,\Fs)$ is strictly less than the ratio between the 
radii of $D_{S,i}(x,\Fs^*)$ and $D(x,S)$. In other words, if we chose a 
coordinate in which the radius of $D(x,S)$ is $1$, then the value of 
$\R_{S,i}(x,\Fs)$ must be strictly less than that of $D$ with respect to that 
coordinate. Hence, the constancy of $\R_{S,i}(-,\Fs)$ over $D$ 
would imply that when $z$ approaches the boundary of $D$ the radius 
$\R_{S,i}(z,\Fs)$ becomes spectral non solvable, which is not the case 
by assumption. 
Hence, we must have $\Gamma_{S,i}(\Fs)\cap D_{S,i}(x,\Fs^*)\neq\emptyset$.

Now, consider an end-point $z_0$ of $\Gamma_{S,i}(\Fs)$ in 
$D_{S,i}(x,\Fs^*)$. 
Since $\R_{S,i}(-,\Fs)=\R_{S,1}(-,\Fs_{\geq i})$,
then  $\R_{S,i}(-,\Fs)$ is a super-harmonic function on 
$D_{S,i}(x,\Fs^*)$ (cf. \cite[Theorems 3.9 and Proposition 4.7]{NP-I}). However, super-harmonicity 
is violated at $z_0$ because the function is constant on each open-disk 
with boundary $z_0$, and solvable along the segment $I$ connecting 
$z_0$ to the boundary of $D_{S,i}(x,\Fs^*)$. 
Indeed this implies that the slope of $\R_{S,1}(-,\Fs_{\geq i})$ 
along $I$ is $1$, while the slope along each other germ of segment 
out of $x$ is zero. Its super-harmonicity is therefore violated at $z_0$. 
Since  
$\R_{S,1}(-,\Fs_{\geq i})$ must be  super-harmonicity at $z_0$, this is 
absurd and the claim follows.
\end{proof}

\subsubsection{A criterion for the direct sum decomposition involving the dual.}

The following Theorem provide conditions on the radii of the dual 
differential equation that guarantee that the sequence \eqref{eq : sequence F>i --> F --> F<i} splits.

Recall that $X$ is connected (cf. Setting \ref{Hypothesis : X connected}).
\begin{theorem}\label{Thm : 5.7 deco good direct siummand}
Assume, as in Proposition \ref{Prop. i sep F and F^* then R=R}, 
that $i$ separates the radii of $\Fs$ and of 
$\Fs^*$. 
Assume moreover that we are in one of 
the following situations:
\begin{enumerate}
\item\label{Thm : 5.7 deco good direct siummand-i} $X$ is not a virtual disk with empty weak triangulation. 
\item\label{Thm : 5.7 deco good direct siummand-ii} $X$ is a virtual disk with empty weak triangulation, and there 
exists $x\in X$ such that one of the radii  
$\R_{\emptyset,i-1}(x,\Fs)$ or $\R_{\emptyset,i-1}(x,\Fs^*)$ 
is spectral non solvable.
\end{enumerate}
Then 
$\Fs_{\geq i}$ and $(\Fs^*)_{\geq i}$ are direct summands of 
$\Fs$ and $\Fs^*$ respectively.
\end{theorem}
\begin{proof}
It is enough to prove that the canonical composite morphism 
\begin{equation}\label{eq : can map iso}
c\;:\;(\Fs^*)_{\geq i}\;\to\;\Fs^*\;\to\;(\Fs_{\geq i})^*
\end{equation} 
is an isomorphism. 
This is a local matter : 
by Proposition \ref{Prop. : iso on a point implies iso global} it is
enough to show that $c$ is an isomorphism at an individual point $x$.
By Proposition \ref{Prop. i sep F and F^* then R=R} 
one has $\R_{S,i}(x,\Fs)=\R_{S,i}(x,\Fs^*)$ for all $x\in X$, 
so the two functions have 
the same controlling graphs $\Gamma :=
\Gamma_{S,i}(\Fs)=\Gamma_{S,i}(\Fs^*)$. 
If we are in the case \eqref{Thm : 5.7 deco good direct siummand-i}, then $\Gamma\neq\emptyset$, and if 
$x\in\Gamma$, the index $i$ is spectral at $x$, and hence $i-1$ is 
spectral non solvable at $x$ (because $i$ separates the radii). 
So $c$ is an isomorphism at $x$ by Proposition
\ref{dual} (cf. Section 
\ref{Local nature of spectral radii.}). 
The same holds in case \eqref{Thm : 5.7 deco good direct siummand-ii} by 
the assumption.
\end{proof}
\begin{remark}
The only pathological case which not contemplated by 
Theorem \ref{Thm : 5.7 deco good direct siummand} is that where $X$ is a virtual open disk, with empty triangulation, 
on which $\R_{S,i}(-,\Fs)$ and $\R_{S,i}(-,\Fs^*)$ 
are both the constant function with value $1$, and such that 
$\R_{S,i-1}(-,\Fs)$ and $\R_{S,i-1}(-,\Fs^*)$ are both solvable or 
over-solvable at all points of $X$. In this case if $i$ separates the radii 
we do not know if $\Fs_{\geq i}$ is a direct factor of $\Fs$. 
Both $(\Fs^*)_{\geq i}$ and $\Fs_{\geq i}$ are the maximum 
trivial submodules of $\Fs^*$ and $\Fs$ respectively and they share the 
same dimension, but $(\Fs_{<i})^*$ might have a non zero trivial 
submodule as in the example in Section \ref{optimality},
and we do not see why the morphism $c$ should be an isomorphism 
(cf. Remark \ref{Remark : iso trivial dual}). The whole problem comes 
from the fact that the functor associating to a differential module 
$\Fs$ on a disk $X$ its solutions on the whole disk 
$X$ is not exact (cf. Lemma \ref{Lemma :devisage}). 
 
However, as soon as one of $\R_{S,i}(-,\Fs)$ and $\R_{S,i}(-,\Fs^*)$ is 
not identically equal to $1$, 
then the index $i-1$ is spectral for $\Fs$, or for $\Fs^*$, 
at some point close to the boundary of the disk 
(because $i$ separates the radii), and the conditions of the 
Theorem \ref{Thm : 5.7 deco good direct siummand} are fulfilled.
\end{remark}

\begin{remark} Assume that $X$ is a virtual open disk.

If $S$ is non empty, and is $i$ separates the $S$-radii of $\Fs$, 
this implies that $\R_{S,i-1}(-,\Fs)$ is spectral non solvable at the 
points of  $S$ (at which $\R_{S,i}(-,\Fs)$ is not over-solvable). 
Therefore, the assumptions of item \eqref{Thm : 5.7 deco good direct siummand-ii} 
of Theorem \ref{Thm : 5.7 deco good direct siummand} are fulfilled. 

In particular, is interesting to notice that
condition \eqref{Thm : 5.7 deco good direct siummand-i} of Theorem 
\ref{Thm : 5.7 deco good direct siummand} is 
automatic if $S$ is a (non weak) triangulation of $X$. 
Indeed, the definition of triangulation 
of \cite{Duc} always prescribes $S\neq \emptyset$.\footnote{Besides, 
we notice that the case of a virtual open disk with empty 
weak-triangulation does not fit in the framework of the semi-stable 
reduction theorem.} 

We might therefore be temped to impose $S\neq\emptyset$ 
in order to satisfy the assumptions of Theorem 
\ref{Thm : 5.7 deco good direct siummand}. However, there are 
situations where $i$ separates the radii with respect to the empty 
triangulation, but not with respect to any non empty one. This is the case of the example in Section \ref{optimality} 
(cf. also Remarks \ref{Remark : changing tr F_S,I} and 
\ref{Remark : bad condition}). 
\end{remark}

The following Corollary shows how to use 
Proposition \ref{Prop : radii= implies Graph=} to control the 
assumptions of Theorem \ref{Thm : 5.7 deco good direct siummand}.
\begin{corollary}
Assume that 
\begin{enumerate}
\item $i$ separates the radii of $\Fs$,
\item $\Gamma_{S,i}(\Fs)=\Gamma_{S,i}(\Fs^*)$ and $\Gamma_{S,i-1}(\Fs)=\Gamma_{S,i-1}(\Fs^*)$. 
\end{enumerate}
Then, by Proposition \ref{Prop : radii= implies Graph=}, we have
$\R_{S,i}(-,\Fs)=\R_{S,i}(-,\Fs^*)$ and 
$\R_{S,i-1}(-,\Fs)=\R_{S,i-1}(-,\Fs^*)$. In particular, the index 
$i$ separates also the radii of $\Fs^*$. 

Assume now that 
we are not in the situation where $X$ is an open disk with 
empty weak triangulation and both $\R_{S,i-1}(-,\Fs)$ and 
$\R_{S,i-1}(-,\Fs^*)$ are solvable or over-solvable at each point of 
the disk $X$. Then, $\Fs$ satisfy at least one of the conditions \eqref{Thm : 5.7 deco good direct siummand-i} or \eqref{Thm : 5.7 deco good direct siummand-ii} of Theorem \ref{Thm : 5.7 deco good direct siummand}, and the sub-objects $\Fs_{\geq i}$ and 
$(\Fs^*)_{\geq i}$ are direct summands of $\Fs$ and $\Fs^*$
respectively. \hfill$\qed$
\end{corollary}

\subsubsection{A topological criterion for the direct sum decomposition  
not involving the dual.}
We now provide a criterion of topological nature on the controlling 
graphs of $\Fs$ that guarantee that sequence \eqref{eq : sequence F>i --> F --> F<i} splits. 
The criterion involves only properties of $\Fs$ and not of its dual 
(cf. Theorem \ref{Thm : criterion self direct sum}). 

\if{MODIFICARE :

The hearth of that criterion is the following 
Proposition \ref{THM : Decomposition on a disk}. Its proof is based 
on the Grothendieck-Ogg-Shafarevich (G.O.S.) formula  
in the form expressed in 
\eqref{GOS-changed}. We apply the G.O.S. formula to a 
differential module over an open disk having all the radii equal and 
constant as functions on the disk. 
So the G.O.S. 
formula does not need any decomposition theorem in that case, 
and the proof of Proposition 
\ref{THM : Decomposition on a disk} is then not circular. 
Note moreover that Proposition \ref{THM : Decomposition on a disk} 
follows from the more general\footnote{As 
explained in section
\ref{Rk : theoreme de deco de Kedlaya sur un disque...} the 
decomposition theorem 
\cite[12.4.1]{Kedlaya-book-2} is more general because it does not 
assume \emph{a priori} that the radii are separated.} 
 result 
\cite[Thm. 12.4.1]{Kedlaya-book-2},
if $D$ is $K$-rational, 
and by a Galois descent from 
$\widehat{K^{\mathrm{alg}}}$ to $K$ otherwise.
}\fi

\begin{lemma}\label{Lemma : const at the wedge implies globally constant-1}
Let $\Fs$ be a differential equation over the open unit disk  $D^-(0,1)$ 
endowed with the empty weak-triangulation. 
Assume that there exist $R,R'<1$ and an index $i$ such that 
for all $x\in]x_{0,R'},x_{0,1}[$ and all $k=1,\ldots,i$ 
we have $\R_{\emptyset,k}(x,\Fs)=R$. Then, 
the equality $\R_{\emptyset,k}(x,\Fs)=R$ holds for all $x\in D$ and all 
$k=1,\ldots,i$.
\end{lemma}
\begin{proof} 
The first radius function $\R_{\emptyset,1}(-,\Fs)$ is super-harmonic 
on $D$. In particular, it is concave along every segment $[c,x_{0,1}[$ 
for every $L$-rational point $c\in D_L$ of the disk belonging to an 
unspecified field extension $L$ of $K$ (cf. \cite[Theorem 3.9]{NP-I}). 
Using the fact that the radii are insensitive to scalar extension, we 
deduce that $\R_{\emptyset,1}(-,\Fs)$ is a constant function on the 
whole $D$. 

We claim that the equality $\R_{\emptyset,i}(x,\Fs)=R$ must hold for 
all $x\in]x_{c,R},X_{0,1}[$. Indeed, let 
$R_*\in [R,1]$ be the largest real number such that the radius 
$\R_{\emptyset,i}(-,\Fs)$ equals $R$. We know that the partial height 
$H_{\emptyset,i}(-,\Fs)$ is concave on $]x_{c,R},x_{0,1}[$. 
If $R_*>R$, its concavity is violated at $x_{c,R_*}$ because the radii 
are all constant on one direction and bounded below by the constant 
function $\R_{\emptyset,1}(x,\Fs)=R$ on the other direction. 
It follows that $R_*=R$ and that $\R_{\emptyset,i}(x,\Fs)=R$ for all $x\in]x_{c,R},x_{0,1}[$.

On the other hand, over the segment $[c,x_{c,R}[$ the radius
$\R_{\emptyset,i}(x,\Fs)$ is over-solvable, hence constant. By 
continuity, we must have $\R_{\emptyset,i}(x,\Fs)=R$ for all 
$x\in[c,x_{0,1}[$. As above, this implies that function 
$\R_{\emptyset,i}(-,\Fs)$ is constant on the whole $D$ and the claim 
follows.
\end{proof}

\begin{proposition}\label{THM : Decomposition on a disk}
Let $X=D$ be a virtual open disk with empty triangulation. 
Let $\Fs$ be a differential equation on $D$ such that, 
for all $k=1,\ldots,i$, the radii
$\R_{\emptyset,k}(-,\Fs)$ are constants functions on $D$.
Let $j\in\{1,\ldots,i\}$ be an index separating the radii of $\Fs$. 
Then, the index $j$ separates also the radii of $\Fs^*$ 
(at every $x\in D$). Moreover, the 
assumptions of Proposition \ref{Prop. i sep F and F^* then R=R} and of item 
\eqref{Thm : 5.7 deco good direct siummand-ii} of 
Theorem \ref{Thm : 5.7 deco good direct siummand} are fulfilled. 
In particular 
\begin{enumerate}
\item\label{THM : Decomposition on a disk-i} $(\Fs_{\geq j},\nabla_{\geq j})$ is a 
direct summand of $(\Fs,\nabla)$;
\item\label{THM : Decomposition on a disk-ii} $\R_{\emptyset,k}(x,\Fs)=\R_{\emptyset,k}(x,\Fs^*)$ for all 
$k=1,\ldots,i$, and all $x\in D$.
\end{enumerate}

\end{proposition}
\begin{proof}
By definition the radii are insensitive to scalar extension of the ground 
field $K$, therefore without loss of generality we may assume that 
$D$ is isomorphic to the open unit disk $D^-(0,1)$. 
We proceed by induction on the rank of $\Fs$. If the rank is $1$, there 
is nothing to show (cf. Proposition \ref{Prop : 1-th radius and dual}).
Assume the assertion proved for all differential equations 
whose rank is strictly less than that of $\Fs$.

\emph{Step 1}. We proceed by induction on $i$. If $i=1$ we have 
$\Fs_{\geq 1}=\Fs$, therefore 
item \eqref{THM : Decomposition on a disk-i} holds (the same is true 
for $\Fs^*$). 
Moreover, by Proposition \ref{Prop : 1-th radius and dual} we 
always have $\R_{\emptyset,1}(x,\Fs)=\R_{\emptyset,1}(x,\Fs^*)$ for all $x\in D$ 
and item \eqref{THM : Decomposition on a disk-ii} holds too. The claim 
is proved for $i=1$.

\if{

If $R=1$, then $\Fs$ is trivial on $D$ and hence $\Fs^*$ is trivial too 
and the function $\R_{\emptyset,1}(-,\Fs^*)$ is constant with value 
$1$. 

If $R<1$, then the function $\R_{\emptyset,1}(-,\Fs)$ is spectral non 
solvable along the segment $]x_{0,R},x_{0,1}[$. Therefore, 
for all $x$ in this segment we have 
$\R_{\emptyset,1}(x,\Fs)=\R_{\emptyset,1}(x,\Fs^*)$ (cf. \eqref{Prop : 1-th radius and dual}).

concave and decreasing (cf. 
\cite[Items (i) and (iii) of Theorem 3.9 and Proposition 2.14]{NP-I}).

...

\cite{Potentiel}.
}\fi

\emph{Step 2.} Assume $i>1$ and that for all $k=1,\ldots,i$ and all 
$x\in D$ we have $\R_{\emptyset,k}(x,\Fs)=\R_{\emptyset,1}(x,\Fs)$. In this case, we necessarily have $j=1$, and the only difference with \emph{Step 1} is item \eqref{THM : Decomposition on a disk-ii}. 
Proposition \ref{Prop : 1-th radius and dual} 
guarantees that 
$\R_{\emptyset,1}(-,\Fs)=\R_{\emptyset,1}(-,\Fs^*)$. Therefore, 
we need to show that for all $k=1,\ldots,i$ and all 
$x\in D$ we have $\R_{\emptyset,k}(x,\Fs^*)=\R_{\emptyset,1}(x,\Fs^*)$.

We know that 
$\R_{\emptyset,1}(-,\Fs)=\R_{\emptyset,1}(-,\Fs^*)$
is a constant function on $D$. 
Let $R:=\R_{\emptyset,1}(-,\Fs)$ be its value. 
If $R=1$, then $\Fs$ is 
trivial over $D=D^-(0,1)$ and so is $\Fs^*$, in this case there is nothing to 
prove. If $R<1$, then for all $x\in]x_{0,R},x_{0,1}[$ the radii 
$\R_{\emptyset,k}(x,\Fs)$ are spectral non solvable for all 
$k=1,\ldots,i$. Hence \eqref{THM : Decomposition on a disk-ii} 
holds over $]x_{0,R},x_{0,1}[$, by Proposition 
\ref{Prop : 1-th radius and dual}. In particular, the 
radii $\R_{\emptyset,k}(x,\Fs^*)$ are constant on 
$]x_{0,R},x_{0,1}[$ with value $R$. By Lemma \ref{Lemma : const at the wedge implies globally constant-1}, 
they are constant on the whole $D$.
\if{
Now, for all $k=1,\ldots,i$, one has by definition 
$\R_{\emptyset,k}(x,\Fs^*)\geq
\R_{\emptyset,1}(x,\Fs^*)=\R_{\emptyset,1}(x,\Fs)$, 
therefore for all  $x\in[0,x_{0,R}[$ the 
radii $\R_{\emptyset,k}(x,\Fs^*)$ are over-solvable (because 
$\R_{\emptyset,1}(x,\Fs)$ is so). In particular, the 
radii $\R_{\emptyset,k}(x,\Fs^*)$ are also constant on $[0,x_{0,R}[$. 
By continuity, they are constant on the whole segment $[0,x_{0,1}[$ 
with value $R$. 
Arguing similarly for every segment $[c,x_{0,1}[$, where $c\in D_L$ 
is an $L$-rational point of the disk belonging to an unspecified 
field extension $L$ of $K$, and using the fact that the radii are insensitive to scalar extension, we deduce that 
for all $k=1,\ldots,i$, one has 
$\R_{\emptyset,k}(x,\Fs^*)=
\R_{\emptyset,1}(x,\Fs^*)$.}\fi 
The Proposition is proved in this case.

\emph{Step 3}. 
Maintain the assumption $i>1$ and assume now that the radii 
$\R_{\emptyset,k}(-,\Fs)$ are not all equal to 
$\R_{\emptyset,1}(-,\Fs)$, which is a constant function with value $R$. 
In particular, notice that in this situation $\Fs$ is not trivial over $D$.
Let $i_1>1$ be the smallest index separating the radii of $\Fs$. Then, 
for all $k=1,\ldots,i_1-1$ the radii $\R_{\emptyset,k}(-,\Fs)$ are 
all constant functions equal to $\R_{\emptyset,1}(-,\Fs)=R$. 
Moreover, we must have $R<1$, because otherwise $\Fs$ is trivial. 
In particular, the radii $\R_{\emptyset,k}(-,\Fs)$ are 
all spectral non solvable over the segment $]x_{0,R},x_{0,1}[$.
By \emph{Step 2}, for all $k=1,\ldots,i_1-1$ and all $x\in D$ 
we have $\R_{\emptyset,k}(x,\Fs)=\R_{\emptyset,k}(x,\Fs^*)=R$. 

Now, by \cite[Theorem 12.4.1]{Kedlaya-book-2} and Proposition 
\ref{Prop. Exact sequence prop separate radii in the right order}, 
$\Fs$ factorizes as $\Fs=\Fs'\oplus\Fs''$ where for all 
$x\in]x_{0,R},x_{0,1}[$ and all $k=1,\ldots,i_1-1$ and 
$k'=1,\ldots,\mathrm{rank}(\Fs)-i_1+1$ we have
\begin{equation}
\R_{\emptyset,k}(x,\Fs')\;=\;
\R_{\emptyset,k}(x,\Fs)\;,\qquad
\R_{\emptyset,k'}(x,\Fs'')\;=\;\R_{\emptyset,i_1-1+k'}(x,\Fs)\;.
\end{equation}
Now, by Remark \ref{Remark : an explicit quote to counterexample}, 
the radii can only increase by morphisms, therefore the image of 
$\Fs_{\geq i_1}\subset\Fs$ by the projection $\Fs\to \Fs'$ is zero. 
This implies that $\Fs_{\geq i_1}\subset\Fs''$ and hence they are equal 
because they have the same rank. This proves that $\Fs_{\geq i_1}$ 
is a direct summand of $\Fs$ and therefore that 
$\Fs_{<i_1}\cong\Fs'$.

By \emph{Step 2} and Lemma 
\ref{Lemma : const at the wedge implies globally constant-1}, for all 
$k=1,\ldots,i_1-1$ and all $x\in D$ 
we have $\R_{\emptyset,k}(x,\Fs_{<i_1})=
\R_{\emptyset,k}(x,\Fs_{<i_1}^*)=R$. The claim is then proved for $\Fs_{<i_1}$.

\emph{Step 4}. By Theorem \ref{MAIN Theorem}, for all 
$k=1,\ldots,\mathrm{rank}(\Fs)-i_1+1$ and all $x\in D$ 
we have $\R_{\emptyset,k}(x,\Fs_{\geq i_1})=
\R_{\emptyset,i_1+k-1}(x,\Fs)>R$.
By induction on the rank, the statement holds for 
$\Fs_{\geq i_1}$.

The claim follows.
\end{proof}

\if{

...

Let us first assume that
$i_1$ separates also the radii of $\Fs^*$ (at all $x\in D$) and prove 
the other assertions. Since $i_1$ separates the radii of $\Fs$,  
the radius $\R_{\emptyset,i_1-1}(-,\Fs)$ is a constant function 
with value $<1$. Therefore, if $x\in D$ is close enough to the open 
boundary of $D$, the index $i_1-1$ is spectral non solvable at $x$.
Hence, Theorem
\ref{Thm : 5.7 deco good direct siummand} applies and
$\Fs_{\geq i_1}\subset\Fs$ is a direct summand of $\Fs$
by item \eqref{Thm : 5.7 deco good direct siummand-ii} 
of Theorem \ref{Thm : 5.7 deco good direct siummand}. 
Moreover, for all $j\leq i_1$ we must prove 
\eqref{THM : Decomposition on a disk-ii}.

 Finally by 
\eqref{eq : equality of radii of F_Si} 
the induction is evident replacing $\Fs$ by $\Fs_{\geq i_1}$.

It remains to prove that $i_1$ separates the radii of $\Fs^*$.
We start by looking to the radii of $(\Fs_{<i_1})^*$. 
By \eqref{eq : equality of radii of F_Si} 
the radii of $\Fs_{<i_1}$ are 
constant functions on $D$ all equal to $R:=\R_{\emptyset,1}(-,\Fs)$. 
Now the same is true for its dual: 
\begin{lemma}\label{Lemma : dualiti if constant on D}
For all $j=1,\ldots,i_1-1=\mathrm{rank}(\Fs_{<i_1})$, 
and all $x\in D$, one has 
\begin{equation}\label{eq : comp dual D D}
\R_{\emptyset,j}(x,(\Fs_{<i_1})^*)\;=\;
\R_{\emptyset,j}(x,\Fs_{<i_1})\;=\;R\;<\;1\;.
\end{equation}  
\end{lemma}
\begin{proof}
Since the radii are insensitive to scalar extension of $K$, we may assume 
without loss of generality that there is 
an isomorphism $D\simto D^-(0,1)$. 
Since $i_1$ separates the radii of $\Fs$, we have 
$R=\R_{\emptyset,1}(-,\Fs)<1$. 
So the indexes $1,\ldots,i_1-1$ are all spectral 
non solvable for $\Fs_{<i_1}$ at each point of the open 
segment $]x_{0,R},x_{0,1}[$. Hence compatibility with the dual 
holds over $]x_{0,R},x_{0,1}[$ by Proposition \ref{dual}. 
Moreover, \eqref{eq : comp dual D D} holds for $j=1$ by Proposition 
\ref{Prop : 1-th radius and dual}, 
hence $\R_{\emptyset,j}(0,\Fs^*)\geq\R_{\emptyset,1}(0,\Fs)=R$. 
Since $\R_{\emptyset,j}(-,(\Fs_{<i_1})^*)$ is constant on 
$[0,x_{0,\R_{\emptyset,j}(0,(\Fs_{<i_1})^*)}[$, 
that contains $[0,x_{0,R}[$, 
this forces $\R_{\emptyset,j}(-,(\Fs_{<i_1})^*)$ 
to be constant on $[0,x_{0,1}[$ with value $R$. 
This is independent on the chosen isomorphism $D\simto D^-(0,1)$, 
so $\R_{\emptyset,j}(-,(\Fs_{<i_1})^*)$ is constant on 
each segment $[z,x_{0,R}[$ of $D$, with $z$ a rational point. 
If $K$ is spherically complete this proves the claim. In general, the 
claim over $K$ is deduced from that over a spherically complete 
extension $\Omega/K$, since the radii are insensitive to scalar 
extension of $K$.
This proves that $\R_{\emptyset,j}(-,(\Fs_{<i_1})^*)$ is constant 
over $D$. 
\end{proof}

In  addition to Lemma \ref{Lemma : dualiti if constant on D}, 
for all $x\in D$, and 
all $j\geq 1$, we have 
\begin{equation}\label{eq : rpopoliuh}
\R_{\emptyset,j}(x,\Fs_{\geq i_1})\;=\;
\R_{\emptyset,j+i_1-1}(x,\Fs)\;.
\end{equation}
Moreover by Prop. \ref{Prop : 1-th radius and dual}, the first radius 
$\R_{\emptyset,1}(-,(\Fs_{\geq i_1})^*)$ is a constant 
function on $D$ and equal to 
$\R_{\emptyset,1}(-,\Fs_{\geq i_1})$.

Unfortunately this is not enough to guarantee that $i_1$ separates the 
radii of $\Fs^*$ because $(\Fs_{\geq i})^*$ is a quotient of 
$\Fs^*$, and the radii of the latter can be different. One has to prove 
that one does not have a pathology as in 
\eqref{eq : picture of radii of p^*}. For this  
Lemma \ref{Lemma : A1} 
below proves that locally at each point $\Fs$ is the direct sum of 
$\Fs_{\geq i_1}$ and $\Fs_{<i}$, so that we have compatibility with duals by 
Prop. \ref{prop : radii of the direct sum}.

Lemma \ref{Lemma : A1} is based on the G.O.S. 
formula \eqref{GOS-changed}.
Since the radii are insensitive to scalar extension of $K$, 
we can assume that $K$ is spherically complete. This guarantee that 
the G.O.S. formula holds. 
\begin{lemma}\label{Lemma : A1}
For all $x\in D$ there exists a spherically complete field extension 
$\Omega_x/K$, 
and a star-shaped open neighborhood $Y_x$ of $x$ in 
$D$ such that 
\begin{enumerate}
\item $Y_x$ contains strictly $D_{\emptyset,i_1-1}(x,\Fs)$. 
In particular one has $\R_{S_{Y_x},i_1-1}(x,\Fs_{|Y_x})<
\R_{S_{Y_x},i_1}(x,\Fs_{|Y_x})$ (the radii remains separated after 
localization).\footnote{Here $S_{Y_x}$ denotes the canonical 
weak triangulation of $Y_x$  (cf. Definition~\ref{def:starshapedopen}),
which is the empty set if $Y_x$ is a virtual open disk.} 
Moreover if $i_1-1$ is solvable or over-solvable at $x$, then we may chose 
$Y_x$ to be a virtual open disk $Y_x=D_x$ such that 
\begin{equation}\label{eq : inequality preserved by loc rty}
D_{\emptyset,i_1-1}(x,\Fs)\;\subset\; D_x\;\subseteq\; 
D_{\emptyset,i_1}(x,\Fs)\;.
\end{equation}
\item The restriction to 
$Y_x\widehat{\otimes}_K\Omega_x$ of the sequence 
\begin{equation}\label{eq : sequence ufsdq }
E\;:\;0\to\Fs_{\geq i_1}\to\Fs\to\Fs_{<i_1}\to 
0\;,
\end{equation}
splits over $Y_x\widehat{\otimes}_K\Omega_x$.
\end{enumerate}
\end{lemma}
\begin{proof}
If $i_1-1$ is spectral non solvable at $x$ the statement follows from 
Dwork-Robba Theorem \ref{Dw-Robba}. 
Assume that $i_1-1$ is solvable or over-solvable at $x$, so that 
$D(x)\subseteq D_{\emptyset,i_1-1}(x,\Fs)$. Note that 
$D_{\emptyset,i_1-1}(x,\Fs)$ is not equal to $D$ because $i_1$ 
separates the radii.
\if{
We claim that 
there exists a disk $D_x$ with strict inclusions
$D_{\emptyset,i_1-1}(x,\Fs)\subset D_x \subset 
D_{\emptyset,i_1}(x,\Fs)$ 
such that \eqref{eq : sequence ufsdq } splits over $D_x$. 
}\fi
Now let $\Omega_x/K$ be a field extension 
such that the connected components of 
$D_{\emptyset,i_1-1}(x,\Fs)\otimes\Omega_x$ 
are isomorphic to $D':=D^-_{\Omega_x}(0,1)$. 
We consider the sequence $E\otimes\O^\dag(D')$, 
and we use the notations of section 
\ref{A remark on the Grothendieck-Ogg-Shafarevich formula}. We now 
prove that the sequence splits over $D'_\varepsilon$ for a convenient 
$\varepsilon>0$. Since by assumption $D(x)\subseteq D'$, then 
$D'_\varepsilon$ comes by scalar extension from a virtual disk 
$D_x:=D'_{\varepsilon,K}$. 

By \eqref{Localization of sol-1}, 
for all $\varepsilon>0$ the radii of 
$(\Fs_{<i_1})_{|D_\varepsilon'}$ and 
$(\Fs_{\geq i_1})_{|D_{\varepsilon}'}$ are constant over 
$D'_\varepsilon$, and the index $i_1$ separates the radii of 
$\Fs_{|D_\varepsilon'}$. 
In particular  $(\Fs_{<i_1})_{|D_\varepsilon'}$ 
does not verify \eqref{eq : negative break of R_i NL}, so the 
Grothendieck-Ogg-Shafarevich \eqref{GOS-changed} 
formula holds.
Hence there exists a virtual open disk $D_{\varepsilon}'$ as 
above such that $\hdr^1((\Fs_{<i_1})|_{D_\varepsilon'},
\O(D_\varepsilon'))=0$ since the slope of 
$-\partial_bH_{\emptyset,i_1-1}(x_{0,1},
(\Fs_{<i_1})|_{D_\varepsilon'})$ is zero. 

\comment{Ce sont des modules NON solubles !!!
Il se passe que c'est vrai quand m?me, mais il faut 
utiliser le th?or?me d'indice sur un disque de Christol-
Mebkhout, et ce dernier utilise la d?composition en 
somme directe sur une couronne et rayons log-affines 
... ?  ce moment la peut-?tre qu'il vaut mieux de 
tronquer la chose et appliquer imm?diatement le 
th?or?me de d?composition de Kedlaya sur un disque ?}
The same happens for the dual 
$((\Fs_{<i_1})|_{D_\varepsilon'})^*$ 
by Lemma \ref{Lemma : dualiti if constant on D}. 
Now, since $(\Fs_{\geq i_1})_{|D'_\varepsilon}$ is trivial, 
then by \eqref{GOS-changed}, one 
has 
\begin{equation}
\Hdr^1(((\Fs_{<i})_{|D'_\varepsilon})^*\otimes 
(\Fs_{\geq i_1})_{|D'_\varepsilon}) 
\;=\; 
\Hdr^1(((\Fs_{<i})_{|D'_\varepsilon})^*)^{r-i_1+1}\;=\;0\;.
\end{equation}
 Now the Yoneda group $\mathrm{Ext}^1(\M,\N)$ 
of extensions $0\to \N\to \P \to \M\to 0$ of differential modules 
can be identified with $\Hdr^1(\M^*\otimes \N)$ 
(cf. Lemma \ref{Lemma : Ext^1=H^1}). So the sequence splits over 
$D_\varepsilon'$.
\end{proof}
The behavior of the radii by localization to $Y_x$ is expressed by 
\eqref{Localization of sol}. 
To show that $i_1$ separates the radii of $\Fs^*$, we then 
prove that $i_1$ separates the radii of $(\Fs_{|Y_x})^*$. 
Since the sequence \eqref{eq : sequence ufsdq } splits over 
$Y_x\widehat{\otimes}_K\Omega_x$, 
then so does the dual sequence 
$0\to (\Fs_{<i_1})^* \to \Fs^*\to 
(\Fs_{\geq i_1})^*\to 0$. 
Now, by Lemma \ref{Lemma : dualiti if constant on D}, the radii of 
$(\Fs_{<i_1})^*$ are all equal to $R=\R_{\emptyset,i_1-1}(-,\Fs)$. 
Hence, by \eqref{eq : rpopoliuh}, 
they are strictly smaller than those of $(\Fs_{\geq i_1})^*$. 
This inequality is preserved by restriction to 
$Y_x\widehat{\otimes}_K\Omega_x$ 
by \eqref{eq : inequality preserved by loc rty}.
We then apply Proposition \ref{prop : radii of the direct sum} 
to prove that the radii of 
$\Fs^*|_{Y_x\widehat{\otimes}_K\Omega_x}$ are 
the union of these two families of radii, hence $i_1$ separates the radii of 
$\Fs^*|_{Y_x\widehat{\otimes}_K\Omega_x}$. 
So the same holds for $\Fs^*$ by \eqref{Localization of sol-1}. 
This concludes the proof of Proposition 
\ref{THM : Decomposition on a disk}.
\if{
...

...

...

We now prove that the sequence splits globally on $X=D$ over $K$.
As a consequence of Lemma \ref{Lemma : A1} one has 
$\R_{\emptyset,i_1}(x,\Fs)=\R_{\emptyset,i_1}(x,
\Fs^*)>\R_{\emptyset,i_1-1}(x,\Fs^*)=\R_{\emptyset,i_1-1}(x,\Fs)$ 
for all $x\in X$. 
Indeed this is true if the $\R_{\emptyset,i_1-1}(x,\Fs^*)$ is 
spectral non solvable (cf. Prop. \ref{Prop : 1-th radius and dual}), and 
in the other cases one uses Lemma \ref{Lemma : A1} together with  
Prop. \ref{prop : radii of the direct sum} and 
\eqref{Localization of sol}.\footnote{The radii are invariant by scalar 
extension of $K$, so we can assume $K$ spherically complete to fulfill 
the assumptions of Lemma \ref{Lemma : A1}.}
So by the first part of the proof of Theorem \ref{MAIN Theorem} 
one has a sequence 
\begin{equation}\label{eq : sequence ufsdq -2}
0\to(\Fs^*)_{\geq i_1}\to\Fs^*
\to(\Fs^*)_{<i_1}\to 0\;.
\end{equation}
It is then enough to check that  
the composite morphism 
$(\Fs^*)_{\geq i_1}\to\Fs^*
\to(\Fs_{\geq i_1})^*$ is an isomorphism. 
One can check this locally around each point. 
By Dwork-Robba 
Thm. \ref{Dw-Robba} this is true if $x$ has type $2$, $3$, or $4$, so 
we can assume that $x$ is of type $1$. 
By Lemma \ref{Lemma : A1} we extend the scalars to a spherically 
complete field $\Omega/K$ in order to find a $\Omega$-rational 
disk $D_{x,\Omega}$ with  
strict inclusions $D_{\emptyset,i_1-1}(x,\Fs)\subset D_{x,\Omega} 
\subset D_{\emptyset,i_1}(x,\Fs)$ on which 
\eqref{eq : sequence ufsdq } splits. Such a disk comes by scalar 
extension from a virtual disk $D_x$ over $K$, because 
it contains $x$ which is a point of type $1$.
Let $det(T)$ be the determinant of the canonical morphism
$(\Fs^*)_{\geq i_1}\to\Fs^*
\to(\Fs_{\geq i_1})^*$ over $D_x$. 
By Prop. \ref{prop : radii of the direct sum} this is an 
isomorphism over $D_{x,\Omega}$. 
So the image of $det(T)$ in $\O(D_{x,\Omega})$ is invertible. 
Then it is a function without zeros. Hence $det(T)$ is
invertible as a function on $\O(D_x)$.
\if{This property descends to $K$ by faithfully flatness. Namely
such a morphism can be identified to a solution of 
$((\Fs^*)_{\geq i_1})^*\otimes 
(\Fs^*)_{<i_1}$. Such solutions descends to $K$ by 
\cite[6.9.1]{Kedlaya-book-2}.
}\fi
\if{
In this case we consider the virtual disk $D_{\emptyset,i_1}(x,\Fs)$ 
as a neighborhood of $x$. Both $(\Fs^*)_{\geq i_1}$ and 
$(\Fs_{\geq i_1})^*$ are trivial over 
$D_{\emptyset,i_1}(x,\Fs)$.

$(\Fs^*)_{\geq i_1}$ is trivial over 
$D_{\emptyset,i_1}(x,\emptyset)$. 

Moreover
we can assume that $K$ is algebraically closed and spherically 
complete, by faithfully flatness of scalar extensions of $K$.
So the claim follows from  \eqref{eq : sequence ufsdq }, 
and again by Prop. \ref{prop : radii of the direct sum}. 
This proves that $\Fs_{<i_1}$ is a direct factor. 
}\fi
The proof goes identically by induction for the other 
indexes separating the radii.
}\fi
\end{proof}
}\fi

The assumptions of the following result only involve the properties of 
$\Fs$. A possible criterion to guarantee condition 
\eqref{eq : condition on graphs included} is discussed in 
\cite[Section 2]{NP-IV} where we provide an operative description of 
the controlling graphs.

\begin{theorem}\label{Thm : criterion self direct sum}
Assume that $i$ separates the radii of $\Fs$, and that 
\begin{equation}\label{eq : condition on graphs included}
\Gamma_{S,1}(\Fs)\cup\cdots\cup\Gamma_{S,i-1}(\Fs)
\;\subseteq\; \Gamma_{S,i}(\Fs)\;.
\end{equation} 
Then 
\begin{enumerate}
\item\label{Thm : criterion self direct sum-i} The index $i$ separates the radii of $\Fs^*$ and all the 
assumptions of Theorem \ref{Thm : 5.7 deco good direct siummand} 
are fulfilled.
In particular $(\Fs_{\geq i},\nabla_{\geq i})$ is a direct summand of 
$(\Fs,\nabla)$. 
\item\label{Thm : criterion self direct sum-ii} For all $j=1,\ldots,i$ and all $x\in X$ 
one has $\R_{S,j}(x,\Fs)=\R_{S,j}(x,\Fs^*)$, hence also
\begin{equation}
\Gamma_{S,j}(\Fs)\;=\;\Gamma_{S,j}(\Fs^*)\;.
\end{equation}
In particular 
$\Gamma_{S,1}(\Fs^*)\cup\cdots\cup\Gamma_{S,i-1}(\Fs^*)
\;\subseteq\; \Gamma_{S,i}(\Fs^*)$.
\end{enumerate}
\end{theorem}
\begin{proof}
If $\Gamma_{S,i}(\Fs)=\emptyset$, then $X$ is a virtual open disk 
with empty triangulation, and the claim reduces to 
Proposition \ref{THM : Decomposition on a disk}.

Assume now that $\Gamma_{S,i}(\Fs)\neq \emptyset$. Then 
$X-\Gamma_{S,i}(\Fs)$ is a disjoint union of virtual open disks. 
We shall prove that the index $i$ separates the radii of $\Fs^*$, and apply 
Theorem \ref{Thm : 5.7 deco good direct siummand}.

Firstly observe that $i$ separates the radii of $\Fs^*$ at the points of 
$\Gamma_{S,i}(\Fs)$. Indeed if $x\in\Gamma_{S,i}(\Fs)$, then the radius 
$\R_{S,i}(x,\Fs)$ is spectral at $x$ 
(cf. Remark \ref{Remark : R_i solvable on Gamma_i}). 
Hence $\R_{S,1}(x,\Fs),\ldots,\R_{S,i-1}(x,\Fs)$ are all 
spectral and \emph{non solvable} at $x$, because $i$ separates the 
radii. So by compatibility with duality in the spectral case 
(cf. 
Prop. \ref{dual}) 
one has
\begin{equation}\label{eq : radii fr equel jh}
\R_{S,j}(x,\Fs)\;=\;\R_{S,j}(x,\Fs^*)\;,\textrm{ for all }j=1,\ldots,i,
\textrm{ and all }x\in\Gamma_{S,i}(\Fs)\;.
\end{equation}
Let now $D$ be a virtual open disk in $X$ with boundary $x\in 
\Gamma_{S,i}(\Fs)$.  
The assumption \eqref{eq : condition on graphs included} implies that 
the radii $\R_{S,1}(-,\Fs),\ldots,\R_{S,i}(-,\Fs)$ are all constant on $D$. 
Now consider the empty weak triangulation on $D$.
The rule \eqref{Localization of sol-3} proves that the radii 
$\R_{\emptyset,1}(-,\Fs_{|D}),\ldots,\R_{\emptyset,i}(-,\Fs_{|D})$ 
are all constants and that 
$\R_{\emptyset,i-1}(-,\Fs_{|D})<\R_{\emptyset,i}(-,\Fs_{|D})$ 
(i.e. the radii remains separated after localization). 
Indeed, the radii are not truncated by localization to $D$.
More precisely, if $z\in D$, we have 
$D= D_{S,i}^c(z,\Fs)$, and hence 
$D_{S,i}(z,\Fs)\subseteq D$ by \eqref{eq : D^c}. Therefore, for all $j=1,\ldots,i-1$ 
one also has 
\begin{equation}\label{eq : bobobobobez}
D_{S,j}(z,\Fs)\;\subset\; D_{S,i}(z,\Fs)\;\subseteq\; D\;.
\end{equation}
Now, by Proposition \ref{THM : Decomposition on a disk} the first $i$ 
radii of $(\Fs_{|D})^*$ are constant functions on $D$ equal to those 
of $\Fs_{|D}$. 
We shall now deduce from this the same equality between the 
radii before localization to $D$. 
The equality of the radii over $D$ shows that the inclusion 
$D_{\emptyset,j}(z,\Fs^*_{|D})\;\subset\; D$ is strict. 
Then, we deduce from \eqref{Localization of sol-3}, 
that for all $j=1,\ldots,i-1$ one has $D_{S,j}(z,\Fs^*)=D_{S,j}(z,\Fs)\subset D$. 
Hence, the radii $\R_{S,1}(-,\Fs^*),\ldots,\R_{S,i-1}(-,\Fs^*)$ (before localization) are constant 
functions on $D$, and are equal to $\R_{S,1}(-,\Fs),\ldots,\R_{S,i-1}(-,\Fs)$ respectively. Together with \eqref{eq : radii fr equel jh} this gives 
$\R_{S,j}(-,\Fs^*)=\R_{S,j}(-,\Fs)$ over the whole $X$ for all $j=1,\ldots,i-1$.

Now, \eqref{Localization of sol-3} does not give information about the 
exact value of $\R_{S,i}(-,\Fs^*)$ because this radius might be 
truncated by localization to $D$. 
However, since we have a strict inclusion
$D_{\emptyset,i-1}(z,\Fs^*_{|D})\subset 
D_{\emptyset,i}(z,\Fs^*_{|D})$, then 
\eqref{Localization of sol-3} implies that 
the inclusion $D_{S,i-1}(z,\Fs^*)\subset D_{S,i}(z,\Fs^*)$ is strict. 
In other words, $i$ separates the $S$-radii of 
$\Fs^*$ at all $z\in D$ (before localization). 
By \eqref{eq : radii fr equel jh}, the index $i$ separates the radii  
of $\Fs^*$ on the whole $X$. 
By Proposition \ref{Prop. i sep F and F^* then R=R}, it follows that 
the $i$-th radius of $\Fs$ and $\Fs^*$ coincide, and 
hence also that $\Gamma_{S,i}(\Fs)=\Gamma_{S,i}(\Fs^*)$. 
Item \eqref{Thm : 5.7 deco good direct siummand-ii} follows.

Item \eqref{Thm : 5.7 deco good direct siummand-i} is now a 
consequence of Theorem \ref{Thm : 5.7 deco good direct siummand}.
\if{

The same holds for all $y\in X-\Gamma_{S,i}(\Fs)$ 
by the following argument. 
Since $\Gamma_S\subseteq\Gamma_{S,i}(\Fs)$, then 
$X-\Gamma_{S,i}(\Fs)$ is a disjoint union of virtual open disks.
Let $D\subseteq X$ be such a virtual open disk, and let 
$\{x\} = \overline{D}-D$ be the point of 
$\Gamma_{S,i}(\Fs)$ in its boundary.
Since $i$ separates the radii, then there exists a 
segment $[z,x]$, connecting a point $z\in D$ to $x$, such that all the 
radii $\R_{S,1}(-,\Fs),\ldots,\R_{S,i-1}(-,\Fs)$ are spectral and 
non solvable over $[z,x]$ (and constant on $D$ by assumption). 

\framebox{
\begin{minipage}{450pt}
ATTENZIONE : Nella dimostrazione del teorema io separo lo studio dei punti di $\Gamma$ dallo studio dei punti non in $\gamma$.

Poi uso il teorema di decomposizione di Kedlaya sui dischi tangenti al 
grafico controllante.

C'e' un problema : siccome questi sono dischi \framebox{virtuali (!!)} 
devo discendere il teorema di Kedlaya ai dischi virtuali. Non credo che ci 
siano problemi, ma bisogna dirlo ...
\end{minipage}
}

\begin{theorem}
Let $X=D$ be a virtual open disk with empty triangulation. 
Let $\Fs$ be a differential equation on $D$ such that 
$\R_{\emptyset,j}(-,\Fs)$ is a constant function on 
$D$ for all $j=1,\ldots,i$.
Let $j\in\{1,\ldots,i\}$ be an index separating the radii. 
Then $\Fs_{\geq j}$ is a 
direct summand of $\Fs$. 
\end{theorem}
\begin{proof}
The claim follows from 
\cite[12.4.1]{Kedlaya-book-2} if $D$ is $K$-rational, 
and by a Galois descent from 
$K^{\mathrm{alg}}$ to $K$ otherwise. 
We provide in Proposition \ref{THM : Decomposition on a disk} another proof 
based on Grothendieck-Ogg-Shafarevich formula 
(cf. \eqref{GOS-changed}). 
\end{proof}

By Proposition \ref{THM : Decomposition on a disk} 
$\Fs_{|D}$ decomposes into a direct sum of sub-objects 
$(\Fs_{|D})^\rho$ satisfying 
$\R_{\emptyset,j}(y,(\Fs_{|D})^\rho)=\rho$ for all 
$j=1,\ldots,\mathrm{rank}((\Fs_{|D})^\rho)$.
So $(\Fs_{\geq i})_{|D}$ is a direct summand of $\Fs_{|D}$. 
From this we obtain that for all $y\in D$, 
and for all $j=1,\ldots,i$, the $j$-th radius 
$\R_{\emptyset,j}(y,\Fs_{|D})$ coincides with the radius 
$\R_{\emptyset,1}(y,\Fs_{|D}')$ where $\Fs_{|D}'$ is a direct 
summand of $\Fs_{|D}$ (see for example 
\cite[Proposition 5.5 and Lemma 7.6]{NP-I}, see also 
Prop. \ref{prop : radii of the direct sum}). Since the duality commutes 
with the $1$-th radius (cf. Prop. \ref{Prop : 1-th radius and dual}), 
one has $(\Fs_{|D}^*)_{\geq j}=
((\Fs_{|D})_{\geq j})^*$ for all $j=1,\ldots,i$. 
This proves that 
$\omega_{S,j}(y,\Fs)^*=\omega_{S,j}(y,\Fs^*)$ for all $j=1,\ldots,i$ 
and all $y\in D$ (cf. \eqref{Localization of sol}). 
Hence the dimension of $\omega_{S,i}(-,\Fs^*)$ is 
constant on $X$ with value $r-i+1$ and $(\Fs^*)_{\geq i}$ 
exists by i) and ii). 
To prove \eqref{eq : can map iso} we observe that the arrow 
$((\Fs^*)_{\geq i})_x\to\Fs^*_x\to(\Fs_{\geq i})^*_x$ 
is an isomorphism for all $x$ for which $i$ is spectral by 
Thm. \ref{Dw-Robba}. In particular it holds for 
$x\in\Gamma_{S,i}(\Fs)$. So the map 
$(\Fs^*)_{\geq i}\to\Fs^*\to(\Fs_{\geq i})^*$ 
is an isomorphism over a 
neighborhood of each point of $\Gamma_{S,i}(\Fs)$. 
This also holds over each $D$ thanks 
to the above decomposition. This concludes the proof of iv).
Equality v) follows from \eqref{eq : compatibility with duals} 
on the points $x\in\Gamma_{S,i}(\Fs)$, since the radius is spectral, 
and by the the above decomposition over each $D$. vi) is a 
consequence of v). This proves iii) as in Prop. 
\ref{Prop : direct summand}.
}\fi
\end{proof}

\begin{remark}\label{Remark : bad condition}
Let $S,S'$ be two weak triangulations of $X$ such that 
$\Gamma_S\subseteq\Gamma_{S'}$. 
The best choice in order to fulfill condition 
\eqref{eq : condition on graphs included} will be $S'$. 
Indeed, by definition we always the inclusion 
$\Gamma_{S'}\subseteq\Gamma_{S',i}(\Fs)$, therefore 
we might be tempted to increase $S'$ in order to have 
$\Gamma_{S,1}(\Fs)\cup\cdots\cup\Gamma_{S,i-1}(\Fs)\subseteq\Gamma_{S'}$.\footnote{This is possible because none of the 
$\Gamma_{S,j}(\Fs)$ contains points of type $1$ or $4$ by 
\cite{Kedlaya-draft}.} By the rule \eqref{eq: change of tri eq 2.32}, this automatically guarantee 
$\Gamma_{S',1}(\Fs)\cup\cdots\cup\Gamma_{S',i-1}(\Fs)
\subseteq\Gamma_{S',i}(\Fs)$. 

Unfortunately, the fact that we increase of $S'$ has the effect to 
truncate the radii (cf. \eqref{Prop: A-4 fgh -eq}). 
Therefore, if $S'$ is too 
large, the index $i$ may not separate the $S'$-radii.
In fact, by Remark \ref{Remark : changing tr F_S,I} we 
know that, in order to guarantee that $i$ separates the radii, 
the best choice of triangulation is $S$. 
The triangulation $S$ is also convenient for Theorem 
\ref{Thm : 5.7 deco good direct siummand} which is for this 
reason more ``\emph{natural}'' than Theorem 
\ref{Thm : criterion self direct sum}.
\end{remark}

\begin{remark}
The statement of 
Theorem \ref{Thm : criterion self direct sum} implies that the 
inclusion of graphs \eqref{eq : condition on graphs included} holds for 
$\Fs$ if, and only if, it holds for $\Fs^*$.
\end{remark}
Thanks to this remark, we may test 
\eqref{eq : condition on graphs included} on $\Fs$ or on $\Fs^*$.
In view of this, the following Corollary is helpful in 
order to check if Theorem \ref{Thm : criterion self direct sum} applies.
\begin{corollary}
\label{Cor. dual compatibility}
For all $i=1,\ldots,\mathrm{rank}(\Fs)$ one always has
\begin{equation}
\Gamma_{S,1}(\Fs)\cup\cdots\cup\Gamma_{S,i}(\Fs)\;=\;
\Gamma_{S,1}(\Fs^*)\cup\cdots\cup\Gamma_{S,i}(\Fs^*)\;.
\end{equation}
\end{corollary}
\begin{proof}
Without loss of generality, we can assume $K$ spherically complete.

Set $\Gamma_{S,i}'(\Fs):=
\Gamma_{S,1}(\Fs)\cup\cdots\cup\Gamma_{S,i}(\Fs)$. 
 By a result of Kedlaya, 
none of the $\Gamma_{S,j}(\Fs)$ contains points of type 
$1$ or $4$ (cf. \cite{Kedlaya-draft}). Therefore, we may 
consider a triangulation $S'$ such that 
$\Gamma_{S'}=\Gamma_{S,i}'(\Fs)$. 
Proposition \ref{Prop : A-5 ecp} we have 
$\Gamma_{S,i}'(\Fs)=\Gamma_{S',i}'(\Fs)$. 

We claim that it is enough to prove that for all $x\notin\Gamma_{S',i}'(\Fs)=\Gamma_{S'}$ the 
radii $\R_{S',j}(-,\Fs^*)$, $j=1,\ldots,i$ are all constant on $D(x,S')$. 
Indeed, by definition of the radii this imply $\Gamma_{S',i}'(\Fs^*)\subseteq\Gamma_{S'}$. On the other hand, 
by Proposition \ref{Prop : A-5 ecp}, we have 
\begin{equation}
\Gamma_{S,i}'(\Fs^*)\cup\Gamma_{S'}\;=\; 
\Gamma_{S',i}'(\Fs^*)\;\subseteq\;
\Gamma_{S'}\;=\;
\Gamma_{S,i}'(\Fs)
\end{equation}
and therefore $\Gamma_{S,i}'(\Fs^*)\subseteq\Gamma_{S,i}'(\Fs)$. 
A dual argument will gives the converse inclusion.

Let us then sow that the, for all $j=1,\ldots,i$ 
the radii $\R_{S',j}(-,\Fs^*)$ are constants outside $\Gamma_{S'}$. 
By their definition, they are not truncated when we localize to a 
connected component $D$ of $X-\Gamma_{S'}$ (cf. \eqref{Localization of sol-3}). Therefore, we can assume that $X$ 
is a virtual open disk with empty triangulation 
such that the first $i$ radii of $\Fs$ are constant functions on it. 
In this case the claim follows by Proposition \ref{THM : 
Decomposition on a disk}.
\end{proof}

\if{
\begin{remark}\label{Rk : another condition?}
We want to know whether the sequence 
$0\to\Fs_{\geq i}\to\Fs\to\Fs_{<i}\to 0$ splits.
This information should be contained in the cokernel
$\mathcal{H}^1(\Fs_{<i}^*\otimes\Fs_{\geq i})$ of the 
connection (cf. Lemma \ref{Lemma : non splitting extension}).
Hence it should be intelligible in term of $\Fs_{<i}^*$ and 
$\Fs_{\geq i}$. Since all the 
extensions $0\to\Fs_{\geq i}\to \mathscr{E}\to\Fs_{<i}\to 0$ have 
the same extrema $\Fs_{<i}^*$ and $\Fs_{\geq i}$, 
the unique way to distinguish the direct sum 
$\Fs_{\geq i}\oplus\Fs_{<i}$ from the other extensions is to 
impose that the direct sum is the unique possible extension. 
This means that the cokernel of the restriction
$(\Fs_{<i}^*\otimes\Fs_{\geq i})_{|D}$ is zero on each disk 
$D\subseteq X$, in relation with the Grothendieck-Ogg-Shafarevich 
formula \eqref{GOS-changed}. 

In section \ref{section counterexample to duality and exactness} we 
provide an example of non splitting sequence 
$0\to\Fs_{\geq i}\to\Fs\to\Fs_{<i}\to 0$. Such a sequence is the 
unique contradicting the assumptions of Theorems 
\ref{Thm : 5.7 deco good direct siummand} and 
\ref{Thm : criterion self direct sum}. This was the 
starting point of this section (cf. Remark 
\ref{remark - radii non separated VS H^1}).
\end{remark}
}\fi

\subsection{Direct sum over annuli.}
\label{Some interesting particular cases (annuli).}

In this section we show how to derive from the above criterion some 
well known decompositions results over open annuli (cf. \cite[section 12]{Kedlaya-book-2}, \cite{Ch-Me-III}). We provide some new informations 
about the nature of the controlling graphs in that context.
\begin{corollary}\label{Cor : linear on annulus}
Let $X$ be a virtual open annulus with empty triangulation 
$S=\emptyset$. Let $I:=\Gamma_X$ be its skeleton 
and let $i\leq r=\mathrm{rank}(\Fs)$. 
Then, the following conditions are equivalent :
\begin{enumerate}
\item\label{Cor : linear on annulus-i} For all 
$j\in\{1,\ldots,i-1\}$, the partial height $H_{\emptyset,j}(-,\Fs)$ 
(cf. \eqref{eq : partial height}) is a $\log$-affine map on $I$;
\item\label{Cor : linear on annulus-ii} For all 
$j\in\{1,\ldots,i-1\}$, the radius $\R_{\emptyset,j}(-,\Fs)$ 
is a $\log$-affine map on $I$.
\end{enumerate}
If they are satisfied, then for all $j=1,\ldots,i-1$ one has 
\begin{equation}
\Gamma_{\emptyset,j}(\Fs)\;=\;I\;.
\end{equation}

Assume that the above conditions are satisfied and, moreover, 
that the index $i$ separates the radii at every point of $I$: 
 $\R_{\emptyset,i-1}(x,\Fs)<\R_{\emptyset,i}(x,\Fs)$ for all 
$x\in I$.

Then,  the index $i$ separates the radii of $\Fs$ globally on $X$, 
the inclusion \eqref{eq : condition on graphs included} holds and 
Theorem \ref{Thm : criterion self direct sum} applies. In particular,
\begin{equation}\label{eq: FS=direct over annulus}
\Fs\;=\;\Fs_{<i}\oplus\Fs_{\geq i}\;.
\end{equation}
\end{corollary}
\begin{proof}
The equivalence between \eqref{Cor : linear on annulus-i} and 
\eqref{Cor : linear on annulus-ii} follows from an easy induction on $j$ 
using Definition \eqref{eq : partial height}.

Let us prove that for all $j=1,\ldots,i-1$ one has 
$\Gamma_{\emptyset,j}(\Fs)=I$. 
For every $z\notin I$ the maximal disk $D(z,\emptyset)$ (cf. Definition \ref{Maximal disksksks}) 
is the connected component of $X-I$ containing $z$. Therefore, the 
localization to $D(z,\emptyset)$ does not change the value of the radii 
nor the total heights (cf. \eqref{Localization of sol-1}). 
It is then enough to prove that for all $j=1,\ldots,i-1$
the $j$-th radius of $\Fs$ is constant on 
every maximal disk $D(z,\emptyset)$ either before or after localization to it.

Since $\R_{\emptyset,j}(-,\Fs)$ is $\log$-affine along $I$, 
we have only two possibilities: 
either $\R_{\emptyset,j}(x,\Fs)$ is spectral non-solvable at every 
$x\in I$, or $\R_{\emptyset,j}(x,\Fs)$ is solvable at every $x\in I$.

Let us first consider the first case, 
the total height $H_{\emptyset,j}(-,\Fs)$ is harmonic 
at every $x\in I$ (cf. \cite[Theorem 3.9]{NP-I}), therefore the $\log$-linearity along 
$I$ implies that for every $x\in I$ and every direction $b$ out of $x$ 
such that $b\not\subset I$ we have 
$\partial_bH_{\emptyset,j}(-,\Fs)=0$, because every $\partial_bH_{\emptyset,j}(-,\Fs)$ is positive. By assumption, this property also holds for every $j'\leq j$ and by definition we 
have $\partial_bH_{\emptyset,j}(-,\Fs)=\sum_{1\leq j'\leq j}\partial_b\R_{\emptyset,j'}(-,\Fs)$. 
Therefore, the radii $\R_{\emptyset,j'}(-,\Fs)$ also satisfy 
$\partial_b\R_{\emptyset,j'}(-,\Fs)=0$. 
As we have seen, this implies that $\partial_b\R_{\emptyset,j'}(-,\Fs_{|D(z,\emptyset)})=0$ for the slopes of the localization of $\Fs$ 
to the maximal disk $D(z,\emptyset)$. 
Proposition \ref{THM : Decomposition on a disk} 
then applies and the radii are all constant on every maximal disk 
$D(z,\emptyset)$ as required.

Let us now consider the second case where $\R_{\emptyset,j}(x,\Fs)$ 
is solvable at every $x\in I$. Let $j_1:=i_x^{\textrm{sol}}$ (cf. 
Definition \ref{eq : over-solvable cutoff}) be the first solvable index 
less than or equal to $j$. Notice that, for every $x'\in I$ we have 
$i_{x'}^{\textrm{sol}}=i_x^{\textrm{sol}}$ because the radii are 
$\log$-affine along it. For all $j'< j_1$ we 
have just proved that 
the radii $\R_{\emptyset,j'}(-,\Fs)$ are all constant on 
every connected component $D(z,\emptyset)$ of $X-I$. Moreover, they 
are not equal to $1$ otherwise they would be solvable at $x$. Arguing 
as in \emph{Step 3} of the proof of Proposition 
\ref{THM : Decomposition on a disk}, we may apply 
\cite[Theorem 12.4.1]{Kedlaya-book-2} and Proposition 
\ref{Prop. Exact sequence prop separate radii in the right order}, 
to prove that $\Fs_{|D(z,\emptyset)}$ factorizes as 
$\Fs_{|D(z,\emptyset)}=\Fs'\oplus\Fs''$ where for all 
$z'\in D(z,\emptyset)$ close enough to $x$ (i.e. close to the boundary of $D(z,\emptyset)$), 
for all $k=1,\ldots,j_1-1$ and 
for all $k'=1,\ldots,\mathrm{rank}(\Fs)-j_1+1$ we have
\begin{equation}
\R_{\emptyset,k}(z',\Fs')\;=\;
\R_{\emptyset,k}(z',\Fs_{|D(z,\emptyset)})\;,\qquad
\R_{\emptyset,k'}(z',\Fs'')\;=\;
\R_{\emptyset,j_1-1+k'}(z',\Fs_{|D(z,\emptyset)})\;.
\end{equation}
This proves in particular that over $D(z,\emptyset)$ 
the $j_1$-th radius of $\Fs$ equals the first radius of $\Fs''$. 
In particular, the function 
$\R_{\emptyset,j_1}(-,\Fs)=\R_{\emptyset,j_1}(-,\Fs_{|D(z,
\emptyset)})$ is decreasing on the segments $[c,x[\subset D(z,
\emptyset)$ (cf. \cite[Theorem 3.9]{NP-I}), 
where $c\in D_L$ is an $L$-rational points for an 
unspecified extension $L/K$. Since the $j_1$-th radius 
is solvable at $x$ (hence equal to $1$), it follows by continuity that 
$\R_{\emptyset,j_1}(-,\Fs)$ is constant along $[c,x[$ with value $1$. 
Since the radii are insensitive to scalar extension of $K$, we deduce 
that $\R_{\emptyset,j_1}(z',\Fs)=1$ for all $z'\in D(z,\emptyset)$. 
Moreover, since for all $j'\geq j_1$ we have $\R_{\emptyset,j'}(z',\Fs)
\geq\R_{\emptyset,j_1}(z',\Fs)$, then $\R_{\emptyset,j'}(-,\Fs)$ is 
constant with value $1$ on every $D(z,\emptyset)$, 
for every $j'\ge j_1$.

To summarize, we have proved that for all $j=1,\ldots,i-1$, 
the function $\R_{\emptyset,j}(-,\Fs)$ is constant on each connected 
component of $X-I$. In particular $\Gamma_{\emptyset,j}(\Fs)=I$.

Let us now assume that 
$\R_{\emptyset,i-1}(x,\Fs)<\R_{\emptyset,i}(x,\Fs)$, for all 
$x\in I$. By definition $I$ is contained in the controlling graph 
$\Gamma_{\emptyset,i}(\Fs)$, therefore we 
have $\Gamma_{\emptyset,1}(\Fs)\cup\cdots\cup\Gamma_{\emptyset,i-1}(\Fs)=I\subset\Gamma_{\emptyset,i}(\Fs)$ and \eqref{eq : condition on graphs included} holds.
As a consequence, if we show that $i$ separates 
the radii of $\Fs$ on the whole $X$, 
then the direct sum decomposition \eqref{eq: FS=direct over annulus} 
will follow from Theorem \ref{Thm : criterion self direct sum}.

Let us prove that $i$ separates the radii on the whole $X$. 
Let $D$ be a connected component of $X-I$ and  $x\in I$ its relative 
boundary in $X$.
Since $i$ separates the radii along $I$, 
for all $j<i$ the $j$-th radius is spectral non-solvable at 
$x\in I$. We know that for every $j<i$ the $j$-th 
radius functions $\R_{\emptyset,j}(-,\Fs)$ is constant on $D$. We can 
therefore apply again \cite[Theorem 12.4.1]{Kedlaya-book-2} to obtain  $\Fs_{|D}=\Fs'\oplus\Fs''$ over 
$D$ and express the $i$-th radius function of $\Fs_{|D}$ as a first 
radius of another differential equation over $D$. 
It follows that the $i$-th radius has the decreasing property as above 
and, as a consequence, the index $i$ separates the radii of $\Fs$ over 
$D$. The claim follows.
\end{proof}

\subsubsection{Solvable equations over the Robba ring.}
\label{Ch-Me-Robba}
The so called Robba ring is the ring  
\begin{equation}\label{eq : def of Robba ring first given}
\mathfrak{R}\;:=\;\bigcup_{\varepsilon>0}\O(C_\varepsilon)\;,
\end{equation} 
where $C_\varepsilon:=C_K^-(0;1-\varepsilon,1)$. 
Following a terminology of Gilles Christol and Zoghman Mebkhout \cite{Ch-Me-III, Ch-Me-IV}, a differential 
module $\M$ over $\mathfrak{R}$ is \emph{solvable} if 
\begin{equation}\label{eq : solvability over the Robba ring}
\lim_{\rho\to 1^-}\R_{\emptyset,1}(x_{0,\rho},\M)\;=\;1\;.
\end{equation} 
By  \cite[Theorem 4.2-1 and Corollary 6.2-6]{Ch-Me-III} and 
\cite[Lemma 12.6.2]{Kedlaya-book-2}, as a consequence of 
the integrality property \eqref{eq : integrality} one shows that,
if $\M$ is solvable, then for all $i=1,\ldots,r=\mathrm{rank}(\M)$, 
there exists $\varepsilon>0$ such that 
$\R_{S,i}(x_{0,\rho},\M)=\rho^{\beta_{r-i+1}}$, 
for all $\rho\in]1-\varepsilon,1[$. 
The numbers $\beta_1\leq\beta_2\leq\cdots\leq\beta_r$ are called the 
\emph{slopes} of $\M$.\footnote{The original terminology of 
Christol and Mebkhout is \emph{$p$-adic slopes}. 
We drop the term $p$-adic because we allow the absolute value of 
$K$ to be the extension of a trivially valued field, or to have a residual 
field of characteristic $0$. 
Indeed, everything works in a more general context.} 
We say that $\M$ is pure of slope 
$\beta$ if $\beta_1=\cdots=\beta_r=\beta$.

As a consequence of Corollary \ref{Cor : linear on annulus} we recover 
Christol-Mebkhout decomposition theorem:
\begin{corollary}[\cite{Ch-Me-III}]\label{Cor : Ch-Me deco}
Any solvable differential module over $\mathfrak{R}$
admits a direct sum decomposition 
\begin{equation}\label{eq : M(beta)}
\M\;:=\;\bigoplus_{\beta \geq 0}\M(\beta)
\end{equation}
into sub-modules $\M(\beta)$ that are pure of slope $\beta$. 
\end{corollary}
\begin{proof}
The radii $\R_{S,i}(-,\Fs)$ are $\log$-affine on the skeleton of a 
conveniently small annulus $C_\varepsilon$. 
Then apply Corollary \ref{Cor : linear on annulus}.
\end{proof}

\begin{definition}[Irregularity]\label{Def : Ch-Me irreg}
Let $\M$ be a solvable module over $\mathfrak{R}$. 
We define the \emph{Christol-Mebkhout Newton polygon} of $\M$ as
the polygon whose slopes are 
$\beta_1\leq\beta_2\leq\cdots\leq\beta_r$. 
It is the epigraph of the convex function 
$h:[0,r]\to\mathbb{R}_{\geq 0}$ such that $h(0)=0$, 
$h(i)=\beta_1+\cdots+\beta_i$ for all $i=1,\ldots,n$, 
and it is affine on each interval $[i,i+1]$.
We define the irregularity of $\M$ as the height 
\begin{equation}
\mathrm{Irr}(\M)\;=\;\sum_{i=1}^r\beta_i\;.
\end{equation}
of the Christol-Mebkhout Newton polygon. 
\end{definition}

\begin{remark}\label{Rk : irr eq slope}
By definition, the irregularity $\mathrm{Irr}(\M)\geq 0$
is the slope of the function $\partial_bH_{\emptyset,r}(-,\M)$,
where $b$ is the germ of segment at the boundary of the open 
unit disk, oriented as outside the disk.
The fact that $\mathrm{Irr}(\M)$ is a non negative integer, is due to 
the fact that the slopes of $H_{\emptyset,r}(-,\M)$ are integers 
numbers (cf. Proposition \ref{Prop : Integrality of the Partial heights}), 
and to the solvability condition. 
\end{remark}

\subsection{Crossing points and filtration outside a locally finite 
set.}
\label{Crossing points and global decomposition.}

In this section we do not assume that the radii of $\Fs$ are separated 
globally on $X$. We investigate the existence of a filtration of $\Fs$ 
over some regions of $X$. A first elementary approach consists in 
localizing on a given region and, if the radii are separated after 
localization, apply the above results to the localized module.
The principal issue in this procedure is that the localization process 
may truncate some radii (cf. Section \ref{Localization}), therefore they may 
not remain separated after 
localization. A possible approach consists in working 
with the original radii before localization, somehow avoiding 
the localization process, and show that a sub-differential equation exists 
on the regions where the radii are separated (cf. Remark 
\ref{Rk : no need of S_i}). However, a more accurate analysis of the 
geometrical locus where the radii are separated is interesting in itself. 
Therefore we begin in Sections \ref{Subsubsection: cr pts} and 
\ref{subsubsection : ith sep sub} by obtaining a precise description
of the nature of the regions where the radii are 
separated and we prove that they admit weak-triangulations for which the 
localization does not truncate the radii.

Let us begin by stating a Lemma which will be useful in this 
section. Recall that $X$, and hence also $\Gamma_S$, are connected 
(cf. Setting \ref{Hypothesis : X connected}).

\begin{lemma}\label{Lemma : cc of X-mathfrak(f)}
Let $\Gamma\subseteq X$ be a non empty, connected, 
locally finite graph in $X$ containing $\Gamma_S$ such that $X-\Gamma$ is a disjoint union of virtual open disks. 
Let $A\subseteq\Gamma$ be any subset. 
Then, a connected component of $X-A$ is necessarily one of the 
following:
\begin{enumerate}
\item\label{Lemma : cc of X-mathfrak(f)-i} a virtual open disk 
coinciding with a connected component of 
$X-\Gamma$ whose relative boundary in $X$ is a point of $A$;
\item\label{Lemma : cc of X-mathfrak(f)-ii} the inverse image in $X$ 
by the canonical retraction 
$X\to\Gamma$ of a connected component of $\Gamma-A$.
\end{enumerate}
\hfill$\qed$
\end{lemma}

\if{\begin{lemma}\label{Lemma : Gamma-A and Gamma-B}
Let $\Gamma_1\subseteq\Gamma_2$ be two non empty, connected, 
closed, locally finite graphs in $X$ containing $\Gamma_S$ such that 
$X-\Gamma_i$ is a disjoint union of virtual open disks. 

Let $A\subseteq \Gamma_1$ be any subset and let 
$B\subseteq\Gamma_2$ be a locally finite subset.

Assume that the relative boundary of $A$ in $\Gamma_1$ 
is contained in $B$. 

\comm{In $\Gamma_1$ or $\Gamma_1\cup\Gamma_2$ ?}

Then, every connected component of $X-B$ is either 
contained in a connected component of $X-A$ or it is 
disjoint from all of them.
\end{lemma}
\begin{proof}
Denote by $\tau_i\;:\;X\to\Gamma_i$ the canonical retractions.
Recall that $X-\Gamma_i$ is a disjoint union of virtual open disks. 
Then, the retraction $\tau_i$ is the identity on $\Gamma_i$ and it 
sends a point $x\in X-\Gamma_i$ into the relative boundary 
$y\in\Gamma_i$ of the connected component of $X-\Gamma_i$ 
containing $x$. It follows that 
$\tau_1=\tau_{1,2}\circ\tau_2$, where $\tau_{1,2}:
\Gamma_2\to\Gamma_1$ is the restriction to $\Gamma_2$ of the 
retraction $\tau_1$ : $\tau_{1,2}=\tau_{1|\Gamma_2}$.

Let $Q$ be a connected component of $X-B$. 
There is a unique connected component $Q'$ of 
$X-(B\cap\Gamma_1)$ containing $Q$. It is enough to 
show that $Q'$ has the claimed property and, therefore, we may 
assume without loss of generality that $\Gamma_1=\Gamma_2$ 
coincide. Let us call $\Gamma$ this graph.

In this case, the connected components of  $X-A$ and $X-B$ are 
described by Lemma \ref{Lemma : cc of X-mathfrak(f)}. 

If $Q'$ satisfies item \eqref{Lemma : cc of X-mathfrak(f)-i} of Lemma 
\ref{Lemma : cc of X-mathfrak(f)}, it is obviously contained in a 
connected component of $X-A$ or none of them. 
The claim is trivial in this case.
 
On the other hand, the set of connected components of $X-B$ (resp. 
$X-A$) that satisfy item \eqref{Lemma : cc of X-mathfrak(f)-ii} 
of Lemma \ref{Lemma : cc of X-mathfrak(f)} are in bijection with the 
connected components of $\Gamma-B$ (resp. $\Gamma-A$). It is then 
enough to show that a connected component of $\Gamma-B$ is either 
contained in a connected component of $\Gamma-A$ or in none of 
them. 

Since $\Gamma$ is closed in $X$, 
the boundary of $A$ in $X$ coincides with the boundary of $A$ in $\Gamma$
	
\comment{C'est vrai \c ca ? Si on a $A=\Gamma$ par exemple, le bord de $A$ dans $\Gamma$ va \^etre vide, mais pas forc\'ement dans $X$.}

\comm{Il y a un problème ici ... je vais y refléchir}
														
 and it is, by assumption, contained in $B$. On the other 
hand, $B$ is the boundary of $\Gamma-B$ in $\Gamma$. This shows 
that the boundary of $\Gamma-B$ contains that of $\Gamma-A$.
It follows that the interior of $\Gamma-B$ is disjoint from the 
boundary of $A$. In particular, every connected component of the 
interior of $\Gamma-B$ is either contained in a connected component 
of the interior of $A$ or of the interior of $\Gamma-A$. In particular, it 
is contained either in a connected component of $\Gamma-A$ or it is 
contained in $A$.

Since $B$ is locally finite, the interior of $\Gamma-B$ equals the 
disjoint union of the connected components of $\Gamma-B$ and the 
claim follows.
\end{proof}
}\fi

\subsubsection{Crossing points.}\label{Subsubsection: cr pts}
We now describe the nature of some regions where the radii cross, or stay 
separated.

Let $I\subseteq\{1,\ldots,r=\mathrm{rank}(\Fs)\}$ be a non empty 
subset. Set
\begin{equation}
\Gamma_{S,I}(\Fs)\;:=\;\bigcup_{i\in I}\Gamma_{S,i}(\Fs)\;.
\end{equation}
\begin{definition}\label{Definition : Sep ne}
An $I$-\emph{separating neighborhood} of $x\in X$ is a open neighborhood 
$U$ of $x$ such that for all $i,j\in I$ one has either 
$\R_{S,i}(y,\Fs)=\R_{S,j}(y,\Fs)$ for all $y\in U$, or 
$\R_{S,i}(y,\Fs)\neq \R_{S,j}(y,\Fs)$ for all  $y\in U$. 
In particular, if $I=\{1,\ldots,r\}$, then each index $i$ separates the radii 
either everywhere or nowhere on $U$. In this case we simply say that $U$ 
is a separating neighborhood.
\end{definition}
\begin{definition}\label{Def : Crossing points}
A point $x\in X$ is called $I$-\emph{crossing point for $\Fs$} if it has no 
$I$-separating neighborhoods. We denote by 
\begin{equation}\label{eq : Cr_S set}
\mathrm{Cr}_S(I,\Fs)
\end{equation}
the subset of $I$-crossing points for $\Fs$ and by 
$\mathrm{Cr}_S(\Fs):=\mathrm{Cr}_S(\{1,\ldots,r\},\Fs)$.
\end{definition}
\begin{lemma}
The subset $\mathrm{Cr}_S(I,\Fs)$ is a locally finite closed subset of $X$ 
contained in $\Gamma_{S,I}(\Fs)$.
\end{lemma}
\begin{proof}
This results by finiteness (cf. Theorem \ref{Thm : finiteness}), 
together with the fact that along a compact segment in $X$
each radius $\R_{S,i}(-,\Fs)$ is a piecewise log-affine function with 
a finite number of breaks (cf. \cite{NP-I,NP-II}).
\end{proof}
\begin{lemma}
Let $\mathfrak{F}\subseteq X$ be a locally finite subset of $X$
containing $\mathrm{Cr}_{S}(I,\Fs)$ and let $Y$ 
be a connected component of $X-\mathfrak{F}$. 
Then $Y$ is an $I$-separating open neighborhood of all its points. In other 
words for all $i,j\in I$ one has either $\R_{S,i}(y,\Fs)<\R_{S,j}(y,\Fs)$ for all 
$y\in Y$, or $\R_{S,i}(y,\Fs)=\R_{S,j}(y,\Fs)$ for all $y\in Y$.
\hfill$\qed$
\end{lemma}

\begin{remark}\label{Remark : cc of X-Cr}
Assume that $\mathrm{Cr}_S(I,\Fs)$ is not empty.\footnote{In 
particular, this implies $\Gamma_{S,I}(\Fs)\neq\emptyset$ because some of the radii are not constants.}
Let $Q$ be a connected component of $X-\mathrm{Cr}_S(I,\Fs)$. 
Then, by Lemma \ref{Lemma : cc of X-mathfrak(f)}, 
$Q$ coincides with one of the following subsets  
\begin{itemize}
\item a virtual open disk coinciding with a connected component of 
$X-\Gamma_{S,I}(\Fs)$ whose relative boundary in $X$ is a point of $\mathrm{Cr}_S(I,\Fs)$; 
\item the inverse image in $X$ of a connected component of 
$\Gamma_{S,I}(\Fs)-\mathrm{Cr}_S(I,\Fs)$ by the canonical 
retraction (cf. \eqref{eq : retraction global})
\begin{equation}
\delta_{\Gamma_{S,I}(\Fs)}\;:\;X\xrightarrow{\quad}
\Gamma_{S,I}(\Fs)\;.
\end{equation}
\end{itemize}
More generally, let $\mathfrak{F}\subseteq X$ be a locally finite subset 
containing $\mathrm{Cr}_{S}(I,\Fs)$ and $Y$ a connected component 
of $X-\mathfrak{F}$. Then $Y$ is contained in a connected component 
$Q$ of $X-\mathrm{Cr}_S(I,\Fs)$. More precisely, $Y$ is equal to a 
connected component of $Q-\mathfrak{F}$.   
\end{remark}

\subsubsection{The $i$-th separating subset $X_{S,i}(\Fs)$.}
\label{subsubsection : ith sep sub}
\begin{definition}
We denote by 
\begin{equation}\label{Def : eq : X_SiF}
X_{S,i}(\Fs)
\end{equation}
the open subset of $X$ formed by 
the  points $x\in X$ such that the index $i$ separates the radii of $\Fs$ at $x$. We simply write $X_{S,i}$ if $\Fs$ is clear from the context.
\end{definition}
\begin{remark}\label{Rk : Y subset X_S,i}
 By definition we have $X_{S,1}(\Fs)=X$. 

 If $X_{S,i}(\Fs)=X$, then the index $i$ separates globally the radii of 
$\Fs$, and $\Fs$ decomposes by Theorem \ref{MAIN Theorem}.

 If $X_{S,i}(\Fs)=\emptyset$, then $\R_{S,i}(x,\Fs)=
\R_{S,i-1}(x,\Fs)$ for all $x\in X$, and  
$\Gamma_{S,i}(\Fs)=\Gamma_{S,i-1}(\Fs)$. In this case there is no 
regions where a generalization of 
Theorem \ref{MAIN Theorem} may exist.

 If $X_{S,i}(\Fs)\neq \emptyset,X$, 
then $\Gamma_{S,i}(\Fs)\cup\Gamma_{S,i-1}(\Fs)\neq\emptyset
$,\footnote{Indeed, if $\Gamma_{S,i}(\Fs)\cup\Gamma_{S,i-1}(\Fs)=\emptyset$, then $X$ is a virtual open disk with 
empty triangulation on which the $i$-th and the $(i-1)$-th radii are 
constants functions. In this case, the two radii either coincide or are distinct globally on $X$. In the first case we would have $X_{S,i}(\Fs)=\emptyset$, while in the second $X_{S,i}(\Fs)=X$.\label{Footnote : Gamma not empty}} 
and in this case $X_{S,i}(\Fs)$ is the inverse image by the canonical 
retraction 
\begin{equation}\label{eq : retract tau i-1,i}
\tau_{i,i-1}\;:\;X\;\xrightarrow{\quad}\; 
\Gamma_{S,i}(\Fs)\cup\Gamma_{S,i-1}(\Fs)
\end{equation}
of a possibly not connected open sub-graph. More precisely, let us 
denote by 
\begin{equation}
C_{S,i,i-1}(\Fs)\;\;\subseteq\;\; 
\Gamma_{S,i}(\Fs)\cup\Gamma_{S,i-1}(\Fs)
\end{equation}
 the subset of points $x\in\Gamma_{S,i}(\Fs)
\cup\Gamma_{S,i-1}(\Fs)$ such that 
$\R_{S,i-1}(x,\Fs)=\R_{S,i}(x,\Fs)$. 
Since the radii are continuous, and 
$\Gamma_{S,i}(\Fs)\cup\Gamma_{S,i-1}(\Fs)$ is closed and locally 
compact in $X$, the subset
$C_{S,i,i-1}(\Fs)$ is closed and locally compact too. 
Moreover, 
since the radii have a locally finite number of breaks along $\Gamma_{S,i}(\Fs)\cup\Gamma_{S,i-1}(\Fs)$ (cf. Theorem \ref{Thm : finiteness}), the set $C_{S,i,i-1}(\Fs)$ and its complement in $\Gamma_{S,i}(\Fs)\cup\Gamma_{S,i-1}(\Fs)$ both have a locally finite number of connected components. 
Therefore, 
$\Gamma_{S,i}(\Fs)\cup\Gamma_{S,i-1}(\Fs)-C_{S,i,i-1}(\Fs)$ is an 
open subset of the graph $\Gamma_{S,i}(\Fs)\cup\Gamma_{S,i-1}(\Fs)$ with 
locally finite number of connected components and 
\begin{equation}\label{eq : X_S,i = tau-1(gamma-C)}
X_{S,i}(\Fs)\;=\;\tau_{i,i-1}^{-1}\Bigl(\Gamma_{S,i}(\Fs)\cup\Gamma_{S,i-1}(\Fs)-C_{S,i,i-1}(\Fs)\Bigr)\;.
\end{equation}
In particular, the connected components of $X_{S,i}(\Fs)$ are in 
bijection with those of 
$\Gamma_{S,i}(\Fs)\cup\Gamma_{S,i-1}(\Fs)-C_{S,i,i-1}(\Fs)$. 
\end{remark}

\begin{lemma}\label{Lemma : crossing set}
The crossing set $\mathrm{Cr}_{S}(\{i-1,i\},\Fs)$ (cf. \eqref{eq : Cr_S set})
coincides with
the boundary of $C_{S,i,i-1}(\Fs)$ as a subset of the topological space 
$\Gamma_{S,i-1}(\Fs)\cup\Gamma_{S,i}(\Fs)$.\footnote{That is the 
(topological) relative boundary of $C_{S,i,i-1}(\Fs)$ in 
$\Gamma_{S,i-1}(\Fs)\cup\Gamma_{S,i}(\Fs)$.} 
In particular 
\begin{equation}\label{eq : Cr_Si,i in CSii}
\mathrm{Cr}_{S}(\{i-1,i\},\Fs)\;\subseteq \;C_{S,i,i-1}(\Fs)\;.
\end{equation}

%
\end{lemma}
\begin{proof}
Let us denote by $\Gamma:=\Gamma_{S,i-1}(\Fs)\cup\Gamma_{S,i}(\Fs)$.
If $x\in C_{S,i,i-1}(\Fs)$ is point in the interior of $C_{S,i,i-1}(\Fs)$ in $\Gamma$, 
there is an open neighborhood 
$\Gamma_x$ of $x$ in $\Gamma$ on 
which the $(i-1)$-th and the $i$-th radii coincide. The inverse image of 
$\Gamma_x$ by the retraction \eqref{eq : retract tau i-1,i} 
is an open subset $U$ of $X$ on which the $(i-1)$-th and the $i$-th radii 
coincide. By Definition \ref{Definition : Sep ne}, 
$U$ is then an $\{i-1,i\}$-separating neighborhood of $x$ and 
$x\notin \mathrm{Cr}_{S}(\{i-1,i\},\Fs)$. Since $C_{S,i,i-1}(\Fs)$ is 
closed in $\Gamma$, if $x\in \Gamma - C_{S,i,i-1}(\Fs)$ it has an open 
neighborhood in $\Gamma$ on which the radii are separated and we may 
conclude as before that $x\notin \mathrm{Cr}_{S}(\{i-1,i\},\Fs)$. 
This proves that $\mathrm{Cr}_{S}(\{i-1,i\},\Fs)$ is contained 
in the boundary of $C_{S,i,i-1}(\Fs)$ in $\Gamma$.

Reciprocally, if $x$ is a point of the boundary of 
$C_{S,i,i-1}(\Fs)$ in $\Gamma$, by definition it has no neighborhood contained in 
$C_{S,i,i-1}(\Fs)$ nor in $\Gamma-C_{S,i,i-1}(\Fs)$. It follows that in every 
open neighborhood of $x$ in $\Gamma$, there is at least a point on which 
the $(i-1)$-th and the $i$-th radii coincide and another point where 
they differs. It follows that $x\in \mathrm{Cr}_{S}(\{i-1,i\},\Fs)$.
%
%
\end{proof}
\begin{remark}
Pay attention to the fact that the end points of 
$\Gamma_{S,i-1}(\Fs)\cup\Gamma_{S,i}(\Fs)$ which belong to 
$C_{S,i,i-1}(\Fs)$ but are not isolated points of $C_{S,i,i-1}(\Fs)$ 
do not belong to the boundary of $C_{S,i,i-1}(\Fs)$ as a 
subspace of $\Gamma_{S,i-1}(\Fs)\cup\Gamma_{S,i}(\Fs)$, because it is a boundary with respect to the induced topology of this graph.\footnote{By 
definition $x$ is an end point of 
$\Gamma_{S,i-1}(\Fs)\cup\Gamma_{S,i}(\Fs)$ if there exists a 
neighborhood $U_x$ of $x$ in $X$ such that 
$\Gamma_x:=U_x\cap(\Gamma_{S,i-1}(\Fs)\cup\Gamma_{S,i}(\Fs))$ is 
homeomorphic to the interval $[0,1[$ and the homeomorphism send $x$ 
to $0$. Since $C_{S,i,i-1}(\Fs)$ has a locally finite number of connected 
components, if $x$ is not isolated in $C_{S,i,i-1}(\Fs)$, then we may chose 
$\Gamma_x\subset C_{S,i,i-1}(\Fs)$. Therefore, $x$ is an internal point of 
$C_{S,i,i-1}(\Fs)$ in the topology induced by 
$\Gamma_{S,i-1}(\Fs)\cup\Gamma_{S,i}(\Fs)$.}
\end{remark}

%

%
%

\begin{lemma}[Relation with crossing points]\label{Lemma: connected CrF}
Every connected component of $X_{S,i}(\Fs)$ coincides with 
a connected component of $X-\mathrm{Cr}_{S}(\{i-1,i\},\Fs)$.
%
\end{lemma}
\begin{proof}
Let $C$ be a connected component of $X_{S,i}(\Fs)$. 
By  \eqref{eq : X_S,i = tau-1(gamma-C)} and \eqref{eq : Cr_Si,i in CSii}, $C$ does not contain points of $\mathrm{Cr}_{S}(\{i-1,i\},\Fs)$. Hence, $C$ is contained in a connected component of $X-\mathrm{Cr}_{S}(\{i-1,i\},\Fs)$. 

Let now $Q$ be a connected component of
$X-\mathrm{Cr}_{S}(\{i-1,i\},\Fs)$ such that $Q\cap X_{S,i}(\Fs)\neq\emptyset$. 
Denote by $\Gamma:=\Gamma_{S,i}(\Fs)\cup\Gamma_{S,i-1}(\Fs)$.
Recall that $X_{S,i}(\Fs)$ is the inverse image by the retraction $X\to\Gamma$ of 
the set $\Gamma-C_{S,i,i-1}(\Fs)$ (cf. \eqref{eq : X_S,i = tau-1(gamma-C)}). 
By Lemma \ref{Lemma : crossing set} and  
Remark \ref{Remark : cc of X-Cr}, $Q$ is necessarily 
the inverse image in $X$ of a connected component of 
$\Gamma-\mathrm{Cr}_S(\{i-1,i\},\Fs)$ 
 (cf. \eqref{eq : retraction global}), because $Q$ intersects $X_{S,i}(\Fs)$. 
More precisely, 
$\mathrm{Cr}_S(\{i-1,i\},\Fs)$ is the boundary of $C_{S,i,i-1}(\Fs)$ in $\Gamma$ and $C_{S,i,i-1}(\Fs)$ has a locally finite number of connected components. It follows that $Q$ is either the inverse image in $X$ of a connected component of the interior of $C_{S,i,i-1}(\Fs)$ in $\Gamma$, 
or of a connected component of the complement of 
$C_{S,i,i-1}(\Fs)$ in $\Gamma$.
Since $Q$ intersects $X_{S,i}(\Fs)$, then necessarily $Q$ is inverse image of a connected component of the complement of 
$C_{S,i,i-1}(\Fs)$ in $\Gamma$. Hence $Q\subseteq X_{S,i}(\Fs)$. The claim follows.
%
%
\end{proof}

\begin{lemma}\label{Lemma: cross pts are spectral}
If $x\in\mathrm{Cr}_S(\{i-1,i\},\Fs)$, then $i$ is a spectral index of 
$\Fs$ at $x$.
\end{lemma}
\begin{proof}
If $i$ is oversolvable at $x$, the image of 
$D_{S,i}(y,\Fs)$ in $X$ is a virtual 
open disk $D_i$ in $X$ containing $y$ and $\R_{S,i}(-,\Fs)$ is constant on 
$D_i$ (cf. Remark \ref{rk : solvable not in gamma}). 
Since $x\in\mathrm{Cr}_S(\{i-1,i\},\Fs)$, we have 
$\R_{S,i-1}(x,\Fs)=\R_{S,i}(x,\Fs)$, that is $D_{S,i-1}(y,\Fs)=D_{S,i}(y,\Fs)$. This 
implies that $i-1$ is also oversolvable at $x$, and that $\R_{S,i-1}(-,\Fs)$ is 
constant on $D_i$. Therefore the equality $\R_{S,i-1}(-,\Fs)=\R_{S,i}(-,\Fs)$ 
holds at every point of $D_i$, which is then an $\{i-1,i\}$-separating 
neighborhood of $x$. 
This contradicts the fact that $x\in\mathrm{Cr}_S(\{i-1,i\},\Fs)$ (cf. Definition \ref{Def : Crossing points}).
\end{proof}
An immediate consequence of the above lemma gives
\begin{lemma}\label{Lemma: D_i cap Cr = empty}
Assume that $i$ is an oversolvable index of $\Fs$ at $x\in X$. 
Let $D_i\subseteq X$ be the image of $D_{S,i}(x,\Fs)$ in $X$ 
(cf. Remark \ref{rk : solvable not in gamma}). Then
\begin{equation}
\phantom{.}\qquad\qquad
\mathrm{Cr}_S(\{i-1,i\},\Fs)\cap D_i\;=\;\emptyset\;.\qquad\qquad\qed
\end{equation}
\end{lemma}

\begin{lemma}\label{Lemma: D_S,i in X_S,i}
Let $\pi_{\Omega/K}:X_\Omega\to X$ be the canonical projection (cf. 
\eqref{eq : def of pi_L/K}) and $S_{\Omega}$ the canonical weak 
triangulation on $X_\Omega$ associated with $S$ (cf. \eqref{eq : S_L}). 
Then, the following properties hold
\begin{enumerate}
\item\label{Lemma: D_S,i in X_S,i -i}  One has
\begin{equation}
\pi_{\Omega/K}^{-1}(X_{S,i}(\Fs))
\;=\;(X_{S,i}(\Fs))_\Omega\;=\;(X_\Omega)_{S_\Omega,i}(\Fs)\;;
\end{equation}
\item\label{Lemma: D_S,i in X_S,i -ii} For every $y\in X_{S,i}$ we have
\begin{equation}
D_{S,i}(y,\Fs)\;\subseteq\;(X_{S,i}(\Fs))_{\Omega}\;.
\end{equation}
\end{enumerate}
\end{lemma}
\begin{proof}
Item \eqref{Lemma: D_S,i in X_S,i -i} follows from the fact that the radii are 
insensitive to scalar extension of the ground field $K$ (cf. Proposition 
\ref{Prop: insensitive to scalar ext}). 

Let us prove item \eqref{Lemma: D_S,i in X_S,i -ii}. 
By item \eqref{Lemma: D_S,i in X_S,i -i} we are reduced to 
prove that for every $t\in D_{S,i}(y,\Fs)$ we have 
$\R_{S_\Omega,i-1}(t,\Fs_\Omega)<\R_{S_\Omega,i}(t,\Fs_\Omega)$. 

If $i$ is a spectral index, this is obvious. Indeed, since $y\in X_{S,i}$ then $i$ 
separates the radii at $y$ and hence at $t_y$ 
(cf. Proposition \ref{Prop: insensitive to scalar ext}). 
On the other hand, since $i$ is spectral, we have 
$D_{S,i}(y,\Fs)\subset D(y)$ and therefore the radii are constant 
functions on $D(y)$ by \eqref{eq : D^c}. 

If $i$ is over-solvable, the image of $D_{S,i}(y,\Fs)$ in $X$ is a virtual 
open disk $D_i$ containing $y$ (cf. Remark \ref{rk : solvable not in gamma}). In this case, 
since the radii are insensitive to scalar extension it is enough to show that 
$i$ separates the $S$-radii of $\Fs$ at every point of $D_i$. 
We may proceed by contrapositive. Let $z\in D_i$ be a point such that 
$\R_{S,i-1}(z,\Fs)=\R_{S,i}(z,\Fs)$. This implies that  $i-1$ is 
oversolvable at $z$ and $D_{S,i-1}(z,\Fs)=D_{S,i}(z,\Fs)$. 
Again by \eqref{eq : D^c}, both the $(i-1)$-th and the $i$-th $S$-radii 
are constants on $D_i$ and distinct at $y$, hence distinct everywhere on 
$D_i$. This contradicts the equality $\R_{S,i-1}(z,\Fs)=\R_{S,i}(z,\Fs)$. The 
claim follows.
\end{proof}

\begin{proposition}\label{Prop. loc to X_Si triang-hh}
There exists a weak-triangulation $S_i$ of $X_{S,i}$ such that 
$S_i\subset \Gamma_{S,i-1}(\Fs)\cup\Gamma_{S,i}(\Fs)$ and
\begin{equation}
\Gamma_{S_i}\;\subseteq\;
\Gamma_{S_i,i-1}(\Fs_{|X_{S, i}})\cup \Gamma_{S_i,i}(\Fs_{|X_{S, i}})\;=\;
(\Gamma_{S,i-1}(\Fs)\cup\Gamma_{S,i}(\Fs))\cap X_{S,i}\;.
\end{equation}
Moreover, the radii of $\Fs$ may change value by localization 
to $(X_{S,i},S_i)$, however the first $i$ radii are not truncated by the 
change of weak-triangulation. 
In other words, for all $x\in X_{S,i}$ and all $j=1,\ldots,i$ 
we have 
\begin{equation}\label{eq : localization to Gii -XSI}
D_{S_i,j}(x,\Fs_{|X_{S,i}})\;=\;
D_{S,j}(x,\Fs)\;.
\end{equation}
In particular, for all 
$x\in X_{S,i}$  we have 
$\R_{S_i,i-1}(x,\Fs_{|X_{S,i}})<
\R_{S_i,i}(x,\Fs_{|X_{S,i}})$ and, for all $j=1,\ldots,r=\mathrm{rank}(\Fs)$, 
we have
\begin{equation}\label{eq : incl gr afer lox gii -XSI}
(\Gamma_{S,j}(\Fs)\cap X_{S,i})\cup\Gamma_{S_i}\;=\;
\Gamma_{S_i,j}(\Fs_{|X_{S,i}})\;.
\end{equation}
\end{proposition}
\begin{proof}
Since $\mathrm{Cr}_S(\{i-1,i\},\Fs)$ is locally finite and contained in the 
locally finite graph $\Gamma_{S,i-1}(\Fs)\cup\Gamma_{S,i}(\Fs)$, there 
exists a minimum weak-triangulation $S'$ of $X$ containing $S$ and
$\mathrm{Cr}_S(\{i-1,i\},\Fs)$. Indeed, 
we may consider the locally finite graph $\Gamma$ in 
$\Gamma_{S,i-1}(\Fs)\cup\Gamma_{S,i}(\Fs)$ obtained 
by adding to $\Gamma_S$ the segments joining every point of 
$\mathrm{Cr}_S(\{i-1,i\},\Fs)$ to $\Gamma_S$. Since 
the end points of 
$\Gamma_{S,i-1}(\Fs)\cup\Gamma_{S,i}(\Fs)$ are not of
type $1$ (cf. \cite[Theorem 3.9]{NP-I}), nor of type $4$ by a 
result of K.S.Kedlaya (cf. \cite[Theorem~4.5.15]{Kedlaya-draft}), 
then $\Gamma$ has no end points of type $1$ nor $4$ too. 
Therefore, we may obtain 
$S'$ just by adding to $S$ the end points and the bifurcation points of 
$\Gamma$, as well as the points of $\mathrm{Cr}_S(\{i-1,i\},\Fs)$. 
Clearly $\Gamma=\Gamma_{S'}\subseteq
\Gamma_{S,i-1}(\Fs)\cup\Gamma_{S,i}(\Fs)$. 
\begin{lemma}\label{Lemma: intermediate triang S' for X_S,i}
The following properties hold
\begin{enumerate}
\item\label{Lemma: intermediate triang S' for X_S,i - i} 
If $x\in\Gamma_{S'}$, then $i$ is a spectral index for $\Fs$ with respect to 
both $S$ and $S'$.
\item\label{Lemma: intermediate triang S' for X_S,i - ii} 
For all $x\in X$ and all $j\leq i$ we have
\begin{equation}\label{eq : localization to Gii -XSI-bis-bis}
D_{S',j}(x,\Fs)\;=\;D_{S,j}(x,\Fs)\;;
\end{equation}
\item\label{Lemma: intermediate triang S' for X_S,i - iii} 
Let $x\in X$, then $i$ separates the $S$-radii of $\Fs$ at $x$ if, 
and only if, it separates the $S'$-radii of $\Fs$ at $x$; 
\item\label{Lemma: intermediate triang S' for X_S,i - iv} One has
\begin{equation}\label{eq : X_S,i=X_S',i}
X_{S,i}\;=\;X_{S',i}\;,\qquad
C_{S,i,i-1}(\Fs)\;=\;C_{S',i,i-1}(\Fs)\;,\qquad 
\textrm{Cr}_{S}(\{i-1,i\},\Fs)\;=\;\textrm{Cr}_{S'}(\{i-1,i\},\Fs)\;.
\end{equation}
\end{enumerate}
Moreover for all $j=1,\ldots,r$ we have 
$\Gamma_{S',j}(\Fs)=\Gamma_{S,j}(\Fs)\cup\Gamma_{S'}$ 
(cf. \eqref{eq: change of tri eq 2.32}).
\end{lemma}
\begin{proof}
The fact that the $i$-th $S'$-radius (resp. $S$-radius) is spectral at every 
$x\in\Gamma_{S'}$ (resp. $x\in\Gamma_S$) is immediate, 
because $D(x,S')=D(x)$ for all $x\in\Gamma_{S'}$ (resp. $D(x,S)=D(x)$ 
for all $x\in\Gamma_S$).  
Therefore, we only need to prove that if $x\in\Gamma_{S'}-\Gamma_S$, 
then $i$ is spectral with respect to $S$. 
We may proceed by contrapositive. 
Assume that $i$ is oversolvable at $x$ with respect to $S$, 
then the image of $D_{S,i}(x,\Fs)$ in $X$ is  
a virtual open disk $D_i$ containing $x$ and such that 
$D_i\cap\Gamma_S=\emptyset$ 
(cf. Remark \ref{rk : solvable not in gamma}). 
In $D_i$ there are no points of $\mathrm{Cr}_S(\{i-1,i\},\Fs)$ 
(cf. Lemma \ref{Lemma: D_i cap Cr = empty}) and therefore 
$x\notin\Gamma_{S'}$, which is absurd. 
Item \eqref{Lemma: intermediate triang S' for X_S,i - i} follows.

Let us prove item \eqref{Lemma: intermediate triang S' for X_S,i - ii}. 
By \eqref{Prop: A-4 fgh -eq} and \eqref{disks}, 
it is enough to prove it for $j=i$.

If $i$ is spectral with respect to $S$ at $x$, 
then $D_{S,i}(x,\Fs)\subseteq D(x)$ and by 
\eqref{Prop: A-4 fgh -eq} we deduce 
\eqref{eq : localization to Gii -XSI-bis-bis}. 

If $i$ is oversolvable at $x$ with respect to $S$, then 
$x\notin\Gamma_{S'}$ by item 
\eqref{Lemma: intermediate triang S' for X_S,i - i}. 
Let $D_i$ be the image in $X$ of $D_{S,i}(x,\Fs)$. Recall that 
$D_i\cap\Gamma_S=\emptyset$ 
(cf. Remark \ref{rk : solvable not in gamma}).
By Lemma \ref{Lemma: D_i cap Cr = empty}, 
$D_i\cap\mathrm{Cr}_S(\{i-1,i\},\Fs)=\emptyset$.
It follows that $D_i\cap \Gamma_{S'}=\emptyset$ by the 
definition of $\Gamma_{S'}$. 
Therefore, $D_i\subseteq D(x,S')$ and 
\eqref{eq : localization to Gii -XSI-bis-bis} follows from 
\eqref{Prop: A-4 fgh -eq}.
The equalities \eqref{eq : X_S,i=X_S',i} are now direct consequences of the 
definitions.
\end{proof}
\emph{Continuation of Proof of 
Proposition \ref{Prop. loc to X_Si triang-hh}.} 
By Lemma \ref{Lemma: intermediate triang S' for X_S,i} and by the definition of $S'$ we have
\begin{equation}\label{eq : Cr in S}
\textrm{Cr}_{S'}(\{i-1,i\},\Fs)\;\subseteq\; S'\;.
\end{equation}
By \eqref{eq : X_S,i = tau-1(gamma-C)} and 
Lemma \ref{Lemma : crossing set}, every connected 
component of $X_{S',i}$ is a connected component of 
$X -\mathrm{Cr}_{S'}(\{i-1,i\},\Fs)$ 
(cf. Lemma \ref{Lemma: connected CrF}). 
Together with \eqref{eq : Cr in S} this implies that $S_i:=S'\cap X_{S',i}$ 
is a weak-triangulation of $X_{S',i}$ and that 
$\Gamma_{S_i}=\Gamma_{S'}\cap X_{S',i}$.

Now, by \eqref{eq : Res to Y equal radii f} and 
\eqref{eq : Res to Y equal graphs f}, for all $k=1,\ldots,r$ 
and all $x\in X_{S',i}$we have
\begin{equation}
\R_{S',k}(x,\Fs)\;=\;
\R_{S_i,k}(x,\Fs_{|X_{S',i}})\;,\qquad
\Gamma_{S',k}(\Fs)\;=\;
\Gamma_{S_i,k}(\Fs_{|X_{S',i}})\;.
\end{equation}
Since $X_{S,i}=X_{S',i}$ the remaining claims are now straightforward.
\end{proof}

\subsubsection{Existence of $\Fs_{\geq i}$.}\label{subsubsection : existence of fi}
We here study the existence of $\Fs_{\geq i}$ over some regions of 
$X$. 

\begin{proposition}\label{Prop. Existence on X_Si}
The restriction $\Fs_{|X_{S,i}(\Fs)}$ of $\Fs$ to $X_{S,i}=X_{S,i}(\Fs)$ 
admits a unique sub-object 
\begin{equation}
(\Fs_{|X_{S,i}})_{\geq i} \;\subseteq\;\Fs_{|X_{S,i}}
\end{equation}
of constant rank $r-i+1$ such that, for all $y\in X_{S,i}$, one has
\begin{equation}
\Hdr^0(y,(\Fs_{|X_{S,i}})_{\geq i})\;=\;\omega_{S,i}(y,\Fs)\;.
\end{equation}
Let $S_i$ be the weak-triangulation obtained in Proposition 
\ref{Prop. loc to X_Si triang-hh}. Then, for all $k=1,\ldots,i$ and all $y\in X_{S,i}(\Fs)$ we have
\begin{equation}
\omega_{S_i,k}(y,\Fs_{|X_{S,i}})\;=\;\omega_{S,k}(y,\Fs)\;.
\end{equation}
For all $j=1,\ldots,r-i+1$ and all $y\in X_{S,i}$, the canonical inclusion 
$\Hdr^0(y,(\Fs_{|X_{S,i}})_{\geq i})\subset\Hdr^0(y,\Fs_{|X_{S,i}})$ identifies 
\begin{equation}\label{eq: omegaS_i,j=omegaS,j}
\omega_{S_i,j}(y,(\Fs_{|X_{S,i}})_{\geq i})\;=\;
\omega_{S_i,j+i-1}(y,\Fs_{|X_{S,i}})\;.
\end{equation}

Set $(\Fs_{|X_{S,i}})_{<i}:=\Fs_{|X_{S,i}}/(\Fs_{|X_{S,i}})_{\geq i}$.
Then, for all $y\in X_{S,i}$, one has 
\begin{equation}\label{eq : equality of radii of F_Si}
\R_{S_i,j}(y,\Fs_{|X_{S,i}})\;=\;\left\{\begin{array}{lcl}
\R_{S_i,j}(y,(\Fs_{|X_{S,i}})_{<i})&\textrm{ if }&j=1,\ldots,i-1\\
\R_{S_i,j-i+1}(y,(\Fs_{|X_{S,i}})_{\geq i})&\textrm{ if }&j=i,\ldots,r\;.\\
\end{array} \right.
\end{equation}
\end{proposition}
\begin{proof}
By Proposition \ref{Prop. loc to X_Si triang-hh},  the weak-triangulation 
$S_i$ on $X_{S,i}(\Fs)$ is such that the index $i$ separates the $S_i$-radii 
the restriction $\Fs_{|X_{S,i}}$. In particular \eqref{eq: omegaS_i,j=omegaS,j} 
follows from \eqref{eq : localization to Gii -XSI}. The remaining assertions 
are then a direct application of Theorem \ref{MAIN Theorem}. 
\end{proof}

\begin{remark}\label{Rk : no need of S_i}
The existence of $S_i$ is interesting in itself, however it is not necessary for 
the existence and uniqueness of $(\Fs_{|X_{S,i}(\Fs)})_{\geq i}$.
Indeed, the (global) radii $\{\R_{S,i}(-,\Fs)\}_i$ induce a filtration 
$\{\omega_{S,i}(x,\Fs)\}_i$ on $\Hdr^0(x,\Fs)$. 
From that filtration we have defined in 
Proposition \ref{Prop : existence of (F_S,i)_x} the augmented 
Dwork-Robba filtration of $\Fs_x$ for all $x\in X$.  
Now  $X_{S,i}(\Fs)$ is the region where 
the $i$-th radius stays separated from the $(i-1)$-th one. 
So the gluing process of the proof of Theorem \ref{MAIN Theorem} 
works over such $X_{S,i}(\Fs)$, and it gives 
a sub-object of $\Fs_{|X_{S,i}}$. 
Notice that the radii considered here are those of $\Fs$ before 
localization to $X_{S,i}$. 

In this case, the global uniqueness of $(\Fs_{|X_{S,i}})_{\geq i}$ follows from 
the local uniqueness (cf. Proposition \ref{Prop : existence of (F_S,i)_x}). 
Indeed, if $(\Fs_{|X_{S,i}})_{\geq i}'$ is another sub-object of $\Fs_{|X_{S,i}}$ 
with the same properties, it is enough to prove that the composite map 
$(\Fs_{|X_{S,i}})_{\geq i}'\to\Fs_{|X_{S,i}}/(\Fs_{|X_{S,i}})_{\geq i}$ is zero. This 
is a local matter which follows from the fact that we locally have 
$(\Fs_{|X_{S,i}})_{\geq i}=(\Fs_{|X_{S,i}})_{\geq i}'$.
\end{remark}

\begin{remark}\label{remark : }
Let $1\leq i_1<i_2<\cdots<i_h\leq r=\mathrm{rank}(\Fs)$ 
be a set of indices.
Let $Y\subseteq X$ be a connected analytic domain such that 
for all $k=1,\ldots,h$ the index $i_k$ separates the $S$-radii of $\Fs$ at 
every point of $Y$. Then, $Y$ is contained in a connected component of 
$X-\bigcup_{k=1}^h\mathrm{Cr}_S(\{i_{k-1},i_k\},\Fs)$. 
Indeed, by the definition of the $i$-th separating open $X_{S,i}(\Fs)$, we 
have $Y\subseteq X_{S,i_1}\cap X_{S,i_2}\cap \cdots\cap X_{S,i_h}$. More 
precisely, for every $k=1,\ldots,h$, $Y$ is contained in a connected 
component of $X_{S,i_k}$, which coincides with a connected component of 
$X-\mathrm{Cr}_S(\{i_k-1,i_k\},\Fs)$ by Lemma \ref{Lemma: connected CrF}.

On the other hand, every connected component $Q$ of $X-\bigcup_{k=1}^h\mathrm{Cr}_S(\{i_{k-1},i_k\},\Fs)$ is an open subset of $X$ such that 
for all $k=1,\ldots,h$ one of the two conditions hold: 
\begin{itemize}
\item For all $y\in Q$ we have $\R_{S,i_k-1}(y,\Fs)<\R_{S,i_k}(y,\Fs)$;
\item For all $y\in Q$ we have $\R_{S,i_k-1}(y,\Fs)=\R_{S,i_k}(y,\Fs)$.
\end{itemize}
\end{remark}
\begin{theorem}\label{Prop : Crossing points shsh-1}
Let $1\leq i_1<i_2<\cdots<i_h\leq r=\mathrm{rank}(\Fs)$ 
be a set of indices.
Let $Y\subseteq X$ be a connected analytic domain such that 
for all $k=1,\ldots,h$ the index $i_k$ separates the $S$-radii of $\Fs$ at 
every point of $Y$.
The restriction $\Fs_{|Y}$ of $\Fs$ to $Y$ admits a unique filtration 
\begin{equation}\label{eq : (5.20)}
0\;\neq\; (\Fs_{|Y})_{\geq i_h}\; 
\subsetneq\; (\Fs_{|Y})_{\geq i_h-1}\; 
\subsetneq\; \cdots\; 
\subsetneq\; (\Fs_{|Y})_{\geq i_1}
\;=\; \Fs_{|Y}
\end{equation} 
such that for all $y\in Y$ and all $k=1,\ldots,h$ the rank of 
$(\Fs_{|Y})_{\geq i_k}$ is $r-i_k+1$ and one has
\begin{equation}
\Hdr^0(y,(\Fs_{|Y})_{\geq i_k})\;=\;\omega_{S,i_k}(y,\Fs)\;.
\end{equation}
\end{theorem}
\begin{proof}
\if{
Let $i$ be any index. 
Every connected components of $X-\mathfrak{F}$ is 
contained in a connected component of $X-\mathrm{Cr}_S(\{i-1,i\},\Fs)$.
By Lemma 
\ref{Lemma: connected CrF}, $Y$ is either contained in a connected 
component of $X_{S,i}(\Fs)$ or $Y\cap X_{S,i}(\Fs)=\emptyset$.
By assumption, $Y$ is contained in 
$X_{S,i_1}(\Fs)\cap X_{S,i_2}(\Fs)\cap\cdots\cap X_{S,i_h}(\Fs)$. 
}\fi
The claim then follows directly 
form Proposition \ref{Prop. Existence on X_Si}.
\end{proof}
\begin{corollary}\label{Prop : Crossing points shsh}
Let $\mathfrak{F}$ be a locally finite subset of $X$ containing 
$\mathrm{Cr}_S(\Fs)$ (cf. Definition \ref{Def : Crossing points}). 
Let $Y$ be a connected analytic domain of $X$
not intersecting $\mathfrak{F}$. Then, for every index $i\in\{1,\ldots,\mathrm{rank}(\Fs)\}$, one of the two conditions hold : 
\begin{itemize}
\item For all $y\in Y$ we have $\R_{S,i-1}(y,\Fs)<\R_{S,i}(y,\Fs)$;
\item For all $y\in Y$ we have $\R_{S,i-1}(y,\Fs)=\R_{S,i}(y,\Fs)$.
\end{itemize}
Let $1\leq i_1<i_2<\cdots<i_h\leq\mathrm{rank}(\Fs)$ be the indexes separating the $S$-radii of $\Fs$ 
over $Y$. Then Theorem 
\ref{Prop : Crossing points shsh-1} applies to $\Fs_{|Y}$.\hfill\qed
\end{corollary}
\subsubsection{Direct sum decomposition: 
condition on $\Fs$ and $\Fs^*$.}
Here we provide conditions to have a direct sum decomposition. 
We firstly consider the open subset 
\begin{equation}\label{eq : XSiFF*}
X_{S,i}(\Fs,\Fs^*) \;:=\; X_{S,i}(\Fs)\cap X_{S,i}(\Fs^*)
\end{equation} 
over which $i$ separates the radii of both $\Fs$ and $\Fs^*$. Set 
\begin{equation}
\Gamma\;:=\;
\Gamma_{S,i}(\Fs)\cup\Gamma_{S,i}(\Fs^*)\cup 
\Gamma_{S,i-1}(\Fs)\cup\Gamma_{S,i-1}(\Fs^*)\;.
\end{equation} 

\begin{remark}
By definition we have $X_{S,1}(\Fs,\Fs^*)=X$. 

If $X_{S,i}(\Fs,\Fs^*)=X$, then the index $i$ separates globally on $X$ 
the radii of $\Fs$ and $\Fs^*$, and 
Theorem \ref{Thm : 5.7 deco good direct siummand} applies.

If $X_{S,i}(\Fs,\Fs^*)=\emptyset$, there is no 
regions where Theorem \ref{Thm : 5.7 deco good direct siummand} 
applies, even though we may have a non empty region 
$X_{S,i}(\Fs)$ where $\Fs$ decomposes.

 If $X_{S,i}(\Fs,\Fs^*)\neq \emptyset,X$, 
then arguing as in Remark \ref{Rk : Y subset X_S,i} we have 
$\Gamma\neq\emptyset$,
and in this case $X_{S,i}(\Fs,\Fs^*)$ is the inverse image by the canonical retraction 
\begin{equation}
\tau_{\Gamma}\;:\;X\;\xrightarrow{\quad}\; 
\Gamma
\end{equation}
of a possibly not connected open sub-graph. More precisely, let us 
denote by 
\begin{equation}
C_{S,i,i-1}(\Fs,\Fs^*)\;\;\subseteq\;\; 
\Gamma
\end{equation}
 the subset of points $x\in\Gamma$ such that we have either
$\R_{S,i-1}(x,\Fs)=\R_{S,i}(x,\Fs)$ or $\R_{S,i-1}(x,\Fs)=\R_{S,i}(x,\Fs)$ or both. If  $f:X\to\mathbb{R}^2_{\geq 0}$ is the function associating to $x$ the pair $f(x)=(\R_{S,i}(x,\Fs)-\R_{S,i-1}(x,\Fs),\R_{S,i}(x,\Fs^*)-\R_{S,i-1}(x,\Fs^*))$, then $C_{S,i,i-1}(\Fs,\Fs^*)$ is the inverse image of the union $(\mathbb{R}_{\geq 0}\times \{0\})\cup(\{0\}\times\mathbb{R}_{\geq 0})$. By 
definition we have
\begin{equation}
C_{S,i,i-1}(\Fs,\Fs^*)\;=\;C_{S,i,i-1}(\Fs)\bigcup C_{S,i,i-1}(\Fs^*)\;.
\end{equation}
Arguing as in Remark \ref{Rk : Y subset X_S,i}, 
we deduce that $C_{S,i,i-1}(\Fs,\Fs^*)$ is closed in $\Gamma$, locally 
compact, and it has a locally finite number of connected components. 
Therefore, 
$\Gamma-C_{S,i,i-1}(\Fs,\Fs^*)$ is an 
open subset of the graph $\Gamma$ and we obviously have 
\begin{equation}
X_{S,i}(\Fs,\Fs^*)\;=\;\tau_{\Gamma}^{-1}\Bigl(\Gamma-C_{S,i,i-1}(\Fs,\Fs^*)\Bigr)\;.
\end{equation}
In particular, the connected components of $X_{S,i}(\Fs,\Fs^*)$ are in 
bijection with those of 
$\Gamma-C_{S,i,i-1}(\Fs,\Fs^*)$.
\end{remark}
The following Proposition gives a condition to have direct sum 
decomposition. The dichotomy of the statement reflects 
that of 
Theorem \ref{Thm : 5.7 deco good direct siummand}.

\begin{proposition}\label{Prop : direct sum by duality liocal}
Let $\Fs$ be a differential equation over $X$, and let 
$Z$ be a connected component of $X_{S,i}(\Fs,\Fs^*)$. 
Assume that $Z$ satisfies one of the following 
\begin{enumerate}
\item\label{Prop : direct sum by duality liocal-i} $Z\cap\bigl(\Gamma_{S,i}(\Fs)\cup 
\Gamma_{S,i}(\Fs^*)\bigr)\neq\emptyset$;
\item\label{Prop : direct sum by duality liocal-ii} $Z\cap\bigl(\Gamma_{S,i}(\Fs)\cup 
\Gamma_{S,i}(\Fs^*)\bigr)=\emptyset$, and there exists 
$x\in Z$ such that $i-1$ is spectral non solvable at $x$
either for $\Fs$, or for $\Fs^*$.
\end{enumerate}
Then $(\Fs_{|Z})_{\geq i}$ is a direct summand of $\Fs_{|Z}$.
\end{proposition}
\begin{proof}
By Proposition \ref{Prop. Existence on X_Si}, both 
$(\Fs_{|Z})_{\geq i}$ and $(\Fs_{|Z}^*)_{\geq i}$ exist over $Z$.
It is enough to prove that the canonical 
composite morphism $c:(\Fs_{|Z}^*)_{\geq i}\subseteq 
\Fs_{|Z}^*\to((\Fs_{|Z})_{\geq i})^*$ is an isomorphism. 
As in the proof of Theorem 
\ref{Thm : 5.7 deco good direct siummand}, 
we only need to find a point $x\in Z$ at which the index $i-1$ is 
spectral non solvable for $\Fs$ or $\Fs^*$ 
(cf. Proposition \ref{Prop. : iso on a point implies iso global}).   
In the case \eqref{Prop : direct sum by duality liocal-ii}, such a point 
exists by assumption.

In case \eqref{Prop : direct sum by duality liocal-i} such a point also 
exists because, by Remark \ref{Remark : R_i solvable on Gamma_i}, 
the radius $\R_{S,i}(x,\Fs)$ is spectral for $\Fs$ (resp. $\Fs^*$) at 
each point $x$ of $\Gamma_{S,i}(\Fs)$ (resp. $\Gamma_{S,i}(\Fs^*)$) and hence $\R_{S,i-1}(x,\Fs)$ is spectral non solvable at $x$, 
because $i$ separates the radii over $Z$. 
\end{proof}
\begin{theorem}\label{Prop : direct sum local- finite-1}
Let $1\leq i_1<i_2<\cdots<i_h\leq r=\mathrm{rank}(\Fs)$ 
be a set of indices.
Let $Y\subseteq X$ be a connected analytic domain such that 
for all $k=1,\ldots,h$ the index $i_k$ separates the $S$-radii of $\Fs$ and 
$\Fs^*$ at every point of $Y$. Assume that for every $k=1,\ldots,h$ 
one of the conditions of Proposition \ref{Prop : direct sum by duality liocal} 
are satisfied by $Z=Y$. 
Then, the terms of the filtration \eqref{eq : (5.20)} at the indexes 
$\{i_{n_k}\}_{k=1,\ldots,s}$ define direct summands of 
$\Fs_{|Y}$:
\begin{equation}
\Fs_{|Y}\;=\;
(\Fs_{|Y})_{\geq i_{n_k}}
\bigoplus
(\Fs_{|Y})_{<i_{n_k}}\;,\qquad k=1,\ldots,s\;
\end{equation}
and we have
\begin{equation}
\phantom{.}\qquad\qquad
\Fs_{|Y}\;=\;
\frac{\Fs_{|Y}}{(\Fs_{|Y})_{<i_{n_s}}}
\oplus
\frac{(\Fs_{|Y})_{<i_{n_s}}}{(\Fs_{|Y})_{<i_{n_{s-1}}}}
\oplus
\cdots
\oplus
\frac{(\Fs_{|Y})_{<i_{n_2}}}{(\Fs_{|Y})_{<i_{n_1}}}
\oplus
(\Fs_{|Y})_{<i_{n_1}}\;.\qquad\qed
\end{equation}
\end{theorem}

\begin{theorem}\label{Prop : direct sum local- finite}
Let $\mathfrak{F}$ be a locally finite subset of $X$ containing 
$\mathrm{Cr}_S(\Fs)\cup\mathrm{Cr}_S(\Fs^*)$.
The statement of Corollary \ref{Prop : Crossing points shsh} holds for 
$(\Fs,\mathfrak{F})$.

Let $Y$ be a connected component of 
$X-\mathfrak{F}$.

For every given index $i\in\{1,\ldots,\mathrm{rank}(\Fs)\}$, at least 
one of the following conditions hold:
\begin{enumerate}
\item $i$ separates both the radii of $\Fs$ and $\Fs^*$ at every point 
of $Y$;
\item for every $x\in Y$ we have 
$\R_{S,i-1}(x,\Fs)=\R_{S,i}(x,\Fs)$; 
\item for every $x\in Y$ we have  
$\R_{S,i-1}(x,\Fs^*)=\R_{S,i}(x,\Fs^*)$.
\end{enumerate}

Let $1=i_1<i_2<\ldots<i_h=\mathrm{rank}(\Fs)$ be the 
indexes separating the radii 
$\{\R_{S,i}(-,\Fs)\}_i$ over $Y$ (cf. \eqref{eq : (5.20)}) 
and let 
$1=i_{n_1}<i_{n_2}<\ldots<i_{n_s}$ be the subset of indexes 
such that for every $k=1,\ldots,s$ both the following conditions hold
\begin{itemize}
\item the index $i_{n_k}$ separates both the radii of $\Fs$ and of $\Fs^*$ over $Y$;
\item  at least one of the two conditions of Proposition \ref{Prop : direct sum by duality liocal} holds for  $Z=Y$.
\end{itemize}
Then, Theorem \ref{Prop : direct sum local- finite-1} applies.
\end{theorem}
\begin{proof}
By Lemma
\ref{Lemma: connected CrF}, for a given index $i$, $Y$ is either contained in a connected component of $X_{S,i}(\Fs,\Fs^*)$ or 
it is disjoint from all of them.

Assume that there exists a connected component $Z$  of 
$X_{S,i}(\Fs,\Fs^*)$ containing $Y$. 
If $Z$ verifies one of the conditions of Proposition 
\ref{Prop : direct sum by duality liocal}, then the term 
$(\Fs_{|Y})_{\geq i}$ of the filtration 
\eqref{eq : (5.20)} of $\Fs$ over $Y$ is a direct 
summands of $\Fs_{|Y}$.
\end{proof}

\subsubsection{Direct sum decomposition: conditions on the 
controlling graphs.}
In analogy with Theorem \ref{Thm : criterion self direct sum}, 
we now provide conditions on the controlling graphs that guarantee the 
direct sum decomposition on some regions of $X$. We proceed in 
analogy with Remark \ref{Remark : remove maximal disks to have ...}
to obtain the maximal analytic domain in $X$ where 
\eqref{eq : condition on graphs included} is satisfied.

Denote by 
\begin{equation}
X_{\langle i \rangle}\subseteq X\;.
\end{equation}
the analytic domain obtained by removing from $X$ 
the connected components of $X-\Gamma_{S,i}(\Fs)$ 
intersecting some $\Gamma_{S,j}(\Fs)$ with $j=1,\ldots,i-1$.
By the finiteness theorem of \cite{NP-I,NP-II}, such connected 
components form a locally finite family of virtual open disks. 

\begin{proposition}\label{Prop. loc to X_langle i rangle}
There exists a weak-triangulation $S_{\langle i\rangle}$ of 
$X_{\langle i\rangle}$ such that 
\begin{equation}
\Gamma_{S_{\langle i\rangle}}=\Gamma_{S_{\langle i\rangle},i}(\Fs_{|X_{\langle i\rangle}})=\Gamma_{S,i}(\Fs)\;.
\end{equation}
Moreover, the radii of $\Fs$ may change value by localization 
to $X_{\langle i\rangle}$, however the first $i$-radii are not truncated 
by the localization process. 
In other words, for all $x\in X_{\langle i\rangle}$ and all $j=1,\ldots,i$ 
we have 
\begin{equation}\label{eq : localization to Gii}
D_{S_{\langle i\rangle},j}(x,\Fs_{|X_{\langle i\rangle}})\;=\;
D_{S,j}(x,\Fs)\;.
\end{equation}
In particular, for all 
$x\in X_{\langle i\rangle}$  we have 
$\R_{S,i}(x,\Fs)>\R_{S,i-1}(x,\Fs)$ if, and only if,
$\R_{S_{\langle i\rangle},i}(x,\Fs_{|X_{\langle i\rangle}})>
\R_{S_{\langle i\rangle},i-1}(x,\Fs_{|X_{\langle i\rangle}})$.

Therefore, for all $j=1,\ldots,i$, we have
\begin{equation}\label{eq : incl gr afer lox gii}
\Gamma_{S,j}(\Fs)\cap X_{\langle i\rangle}\;=\;
\Gamma_{S_{\langle i\rangle},j}(\Fs_{|X_{\langle i\rangle}})
\;\subseteq\;
\Gamma_{S_{\langle i\rangle},i}(\Fs_{|X_{\langle i\rangle}})\;=\;
\Gamma_{S,i}(\Fs)\;.
\end{equation}
\end{proposition}
\begin{proof}
$X-S$ is a disjoint union of virtual open disks and virtual open annuli. 
By Propositions \ref{Prop: Localization} and \ref{Prop : immersion}, the localization to one of these 
disks or annuli $C$ does not change the value of the radii and 
$\Gamma_{S,i}(\Fs)\cap C=\Gamma_{\emptyset,i}(\Fs_{|C})$. 
Moreover, by \cite{NP-I,NP-II}, $\Gamma_{S,i}(\Fs)$ is a locally finite 
graph in $X$. 
This implies that, if we increase $S$ by adding 
a finite number of points in 
every connected component of $X-S$ intersecting 
$\Gamma_{S,i}(\Fs)$, the set so obtained will be locally finite.
Therefore, without loss of generality, 
we may assume that $X$ is a virtual open annulus or disk with empty 
weak triangulation.
It is then enough to define $S_{\langle i\rangle}$ as 
the set of points of $X$ that are either end points of 
$\Gamma_{S,i}(\Fs)$ or bifurcation points of it. 
It is a locally finite set 
because $\Gamma_{S,i}(\Fs)$ is locally finite. The end points are not 
type $1$ (cf. \cite[Theorem 3.9]{NP-I}), nor type $4$ points by a 
result of K.S.Kedlaya (cf. \cite[Theorem~4.5.15]{Kedlaya-draft}). The existence of $S_{\langle i\rangle}$ 
follows.
Equality \eqref{eq : localization to Gii} is now a consequence 
of \eqref{eq : D^c} and \eqref{eq : incl gr afer lox gii} 
is a consequence of Proposition \ref{Prop : D^c}. The claim follows.
\end{proof}
It follows from Proposition \ref{Prop. loc to X_langle i rangle} 
and the definition of $X_{S,i}(\Fs)$ (cf. \eqref{Def : eq : X_SiF})
that
\begin{equation}
X_{S,i}(\Fs)\cap X_{\langle i\rangle}\;=\;
(X_{\langle i\rangle})_{S_{\langle i\rangle},i}(\Fs_{|X_{\langle i\rangle}})\;.
\end{equation}
That is, the locus in $X$ where the index $i$ separates the $S$-radii 
intersected with $X_{\langle i\rangle}$ coincides with the locus in 
$X_{\langle i\rangle}$ where the index $i$ separates the 
$S_{\langle i\rangle}$-radii. 

\begin{proposition}\label{Prop. deco graph dir dsum}
Let $Y$ be an analytic domain in $X$ contained in 
$X_{S,i}(\Fs)\cap X_{\langle i\rangle}$.
The sub-object $(\Fs_{|Y})_{\geq i}$ obtained in 
Corollary \ref{Prop : Crossing points shsh} 
(using the radii $\{\R_{S,k}(-,\Fs)\}_{k}$) is a direct summand of 
$\Fs_{|Y}$.
\end{proposition}
\begin{proof}
We may first restrict $\Fs$ to $X_{\langle i\rangle}$ and then to $Y$.
Proposition \ref{Prop. loc to X_langle i rangle} ensures that the first $i$ 
radii remain intact by localization to $X_{\langle i\rangle}$ and the 
index $i$ separates the $S_{\langle i\rangle}$-radii of 
$\Fs_{X_{\langle i\rangle}}$ over 
$X_{\langle i\rangle}\cap X_{S,i}(\Fs)$. Moreover, by local uniqueness, 
the object $(\Fs_{|(X_{\langle i\rangle})_{S_{\langle i\rangle},i}(\Fs_{|X_{\langle i\rangle}})})_{\geq i}$ 
obtained from $\Fs_{|X_{\langle i\rangle}}$ 
applying Corollary \ref{Prop : Crossing points shsh} 
to the $S_{\langle i\rangle}$-radii of $\Fs_{|X_{\langle i\rangle}}$ 
coincides with the restriction of $(\Fs_{|X_{S,i}(\Fs)})_{\geq i}$ to 
$X_{\langle i\rangle}\cap X_{S,i}(\Fs)$. Indeed, they are obtained by 
gluing the same local modules (cf. Remark \ref{Prop : Crossing points shsh}).

It follows that, without loss of generality, we can assume 
$X=X_{\langle i\rangle}$. That is, we can assume that $\Fs$ is a 
differential equation on $X$ satisfying 
\begin{equation}
\Gamma_S=\Gamma_{S,j}(\Fs)\;,\qquad\textrm{for all }j=1,\ldots,i\;.
\end{equation}
In particular the inclusion of graphs 
\eqref{eq : condition on graphs included} is realized.\footnote{However, 
Theorem \ref{Thm : criterion self direct sum} does not apply directly because $i$ does not 
necessarily separate the radii globally on $X$.} 
For such an equation we want to show that 
the sub-object $(\Fs_{|X_{S,i}(\Fs)})_{\geq i}$ furnished by 
Proposition \ref{Prop. Existence on X_Si} 
is a direct factor of $\Fs_{|X_{S,i}(\Fs)}$.

By Corollary \ref{Cor. dual compatibility}, one has
\begin{equation}
\Gamma_S=\Gamma_{S,j}(\Fs^*)\;,\qquad
\textrm{for all }j=1,\ldots,i\;.
\end{equation}
Since, by definition, the radii are always spectral on $\Gamma_S$, the 
radii of $\Fs$ and $\Fs^*$ coincide along $\Gamma_S$. On the other 
hand, the first $i$ radii of $\Fs$ and $\Fs^*$ are constant on every 
connected component $D$ of $X-\Gamma_S$ (which is necessarily a 
maximal disk, cf. Remark \eqref{Remark : D(x,S) cases}). 
It follows 
that for all $j=1,\ldots,i$ and all $x\in X$ we have
\begin{equation}\label{eq : equality radii }
\R_{S,j}(x,\Fs)\;=\;
\R_{S,j}(x,\Fs^*)\;.
\end{equation}
In particular, for all $j=1,\ldots,i$ one has (cf. \eqref{eq : XSiFF*})
\begin{equation}
X_{S,j}(\Fs)\;=\;X_{S,j}(\Fs^*)\;=\;X_{S,j}(\Fs,\Fs^*)\;.
\end{equation}
If $X_{S,i}(\Fs)=\emptyset$, there is nothing to prove.

If $X_{S,i}(\Fs)=X$, the claim reduces to 
Theorem \ref{Thm : criterion self direct sum}.

Let us assume that $X_{S,i}(\Fs)\neq \emptyset,X$. In this case, we 
necessarily have $\Gamma_S\neq \emptyset$ as in 
Footnote \ref{Footnote : Gamma not empty} of 
Remark \ref{Rk : Y subset X_S,i}.

Let $Z$ be connected component of $X_{S,i}(\Fs)$. Since the radii are 
locally constant outside $\Gamma_S$, it follows that $Z$ is necessarily 
the inverse image by the retraction $\tau:X\to\Gamma_S$ of an open 
connected sub-graph of $\Gamma_S$ 
(cf. Remark \ref{Rk : Y subset X_S,i}). In particular 
$Z\cap\Gamma_S\neq\emptyset$. Then $(\Fs_{|Z})_{\geq i}$ is a 
direct factor of $\Fs_{|Z}$ by item 
\eqref{Prop : direct sum by duality liocal-i} of Proposition 
\ref{Prop : direct sum by duality liocal}.
\end{proof}

\if{

If $X_{S,i}(\Fs)=\emptyset$, the index $i$ separates nowhere the radii 
(cf. Remark \ref{Rk : Y subset X_S,i}). 
Let us assume that $X_{S,i}(\Fs)\neq \emptyset$.

We denote the retraction by
\begin{equation}
\tau_{S,i}\;:\;X\;\to\;\Gamma_{S,i}(\Fs)\;.
\end{equation}
By construction $X_{|\Gamma_{S,i}(\Fs)}$ also retracts onto 
$\Gamma_{S,i}(\Fs)$, and the retraction coincides with $\tau_{S,i}$.

Let $\Gamma_{S,i}^{\mathrm{sep}}(\Fs)\subseteq 
\Gamma_{S,i}(\Fs)$ the possibly not 
connected open sub-graph formed by the points 
$x\in\Gamma_{S,i}(\Fs)$ satisfying 
$\R_{S,i-1}(x,\Fs)<\R_{S,i}(x,\Fs)$. 
Then 
\begin{equation}
\tau_{S,i}^{-1}(\Gamma_{S,i}^{\mathrm{sep}}(\Fs))\cap 
X_{|\Gamma_{S,i}(\Fs)}
\;=\;
X_{S,i}(\Fs)\cap X_{|\Gamma_{S,i}(\Fs)}\;.
\end{equation}
\begin{proposition}\label{Prop. deco graph dir dsum}
Let $Y:=X_{S,i}(\Fs)\cap 
X_{|\Gamma_{S,i}(\Fs)}$. 
Then $(\Fs_{|Y})_{\geq i}$ is a direct summand of $\Fs_{|Y}$.
\end{proposition}
\begin{proof}
The existence of $(\Fs_{|Y})_{\geq i}$ follows from 
Proposition \ref{Prop. Existence on X_Si}. 

Let $Z$ be a connected component of $Y$. 
Then $Z$ is the inverse image in $X_{|\Gamma_{S,i}(\Fs)}$ 
by $\tau_{S,i}$ of a connected component $\Gamma_Z$ of 
$\Gamma$.

Each connected component of $Z-\Gamma_Z$ is a disk $D$ 
on which the radii $\R_{S,j}(-,\Fs)$, $j=1,\ldots,i$ are all constant.
So we have a direct sum decomposition 
$\Fs_{|D}=(\Fs_{|D})_{\geq i}\oplus (\Fs_{|D})_{<i}$ by 
Proposition \ref{THM : Decomposition on a disk}.\footnote{Indeed 
Proposition \ref{THM : Decomposition on a disk} 
applies since $D_{S,i}(x,\Fs)\subseteq\tau_{S,i}(\Gamma_Z)$, 
so $i$ separates the radii $\R_{\emptyset,i}(-,\Fs_{|D})$ after 
localization to $D$. 
Moreover $(\Fs_{|D})_{\geq i}=((\Fs_{|Y})_{\geq i})_{|D}$.}

Let $x\in \Gamma_Z$ be the boundary of $D$. Since $i$ is spectral 
along $\Gamma_{S,i}(\Fs)$ we also have the Dwork-Robba 
decomposition $\Fs_x = (\Fs_x)_{\geq i} \oplus (\Fs_x)_{<i}$, where 
$\Fs_{x}:=\Fs\widehat{\otimes}\O_{X_{\Gamma_{S,i}(\Fs)},x}$.

These two decompositions show that $(\Fs_{|Y})_{\geq i}$ is 
\emph{locally} a direct summand of $\Fs_{|Y}$. 
This guarantee that $\R_{S,j}(x,\Fs)=\R_{S,j}(x,\Fs^*)$ for all 
$x\in Y$ and all $j=1,\ldots,i$. 

Hence $Y\subseteq X_{S,i}(\Fs,\Fs^*)$, and each connected 
component $Z$ of $Y$ verifies $Z\cap\Gamma_{S,i}\neq\emptyset$. 
We then apply Proposition \ref{Prop : direct sum by duality liocal}.
\if{

\smallcomment{}

By definition 
$\Gamma_S\subseteq\Gamma_{S,i}(\Fs)\subseteq 
X_{|\Gamma_{S,i}(\Fs)}$. 
One sees that there exists a smallest weak triangulation $S'$ of 
$X_{|\Gamma_{S,i}(\Fs)}$ containing $S$ and the boundaries 
of the disks removed. 
The set $S'$ is a locally finite set in $\Gamma_{S,i}(\Fs)$ satisfying $S'$ verifies $\Gamma_S 
\subseteq\Gamma_{S'}
\subseteq \Gamma_{S,i}(\Fs)$. 
By Prop. \ref{Prop : immersion}, for all $j=1,\ldots,i$ one has 
\begin{equation}
\Gamma_{S',j}(\Fs_{|X'})\;=\;
(\Gamma_{S,j}(\Fs)\cap X')\cup 
\Gamma_{S'}\;\subseteq\;
\Gamma_{S',i}(\Fs_{|X'})\;=\;
\Gamma_{S,i}(\Fs)\;,
\end{equation}
where $X':=X_{|\Gamma_{S,i}(\Fs)}$. Moreover the first $i$ 
$S$-radii are all spectral along $\Gamma_{S,i}(\Fs)$ by Remark 
\ref{Remark : R_i solvable on Gamma_i}. 
So by \eqref{Feq : Radii truncated by S to S'}, the index $i$ 
separates the radii at $y\in X_{|\Gamma_{S,i}(\Fs)}$ with respect to 
$S$ if, and only if, it separates the radii with respect to $S'$.
This proves that 
$X_{S,i}(\Fs)\cap X_{|\Gamma_{S,i}(\Fs)}$ is the 
inverse image by the retraction 
$X_{|\Gamma_{S,i}(\Fs)}\to 
\Gamma_{S',i}(\Fs_{|X_{|\Gamma_{S,i}(\Fs))}})$ of the subset of 
points on which $i$ separates the radii. In other words 
$X_{S,i}(\Fs)\cap X' = X'_{S',i}(\Fs_{|X'})$, where $X':=X_{|
\Gamma_{S,i}(\Fs)}$. 
Hence, replacing $X$ by $X_{|\Gamma_{S,i}(\Fs)}$, 
we can assume $\Gamma_{S,j}(\Fs)\subseteq\Gamma_{S,i}(\Fs)$, for 
all $j=1,\ldots,i$. 

Now with the same proof as Thm. \ref{Thm : criterion self direct sum} 
one shows that $i$ separates the radii of $\Fs^*$ at each point of 
$X_{S,i}(\Fs)$ (i.e. $X_{S,i}(\Fs)=X_{S,i}(\Fs,\Fs^*)$), and the 
assumptions of Prop. \ref{Prop : direct sum by duality liocal} 
are fulfilled.
}\fi
\end{proof}
}\fi
The following Corollary provides a decomposition with respect to the original radii, before localization.
\begin{theorem}\label{Cor : deco dirs outside Cr}
Let $Y$ be a connected analytic domain of $X-\mathrm{Cr}_S(\Fs)$. 
Let $1=i_1<\cdots<i_h$ be the indexes separating the radii 
$\{\R_{S,i}(-,\Fs)\}_i$ over $Y$. 
If
\begin{equation}\label{eq : condition for split Y cap jjj}
Y\;\subseteq\; 
X_{\langle i_1\rangle}\cap 
X_{\langle i_2\rangle}\cap\cdots\cap 
X_{\langle i_h\rangle}
\end{equation}
then $(\Fs_{|Y})_{\geq i_k}$ is a direct summand of $\Fs_{|Y}$ for all 
$k$.
\hfill$\qed$
\end{theorem}
\if{\begin{proof}
The sub-object $(\Fs_{|Y})_{\geq i_k}$ exists by Corollary
\ref{Prop : Crossing points shsh}. The claim is a consequence of
Proposition \ref{Prop. deco graph dir dsum}.
\end{proof}
}\fi

\subsubsection{Clean decomposition.}
Until now, we have restricted the module~$\Fs$ to a subset~$Y$ but 
have always computed its radii with respect to the original weak 
triangulation~$S$, which is a weak triangulation on~$X$. Now, we will 
allow us to change the weak triangulation as well and compute the radii 
with respect to a weak triangulation of~$Y$. In the process, the radii 
may be truncated, which would lead to a decomposition that is less 
precise. On the other hand, we will show that such a decomposition 
always exists.

Recall that, by~\cite[Theorem~4.5.15]{Kedlaya-draft}, every radius of 
convergence is constant in the neighborhood of a point of type~$4$. 
In particular, every graph~$\Gamma_{S,j}(\Fs)$ is the skeleton of a 
weak triangulation of $X$ (cf. for instance the proof of Proposition \ref{Prop. loc to X_langle i rangle}).

\begin{theorem}[Clean decomposition]\label{clean deco}
Set $r=\mathrm{rank}(\Fs)$. There exists a 
weak triangulation~$S_{d}$ of~$X$ containing~$S$ such that the 
following holds:
\begin{enumerate}
\item $\Gamma_{S_{d}} = 
\Gamma_{S}(\Fs)=\Gamma_{S}(\Fs^*)$ 
(cf. \eqref{eq : total controlling graph});

\item for every~$x\in X$ and every $j\in\{1,\dotsc,r\}$, we have 
$\Rc_{S_{d},j}(x,\Fs) = \Rc_{S_{d},j}(x,\Fs^*)$;

\item every connected component of~$X- S_{d}$ is a 
separating neighborhood of all its points,
 for both~$\Fs$ and~$\Fs^*$, with respect to the weak 
triangulation~$S_{d}$ (cf. Definition \ref{Definition : Sep ne}).
\end{enumerate}


Moreover, let~$C$ be a connected component of~$X- S_{d}$ 
(necessarily a virtual open disk or annulus) and endow it with the 
empty weak triangulation. Let $1=i_1< i_2< \dotsb <i_{h}$ be the 
indexes separating the radii of~$\Fs_{|C}$. Then, we have a direct 
sum decomposition
\begin{equation}
\Fs_{|C} \;=\; \bigoplus_{1\le m\le h} (\Fs_{|C})_{i_{m}} 
\end{equation}
such that, for every $m\in \{1,\dotsc,h\}$, every $j\in\{1,
\mathrm{rank}((\Fs_{|C})_{i_{m}})\}$ and every $x\in C$, we have
\begin{equation}
\Rc_{\emptyset,j}(x,(\Fs_{|C})_{i_{m}})\; =\;  
\Rc_{\emptyset,i_{m}}(x,\Fs_{|C}) \;=\; 
\Rc_{S_{d},i_{m}}(x,\Fs)\;.
\end{equation}
The same result hold for~$\Fs^*$ and, for every 
$m\in \{1,\dotsc,h\}$, we have 
$(\Fs_{|C})_{i_{m}}^* = (\Fs_{|C}^*)_{i_{m}}$.
\end{theorem}
\begin{proof}
We may proceed similarly as in the proof of Proposition 
\ref{Prop. loc to X_langle i rangle} to 
find a weak triangulation~$S'$ such that
$\Gamma_{S'} = \Gamma_{S}(\Fs)$. 
By Corollary \ref{Cor. dual compatibility} we also have 
$\Gamma_{S'} = \Gamma_{S}(\Fs^*)$. 
By Proposition \ref{Prop: A-4 fgh} this implies 
$\Gamma_{S'} = \Gamma_{S'}(\Fs)= \Gamma_{S'}(\Fs^*)$.
In analogy with \eqref{eq : equality radii }, 
the $S'$-radii are always spectral along 
$\Gamma_{S'}$ (hence compatible with duals) and the above 
equality of controlling graphs implies 
that they are locally constant outside it.
Therefore, one has $\R_{S',i}(-,\Fs)=\R_{S',i}(-,\Fs^*)$ for all $i$. 
Hence 
$\mathrm{Cr}_{S'}(\Fs^*)=
\mathrm{Cr}_{S'}(\Fs)\subset\Gamma_{S'}$. Define 
$S_d:=S'\cup\mathrm{Cr}_{S'}(\Fs)$. 
Proposition \ref{Prop: A-4 fgh} then ensures that the $S'$-radii 
coincides with the $S_d$-radii. Hence on an edge of 
$\Gamma_{S_d}$ (i.e. on the interior of each connected component 
of $\Gamma_{S_d}-S_d$)
two $S_d$-radii of~$\Fs$ (resp. $\Fs^*$) are either always equal or 
always different. One sees that, with this triangulation we have 
$X_{\langle i\rangle}=X$, for all $i$. 
The statement is now a consequence of 
Theorems \ref{Prop : direct sum local- finite} and 
\ref{Cor : deco dirs outside Cr}.
\if{

\begin{enumerate}
\item $\Gamma_{S'} = \Gamma_{S}(\Fs)\cup \Gamma_{S}(\Fs^*)$;
\item on every edge of~$\Gamma_{S'}$, two radii of~$\Fs$ (resp. 
$\Fs^*$), 
with respect to the weak triangulation~$S$, are either always equal 
or always different.
\end{enumerate}

We may now choose a weak triangulation~$S_{d}$ such that 
$\Gamma_{S_{d}}=\Gamma_{S'}$ (by Prop. \ref{Prop: A-4 fgh} 
this implies $\RR_{S'}=\RR_{S_d}$) in order that on 
every edge of~$\Gamma_{S_{d}}$, two radii of~$\Fs$ 
(resp. $\Fs^*$), with respect to the weak triangulation $S'$, are 
either always equal or always different. 

Since $\Gamma_{S_{d}} = 
\Gamma_{S}(\Fs)\cup \Gamma_{S}(\Fs^*)$, 
the radii of~$\Fs$ and~$\Fs^*$, with respect to~$S$, are locally 
constant outside $\Gamma_{S_{d}}$. 
The same holds with respect to~$S_{d}$ by Prop. \ref{Prop: A-4 fgh}.

Let~$C$ be a connected component of $X- S_{d}$. Let us 
first assume that~$C$ is a virtual open disk. Endow~$C$ with the 
empty weak triangulation. The radii of~$\Fs$ and~$\Fs^*$ are 
constant on~$C$, hence those of~$\Fs_{|C}$ and~$\Fs^*_{|C}$ too. 
Let~$x\in C$. Let $j\in\{1,\dotsc,r\}$ such that 
$\Rc_{\emptyset,j}(x,\Fs_{|C}) < 1$. At a point~$z$ that is close 
enough to the boundary of~$C$, since 
$\Rc_{\emptyset,j}(z,\Fs_{|C}) = \Rc_{\emptyset,j}(x,\Fs_{|C}) < 1$, 
this radius becomes spectral and non-solvable. Hence, by~\ref{} ,we 
have 
\begin{equation}
\Rc_{\emptyset,j}(x,\Fs_{|C}^*) = \Rc_{\emptyset,j}(z,\Fs_{|C}^*) 
= \Rc_{\emptyset,j}(z,\Fs_{|C}) = \Rc_{\emptyset,j}(x,\Fs_{|C}).
\end{equation}
The same equality holds if we assume that 
$\Rc_{\emptyset,j}(x,\Fs_{|C}^*) < 1$. If none of the previous 
hypotheses is satisfied, then we have 
$\Rc_{\emptyset,j}(x,\Fs_{|C}^*) = 
\Rc_{\emptyset,j}(x,\Fs_{|C}) = 1$.

From those equalities and the constancy of the radii on~$C$, we deduce that~$C$ is a separating neighborhood of all its points for~$\Fs$ and~$\Fs^*$ and that condition~(ii) of Theorem~\ref{Thm : 5.7 deco good direct siummand} is satisfied for every separating radius~$i$. The decomposition result follows.

Let us now assume that~$C$ is a virtual annulus. Denote by~$\Gamma_{C}$ its skeleton. Remark that the radii of~$\Fs$ and~$\Fs^*$, with respect to the weak triangulation~$S'$, are all spectral at the points of $\Gamma_{S'}=\Gamma_{S_{d}}$. Hence they remain unchanged if we compute them with respect to~$S'$ or~$S_{d}$. In particular, on $\Gamma_{C}$, two radii of~$\Fs$ (resp. $\Fs^*$), with respect to the weak triangulation~$S_{d}$, are either always equal or always different. Since they are locally constant outside~$\Gamma_{C}$, we deduce that~$C$ is a separating neighborhood of all its points. Note that, on~$C$, it is equivalent to compute the radii of~$\Fs$ and~$\Fs^*$ with respect to~$S_{d}$ or the radii of~$\Fs_{|C}$ and~$\Fs^*_{|C}$ with respect to the empty weak triangulation. The result now follows from Theorem~\ref{Thm : 5.7 deco good direct siummand}, whose condition~(i) is satisfied.
}\fi
\if{\comment{J\'er\^ome,
Je pense qu'on peut \'eviter d'utiliser Kedlaya:
\begin{enumerate}
\item Il faut decomposer apr\`es extension des scalaires ?  $\Omega$ 
o\`u il n'y a pas de points de type 4.
\item Ensuite descendre la decomposition ?  $K$ par Galois... (le galois 
stabilise la decomposition par unicit\'e de cette derni\`ere).
\item Reste quand m\^eme ?  comprendre le type de partition qu'on 
obtient sur $X$ (et non pas $X_\Omega$). 
Je crois qu'on est bien 
oblig\'e d'introduire des pseudo-triangulations. Ca revient plus ou 
moins au Remark \ref{Rk : clean sans Kedlaya} ... ?
\end{enumerate}
Il faudrait arriver \`a \'ecrire cette partie ...
}
}\fi
\end{proof}

\begin{remark}\label{Rk : clean sans Kedlaya}
Without using Kedlaya's result, it is still possible to prove a slightly 
weaker statement by replacing~$X$ by $X- T$, where~$T$ 
is the locally finite subset of~$X$ formed by the type~$4$ 
points in $\Gamma_{S}(\Fs)$. 
Then~$S_{d}$ is a weak triangulation of~$X- T$ whose 
skeleton is $\Gamma_{S_d}=\Gamma_{S}(\Fs)- T$. The 
other properties hold unchanged.
\end{remark}

\begin{remark}\label{Remark : defofcanonical triangulation}
An alternative proof of the global decomposition 
Theorem \ref{MAIN Theorem} is possible by 
using some ideas from the proof of 
Theorem \ref{clean deco}. Namely, let $\SF$ be a 
weak triangulation such that $\Gamma_{\SF}=\Gamma_S(\Fs)$ 
satisfying
\begin{enumerate}
\item\label{Remark : defofcanonical triangulation-i} Each connected component of $X-\SF$ is either an open virtual 
disk or annulus, which is a separating neighborhood of all its points 
with respect to the original $S$-radii $\{\R_{S,i}(-,\Fs)\}_i$; 
\item\label{Remark : defofcanonical triangulation-ii} The $S$-radii $\{\R_{S,i}(-,\Fs)\}_i$ 
are all $\log$-affine on each edge of 
$\Gamma_{\SF}$ (i.e. on the interior of each connected component 
of $\Gamma_{\SF}-\SF$).
\end{enumerate}
By Proposition \ref{Prop: A-4 fgh}, it follows that items \eqref{Remark : defofcanonical triangulation-i} and \eqref{Remark : defofcanonical triangulation-ii} also hold for the $\SF$-radii 
$\{\R_{\SF,i}(-,\Fs)\}_i$. 
Let $C$ be a connected component of $X-\SF$. We now define an 
augmented decomposition of $\Fs$ over $C$ with respect to the 
original radii $\{\R_{S,i}(-,\Fs)\}_i$ (i.e. a decomposition taking in 
account over-convergent radii).
Since by assumption the radii $\R_{\SF,i}(-,\Fs)$ 
are separated over $\Gamma_C$, then,
by Corollary \ref{Cor : linear on annulus} 
(see also \cite[Theorem 12.5.2]{Kedlaya-book-2}), $\Fs_{|C}$ can be 
decomposed by the spectral $\SF$-radii. Since for the spectral indexes 
the radii are not truncated by the change of triangulation 
(cf. \eqref{eq : D^c} and Proposition \ref{Prop: A-4 fgh}), 
then this is 
also a decomposition by the spectral $S$-radii. 
Now, the over-convergent $S$-radii over $C$ can exists if, and only if, 
$C\subseteq D(x,S)$ for some $x$, that is, if $C\cap\Gamma_S=\emptyset$. In this case over-solvable 
$S$-radii generate a trivial submodule of $\Fs_{|C}$, and this 
defines the augmented filtration.
This augmented filtration glues with the augmented Dwork-Robba 
decomposition (with respect to the $S$-radii) at the points of $\SF$, 
and we recover the result of Theorem \ref{MAIN Theorem}.
%
\end{remark}

\subsection{Formal differential equations}
\label{Formal differential equations}
Assume that $K$ is trivially valued. 
This is a somehow degenerate situation since the punctured disk 
$D^-_K(0,1)-\{0\}$ coincides with the open segment $]0,x_{0,1}[$ 
and it is topologically homeomorphic to the real open interval 
$]0,1[\subset\mathbb{R}$. 
\begin{remark}\label{rk : f_rho=rho^vt(f)}
Let $K((T))$ be the field of formal power series with coefficients in 
$K$. If $f=\sum_{i\geq n} a_iT^i\in K((T))$, then 
$|f|_{0,\rho}=\sup_{i\geq n}|a_i|\rho^i$, 
and $|a_i|$ is either equal to $0$ or $1$. 
Hence for all $0<\rho<1$ one has 
\begin{equation}
|f|_{0,\rho}\;=\;\rho^{v_T(f)}\;,
\end{equation}
 where 
\begin{equation}
v_T(f)=\min \{i\;|\; a_i\neq 0\}
\end{equation} is the 
$T$-adic valuation of $f$.
\end{remark}

The Remark shows that we have equalities of rings (without norms or topologies)
\begin{equation}
K((T))\;=\;\left\{
\begin{array}{ll}
\H(x_{0,\rho})&\textrm{for all }\rho\in]0,1[;\\
\O(C)&\textrm{for all annulus }C\subseteq D^-_K(0,1) 
\textrm{ centered at }0;\\
\mathfrak{R}:=\bigcup_{\varepsilon>0}\O(C_\varepsilon)&
\textrm{where }C_\varepsilon:=C_K^-(0;1-\varepsilon,1)\textrm{ (this 
is the Robba ring)}.\\
\end{array}
\right.
\end{equation}
A differential module $\M$ over $K((T))$ then have $3$ kind of 
decompositions:
\begin{enumerate}
\item If $\M$ is viewed as a module over $\H(x_{0,\rho})$, 
one has the Robba's decomposition 
\eqref{eq : deco at x Robba eq};
\item If $\M$ is viewed as a module over $K((T))$, one has
the decomposition \cite[p. 97-107]{Correspondance-Malgrange-Ramis} 
(see below) by the slopes of its formal Newton polygon of B. Malgrange and 
J. P. Ramis (cf. \cite{Ramis-Devissage-Gevrey});
\item If $\M$ is viewed as a module over the Robba's ring 
$\mathfrak{R}$, one has the decomposition by the 
slopes of Christol-Mebkhout  of section \ref{Ch-Me-Robba}
(in fact we prove that such a module is always solvable in the sense of their terminology \eqref{eq : solvability over the Robba ring});
\end{enumerate}
\begin{definition}[Formal Newton polygon]
Let  
\begin{equation}
P\;:=\;\sum_{k=0}^rg_k(T)(\frac{d}{dT})^k
\;\;\in\;\;
K((T))\lr{\frac{d}{dT}}\;,
\end{equation} 
with $g_r=1$, be a differential operator 
corresponding to a cyclic basis of $\M$. 
For $(u,v)\in\mathbb{R}^2$ we define the quadrant with vertex at $(u,v)$ as 
\begin{equation}
\mathcal{Q}(u,v)\;=\;\{(x,y)\in\mathrm{R}^2\;|\; x\leq u,y\geq v\}\;.
\end{equation} 
Then, the \emph{formal Newton polygon} is the convex hull in $\mathbb{R}^2$ of 
the family of quadrants  
\begin{equation}
\{\;\mathcal{Q}(k,v_T(g_k)+r-k)\;\}_{k=0,\ldots,r}\;.
\end{equation} 
The numbers 
\begin{eqnarray}
\mathrm{Irr}_{\mathrm{Formal}}(P)&:=&\max_{0\leq k\leq
r}\{k-v_T(g_k)\}-(r-v_T(g_r))\;,\\
\mu_{\mathrm{max}}(P)&\;:=\;&
\max\Bigl(\;0\;,\;\max_{k=0,\ldots,r-1}\frac{v_T(g_k)}{k-r}-1\;\Bigr)\;.
\end{eqnarray}
is called the Formal irregularity and the Poincaré-Katz rank of $\M$ 
respectively. These are height and the largest slope of the formal 
Newton polygon respectively:
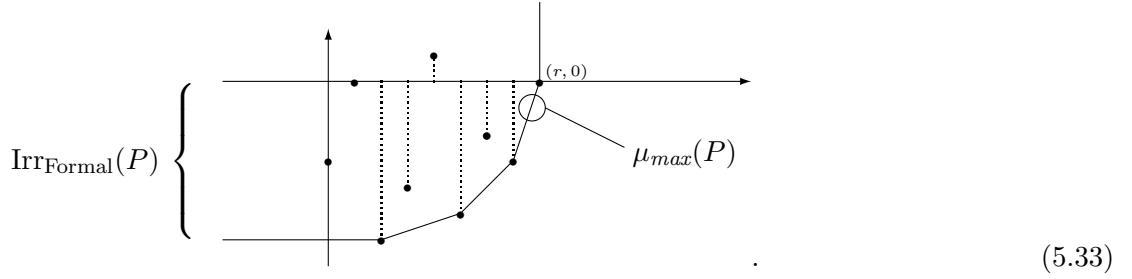
\begin{equation}
\begin{picture}(200,100) %
\put(40,0){\vector(0,1){90}} %
\put(0,70){\vector(1,0){200}} %

\put(0,10){\line(1,0){60}} %
\put(60,10){\line(3,1){30}} %
\put(90,20){\line(1,1){20}}
\put(110,40){\line(1,3){10}} %
\put(120,70){\line(0,1){30}} %

\put(117,60){\circle{10}} %
\put(122,60){\line(2,-1){30}} %
\put(155,40){$\mu_{\textrm{max}}(P)$}

\put(118,68){\begin{tiny}$\bullet$\end{tiny}}
\put(122,72){\begin{tiny}$(r,0)$\end{tiny}}

\qbezier[30](60,10)(60,40)(60,70)   
\put(58,8){\begin{tiny}$\bullet$\end{tiny}}
\qbezier[25](90,20)(90,45)(90,70)    
\put(88,18){\begin{tiny}$\bullet$\end{tiny}}
\qbezier[15](110,40)(110,55)(110,70) 
\put(108,38){\begin{tiny}$\bullet$\end{tiny}}

\qbezier[10](100,50)(100,60)(100,70) 
\put(98,48){\begin{tiny}$\bullet$\end{tiny}}

\qbezier[5](80,70)(80,75)(80,80) 
\put(78,78){\begin{tiny}$\bullet$\end{tiny}}

\qbezier[20](70,70)(70,50)(70,30) 
\put(68,28){\begin{tiny}$\bullet$\end{tiny}}

\qbezier[10](50,70)(50,70)(50,70) 
\put(48,68){\begin{tiny}$\bullet$\end{tiny}}

\put(38,38){\begin{tiny}$\bullet$\end{tiny}}

\qbezier[10](100,50)(100,60)(100,70) 
\put(98,48){\begin{tiny}$\bullet$\end{tiny}}

\put(-80,37){$\mathrm{Irr}_{\mathrm{Formal}}(P)$}
\put(-20,37){$\left\{
\begin{smallmatrix}
\\\\\\\\\\\\\\\\\\\\\\
\end{smallmatrix}\right.$} 
%
\end{picture}.
\end{equation}
It is a classic fact that the formal Newton polygon is independent of the 
chosen cyclic basis of $\M$ (cf. \cite{VS}).
\end{definition}
In order to be coherent with the rest of the paper, by convention we 
say that the formal Newton polygon has precisely 
$r=\mathrm{rank}(\M)$ slopes 
$\mu_1\leq\cdots\leq\mu_r$ defined as $\mu_i:=h(i)-h(i-1)$, 
where $h:[-\infty,r]\to\mathbb{R}$ is the function whose 
epigraph is the formal Newton polygon. Moreover, with a little abuse, 
we identify the polygon with the multi-set of its slopes $\{\mu_1,\ldots,\mu_r\}$.

The module $\M$ is said pure of formal slope $\mu$ if 
$\mu_1=\cdots=\mu_r=\mu$. 
\begin{theorem}[\protect{\cite[p. 97-107]{Correspondance-Malgrange-Ramis}}]\label{Thm. deco formal slopes}
Any solvable differential module over $K((T))$
admits a direct sum decomposition 
\begin{equation}
\M\;:=\;\bigoplus_{\mu \geq 0}\M(\mu)
\end{equation}
into sub-modules $\M(\mu)$ that are pure of formal 
slope $\mu$. Moreover, as a multiset of slopes, 
the formal Newton polygon of $\M$ is the union of 
those of the $\M(\mu)$ appearing in this decomposition.
\hfill$\qed$
\end{theorem}
In this section we prove the following:
\begin{proposition}\label{Prop : 3 deco coincide}
The above three decompositions of $\M$ coincide. More precisely 
let $S$ be a weak triangulation of $D_K^-(0,1)-\{0\}$.\footnote{ In 
this degenerate situation we automatically have 
$\Gamma_S=]0,x_{0,1}[$, and $S$ is a discrete set of points 
$]0,x_{0,1}[$ whose unique accumulation point in 
$\mathbb{R}$ is $0$.} 
Let 
$\mu_1\leq\mu_2\leq\cdots\leq \mu_r$ be the slopes of the formal 
Newton polygon of $\M$. 
Then 
\begin{enumerate}
\item\label{Prop : 3 deco coincide-i} For all $i=1,\ldots,r=\mathrm{rank}(\M)$ and all $\rho\in]0,1[$ 
one has 
\begin{equation}\label{eq : R=rho^mu}
\R_{S,i}(x_{0,\rho},\M)\;=\;\rho^{\mu_{r-i+1}}\;.
\end{equation}
In particular $\M$ viewed as a differential module over $\mathfrak{R}$ 
is solvable following the terminology of Christol-Mebkhout 
(cf. \eqref{eq : solvability over the Robba ring}). 
\item\label{Prop : 3 deco coincide-ii} One has $\Gamma_{S,i}(\M)
=\Gamma_S=]0,x_{0,1}[=D^-_K(0,1)-\{0\}$. 
The radii of $\M$ are separated over $D^-_K(0,1)-\{0\}$, and there 
exists a global decomposition by the radii over 
$K((T))=\O(D_K^-(0,1)-\{0\})$.

\item\label{Prop : 3 deco coincide-iii} The formal slopes $\mu_1\leq\cdots\leq \mu_r$ 
coincide with the slopes $\beta_1\leq\cdots\leq\beta_r$ of 
Christol-Mebkhout. In particular one has
\begin{eqnarray}
\partial_b\R_{S,1}(x_{0,1},\M)&\;=\;& \mu_r=\beta_r=\mu_{max}(P) 
\;,\\
\partial_bH_{S,r}(x_{0,1},\M)&=& \mathrm{Irr}_{Formal}(P)\;,
\end{eqnarray}
where $H_{S,r}(x,\M)=\prod_{i=1}^n\R_{S,i}(x,\M)$ is the highest partial 
height of $\Fs$ (cf. \eqref{eq : partial height}), 
and $b=]x_{0,1-\varepsilon},x_{0,1}[$ is a germ of 
segment oriented as out of $x_{0,1-\varepsilon}$ (i.e. 
towards $+\infty$).
\end{enumerate}

We sum up these facts by saying that the formal Newton polygon 
equals the Christol-Mebkhout Newton polygon, and it coincides with 
the \emph{derivative} of the convergence Newton polygon.
\end{proposition} 
\begin{proof}
Decomposing the module with respect to all decomposition results 
Corollary \ref{corollary : uniqueness at H(x)}, 
Corollary \ref{Cor : Ch-Me deco}, and
Theorem \ref{Thm. deco formal slopes} we can assume that the three 
newton polygons of $\M$ all have an individual slope with multiplicity 
$\mathrm{rank}(\M)$. 
All the statements will 
then follows from \eqref{eq : R=rho^mu} for $i=1$, which we now 
prove.

With the notations of
 \eqref{eq: Y(T,t_x)} one has 
\begin{equation}
\R_{S,1}(x_{0,\rho},\M)\;=\;\min(1,\R^Y(x_{0,\rho})/\rho)\;,
\end{equation}
where $\R^Y(x_{0,\rho})=\liminf_n(|G_n|_{0,\rho}/|n!|)^{-1/n}$.
By Remark \ref{rk : f_rho=rho^vt(f)}, 
the functions $\log(\rho)\mapsto \log(|G_n|_{0,\rho})$ 
are all lines passing through the origin. Hence the same happens for 
the  functions $\log(\rho) \mapsto\log\R^Y(x_{0,\rho})$ and 
$\log(\rho) \mapsto\log\R_{S,1}(x_{0,\rho},\M)$. 
In particular one has 
\begin{equation}
\lim_{\rho\to 1^-}\R_{S,1}(x_{0,\rho},\M)\;=\;1\;.
\end{equation}
Now, by \cite[Theorem 6.2]{Astx} (cf. P.Young's result \cite{Young}) we obtain for all $\rho\in]0,1[$
\begin{equation}
\R_{S,1}(x_{0,\rho},\M)\;=\;\rho^{\mu_{\mathrm{max}}(P)}\;.
\end{equation}
The claim follows.
\end{proof}
\begin{remark}
It is interesting to notice (\emph{a posteriori}) that 
the decomposition by the radii of P.Robba 
\cite{Robba-Hensel-original} exposed in section 
\ref{Robba-deco} is prior to the formal decomposition result of  
\cite[p. 97-107]{Correspondance-Malgrange-Ramis} and coincides with 
it.
\end{remark}

\subsection{Notes.}
\label{Rk : theoreme de deco de Kedlaya sur un disque...}
The decomposition Theorem \ref{MAIN Theorem} is not a simple 
consequence of Robba's and Dwork-Robba's decompositions by the 
spectral radii (cf. Corollary \ref{corollary : uniqueness at H(x)}, 
and Theorem \ref{Dw-Robba}). 
Indeed the proof of Theorem \ref{MAIN Theorem} uses the continuity 
of all the radii (cf. Proposition \ref{Prop : extended by continuity}), 
which is a consequence of the local finiteness of $\Gamma_{S}(\Fs)$. 
The proof of the finiteness involves again Robba's decomposition, 
several results of \cite{Kedlaya-book-2} and, in particular, 
another decomposition result due to Kedlaya 
\cite[Theorem 12.4.1]{Kedlaya-book-2}.
The  \cite[Theorem 12.4.1]{Kedlaya-book-2}  
\emph{does not assume} that the radii are separated, 
but \emph{it implies} the separateness condition on some regions.
In fact the decomposition \cite[Theorem 12.4.1]{Kedlaya-book-2} 
is actually used in \cite{NP-I} and \cite{Kedlaya-draft} to prove that the 
radii are separated and constant on certain regions (see 
\cite[Proposition 6.2]{NP-I}, \cite[item iii) of Remark 2.1.6]{NP-IV} 
and \cite[Lemma 4.3.12]{Kedlaya-draft}),
this is somehow the heart of both proofs of the finiteness result.

\section{Some counterexamples.} 
\label{An explicit counterexample.}
In this section we provide the following counterexamples:
\begin{enumerate}
\item Non compatibility of solvable and over-solvable 
radii with duals (cf. section 
\ref{section counterexample to duality and exactness});
\item Uncontrolled behavior of solvable and over-solvable 
radii by exact sequences (cf. section 
\ref{section counterexample to duality and exactness});
\item An explicit example of differential module over a 
disk for which $\Fs_{\geq i}$ is not a direct summand (cf. section 
\ref{optimality});
\item Some basic relations between the 
Grothendieck-Ogg-Shafarevich formula and 
super-harmonicity of partial heights 
(cf. Section \ref{A remark on the Grothendieck-Ogg-Shafarevich formula}). In particular, a link between 
the failure of super-harmonicity and the presence of $p$-adic Liouville 
exponents in our differential module 
(cf. Remark \ref{rk : super-harmonicity}).
\end{enumerate}
All the examples involve a differential module over an open disk with 
empty triangulation. Indeed, by definition all the radii are spectral on
$\Gamma_S$, and hence 
compatible with duality and spectral sequences 
(cf. Proposition \ref{Prop : 1-th radius and dual} and 
\eqref{eq : compatibility with duals}). Therefore, 
any possible counterexample to the above situations is reduced 
to the case of an open disk with empty weak-triangulation 
by localizing to a maximal disk $D(x,S)$. 

Concerning item iv), the potential failure of super-harmonicity at some 
points was one of the crucial difficulties of \cite{NP-I}, we here relate 
this to the presence of Liouville numbers in the exponents.

\subsection{Setting.}\label{section : setting examples ..htfhg}
Let $D:=D^-(0,1)$ be the open unit disk, and let $T$ be its coordinate.
In this section all differential module will be defined over the ring 
\begin{equation}
\O^\dag(D)\;:=\;\cup_{\varepsilon>0}\O(D_\varepsilon)\;,
\end{equation}
where 
\begin{equation}
D_\varepsilon\;:=\;D^-(0,1+\varepsilon)\;.
\end{equation}
Namely $\O^\dag(D)$ 
is formed by power series $f(T)=\sum_{i\geq 0}a_iT^i$ satisfying 
$\lim_i|a_i|\rho^i=0$, for some unspecified $\rho>1$.
The data of a differential module $\M$ 
over $\O^\dag(D)$ is equivalent to 
that of a differential module $\M_\varepsilon$ over 
$\O(D_\varepsilon)$ for some 
unspecified $\varepsilon>0$. The triangulation on 
$D_\varepsilon$ will always be the empty one. Moreover the 
radii are assumed to be \emph{all solvable or over-solvable} at the 
boundary $x_{0,1}$ of $D$. By \cite[Proposition 4.2 or 
Lemma 6.5]{NP-I}, 
this implies that the restriction of $\M$ to 
all sub-disks of $D_\varepsilon$ with boundary $x_{0,1}$ is trivial. 

Since we are going to make use of the index formula, 
we assume that $K$ is a finite extension the field of $p$-adic numbers 
$\mathbb{Q}_p$.\footnote{cf. \cite{NP-V}  for a generalization of that formula to any complete valued base field, including trivial valuation.} 
All differential modules will have the property $\bs{\mathrm{NL}}$ 
of non-Liouville exponents at the boundary of the disks 
$D_\varepsilon$, along the annuli
\begin{equation}\label{eq : notation C_e}
C_\varepsilon:=\{|T(x)|\in]1,1+\varepsilon[\}\;.
\end{equation} 
In the sequel, we refer to this property simply as 
``\emph{the $\bs{\mathrm{NL}}$ property of $\M$}''.
Under this condition 
$\Hdr^1(\M,\O^\dag(D)):=\textrm{Coker}(\nabla:\M\to\M)$ 
has finite dimension (cf. \cite{Ch-Me-IV}, see also \cite{NP-V,NP-VI} 
for a more extensive treatment). 
More precisely, one shows that the natural 
restrictions maps $\Hdr^i(\M_\varepsilon,\O(D_\varepsilon))
\to\Hdr^i(\M_{\varepsilon'},\O(D_{\varepsilon'}))$ are isomorphisms 
for all $0<\varepsilon'<\varepsilon$ close enough to $0$ and coincide 
with $\Hdr^i(\M,\O^\dag(D))$. We set (cf. \eqref{eq: hidr little})
\begin{equation}
\hdr^i(\M,\O^\dag(D))\;:=\;
\mathrm{dim}_{K}\Hdr^i(\M,\O^\dag(D))\;=\;
\lim_{\varepsilon\to 0^+}\mathrm{dim}_{K}
\Hdr^i(\M_\varepsilon,\O(D_\varepsilon))\;.
\end{equation}

The $\bs{\mathrm{NL}}$ condition is quite implicit. 
The effective way to ensure it is to assume either that the restriction of 
$\M$ to the annulus $C_\varepsilon$
has a Frobenius structure, or that 
none of the radii $\R_{\emptyset,i}(-,\M)$ of 
$\M$ verifies the following Robba condition along 
$b=]x_{0,1},x_{0,1+\varepsilon}[$:
\begin{equation}\label{eq : negative break of R_i NL}
\R_{\emptyset,i}(-,\M)\textrm{ is solvable at $x_{0,1}$, 
and $\partial_b\R_{\emptyset,i}(x_{0,1},\M)=1$}\;.
\end{equation} 
\begin{remark}\label{Rk : NL condition restr}
Condition \eqref{eq : negative break of R_i NL} concerns $\M$ before 
localization to $C_\varepsilon$ and 
differs from the same condition for the radii of $\M_{|C_\varepsilon}$. 
Indeed, over-solvable radii of $\M_\varepsilon$ at 
$x_{0,1}$ are truncated by localization to $C_\varepsilon$ 
and become solvable (cf. Section \ref{Localization}). 
The existence of over-solvable radii at $x_{0,1}$ 
corresponds to the existence of trivial 
submodules of $\M_\varepsilon$ that of course satisfy the condition 
$\bs{\mathrm{NL}}$. Therefore, in this particular case, 
the $\bs{\mathrm{NL}}$ condition really 
arises from the presence of a break of some 
$\R_{\emptyset,i}(-,\M)$ at $x_{0,1}$, as in 
\eqref{eq : negative break of R_i NL}, before localization.
\end{remark}
\if{
will have an \emph{unspecified} Frobenius 
structure, i.e. an isomorphism of differential modules 
$(\varphi^n)^*\M\simto\M$, for an unspecified $n\geq 1$, 
where the first module is the pull-back by the Frobenius map 
$\varphi^n:\O^\dag(D)\to\O^\dag(D)$ associating to 
$\sum_{i\geq 0}a_iT^i$ the series 
$\sum_{i\geq 0}\sigma^n(a_i)T^{ip^n}$, where 
$\sigma:K\to K$ is a lifting of the $p$-th power map 
$x\mapsto x^p$ of the residual field of $K$.
}\fi

\subsection{Grothendieck-Ogg-Shafarevich formula and 
super-harmonicity}
\label{A remark on the Grothendieck-Ogg-Shafarevich formula}

We maintain the above notation: 
 $\M$ is a differential module over $\O^\dag(D)$ of 
rank $r=\mathrm{rank}(\M)$, and $b=]x_{0,1},x_{0,1+\varepsilon}[$ 
is a germ of segment oriented as out of $x_{0,1}$.
\begin{definition}[\cite{Ch-Me-III},\cite{Ch-Me-IV}]
One defines the $p$-adic irregularity of $\M$ 
at $\infty$ as
\begin{equation}\label{Eq : def of irririr}
\mathrm{Irr}_{\infty}(\M)\;:=\;
-\partial_bH_{\emptyset,r}(x_{0,1},\M_{|C_\varepsilon})\;=\; 
-\sum_{i=0}^r\partial_b\R_{\emptyset,i}(x_{0,1},
\M_{|C_\varepsilon})\;.
\end{equation}
\end{definition}
Notice that we restrict $\M$ to $C_\varepsilon$, endowed with the 
empty weak-triangulation and we compute the radii after localization to 
$C_\varepsilon$ (cf. Section \ref{Localization}). This definition 
corresponds to that of G.Christol and Z.Mebkhout 
(cf. Definition \ref{Def : Ch-Me irreg}, \cite{Ch-Me-III, Ch-Me-IV}, see \cite{NP-V} for a more extensive treatment).

The radii of $\M$ that are over-solvable at $x_{0,1}$ 
correspond to solutions of $\M$ on some 
$D_\varepsilon$. Their number equals then $\hdr^0(\M,\O^\dag(D))$. 
These radii are truncated by localization to 
$C_\varepsilon$. As a result their slope remains zero, but the 
localization to $C_{\varepsilon}$ adds $-1$ to the slope of 
each other radius (i.e. to the spectral radii, 
cf. Proposition \ref{Prop : immersion}). 
One finds $\partial_bH_{\emptyset,r}(x_{0,1},\M_{|C_\varepsilon})=
\partial_bH_{\emptyset,r}(x_{0,1},\M)-r+\hdr^0(\M,\O^\dag(D))$:
\begin{equation}\label{eq :IIR as intrinsic def}
\mathrm{Irr}_\infty(\M)\;=\;
\mathrm{rank}(\M)-\partial_bH_{\emptyset,r}(x_{0,1},\M_\varepsilon)- 
\hdr^0(\M,\O^\dag(D))\;.
\end{equation}
Assume now that $\M$ satisfies 
the $\bs{\mathrm{NL}}$ property, and 
that $K$ is spherically complete. Then,
one has the Grothendieck-Ogg-Shafarevich formula 
(often called Euler-Poincaré formula, or simply index formula):
\begin{equation}\label{eq : GOS}
\hdr^0(\M,\O^\dag(D))-\hdr^1(\M,\O^\dag(D))\;=\;
\mathrm{rank}(\M)-\mathrm{Irr}_\infty(\M)\;.
\end{equation}
By \eqref{eq :IIR as intrinsic def} this formula can be written as :
\begin{equation}\label{GOS-changed}
\hdr^1(\M,\O^\dag(D))\;=\;-\partial_bH_{\emptyset,r}(x_{0,1},\M_\varepsilon) \;.
\end{equation}

\begin{proposition}\label{Prop : super-harm everywhere}
Let $\M_\varepsilon$ be a differential module over a disk 
$D_\varepsilon$, 
such that $\R_{\emptyset,1}(-,\M_\varepsilon)$ 
is solvable or over-solvable at 
$x_{0,1}$, and such that $\M_\varepsilon$ 
has the $\bs{\mathrm{NL}}$ property.
Then, up to shrink $\varepsilon>0$, the following properties hold
\begin{enumerate}
\item\label{Prop : super-harm everywhere-i} If $D\subseteq D_\varepsilon$ is a virtual open disk with boundary $x_{0,1}$, then 
$(\M_\varepsilon)_{|D}$ is a trivial differential module on $D$. In particular, for every germ of segment $b_D\neq b$ out of $x_{0,1}$ and all $i=1,\ldots,r$, we have $\partial_{b_D}H_{\emptyset,i}(-,\M_\varepsilon)= 0$; 
\item\label{Prop : super-harm everywhere-ii} For all $i=1,\ldots,r$ we have $\Gamma_{\emptyset,i}(\M_\varepsilon)\subseteq[x_{0,1},x_{0,1+\varepsilon}[$;
\item\label{Prop : super-harm everywhere-iii} for all $i=1,\ldots,r=\mathrm{rank}(\M_\varepsilon)$ 
the $i$-th partial height 
$H_{\emptyset,i}(-,\M_\varepsilon)$ 
satisfies $\partial_bH_{\emptyset,i}(-,\M_\varepsilon)\leq 0$;
\item\label{Prop : super-harm everywhere-iiii} For all $i=1,\ldots,r$ and all $x\in D_\varepsilon$, the partial 
height $H_{\emptyset,i}(-,\M_\varepsilon)$ is super-harmonic at $x$, 
that is (cf. Notatition \ref{Notation branch - germ - direction})
\begin{equation}
dd^cH_{\emptyset,i}(x,\M_\varepsilon)\;=\;\partial_bH_{\emptyset,i}(x,\M_\varepsilon)+
\sum_{D}
\mathrm{deg}(b_D)\partial_{b_D}H_{\emptyset,i}(x,\M_\varepsilon)
\leq 0\;,
\end{equation}
where $D$ runs in the family of virtual open disks with 
boundary $x_{0,1}$ and $\mathrm{deg}(b_D)$ is the number of 
open disks over the algebraic closure of $K$ over $D$ (cf. \cite{NP-IV} 
for more details).
\end{enumerate}
\end{proposition}
\begin{proof}
All the properties are insensitive to scalar extensions of $K$. 
So we can assume that  $K$ is algebraically closed and 
spherically complete.  
If $\R_{\emptyset,1}(-,\M_\varepsilon)$ 
is over-solvable at $x_{0,1}$, then 
$D_{\emptyset,1}(x_{0,1},\M_\varepsilon)=D_\varepsilon$ for some 
$\varepsilon>0$ 
(cf. Remark \ref{Remark : R_i solvable on Gamma_i}), therefore 
$\M_\varepsilon$ is trivial over $D_\varepsilon$ 
for some $\varepsilon$, and there is nothing to prove. Let us assume that $\R_{\emptyset,1}(-,\M_\varepsilon)$ 
is over-solvable at $x_{0,1}$, that is 
$\R_{\emptyset,1}(x_{0,1},\M_\varepsilon)=\frac{1}{1+\varepsilon}$.

\eqref{Prop : super-harm everywhere-i}. If $i=1$, the super-harmonicity follows from \cite[Theorem 3.9]{NP-I}. 
Moreover, $\R_{\emptyset,1}(-,\M_\varepsilon)$ is decreasing on the 
segments of every virtual open disks in $D_\varepsilon$, 
when these segments are oriented as towards outside the disk 
(cf. concavity and transfer \cite[Proposition 4.2]{NP-I}). 
Since by assumption 
$\R_{\emptyset,1}(x_{0,1},\M_\varepsilon)$ is solvable, 
it follows that it has to be a constant function on every virtual open 
disk $D$ with boundary $x_{0,1}$. In other words, if for every $z\in D$ we have $\R_{\emptyset,1}(z,\M_\varepsilon)\geq
\R_{\emptyset,1}(x_{0,1},\M_\varepsilon)$, which implies 
$D\subseteq D_{\emptyset,1}(z,\M_\varepsilon)$  
(cf. see for instance the proof of Corollary 
\ref{Cor : linear on annulus} for a similar reasoning). In particular, 
$\M_\varepsilon$ has a basis of Taylor solutions over ever such disk 
$D$ and it is trivial on it. 

\eqref{Prop : super-harm everywhere-ii}. In particular, item \eqref{Prop : super-harm everywhere-i} also implies that for every 
such virtual open disk 
$D$ with boundary $x_{0,1}$ all the radii are solvable or over-solvable at every point of $D$. In particular they are constant on $D$ and, for all $i=1,\ldots,r$ we have 
$\Gamma_{\emptyset,i}(\M_\varepsilon)\cap D=\emptyset$. By the 
finiteness result of \cite[Theorem 3.9]{NP-I}, the graphs are locally 
finite around $x_{0,1}$ and therefore, up to shrink $\varepsilon$, 
we have
$\Gamma_{\emptyset,i}(\M_\varepsilon)\subseteq[x_{0,1},x_{0,1+\varepsilon}[$.

\eqref{Prop : super-harm everywhere-iii} and
\eqref{Prop : super-harm everywhere-iiii}. 
Over-solvable radii do not contribute to super-harmonicity. 
Moreover, for $\varepsilon>0$ small enough, they are constant of 
value $1$ on the whole $D_\varepsilon$. 
On the other hand, the other radii are by definition 
spectral at $x_{0,1}$. Hence, 
they can not be equal to $1$ at any point $y\in D_\varepsilon$, 
because this would mean by definition that 
$D_{\emptyset,i}(y,\M_\varepsilon)=D_\varepsilon$ and 
that they would be over-solvable. 
In other words, the index $i_{x_{0,1}}^{\mathrm{sol}}+1$ (cf. 
\eqref{eq : over-solvable cutoff}) separates the 
radii of $\M_\varepsilon$ globally on $D_\varepsilon$. 
Therefore, by Theorem \ref{MAIN Theorem}, we can 
decompose $\M_\varepsilon$ accordingly, over the whole 
$D_\varepsilon$, into its larger trivial sub-module and a quotient 
whose radii at $x_{0,1}$ are all spectral. 
By Proposition 
\ref{Prop. Exact sequence prop separate radii in the right order}, the 
first $i_{x_{0,1}}^{\mathrm{sol}}$ radii of $\M_\varepsilon$ coincide
 with those of its quotient at every point of $D_\varepsilon$. 
Therefore, we can assume that $\M_\varepsilon$ has no 
trivial submodules over $D_\varepsilon$, i.e. 
$\R_{\emptyset,i}(x_{0,1},\M_\varepsilon)=
\frac{1}{1+\varepsilon}$ for all $i$.\footnote{That is, the radii are 
solvable at $x_{0,1}$ (not over-solvable, nor spectral non-solvable).}
By \eqref{GOS-changed}, $H_{\emptyset,r}(-,\M_\varepsilon)$ is 
concave at $x_{0,1}$, and by item \eqref{Prop : super-harm everywhere-ii} it is constant on $D^+(0,1)$ 
and on each connected component of 
$C_\varepsilon-[x_{0,1},x_{0,1+\varepsilon}[$, 
hence $H_{\emptyset,r}(-,\M_\varepsilon)$ is super-harmonic on 
$D_\varepsilon$.\footnote{Notice that, 
at the moment we write this 
paper, the G.O.S. formula \eqref{eq : GOS} is proved over a discretely 
valued field $K$ (cf. \cite{Ch-Me-IV}). 
We quote \cite{NP-V,NP-VI} for the general proof and a more 
extensive treatment of the index theory. 
On the other hand it is possible, although not very practical, to allow 
only unspecified discretely valued field extensions $L/K$ 
(this is the approach of \cite{Ch-Me-IV}).}
The assertion for $H_{\emptyset,i}(-,\M_\varepsilon)$ 
is deduced from that of $H_{\emptyset,r}(-,\M_\varepsilon)$ by 
interpolation. Namely by convexity of the 
convergence Newton polygon one has 
$H_{\emptyset,i}(-,\M_\varepsilon)\leq 
\frac{i}{r}H_{\emptyset,r}(-,\M_\varepsilon)$, 
moreover these two functions coincide at $x_{0,1}$. This proves that 
$H_{\emptyset,i}(-,\M_\varepsilon)$ is super-harmonic. 
\end{proof}
\begin{remark}\label{rk : super-harmonicity}
In \cite[Theorem 3.9]{NP-I} one proves that 
$H_{\emptyset,i}(-,\M_\varepsilon)$ 
are super-harmonic over $D_\varepsilon$ 
outside a certain \emph{finite} set $\mathscr{C}_i$ and that 
$H_{\emptyset,1}(-,\M_\varepsilon)$ 
is always super-harmonic on the whole disk.  
Now, under the conditions of 
Proposition \ref{Prop : super-harm everywhere} all the 
$H_{\emptyset,i}(-,\M_\varepsilon)$ are super-harmonic at every point 
of $D_\varepsilon$. 
This shows that, if $H_{\emptyset,i}(-,\M_\varepsilon)$ 
is not super-harmonic, and if $\R_{\emptyset,1}(x_{0,1},
\M_\varepsilon)$ is solvable or over-solvable,
then the $\bs{\mathrm{NL}}$ assumption needs to fail for 
$\M_\varepsilon$. 
\if{
\framebox{
\begin{minipage}{430pt}
QUI COMPLETARE :

1) decomporre con Dwork-Robba separare il caso solubile dal caso non 
solubile. 

2) Guardare la formula di GOS in Robba IV nel caso spettrale non 
soluble. E dimostrare che la formula vale nel caso non solubile.

3) Mostrare che il caso solubile è la formula di GOS fatta sopra 
(risciverla con un numero arbitrario di buchi surconvergenti). 

4) Dedurre che quando GOS funziona abbiamo super-harmonicità .

....

....

The above example then proves then that the super-harmonicity holds 
if the Grothendieck-Ogg-Shafarevich formula holds. Unfortunately the 
assumptions on $\M$ for this formula are Non Liouville exponents ... 
completare !!!! .......

.........

........

\end{minipage}}
}\fi

\end{remark}

\begin{remark}
Since $\mathrm{Irr}_\infty(\M)$ involves spectral radii, these are 
stable by duality and one has 
$\mathrm{Irr}_\infty(\M)=\mathrm{Irr}_\infty(\M^*)$. 
Hence $\hdr^0(\M,\O^\dag(D))-\hdr^1(\M,\O^\dag(D))=\hdr^0(\M^*,\O^\dag(D))-\hdr^1(\M^*,\O^\dag(D))$ as soon as Grothendieck-Ogg-Shafarevich formula holds.
\end{remark}

\subsection{An example where $\Fs_{\geq i}$ is not a direct 
summand.}\label{optimality}

The Yoneda group $\mathrm{Ext}^1(\M,\N)$ 
of extensions 
of differential modules 
can be identified with $\Hdr^1(\M^*\otimes \N)$ 
(cf. Lemma \ref{Lemma : Ext^1=H^1}). 
\begin{lemma}\label{Lemma : non splitting extension}
Assume that $\M$ has the property $\bs{\mathrm{NL}}$.
If $\partial_bH_{\emptyset,r}(x_{0,1},\M_\varepsilon)\leq -1$ there 
exists a non splitting exact sequence 
$0 \to \M_\varepsilon \to \P_\varepsilon \to \O(D_\varepsilon) \to 0$.
\end{lemma}
\begin{proof}
One has 
$\hdr^1(\M, \O^\dag(D))=-\partial_bH_{\emptyset,r}(x_{0,1},\M)
\geq 1$. Therefore,
\begin{equation}
\mathrm{Ext}^1(\O^\dag(D),\M) 
=\Hdr^1(\M\otimes\O^\dag(D),\O^\dag(D))=\Hdr^1(\M,\O^\dag(D))\neq 0\;.
\end{equation}
\end{proof}

Let $\M$ be a rank one differential module over $\O^\dag(D)$ such 
that $\R_{\emptyset,1}(x_{0,1},\M_\varepsilon)$ 
is a solvable radius at $x_{0,1}\in D_\varepsilon$, 
and $\mathrm {Irr}_\infty(\M)\geq 2$ 
(i.e. $\partial_bH_{\emptyset,r}(x_{0,1},\M_\varepsilon)\leq -1$).
Such differential modules have been classified in \cite{Rk1}. 
Consider the dual of the non splitting sequence of  Lemma 
\ref{Lemma : non splitting extension}
\begin{equation}
0\to\O(D_\varepsilon)\to\P_\varepsilon^{*}\to\M_\varepsilon^{*}\to 
0\;.
\end{equation} 
By Proposition \ref{Prop. Exact sequence prop separate radii in the right 
order} the radii of $\P_\varepsilon^*$ are the union of those of 
$\M_\varepsilon^*$ and $\O(D_\varepsilon)$:
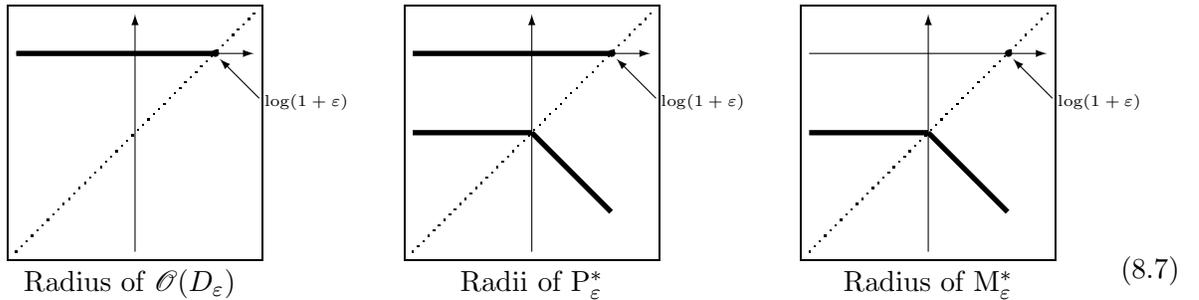
\begin{equation}\label{eq : picture of radii of p^*}
\begin{picture}(400,100)
\put(0,10){
\framebox{\begin{picture}(90,90)
\put(45,0){\vector(0,1){90}}
\put(0,75){\vector(1,0){90}}
\qbezier[45](0,0)(45,45)(90,90)
\put(73.5,73.5){\begin{tiny}$\bullet$\end{tiny}}
\put(92,58){\vector(-1,1){15}}
\put(94,55){\begin{tiny}$\log(1+\varepsilon)$\end{tiny}}
\put(0,75){\linethickness{2pt}\line(1,0){75}}

\end{picture}}}
\put(10,-5){Radius of $\O(D_\varepsilon)$}

\put(150,10){
\framebox{\begin{picture}(90,90)
\put(45,0){\vector(0,1){90}}
\put(0,75){\vector(1,0){90}}
\qbezier[45](0,0)(45,45)(90,90)
\put(73.5,73.5){\begin{tiny}$\bullet$\end{tiny}}
\put(92,58){\vector(-1,1){15}}
\put(94,55){\begin{tiny}$\log(1+\varepsilon)$\end{tiny}}
\put(0,75){\linethickness{2pt}\line(1,0){75}}

\put(0,45){\linethickness{2pt}\line(1,0){45}}
\put(45,45){\linethickness{2pt}\line(1,-1){30}}
\end{picture}}}
\put(175,-5){Radii of $\P^*_\varepsilon$}

\put(300,10){
\framebox{\begin{picture}(90,90)
\put(45,0){\vector(0,1){90}}
\put(0,75){\vector(1,0){90}}
\qbezier[45](0,0)(45,45)(90,90)
\put(73.5,73.5){\begin{tiny}$\bullet$\end{tiny}}
\put(92,58){\vector(-1,1){15}}
\put(94,55){\begin{tiny}$\log(1+\varepsilon)$\end{tiny}}

\put(0,45){\linethickness{2pt}\line(1,0){45}}
\put(45,45){\linethickness{2pt}\line(1,-1){30}}
\end{picture}}}
\put(320,-5){Radius of $\M^*_\varepsilon$}

\end{picture}
\end{equation}
The functions of the pictures are precisely
\begin{equation}
\widetilde{\rho}\;\mapsto\;
\log(\R_{\emptyset,i}(x_{0,\exp(\widetilde{\rho})},\bullet))\;,
\end{equation}
where $\varepsilon$ is \emph{unspecified}, in order to avoid 
complicate behavior of the radius of $\M_\varepsilon^*$.
The dotted line denotes the region of solvability. At the left hand side 
of this line one has over-solvable radii 
(which are always constant functions), 
and on its right hand side one has spectral non solvable radii.
The following proposition provides a example of differential equation 
such that  $\Fs_{\geq i}$ is not a direct factor. 
By Proposition \ref{Prop : super-harm everywhere}, one sees that
\begin{equation}
\Gamma_{\emptyset,1}(\M^*_\varepsilon)=
\Gamma_{\emptyset,1}(\P^*_\varepsilon)=[x_{0,1},x_{0,1+\varepsilon}[\;,\qquad
\qquad
\Gamma_{\emptyset,2}(\P^*_\varepsilon)=
\Gamma_{\emptyset,1}(\O(D_\varepsilon))=\emptyset\;.
\end{equation}
The radii of $\P^*_\varepsilon$ are then separated on the whole disk 
$D_\varepsilon$, if $\varepsilon>0$ is small enough. In this case 
$\O(D_\varepsilon)=(\P^*_\varepsilon)_{\geq 2}$. 
Moreover, by construction, $\P^*_\varepsilon$ in not isomorphic to the 
direct sum of 
$\M_\varepsilon^*=(\P_\varepsilon^*)_{< 2}$ and 
$(\P_\varepsilon^*)_{\geq 2}=\O(D_\varepsilon)$.
\subsection{Non compatibility of solvable or over-solvable radii with 
duality and with exactness.}\label{section counterexample 
to duality and exactness}
Consider now the sequence of Lemma 
\ref{Lemma : non splitting extension}
\begin{equation}\label{eq : exact seq of MPO}
0\to\M_\varepsilon\to\P_\varepsilon\to \O(D_\varepsilon)\to 0\;.
\end{equation}
We now show that over-solvable radii of this sequence do not 
behave as spectral one because the $\Hdr^0(D_\varepsilon,
\M_{\varepsilon})$ functor is not exact on this sequence 
(cf. Theorem \ref{thm : spectral is exact}). More precisely, we claim 
that for every $\varepsilon>0$ one has 
(cf. Remark \ref{Remark : an explicit quote to counterexample})
\begin{equation}\label{eq : Hdr=0 example section 6}
\Hdr^0(D_\varepsilon,\M_\varepsilon)\;=\; 
\Hdr^0(D_\varepsilon,\P_\varepsilon)\;=\;0\;.
\end{equation}
By contrapositive, assume that we a non-zero solution of $\P_\varepsilon$ over $D_\varepsilon$. That solution generates 
a trivial submodule of $\P_\varepsilon$ whose intersection with 
$\M_\varepsilon$ is zero because $\M_\varepsilon$ is not trivial. 
Therefore, $\P_\varepsilon$ is 
direct sum of $\M_\varepsilon$ with that trivial submodule.
This is absurd because the sequence \eqref{eq : exact seq of MPO} 
does not split by its definition. Therefore \eqref{eq : Hdr=0 example section 6} holds.

We now deduce from this that $\P_\varepsilon$ can not have 
over-solvable radii at $x_{0,1}$, because these radii would 
generate a trivial sub-module over $D_\varepsilon$. 
This proves that \emph{the over-solvable radii of 
$\P_\varepsilon^*$ are not stable by duality}. On the other hand, 
Proposition 
\ref{Prop : 1-th radius and dual} guarantees that 
its spectral non-solvable radii are stable by duality  
so the pictures of the radii of \eqref{eq : exact seq of MPO} 
coincides with \eqref{eq : picture of radii of p^*} under the doted 
lines. By continuity of the radii functions, one finds:
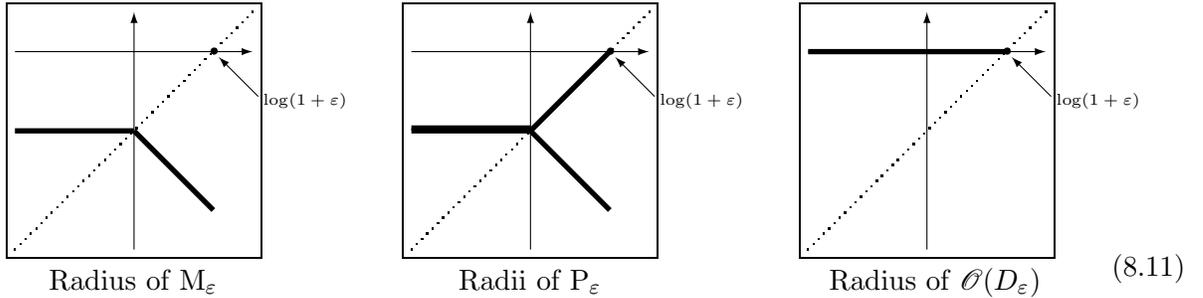
\begin{equation}\label{eq : picture of radii of p^*}
\begin{picture}(400,100)

\put(0,10){
\framebox{\begin{picture}(90,90)
\put(45,0){\vector(0,1){90}}
\put(0,75){\vector(1,0){90}}
\qbezier[45](0,0)(45,45)(90,90)
\put(73.5,73.5){\begin{tiny}$\bullet$\end{tiny}}
\put(92,58){\vector(-1,1){15}}
\put(94,55){\begin{tiny}$\log(1+\varepsilon)$\end{tiny}}

\put(0,45){\linethickness{2pt}\line(1,0){45}}
\put(45,45){\linethickness{2pt}\line(1,-1){30}}
\end{picture}}}
\put(20,-5){Radius of $\M_\varepsilon$}

\put(150,10){
\framebox{\begin{picture}(90,90)
\put(45,0){\vector(0,1){90}}
\put(0,75){\vector(1,0){90}}
\qbezier[45](0,0)(45,45)(90,90)
\put(73.5,73.5){\begin{tiny}$\bullet$\end{tiny}}
\put(92,58){\vector(-1,1){15}}
\put(94,55){\begin{tiny}$\log(1+\varepsilon)$\end{tiny}}


\put(0,46){\linethickness{2pt}\line(1,0){45}}
\put(45,45){\linethickness{2pt}\line(1,1){30}}

\put(0,45){\linethickness{2pt}\line(1,0){45}}
\put(45,45){\linethickness{2pt}\line(1,-1){30}}
\end{picture}}}
\put(175,-5){Radii of $\P_\varepsilon$}

\put(300,10){
\framebox{\begin{picture}(90,90)
\put(45,0){\vector(0,1){90}}
\put(0,75){\vector(1,0){90}}
\qbezier[45](0,0)(45,45)(90,90)
\put(73.5,73.5){\begin{tiny}$\bullet$\end{tiny}}
\put(92,58){\vector(-1,1){15}}
\put(94,55){\begin{tiny}$\log(1+\varepsilon)$\end{tiny}}
\put(0,75){\linethickness{2pt}\line(1,0){75}}

\end{picture}}}
\put(315,-5){Radius of $\O(D_\varepsilon)$}

\end{picture}
\end{equation}
In particular, we have
\begin{equation}
\Gamma_{\emptyset,1}(\M_\varepsilon)=
\Gamma_{\emptyset,1}(\P_\varepsilon)=
\Gamma_{\emptyset,2}(\P_\varepsilon)=
[x_{0,1},x_{0,1+\varepsilon}[\;.
\end{equation}
This shows that over-solvable radii of $\P_\varepsilon$ 
(nor their controlling graphs) are not preserved by duality (cf. Remark 
\ref{rk : quote counterex for duality} and 
Proposition \ref{Prop : 1-th radius and dual}). 

\begin{remark}\label{remark - radii non separated VS H^1}
Among all the extensions of $\O(D_\varepsilon)$ by 
$\M_\varepsilon^*$, 
the unique non trivial extension $\P_\varepsilon^*$ is also the unique 
one for which the radii are not separated. And its controlling graphs 
contradicts the assumption of Theorem 
\ref{Thm : criterion self direct sum}. Notice that that information is 
written in the radii of $\P_\varepsilon$ and $\P^*_\varepsilon$, while 
the radii and the controlling graphs of $\M_\varepsilon$, 
$\O(D_\varepsilon)$, and their duals, are stable by duality.
\end{remark}

\begin{remark}\label{Remark : computation of C_i}
The differential module $\P_\varepsilon$ 
satisfies $\mathscr{C}_1=\mathscr{C}_2=\emptyset$ 
(cf. \cite[Theorem 3.9]{NP-I}). Indeed 
$H_{\emptyset,2}(-,P_\varepsilon)$ is a constant function on 
$D_\varepsilon$, 
and $\Gamma(H_{\emptyset,2}(-,\P_\varepsilon))=\emptyset$ 
(cf. \cite[item (iv-c) of Theorem 3.9]{NP-I}).
Also $\P_\varepsilon^*$ verifies 
$\mathscr{C}_1=\mathscr{C}_2=\emptyset$, 
but the reason is that 
$\R_{\emptyset,2}(-,\P_\varepsilon^*)$ is constant and 
$\Gamma_{\emptyset,2}(\P_\varepsilon^*)=\emptyset$.
\end{remark}

\appendix

\section{A note about the definition of the radius}
\label{Comparison with the general definition of radii}
As already observed the radii only depends on the skeleton 
$\Gamma_S$, in the sense that if 
$\Gamma_{S}=\Gamma_{S'}$ then $\R_{S,i}(-,-)=\R_{S',i}(-,-)$ (cf.  
Remark \ref{changing triang}).

\if{In this section we consider a more general definition of the radii 
based on the idea that the datum defining the radii is a graph 
$\Gamma_S$ instead of a triangulation. So, we define the radii starting 
from a graph which is not necessarily the skeleton of a weak 
triangulation.
We here show that such a point of view is not a real generalization, in 
the sense that the main theorems (finiteness, 
continuity, and decomposition) in this new context can be deduced 
from the same results in the framework of weak triangulations. For this 
reason the framework of weak triangulations seems to us the optimal 
one.
}\fi

In this section, we explore a broader definition of the radii by 
introducing the concept of a graph as the basis for defining these radii, 
rather than relying solely on a triangulation. This means that the radii 
will be derived from a graph that might not necessarily equals the 
skeleton of a weak triangulation.

However, upon closer examination, we prove that this approach does not 
significantly generalize the concept. Not really surprisingly, the 
main theorems concerning finiteness, continuity, and decomposition in 
this new context can be deduced from the same results obtained within 
the framework of weak triangulations. In other words, the crucial 
results that we seek to achieve through this alternative perspective can 
already be derived using the existing approach of weak triangulations.

\subsection{Definition of the radii}\label{Definition of the radii-appendix}
In order to define the radii we need the notion of 
maximal disks. For this we proceed as follow. 
Let $\Fs$ be a differential equation over $X$ of rank $r$, and let $x\in X$.

A graph $\GS\subseteq X$ is called a 
\emph{weakly admissible graph} if 
$X-\GS$ is a disjoint union of virtual open disks.\footnote{This 
generalizes a terminology of Ducros \cite[(5.1.3)]{Duc}, where one 
defines an \emph{analytically admissible graph} 
as a weakly admissible graph such that the disks that are connected 
components of $X-\GS$ are relatively compact in $X$. In 
analogy with the definition of weak triangulations, 
we allow the empty set of a virtual open disk 
to be a weakly admissible graph.} 

For all $x\in X$ we call maximal disk the $\Omega$-rational disk 
$D(x,\GS)$ (cf. Definition \ref{Def : can extensions}). 
The disk $D(x,\GS)$ is empty if, and only if, 
$x\in\GS$ and if $x$ is a point of type $1$. 
In this case set $\R_{\GS,i}(x,\Fs):=1$, for all $i=1,\ldots,r$.

Otherwise, imitating section \ref{multiradius}, 
choose an isomorphism 
$D(x,\GS)\simto D_{\Omega}^-(0,R)$ 
sending $t_x$ into $0$. 
Consider the restriction 
$\widetilde{\Fs}$ of $\Fs$ to $D^-_\Omega(0,R)$. 
And then define 
$\R_{\GS,i}^{\widetilde{\Fs}}(x)$
 as the radius of the largest open disk 
$D$ centered at $0$, contained in $D^-_\Omega(0,R)$, such that 
$\widetilde{\Fs}$ has at least $r-i+1$ linearly 
independent solutions on $D$. 
Now, for all $i=1,\ldots,r$, we set
\begin{equation}
\R_{\GS,i}(x,\Fs)\;:=\;
\R_{\GS,i}^{\widetilde{\Fs}}(x)/R\;.
\end{equation}
We call $\GS$-multiradius of $\Fs$ the tuple 
$(\R_{\GS,1}(x,\Fs),\ldots,\R_{\GS,r}(x,\Fs))$.

The following definition coincides with Definition
\ref{Def. :: controlling graphs} replacing $\Gamma_S$ with 
$\GS$.
\begin{definition}
We call $\GS$-controlling graph, or $\GS$-skeletons, 
of $\R_{\GS,i}(-,\Fs)$ the set of points $x\in X$ that do not 
admit as a neighborhood in $X$ a virtual open disk $D$ on which 
$\R_{\GS,i}(-,\Fs)$ is constant and such that $D\cap\GS=\emptyset$. We denote it by $\Gamma_{\GS,i}(\Fs)$.
\end{definition}
It follows by the definition that 
$\GS\subseteq\Gamma_{\GS,i}(\Fs)$.

\subsection{Properties}
In this section we prove that the radii $\R_{\GS,i}(x,\Fs)$ are 
subjected to analogous properties of the radii $\R_{S,i}(x,\Fs)$ 
attached to a weak triangulation $S$. 
Indeed their restriction to a maximal 
disk $D(x,\GS)$ is determined by 
the functions 
$\R_{\GS,i}^{\widetilde{\Fs}}(x)$
 in the spirit of \cite{NP-I}. 
Below, one shows how to reduce the study of the 
$\GS$-radii to the case of the radii defined by a weak 
triangulation.

\begin{proposition}[\protect{\cite[(1.3.15.2)]{Duc}}]
A weakly admissible graph 
is a closed connected subset of $X$, hence it is a graph 
in the sense of \cite[Definition 1.1.1]{NP-IV}.
\end{proposition}
\if{\begin{proof}
By assumption the connected components $D$ of $X-\GS$ 
are all virtual open disks with boundary $x_D\in\GS$.
The claim is then equivalent to say that the complement $C$ in $X$ of 
a disjoint family of virtual open sub-disks of $X$ is a connected subset. 
Now the complement $C(D)$ of a disk is a connected subset, because 
$X$ is a connected tree. We have to prove that the intersection 
$\cap_D C(D)$ of the complements of such disks is connected.
The existence of a triangulation of 
$X$ implies that we can assume that $X$ is a virtual open annulus or 
disk. In this case the claim is true.
\end{proof}

\comment{JerÃÂŽme: help ! Je ne vois plus pourquoi les deux dernières 
lignes de la preuve sont vraies ...}
}\fi

\begin{proposition}
Let $\GS$ be a weakly admissible graph. 
Then there exists a weak 
triangulation $S$ of $X$ such that $\Gamma_S\subseteq\GS$.
\end{proposition}
\begin{proof}
Let $S'$ be an arbitrary triangulation. The intersection 
$\Gamma:=\GS\cap\Gamma_{S'}$ is a locally finite graph 
whose complement $X-\Gamma$ is a disjoint union of virtual disks.  
Since 
$\Gamma$ is contained in $\Gamma_{S'}$, then $\Gamma$ 
is the skeleton of a weak triangulation $S$ by \cite[5.2.2.3]{Duc}.
\if{Let $D$ be a virtual 
disk in $X$ whose boundary $s$ lies in $\Gamma$ and 
such that $D\cap\Gamma=\emptyset$. 
If $\Gamma_{S'}$ intersects $D$, then one sees that the set 
$S'' := (S'-(S'\cap D))\cup\{s\}$ is a triangulation such that 
$\Gamma_{S''}=\Gamma_{S'}-(\Gamma_{S'}\cap D)$, 
and such that $\Gamma\subseteq \Gamma_{S''}$.
In other words we can remove from $\Gamma_{S'}$ the segments 
that do not lie in $\GS$. 
Now $\Gamma$ is a locally finite graph because it is contained in 
$\Gamma_{S'}$. This proves that $\Gamma$ is the skeleton of a 
triangulation $S$.
}\fi
\end{proof}

\begin{proposition}\label{Prop: A4 fgh}
Let $\Gamma_S\subseteq\GS$ be the skeleton of a 
triangulation contained in the weakly admissible graph 
$\GS$. Then for all 
$i=1,\ldots,r$ one has
\begin{equation}\label{eq : R_Gamma VS R_S}
\R_{\GS,i}(x,\Fs)\;=\;\min\Bigl(\;1\;,\; 
f_{S,\GS}(x)\cdot\R_{S,i}(x,\Fs)\;
\Bigr)
\end{equation}
where $f_{S,\GS}:X\to[1,+\infty[$ is the modulus 
of the inclusion of 
disks $D(x,\GS)\subseteq D(x,S)$ 
(cf. Definition \eqref{def:modulusannulus}).
\hfill $\Box$
\end{proposition}
The following  Proposition \ref{Prop : A5 ecp} together with 
Theorem \ref{Thm : finiteness} shows that 
the $\GS$-controlling graph $\Gamma_{\GS,i}(\Fs)$ 
is locally finite if, and only if, so $\GS$ is. 
\begin{proposition}\label{Prop : A5 ecp}
One has 
$\Gamma_{\GS,i}(\Fs)=\GS\cup\Gamma_{S,i}(\Fs)$.
\end{proposition}
\begin{proof}
The proof coincides with that of Proposition \ref{Prop : A-5 ecp} 
(replacing $S'$ with $\GS$).
\end{proof}

\begin{lemma}\label{Prop. : A6}
The function 
$x\mapsto f_{S,\GS}(x)$ is continuous 
if, and only if, $\GS$ is a locally finite graph.\hfill$\Box$
\end{lemma}
\if{\begin{proof}
La fonction est localement constante en dehors de $\S$. Si le graphe 
$\S$ est localement fini, il suffit donc de tester ça sur les arêtes de 
$\S$ et ça se fait.

    Si le graphe n'est pas localement fini, on considère un point $x$ de 
$\S$ d'où partent une infinité de branches de $\S$. Supposons qu'on a 
$f(x)=R$. La fonction $f$ est strictement décroissante sur une infinité 
de branches de $\S$. Pour $\eps>0$ assez petit, la partie 
$|f| > R-\eps$ va couper une infinité de branches hors de $x$. Ce n'est 
donc pas un ouvert.
\end{proof}
}\fi
\begin{theorem}
If $\GS$ is a locally finite graph, the functions 
$\R_{\GS,i}(-,\Fs)$ are continuous. 
\end{theorem}
\begin{proof}
This follows directly from \eqref{eq : R_Gamma VS R_S} 
and the continuity of both 
$\R_{S,i}(-,\Fs)$ and $f_{S,\GS}$.
\end{proof}

\begin{remark}
The trivial equation $(\O_X,d)$ satisfies 
$\R_{\GS,1}(x,\Fs)=1$ for all $x\in X$ and all weakly admissible graph
$\GS$. So it is always continuous.
\end{remark}

\begin{remark}
If $\GS$ is not locally finite, then there are differential equation with 
non continuous radii. Here we provide a basic explicit  example.

Let $X$ be the the disk $\mathrm{D}^-_K(0,1)$ 
with empty weak triangulation. 
Denote by $\Fs_a$ the rank-one differential equation over $X$, 
given by $y'=a\cdot y$, with $a\in K$.
A direct computation of the Taylor expansion of its solution 
shows that  for all $x\in X$ the function
$\R_{\emptyset,1}(x,\Fs_a)$ is constant on $X$ with value 
$R_a:=\min(1,|p|^{\frac{1}{p-1}}/ |a|)$. 
If the valuation of $K$ is non trivial, it is then possible to have 
differential equations having arbitrarily 
small and constant radii on a disk. 

By \eqref{eq : R_Gamma VS R_S}, it follows that, 
for all bifurcation point $x\in \GS$, 
there exists $a\in K$, such that $\R_{\GS,1}(y,\Fs_a)=
f_{S,\GS}(y)\cdot R_a$ for all $y$ close enough to $x$.
This proves that 
$\Gamma_{\GS,1}(\Fs_a)=\GS$ around $x$. 

If $\GS$ is not locally finite, then, by Lemma \ref{Prop. : A6}, 
the function $\R_{\GS,1}(y,\Fs_a)$ is not continuous. 
\end{remark}


\begin{theorem}
If $\R_{\GS,i-1}(x,\Fs)<\R_{\GS,i}(x,\Fs)$
for all $x\in X$, then there exists a sub-object 
$\Fs_{\geq i}$ of $\Fs$ satisfying analogous 
properties to Theorem \ref{MAIN Theorem}, 
and of section \ref{Conditions to have a direct sum decomposition}, 
replacing everywhere the index $S$ by $\GS$. Moreover 
$\Fs_{\geq i}$ is independent on the choice of $\GS$.
\end{theorem}
\begin{proof}
Let $S$ be a weak triangulation such that 
$\Gamma_S\subseteq\GS$. 
For all $x\in X$ one has $D(x,\GS)\subseteq D(x,S)$. 
So $\omega_{S,j}(x,\Fs)=\omega_{\GS,j}(x,\Fs)$ for all 
$j\leq i$. the index $i$ also separates the $\Gamma_S$-radii. 
By Theorem \ref{MAIN Theorem} we have the existence of 
a submodule $\Fs_{S,\geq i}$ separating the $S$-radii. 
As in Remark \ref{Remark : changing tr F_S,I}, and Proposition 
\ref{Prop. : independence on S} 
one shows that $\Fs_{S,\geq i}$ also separates 
the $\GS$-radii : $\Fs_{S,\geq i}=\Fs_{\GS,\geq i}$.
\end{proof}
\begin{remark}
According to Remark \ref{Remark : bad condition}, to prove that 
$\Fs_{\geq i}$ is a direct factor, it is 
convenient to chose $\GS$ as small as possible.
This can be done by replacing $\GS$ by a convenient weak 
triangulation $S$ such that $\Gamma_S\subseteq\GS$.
\end{remark}


\bibliographystyle{amsalpha}
\bibliography{deco}

\def\cprime{$'$}
\providecommand{\bysame}{\leavevmode\hbox to3em{\hrulefill}\thinspace}
\providecommand{\MR}{\relax\ifhmode\unskip\space\fi MR }
\providecommand{\MRhref}[2]{%
  \href{http://www.ams.org/mathscinet-getitem?mr=#1}{#2}
}
\providecommand{\href}[2]{#2}
\begin{thebibliography}{{Pul}15}

\bibitem[And02]{An}
Y.~Andr{\'e}, \emph{Filtrations de type {H}asse-{A}rf et monodromie
  {$p$}-adique}, Invent. Math. \textbf{148} (2002), no.~2, 285--317.

\bibitem[And09]{Andre-Slope-Filtration}
Yves Andr{\'e}, \emph{Slope filtrations}, Confluentes Math. \textbf{1} (2009),
  no.~1, 1--85 (English).

\bibitem[Bal10]{Balda-Inventiones}
Francesco Baldassarri, \emph{Continuity of the radius of convergence of
  differential equations on {$p$}-adic analytic curves}, Invent. Math.
  \textbf{182} (2010), no.~3, 513--584. \MR{2737705 (2011m:12015)}

\bibitem[Ber90]{Ber}
Vladimir~G. Berkovich, \emph{Spectral theory and analytic geometry over
  non-{A}rchimedean fields}, Mathematical Surveys and Monographs, vol.~33,
  American Mathematical Society, Providence, RI, 1990.

\bibitem[Ber93]{bleu}
\bysame, \emph{\'{E}tale cohomology for non-{A}rchimedean analytic spaces},
  Inst. Hautes \'Etudes Sci. Publ. Math. (1993), no.~78, 5--161 (1994).
  \MR{MR1259429 (95c:14017)}

\bibitem[Bou72]{Bou-Comm-Alg}
Nicolas Bourbaki, \emph{Elements of mathematics. {Commutative} algebra.
  {Translated} from the {French}}, Actualit{\'e}s scientifiques et
  industrielles, {Hermann}. {Adiwes} {International} {Series} in {Mathematics}.
  {Paris}: {Hermann}; {Reading}, {Mass}.: {Addison}-{Wesley} {Publishing}
  {Company}. xxiv, 625 p. (1972)., 1972.

\bibitem[Bou98]{BouAlgCom}
\bysame, \emph{Commutative algebra. {C}hapters 1--7}, Elements of Mathematics
  (Berlin), Springer-Verlag, Berlin, 1998, Translated from the French, Reprint
  of the 1989 English translation. \MR{MR1727221 (2001g:13001)}

\bibitem[Bou07]{Bou-Alg-1-3}
\bysame, \emph{{\'E}l{\'e}ments de math{\'e}matique. {Alg{\`e}bre}. {Chapitres}
  1 {\`a} 3}, reprint of the 1970 original ed., Berlin: Springer, 2007
  (French).

\bibitem[BV07]{Dv-Balda}
F.~Baldassarri and L.~Di Vizio, \emph{Continuity of the radius of convergence
  of $p$-adic differential equations on berkovich spaces}, arXiv, 2007,
  \url{http://arxiv.org/abs/0709.2008}, pp.~1--22.

\bibitem[CD94]{Ch-Dw}
G.~Christol and B.~Dwork, \emph{Modules diff\'erentiels sur des couronnes},
  Ann. Inst. Fourier (Grenoble) \textbf{44} (1994), no.~3, 663--701.
  \MR{MR1303881 (96f:12008)}

\bibitem[Chr83]{Ch}
Gilles Christol, \emph{Modules diff\'erentiels et \'equations diff\'erentielles
  {$p$}-adiques}, Queen's Papers in Pure and Applied Mathematics, vol.~66,
  Queen's University, Kingston, ON, 1983.

\bibitem[Chr12]{Christol-Book}
Gilles Christol, \emph{Le théorème de turritin $p$-adique. (book in
  preparation)}, 2012, \url{http://www.math.jussieu.fr/~christol/courspdf.pdf}.

\bibitem[CM00]{Ch-Me-III}
G.~Christol and Z.~Mebkhout, \emph{Sur le th\'eor\`eme de l'indice des
  \'equations diff\'erentielles {$p$}-adiques. {III}}, Ann. of Math. (2)
  \textbf{151} (2000), no.~2, 385--457. \MR{1765703 (2001k:12014)}

\bibitem[CM01]{Ch-Me-IV}
\bysame, \emph{Sur le th\'eor\`eme de l'indice des \'equations
  diff\'erentielles {$p$}-adiques. {IV}}, Invent. Math. \textbf{143} (2001),
  no.~3, 629--672. \MR{1817646 (2002d:12005)}

\bibitem[CM02]{Astx}
Gilles Christol and Zoghman Mebkhout, \emph{\'{E}quations diff\'erentielles
  {$p$}-adiques et coefficients {$p$}-adiques sur les courbes}, Ast\'erisque
  (2002), no.~279, 125--183, Cohomologies $p$-adiques et applications
  arithm{\'e}tiques, II. \MR{1922830 (2003i:12014)}

\bibitem[CR94]{Ch-Ro}
G.~Christol and P.~Robba, \emph{\'{E}quations diff\'erentielles {$p$}-adiques},
  Actualit\'es Math\'ematiques., Hermann, Paris, 1994, Applications aux sommes
  exponentielles. [Applications to exponential sums].

\bibitem[Del70]{Deligne-Reg-Sing}
Pierre Deligne, \emph{\'{E}quations diff\'erentielles \`a points singuliers
  r\'eguliers}, Lecture Notes in Mathematics, Vol. 163, Springer-Verlag,
  Berlin, 1970. \MR{MR0417174 (54 \#5232)}

\bibitem[DM69]{Semi-Stable-red-thm}
Pierre Deligne and D.~Mumford, \emph{The irreducibility of the space of curves
  of a given genus}, Publ. Math., Inst. Hautes {\'E}tud. Sci. \textbf{36}
  (1969), 75--109 (English).

\bibitem[DMR07]{Correspondance-Malgrange-Ramis}
Pierre Deligne, Bernard Malgrange, and Jean-Pierre Ramis, \emph{Singularit\'es
  irr\'eguli\`eres}, Documents Math\'ematiques (Paris) [Mathematical Documents
  (Paris)], 5, Soci\'et\'e Math\'ematique de France, Paris, 2007,
  Correspondance et documents. [Correspondence and documents]. \MR{2387754
  (2009d:32033)}

\bibitem[DR77]{Dw-Robba}
B.~Dwork and P.~Robba, \emph{On ordinary linear {$p$}-adic differential
  equations}, Trans. Amer. Math. Soc. \textbf{231} (1977), no.~1, 1--46.
  \MR{0447247 (56 \#5562)}

\bibitem[Duc]{Duc}
Antoine Ducros, \emph{La structure des courbes analytiques},
  \url{http://www.math.jussieu.fr/~ducros/livre.html}.

\bibitem[Dwo73]{Dw-II}
B.~Dwork, \emph{On {$p$}-adic differential equations. {II}. {T}he {$p$}-adic
  asymptotic behavior of solutions of ordinary linear differential equations
  with rational function coefficients}, Ann. of Math. (2) \textbf{98} (1973),
  366--376. \MR{0572253 (58 \#27987b)}

\bibitem[Esc16]{Esct}
Alain Escassut, \emph{The corona problem on a complete ultrametric
  algebraically closed field}, \(p\)-Adic Numbers Ultrametric Anal. Appl.
  \textbf{8} (2016), no.~2, 115--124 (English).

\bibitem[Kat87]{Katz-cyclic-vect}
Nicholas~M. Katz, \emph{A simple algorithm for cyclic vectors}, Amer. J. Math.
  \textbf{109} (1987), no.~1, 65--70. \MR{878198 (88b:13001)}

\bibitem[Ked04]{Ked}
Kiran~S. Kedlaya, \emph{A {$p$}-adic local monodromy theorem}, Ann. of Math.
  (2) \textbf{160} (2004), no.~1, 93--184. \MR{MR2119719 (2005k:14038)}

\bibitem[Ked15]{Kedlaya-draft}
\bysame, \emph{Local and global structure of connections on nonarchimedean
  curves}, Compos. Math. \textbf{151} (2015), no.~6, 1096--1156. \MR{3357180}

\bibitem[Ked22]{Kedlaya-book-2}
Kiran~Sridhara Kedlaya, \emph{{{\(p\)}}-adic differential equations}, 2nd
  edition ed., Camb. Stud. Adv. Math., vol. 199, Cambridge: Cambridge
  University Press, 2022 (English).

\bibitem[Laz62]{Lazard}
Michel Lazard, \emph{Les z\'eros des fonctions analytiques d'une variable sur
  un corps valu\'e complet}, Inst. Hautes \'Etudes Sci. Publ. Math. (1962),
  no.~14, 47--75. \MR{0152519 (27 \#2497)}

\bibitem[Lev75]{Levelt}
A.~H.~M. Levelt, \emph{Jordan decomposition for a class of singular
  differential operators}, Ark. Mat. \textbf{13} (1975), 1--27. \MR{0500294 (58
  \#17962)}

\bibitem[Mal74]{Malgrange-Irreg}
Bernard Malgrange, \emph{Sur les points singuliers des \'equations
  diff\'erentielles}, Enseignement Math. (2) \textbf{20} (1974), 147--176.
  \MR{0368074 (51 \#4316)}

\bibitem[Meb02]{Me}
Zoghman Mebkhout, \emph{Analogue {$p$}-adique du th\'eor\`eme de {T}urrittin et
  le th\'eor\`eme de la monodromie {$p$}-adique}, Invent. Math. \textbf{148}
  (2002), no.~2, 319--351.

\bibitem[{Poi}13]{Angie}
J\'er\^ome {Poineau}, \emph{{Les espaces de Berkovich sont ang\'eliques.}},
  {Bull. Soc. Math. Fr.} \textbf{141} (2013), no.~2, 267--297 (French).

\bibitem[PP15a]{Potentiel}
J\'er\^ome {Poineau} and Andrea {Pulita}, \emph{{Continuity and finiteness of
  the radius of convergence of a $p$-adic differential equation via potential
  theory.}}, {J. Reine Angew. Math.} \textbf{707} (2015), 125--147 (English).

\bibitem[PP15b]{NP-II}
\bysame, \emph{{The convergence Newton polygon of a $p$-adic differential
  equation. II: Continuity and finiteness on Berkovich curves.}}, {Acta Math.}
  \textbf{214} (2015), no.~2, 357--393 (English).

\bibitem[PP24a]{Banachoid}
\bysame, \emph{Banachoid spaces},
  \url{https://www-fourier.ujf-grenoble.fr/~pulitaa/Publications/Banachoid.pdf},
  2024.

\bibitem[PP24b]{NP-IV}
\bysame, \emph{The convergence newton polygon of a $p$-adic differential
  equation iv : controlling graphs}, soon on arxiv, 2024, pp.~1--50.

\bibitem[PP24c]{NP-V}
\bysame, \emph{The convergence newton polygon of a $p$-adic differential
  equation v : local index theorems}, arxiv, 2024,
  \url{https://www-fourier.ujf-grenoble.fr/~pulitaa/Publications/local.pdf},
  pp.~1--88.

\bibitem[PP24d]{NP-VI}
\bysame, \emph{The convergence newton polygon of a $p$-adic differential
  equation vi : global index theorems}, arxiv, 2024,
  \url{https://www-fourier.ujf-grenoble.fr/~pulitaa/Publications/global.pdf},
  pp.~1--87.

\bibitem[PT21]{PoineauTurchettiVIASMI}
J{\'e}r{\^o}me Poineau and Daniele Turchetti, \emph{Berkovich curves and
  {Schottky} uniformization. {I}: {The} {Berkovich} affine line}, Arithmetic
  and geometry over local fields. VIASM 2018. Based on lectures given during
  the program ``Arithmetic and geometry of local and global fields'', summer
  2018, Hanoi, Vietnam, Cham: Springer, 2021, pp.~179--223 (English).

\bibitem[Pul07]{Rk1}
Andrea Pulita, \emph{Rank one solvable {$p$}-adic differential equations and
  finite abelian characters via {L}ubin-{T}ate groups}, Math. Ann. \textbf{337}
  (2007), no.~3, 489--555. \MR{MR2274542}

\bibitem[{Pul}15]{NP-I}
Andrea {Pulita}, \emph{{The convergence Newton polygon of a $p$-adic
  differential equation. I: Affinoid domains of the Berkovich affine line.}},
  {Acta Math.} \textbf{214} (2015), no.~2, 307--355 (English).

\bibitem[Ram78]{Ramis-Devissage-Gevrey}
J.-P. Ramis, \emph{D\'evissage {G}evrey}, Journ\'ees {S}inguli\`eres de {D}ijon
  ({U}niv. {D}ijon, {D}ijon, 1978), Ast\'erisque, vol.~59, Soc. Math. France,
  Paris, 1978, pp.~4, 173--204. \MR{542737 (81g:34010)}

\bibitem[Rob75a]{Ro-I}
P.~Robba, \emph{On the index of {$p$}-adic differential operators. {I}}, Ann.
  of Math. (2) \textbf{101} (1975), 280--316. \MR{0364243 (51 \#498)}

\bibitem[Rob75b]{Robba-Hensel-original}
Philippe Robba, \emph{Lemme de {H}ensel pour les op\'erateurs diff\'erentiels},
  Groupe d'\'{E}tude d'{A}nalyse {U}ltram\'etrique, 2e ann\'ee (1974/75),
  {E}xp. {N}o. 16, Secr\'etariat Math\'ematique, Paris, 1975, D'apr{\`e}s un
  travail en commun avec B. Dwork (``On ordinary $p$-adic differential
  equations'', to appear), p.~11. \MR{0572968 (58 \#27989c)}

\bibitem[Rob80]{Robba-Hensel}
P.~Robba, \emph{Lemmes de {H}ensel pour les op\'erateurs diff\'erentiels.
  {A}pplication \`a la r\'eduction formelle des \'equations diff\'erentielles},
  Enseign. Math. (2) \textbf{26} (1980), no.~3-4, 279--311 (1981). \MR{610528
  (82k:12022)}

\bibitem[vdPS03]{VS}
M.~van~der Put and M.~F. Singer, \emph{Galois theory of linear differential
  equations}, vol. 328, Springer-Verlag, Berlin, 2003.

\bibitem[You92]{Young}
Paul~Thomas Young, \emph{Radii of convergence and index for {$p$}-adic
  differential operators}, Trans. Amer. Math. Soc. \textbf{333} (1992), no.~2,
  769--785. \MR{1066451 (92m:12015)}

\end{thebibliography}

\end{document}